\documentclass[a4paper,10pt]{article}
\usepackage[hypertexnames=false,hidelinks]{hyperref}
\usepackage{fullpage}
\usepackage{amsmath,amssymb,latexsym,amsthm,amsfonts}
\usepackage{mathabx}
\changenotsign
\usepackage[pdftex]{graphicx}
\usepackage[utf8]{inputenc}
\usepackage{centernot}
\usepackage{makeidx}
\usepackage[mathscr]{euscript}
\usepackage[columns=1]{idxlayout}

\newcommand{\R}{\mathbb{R}}
\newcommand{\C}{\mathbb{C}}
\newcommand{\N}{\mathbb{N}}

\renewcommand{\leq}{\leqslant}
\renewcommand{\geq}{\geqslant}
\renewcommand{\div}{\operatorname{div\,}}
\newcommand{\curl}{\operatorname{curl}}

\newcommand{\dist}{\operatorname{dist}}

\newcommand{\Lip}{\operatorname{Lip}}
\newcommand{\LL}{\operatorname{{\mathcal L}{\mathcal L}}}
\def\II{\parbox[][0.6cm][c]{0cm}{\ }}
\def\longrightharpoonup{\relbar\joinrel\rightharpoonup}
\def\medcup{\textstyle\bigcup}

\newcommand{\Supp}{\operatorname{Supp}}
\def\eop{\hfill $\Box$ \\ \ \par}
\newtheorem{Theorem}{Theorem}
\newtheorem{Definition}{Definition}[section]
\newtheorem{Corollary}[Definition]{Corollary}
\newtheorem{Proposition}[Definition]{Proposition}
\newtheorem{Lemma}[Definition]{Lemma}
\newtheorem{Remark}[Definition]{Remark}
\numberwithin{equation}{section}

\author{ 
Olivier Glass\footnote{CEREMADE, UMR CNRS 7534, Universit\'e Paris-Dauphine, PSL,
Place du Mar\'echal de Lattre de Tassigny, 75775 Paris Cedex 16, France},   
Franck Sueur\footnote{Institut de Math\'ematiques de Bordeaux, UMR CNRS 5251, Universit\'e de Bordeaux, 351 cours de la Lib\'eration, F33405 Talence Cedex, France  $\&$ Institut  Universitaire de France}}
\date{\today}
\makeindex
\begin{document}
\title{Dynamics of several rigid bodies in a two-dimensional ideal fluid and convergence to vortex systems}
\maketitle

\begin{abstract}
We consider the motion of several solids in a bounded cavity filled with a perfect incompressible fluid, in two dimensions. 
The solids move according to Newton's law, under the influence of the fluid's pressure. On the other hand the fluid dynamics is driven by the 2D incompressible Euler equations, which are set on the time-dependent domain corresponding to the cavity deprived of the sets occupied by the solids. We assume that the fluid vorticity is initially bounded and that the circulations around the solids may be non-zero. 
The existence of a unique corresponding solution, \textit{\`a la Yudovich}, to this system, up to a possible collision, follows from the arguments in
\cite{GS-!}.

The main result of this paper is to identify the limit dynamics of the system when the radius of some of the solids converge to zero, in different regimes, depending on how, for each body, the inertia is scaled with the radius. 
We obtain in the limit some point vortex systems for the solids converging to particles and a form of Newton's law for the solids that have a fixed radius; for the fluid we obtain an Euler-type system. 
This extends the earlier works 
\cite{GLS}, 
which deals with   the case of a single small heavy  body   immersed in an incompressible perfect fluid occupying the rest of the plane, 
\cite{GLS2},
which deals with   the case of a single small light body  immersed in an incompressible perfect fluid occupying the rest of the plane,  
and
\cite{GMS}
which deals with   the case of a single small, heavy or light,  body immersed in a irrotational incompressible perfect fluid occupying a bounded  plane domain. 

In particular we consider for the first time the case of several small rigid bodies, for which the strategy of the previous papers cannot be adapted straightforwardly, despite the partial results  recently obtained  in \cite{GLMS}. 
The main difficulty is to understand the interaction, through the fluid, between several moving solids. 
A crucial point of our strategy is the use of normal forms of the ODEs driving the motion of the solids in a two-steps process. 
First we use a normal form for the system coupling the time-evolution of all the solids to obtain a rough estimate of the acceleration of the bodies. 
Then we turn to some normal forms that are specific to each small solid, with an appropriate modulation related to the influence of the other solids and of the fluid vorticity. 
Thanks to these individual normal forms we obtain some precise uniform \textit{a priori} estimates of the velocities of the bodies, and then pass to the limit. 
In the course of this process we make use of another new main ingredient of this paper, which is an estimate of the fluid velocity with respect to the solids, uniformly with respect to their positions and radii, and which can be seen as an 
refinement of the reflection method for a div/curl system with prescribed circulations.

\end{abstract}

\newpage
\tableofcontents
\newpage
\printindex
\newpage
%
%

%
%
%
%
\section{Introduction and statement of the main result}
%
%
%
%
%
%
\subsection{The fluid-solid system}
The general situation that we describe is that of $N$ solids immersed in a bounded domain of the plane. The total domain (containing the fluid and the solids) is denoted by $\Omega$, that is a nonempty bounded open connected set in $\R^2$, with smooth boundary.  In the domain $\Omega$ are embedded $N$ solids ${\mathcal S}_{1}$, \dots, ${\mathcal S}_{N}$, which are nonempty, simply connected and closed sets with smooth boundaries.
To simplify, we assume that $\Omega$ is simply connected and that the solids ${\mathcal S}_{1}$, \dots, ${\mathcal S}_{N}$ are not discs (though the general case could be treated similarly).
 We will systematically suppose them to be at positive distance one from another and from the outer boundary $\partial \Omega$ during the whole time interval:
\begin{equation} \label{Eq:PositiveDistance}
\forall t, \ \ \forall \kappa \in \{1, \dots, N \}, \ \ 
{\mathcal S}_{\kappa}(t) \subset \Omega , \ \ 
\dist({\mathcal S}_{\kappa}(t), \partial \Omega) >0 \ \text{ and } \
\forall \lambda \neq \kappa, \  \dist({\mathcal S}_{\kappa}(t), {\mathcal S}_{\lambda}(t)) >0.
\end{equation}
Their positions depend on time, so we will denote them ${\mathcal S}_{1}(t)$, \dots, ${\mathcal S}_{N}(t)$. Since they are rigid bodies, each solid ${\mathcal S}_{\kappa}(t)$ is obtained through a rigid movement from ${\mathcal S}_{\kappa}(0)$. The rest of the domain, occupied by the fluid, will be denoted by ${\mathcal F}(t)$ so that
\index{T@Domains and admissible configurations!Q1a@$\Omega$: whole domain with fluid and solids}
\begin{equation*}
{\mathcal F}(t) = \Omega \,\setminus\, ({\mathcal S}_{1}(t) \cup \dots \cup {\mathcal S}_{N}(t)).
\end{equation*}
Let us now describe the dynamics of the fluid and of the solids. \par
\paragraph{Dynamics of the fluid.} The fluid is supposed to be inviscid and incompressible, and consequently driven by the incompressible Euler equation. We denote $u=u(t,x)$ the velocity field (with values in $\R^2$) and $\pi=\pi(t,x)$ the (scalar) pressure field, both  defined for $t$ in some time interval $[0,T]$ and $x \in {\mathcal F}(t)$. The incompressible Euler equation reads
\begin{equation} \label{Eq:Euler}
\left\{ \begin{array}{l}
\partial_{t} u + (u \cdot \nabla) u + \nabla \pi =0, \\
\div u = 0,
\end{array} \right. \ \text{ for } \ t \in [0,T], \ x \in {\mathcal F}(t).
\end{equation}
\index{Velocity fields!U0@$u$: fluid velocity field}
This equation is supplemented by boundary conditions which correspond to the {\it non-penetration condition}, precisely
\begin{equation} \label{Eq:NonPenetration}
u \cdot n = 0 \ \text{ on } \ \partial \Omega,
\ \text{ and } \ 
u \cdot n = v_{{\mathcal S},\kappa} \cdot n \ \text{ on } \ \partial {\mathcal S}_{\kappa} \text{ for } \kappa \in \{1,\dots,N\},
\end{equation}
where $n$ denotes the unit normal on $\partial {\mathcal F}(t)$ directed outside ${\mathcal F}(t)$ and $v_{{\mathcal S},\kappa}$ denotes the velocity field of the solid ${\mathcal S}_{\kappa}$. \par
Hence there is no difference with the classical situation, except the fact that the space-time domain is not cylindrical. \par
\paragraph{Dynamics of the solids.}
To describe the position of the $\kappa$-th solid ${\mathcal S}_{\kappa}$, we denote $h_{\kappa}$ and $\vartheta_{\kappa}$ the position of its center of mass and its angle with respect to its initial position. Correspondingly, the solid's position at time $t$ is obtained by the following rigid movement with respect to its initial position:
\begin{equation} \label{Eq:PositionSolide}
{\mathcal S}_{\kappa}(t) = h_{\kappa}(t) + R(\vartheta_{\kappa}(t)) ({\mathcal S}_{\kappa}(0) - h_{\kappa}(0)),
\end{equation}
where $R(\vartheta)$ is the linear rotation of angle $\vartheta$, that is
\begin{equation} \label{Eq:DefRTheta}
R(\vartheta) = \begin{pmatrix} \cos(\vartheta) & - \sin(\vartheta) \\ \sin(\vartheta) & \cos(\vartheta)
\end{pmatrix}.
\end{equation}
Note also that the velocity field of the solid ${\mathcal S}_{\kappa}$ mentioned in \eqref{Eq:NonPenetration} is given by
\begin{equation} \label{Eq:SolidVelocity}
v_{{\mathcal S},\kappa}(t,x) := h'_{\kappa}(t) + \vartheta'_{\kappa}(t) (x-h_{\kappa}(t))^\perp,
\index{Solid velocity!E4b@{$v_{{\mathcal S},\kappa}$}: $\kappa$-th solid velocity vector field}
\end{equation}
where $(x_{1},x_{2})^\perp:=(-x_{2},x_{1})$.
Now we denote the mass and momentum of inertia  of the solid ${\mathcal S}_{\kappa}$ by $m_{\kappa}$ and $J_{\kappa}$ respectively. 
\index{R@Solid position!D1@$h_{\kappa}$: position of the center of mass of the $\kappa$-th solid}
\index{R@Solid position!D2@$\vartheta_{\kappa}$: angle of the $\kappa$-th solid with respect to its initial position}
The assumption is that the solids evolve according to Newton's law, under the influence of the fluid's pressure on its boundary. Hence the equations of $h_{\kappa}$ and $\vartheta_{\kappa}$ read
\begin{equation} \label{Eq:Newton}
\left\{ \begin{array}{l}
\displaystyle
m_{\kappa} h_{\kappa}''(t) = \int_{\partial {\mathcal S}_{\kappa}(t)} \pi(t,x) n(t,x) \, ds(x), \\
\displaystyle
J_{\kappa} \vartheta_{\kappa}''(t) = \int_{\partial {\mathcal S}_{\kappa}(t)} \pi(t,x) (x-h_{k}(t))^\perp \cdot n(t,x) \, ds(x),
\end{array} \right. \ \text{ in } \ [0,T].
\end{equation}
\begin{Remark}
It could be possible to add some external forces such as gravity in the right hand side of \eqref{Eq:Newton}  with only minor modifications in the reasonings below. 
\end{Remark}
\paragraph{Initial conditions.} The system is supplemented with initial conditions:
\begin{itemize}
\item[--] At initial time the solids ${\mathcal S}_{1}$, \dots, ${\mathcal S}_{N}$ occupy the positions ${\mathcal S}_{1,0}$, \dots, ${\mathcal S}_{N,0}$ such that
\begin{equation} \label{Eq:CondSolidesInit}
\forall \kappa \in \{1, \dots, N \}, \ \ 
{\mathcal S}_{\kappa,0} \subset \Omega , \ \ 
\dist({\mathcal S}_{\kappa,0}, \partial \Omega) >0 \ \text{ and } \
\forall \lambda \neq \kappa, \  \dist({\mathcal S}_{\kappa,0}, {\mathcal S}_{\lambda,0}) >0.
\end{equation}
We introduce the initial values of the centers of masses $h_{1,0}$, \dots, $h_{N,0}$, and the angles $\vartheta_{1,0} = \dots = \vartheta_{N,0}=0$ (by convention), which characterize these positions. We denote ${\mathcal F}_{0}$ the corresponding initial fluid domain.
\item[--] The solids have initial velocities $(h'_{\kappa},\vartheta'_{\kappa})(0)= (h_{\kappa,0}',\vartheta_{\kappa,0}') \in \R^3$  for $\kappa \in \{1, \dots, N \}$,
\item[--] The circulations of velocity around the solids ${\mathcal S_{1}}$, \dots, ${\mathcal S}_{N}$, gathered as ${\boldsymbol\gamma}=(\gamma_{1},\dots,\gamma_{N})$, are given,
\item[--] We consider an initial vorticity $\omega_{0}  \in L^{\infty}({\mathcal F}_{0})$.
\end{itemize}
Note that this data is sufficient to reconstruct the initial velocity field $u_{0} \in C^0(\overline{{\mathcal F}}_{0};\R^2)$ in a unique way, see 
\eqref{Eq:DivCurlU}. In  particular  $ \curl u_0 = \omega_0$ and 
$ \oint_{\partial {\mathcal S}_{\nu}} u_0 \cdot \tau \, ds = \gamma_{\nu} $  for $ \nu =1, \dots, N$, where $\tau$ is the unit clockwise tangent vector field.
\par

\paragraph{Cauchy theory \`a la Yudovich.} The system \eqref{Eq:Euler}-\eqref{Eq:Newton} admits a suitable Cauchy theory in the spirit of Yudovich \cite{Yu}. Precisely, by a straightforward adaptation of  the arguments of \cite{GS-!}, we obtain the following result where initial conditions are given, as described above. 
\begin{Theorem}
\label{Th:Yudovich}
Given the initial conditions above, there is a unique maximal solution $(h_{1},\vartheta_{1}, \dots, h_{N}, \vartheta_{N}, u)$ in the space $C^2([0,T^*))^{3N} \times [L^\infty_\text{loc}( [0,T^*) ;\LL({\mathcal F}(t);\R^2)) \cap C^0([0,T^*);W^{1,q}({\mathcal F}(t);\R^2))]$ (for all $q$ in $[1,+\infty)$) of System \eqref{Eq:Euler}--\eqref{Eq:Newton} for some $T^* >0$. Moreover, as $  t \longrightarrow T^{*}$, 
\begin{multline*}
\min \Big\{  \min \big( \dist({\mathcal S}_{\kappa}(t), \partial \Omega) , \kappa \in \{1,\dots,N\} \big) , 
\min \big( \dist({\mathcal S}_{\kappa}(t), {\mathcal S}_{\lambda}(t)), \kappa, \lambda \in \{1,\dots,N\}, \ \lambda \neq \kappa \big)
\Big\} \longrightarrow 0.
\end{multline*}
Finally  the velocity circulations around the solids ${\boldsymbol\gamma}=(\gamma_{1},\dots,\gamma_{N})$ are constant in time. 
\end{Theorem}
Above, $\LL({\mathcal F}(t);\R^2)$ stands for the space of log-Lipschitz vector fields on ${\mathcal F}(t)$; we recall that $\LL(X)$ that is the space of functions $f \in L^{\infty}(X)$ such that
\begin{equation*}
\| f \|_{\mathcal{LL}(X)} := \| f\|_{L^{\infty}(X)} + \sup_{x\not = y} \frac{|f(x)-f(y)|}{|(x-y)(1+ \ln^{-}|x-y|)|} < +\infty.
\end{equation*}
Also we used the slightly abusive notation $L^\infty(0,T;\LL({\mathcal F}(t);\R^2))$: it describes the space of functions defined for almost all $t$, with values for such $t$ in ${\mathcal{LL}}({\mathcal F}(t))$, with a uniform log-Lipschitz norm. We will quite systematically use such notations from the cylindrical case to describe our situation. There should be no ambiguity coming from this abuse of notation. \par
Theorem~\ref{Th:Yudovich} indicates in particular that the lifespan of the solutions is only limited by a possible collision between solids or between a solid and the boundary.
Regarding  the issue of collisions  we refer to \cite{H}, \cite{HM} and the recent paper \cite{C}.
%
%
%
%
%
%
%
%
%
%

\subsection{The problem of small solids}
\label{sec-small}
The main question raised by this paper is to determine a limit system when some of the solids ${\mathcal S}_{1}$, \dots, ${\mathcal S}_{N}$ shrink to a point. To describe this problem, we will denote the scale of the $\kappa$-th solid by $\varepsilon_{\kappa}$ and suppose that the $\kappa$-th solid ${\mathcal S}_{\kappa}$ is obtained initially by applying a homothety of ratio $\varepsilon_{\kappa}$ and center $h_{\kappa,0}$ on the solid of fixed size ${\mathcal S}_{\kappa,0}^1$:
\begin{equation} \label{Eq:InitSolidSmall}
{\mathcal S}_{\kappa,0}^\varepsilon = h_{\kappa,0} + \varepsilon_{\kappa} \big( {\mathcal S}_{\kappa,0}^1 - h_{\kappa,0} \big).
\end{equation}
\index{T@Domains and admissible configurations!Q1b@${\mathcal S}^\varepsilon_{\kappa}$: $\kappa$-th solid}
\index{Characteristics of the solids!C1@$\varepsilon_{\kappa}$: scale factor for the $\kappa$-th solid}
\paragraph{The three sets of solids.}
Now let us be more specific about the indices $\kappa$. The set of indices $\{ 1, \dots, N \}$ is split in three:
\begin{gather*}
\{ 1, \dots, N \} = {\mathcal P}_{(i)} \cup {\mathcal P}_{(ii)} \cup {\mathcal P}_{(iii)} \ \text{ with } \ \\
\index{I@Indices!I1@${\mathcal P}_{(i)}$: set of indices of large solids}
\index{I@Indices!I2@${\mathcal P}_{(ii)}$: set of indices of small massive solids}
\index{I@Indices!I3@${\mathcal P}_{(iii)}$: set of indices of small light solids}
{\mathcal P}_{(i)}:= \{ 1,\dots,N_{(i)} \}, \ 
{\mathcal P}_{(ii)}:= \{ N_{(i)}+1,\dots,N_{(i)}+N_{(ii)} \}, \ 
{\mathcal P}_{(iii)}:= \{ N_{(i)}+N_{(ii)}+1,\dots,N \},
\end{gather*}
corresponding respectively to the solids: 
\begin{itemize}
\item (i) of fixed size and inertia:
\begin{equation} \label{Eq:Family_i}
\text{for } \kappa \in {\mathcal P}_{(i)}, \ \ 
\varepsilon_{\kappa}=1, \ \ m_{\kappa}^\varepsilon=m^1_{\kappa}, \ \ J_{\kappa}^\varepsilon=J^1_{\kappa},
\end{equation}
\item (ii) of size going to zero but with fixed mass:
\begin{equation} \label{Eq:Family_ii}
\text{for } \kappa \in {\mathcal P}_{(ii)}, \ \ 
\varepsilon_{\kappa} \rightarrow 0^+, \ \ m_{\kappa}^\varepsilon=m^1_{\kappa}, \ \ J_{\kappa}^\varepsilon= \varepsilon_{\kappa}^2 J^1_{\kappa},
\end{equation}
\item (iii) of size and mass converging to zero:
\begin{equation} \label{Eq:Family_iii}
\text{for } \kappa \in {\mathcal P}_{(iii)}, \ \ 
\varepsilon_{\kappa} \rightarrow 0^+, \ \ m_{\kappa}^\varepsilon= \varepsilon_{\kappa}^{\alpha_{\kappa}} m^1_{\kappa}, \ \ J_{\kappa}^\varepsilon= \varepsilon_{\kappa}^{\alpha_{\kappa} + 2} J^1_{\kappa}
\ \ \text{ for some } \alpha_{\kappa}>0.
\end{equation}
\end{itemize}
\begin{Remark}
Case (iii)   encompasses the case of  fixed density, for which $\alpha_{\kappa} =2$. This is actually the main motivation for the difference in the scaling of $m_{\kappa}^\varepsilon$ and $J_{\kappa}^\varepsilon$.
\end{Remark}
It will be useful to consider the indices corresponding to small solids (here $s$ stands for small):
\begin{equation} \label{Eq:PS}
{\mathcal P}_{s} := {\mathcal P}_{(ii)} \cup {\mathcal P}_{(iii)}=\{N_{(i)}+1, \dots, N \}, \ \ N_s :=   N_{(ii)} + N_{(iii)} .
\index{I@Indices!I4@${\mathcal P}_{s}$: set of indices of all small solids, that is ${\mathcal P}_{(ii)} \cup {\mathcal P}_{(iii)}$}
\index{I@Indices!I5@$N_{(i)}, N_{(ii)}, N_{(iii)}, N_{s}$: cardinals of the sets of indices}
\end{equation}
We collect the various $\varepsilon_{\kappa}$ as follows:
\begin{equation*}
{\boldsymbol\varepsilon} = (\varepsilon_{1},\dots,\varepsilon_{N}),  \ \text{ and } \ \overline{\boldsymbol\varepsilon} = (\varepsilon_{N_{(i)}+1},\dots,\varepsilon_{N}).
\index{Characteristics of the solids!C2@$\boldsymbol{\varepsilon}$: scale factors for all solids at once}
\index{Characteristics of the solids!C3@$\overline{\boldsymbol{\varepsilon}}$: scale factors for all shrinking solids at once}
\end{equation*}
The total size of small solids will be denoted as follows
\begin{equation} \label{Eq:TailleTotalePetitsSolides}
|\overline{\boldsymbol\varepsilon}|:= \sum_{\kappa \in {\mathcal P}_{s}} \varepsilon_{\kappa}.
\index{Characteristics of the solids!C4@$\vert\overline{\boldsymbol{\varepsilon}}\vert$: total size of small solids}
\end{equation}
For $\varepsilon_{0} >0$, we will write $\overline{\boldsymbol\varepsilon} < \varepsilon_{0}$ or $\overline{\boldsymbol\varepsilon} \leq \varepsilon_{0}$ to express that the inequality is valid for each coordinate. \par
We assume, for any $\kappa$ in ${\mathcal P}_{s}$, that $h_{\kappa,0} $ is in $\Omega$ so that ${\mathcal S}_{\kappa,0}^\varepsilon  \subset \Omega$ for $\varepsilon_{\kappa}$ small enough. Up to a redefinition of ${\mathcal S}_{\kappa,0}^1$ we may assume that 
\begin{equation} \label{Eq:ToutDansOmega}
{\mathcal S}_{\kappa,0}^\varepsilon  \subset \Omega \ \text{ for all } \varepsilon_{\kappa} \leq 1. 
\end{equation}

\paragraph{Description of the position of the solids.} Grouping the positions of the center of mass and angles together, we denote the position variable as follows:
\begin{equation*}
q_{\kappa} = (h_{\kappa},\vartheta_{\kappa})^T \ \text{ and } \ {\mathbf q} = (q_{1},\dots,q_{N}).
\index{R@Solid position!D3@$q_{\kappa}$: position and angle of the $\kappa$-th solid}
\index{R@Solid position!D4@$\mathbf{q}$: positions and angles of all the solids at once}
\end{equation*}
It follows that the $\kappa$-th solid is determined by $q_{\kappa}$ and $\varepsilon_{\kappa}$; we will denote it by ${\mathcal S}_{\kappa}(\varepsilon_{\kappa}, q_{\kappa})$, or in a simpler manner ${\mathcal S}^{\varepsilon}_{\kappa}(q_{\kappa})$. Moreover when it does not play an important role in the discussion or when it is clear, we will drop the exponent $\varepsilon$ and/or the dependence on $q_{\kappa}$ to lighten the notations. \par
When one considers only the non-shrinking solids, it is useful to introduce
\begin{equation*}
{\mathbf q}_{(i)} = (q_{1},\dots,q_{N_{(i)}}).
\index{R@Solid position!D5@$\mathbf{q}_{(i)}$: positions and angles of all the final solids}
\end{equation*}
\paragraph{Fluid domains.}
Corresponding to the above notations, the fluid domain is 
\begin{equation*}
{\mathcal F}^\varepsilon(q) =\Omega \setminus ({\mathcal S}^{\varepsilon}_{1}(q_{1}) \cup \dots \cup {\mathcal S}^{\varepsilon}_{N}(q_{N})). 
\index{T@Domains and admissible configurations!Q3@${\mathcal F}(q)$: fluid domain}
\end{equation*}
When the small solids have disappeared, it remains merely the final domain
\begin{equation} \label{Eq:AugmentedDomain}
\widecheck{\mathcal F}({\bf q}_{(i)}) =\Omega \setminus ({\mathcal S}_{1}(q_{1}) \cup \dots \cup {\mathcal S}_{N_{(i)}}(q_{N_{(i)}})). 
\index{T@Domains and admissible configurations!Q4@$\widecheck{\mathcal F}(q_{(i)})$: final fluid domain}
\end{equation}

\paragraph{Initial conditions.}
 We consider the initial vorticity $\omega_{0}$, the circulations around the solids ${\boldsymbol\gamma}=(\gamma_{1},\dots,\gamma_{N})$, the initial solid positions
${\boldsymbol q}^{0}=(q_{1,0},\dots,q_{N,0}) = (h_{1}^{0},0, \dots, h_{N}^{0},0)$ 
and the initial solid velocities ${\boldsymbol p}^{0}=(p_{1,0},\dots,p_{N,0}) = (h_{1,0}',\vartheta_{1,0}', \dots, h_{N,0}',\vartheta_{N,0}')$ fixed independently of $\varepsilon$. 
Moreover we assume that $\gamma_{\kappa} \neq 0$ when $\kappa \in {\mathcal P}_{(iii)}$. 
 \par
To be more precise on the vorticity, we set for $\delta>0$ the space $L_{c,\delta}^\infty({\mathcal F})$ of essentially bounded functions $f$ satisfying that for almost all $x \in {\mathcal F}(q)$ such that $d(x,{\mathcal S}_{\kappa}) \leq \delta$ for some $\kappa \in {\mathcal P}_{s}$, one has $f(x)=0$.
Now we suppose that 
\begin{equation*}
\omega_{0} \in L^\infty_{c} \big( \overline{\Omega} \setminus [{\mathcal S}_{1,0} \cup \dots \cup {\mathcal S}_{N_{(i)},0} \cup \{ h^0_{j}, \ j \in {\mathcal P}_{s} \} ] \big).
\end{equation*}
Hence for some $\delta>0$ and for suitably small $\overline{\varepsilon}$, one has $\omega_{0} \in L_{c,\delta}^\infty({\mathcal F}_{0})$. \par
We are now in position to state our main result.
%
%
%
%
%
%

%
\subsection{Main result}
We first introduce a convention.
To express convergences in domains that actually depend on the solutions themselves, we will take the convention to extend the vorticity $\omega$ and the velocity $u$ (defined in ${\mathcal F}(t)$) by $0$ inside ${\mathcal S}_{1}, \dots, {\mathcal S}_{N}$. In the same way, the limit vorticity and velocity (defined in $\widecheck{\mathcal F}(t)$) are extended by $0$ inside ${\mathcal S}_{1}, \dots, {\mathcal S}_{N_{(i)}}$ as well. \par
Our main result is as follows. 
\begin{Theorem}
\label{Joli_Theoreme}
Under the above assumptions there exists $\varepsilon_{0}>0$ and some $T>0$ such that the following holds. To each family $\boldsymbol{\varepsilon}$ of scale factors with $\overline{\boldsymbol{\varepsilon}} \leq \varepsilon_{0}$ we associate the corresponding maximal solution $(q^\varepsilon,u^\varepsilon)$ on $[0,T^{\boldsymbol{\varepsilon}})$ given by Theorem~\ref{Th:Yudovich}. Then the maximal existence times $T^{\boldsymbol{\varepsilon}}$ satisfy $T^{\boldsymbol{\varepsilon}} \geq T$ and, as $\overline{\boldsymbol{\varepsilon}} \rightarrow 0^+$, up to a subsequence, one has
\begin{gather}
\label{Eq:CVu}
u^\varepsilon \longrightarrow u^\star \ \text{ in } \ C^0([0,T];L^q(\Omega)) \text{ for } q \in [1,2), \\
\label{Eq:CVomega}
\omega^\varepsilon \longrightarrow \omega^\star \ \text{ in } \ C^0([0,T];L^\infty(\Omega)-w\star), \\
\label{Eq:CVh}
h_{\kappa}^\varepsilon \longrightarrow h_{\kappa}^\star \ \text{ in } \left\{ \begin{array}{l}
W^{2,\infty}(0,T) \ \text{weak}-\star \text{ for } \kappa \in {\mathcal P}_{(i)} \cup {\mathcal P}_{(ii)}, \\
W^{1,\infty}(0,T) \ \text{weak}-\star \text{ for } \kappa \in {\mathcal P}_{(iii)},
\end{array} \right. \\
\label{Eq:CVtheta}
\vartheta_{\kappa}^\varepsilon \longrightarrow \vartheta_{\kappa}^\star \ \text{ in } W^{2,\infty}(0,T) \ \text{weak}-\star \text{ for } \kappa \in {\mathcal P}_{(i)},
\end{gather}
and at the limit the following system holds in the final domain:
\begin{equation} \label{Eq:DefUStar}
\left\{ \begin{array}{l}
\div u^{\star} = 0 \ \text{ in } \ \widecheck{{\mathcal F}}({\boldsymbol{q}}^\star_{(i)}), \medskip \\
\displaystyle \curl u^{\star} = \omega^\star + \sum_{\kappa \in {\mathcal P}_{s}} \gamma_{\kappa} \delta_{h^\star_{\kappa}} \ \text{ in } \ \widecheck{{\mathcal F}}({\boldsymbol{q}}^\star_{(i)}), \medskip \\
u^{\star}\cdot n = \big[ (h_{\kappa}^\star)' + (\vartheta_{\kappa}^\star)'(x-h_{\kappa}^\star)^\perp \big] \cdot n \ \text{ on } \ \partial {\mathcal S}_{\kappa}(q^\star_{\kappa}) \ \text{ for } \ \kappa \in {\mathcal P}_{(i)}, \medskip \\
u^{\star}\cdot n = 0 \ \text{ on } \ \partial \Omega, \medskip \\
\displaystyle \oint_{\partial {\mathcal S}_{\kappa}(q^\star_{\kappa})} u^{\star} \cdot \tau \, ds = \gamma_{\kappa}  \ \text{ for } \ \kappa \in {\mathcal P}_{(i)},
\end{array} \right.
\end{equation}
where $q^\star_{\kappa} = (h_{\kappa},\vartheta_{\kappa})^T $ and ${\mathbf q}^\star_{(i)} = (q^\star_{1},\dots,q^\star_{N_{(i)}})$, 
\begin{equation} \label{Eq:EvolOmegaStar}
\partial_{t} \omega^\star + \div (u^{\star} \omega^\star) = 0 \ \text{ in } \ [0,T] \times \widecheck{\mathcal F}( {\mathbf q}^\star_{(i)}(t)),
\end{equation}
\begin{multline} \label{Eq:DefPStar}
\text{ for all } t \in [0,t], \ \ 
-(\partial_{t} u^\star + (u^\star \cdot \nabla) u^\star ) \text{ is a gradient in } \widecheck{\mathcal F}(q^\star_{(i)}(t)) \, \setminus \, \{h^\star_{\kappa}(t), \ \kappa \in {\mathcal P}_{s} \}, \\
\text{regular in the neighborhood of } \bigcup_{\kappa=1}^{N_{(i)}} \partial {\mathcal S}_{\kappa}(q^\star_{\kappa}), \ 
 \text{ which we denote } \nabla \pi^\star, 
\end{multline}
\begin{equation} \label{Eq:Evolh(i)Star}
\left\{ \begin{array}{l}
\displaystyle
m_{\kappa} (h_{\kappa}^\star)''(t) = \int_{\partial {\mathcal S}_{\kappa} (q^\star_{\kappa})} \pi^\star(t,x) n(t,x) \, ds(x), \\
\displaystyle
J_{\kappa} (\vartheta^\star_{\kappa})''(t) = \int_{\partial {\mathcal S}_{\kappa}(q^\star_{\kappa}) } \pi^\star(t,x) (x-h^\star_{\kappa}(t))^\perp \cdot n(t,x) \, ds(x),
\end{array} \right. \text{ in } \ [0,T] \ \ \text{ for } \kappa \in {\mathcal P}_{(i)}, \ \ 
\end{equation}
\begin{gather}
\label{Eq:Evolh(ii)Star}
m_{\kappa} (h^\star_{\kappa})''= \gamma_{\kappa} \big[ (h^\star_{\kappa})' - u^{\star}_\kappa (t,h^\star_{\kappa}) \big]^\perp \ \text{ in } \ [0,T] \ \ \text{ for } \  \kappa \in {\mathcal P}_{(ii)}, \medskip \\
\label{Eq:Evolh(iii)Star}
(h^\star_{\kappa})'= u^{\star}_\kappa (t,h^\star_{\kappa}) \ \text{ in } \ [0,T] \ \ \text{ for } \  \kappa \in {\mathcal P}_{(iii)},
\end{gather}
where $u^\star_{\kappa}$ is the ``desingularized version'' of $u^\star$ at $h^\star_{\kappa}$ defined by
\begin{equation} \label{DefUellKappa}
u^{\star}_\kappa(t,x) = u^{\star}(t,x) - \frac{\gamma_{\kappa}}{2\pi} \frac{(x-h_{\kappa}^\star(t))^\perp}{|x-h^\star_{\kappa}(t)|^2},
\ \ t \in [0,T], \ \ x \in \widecheck{\mathcal F}( {\mathbf q}^\star_{(i)}(t)).
\end{equation}
\end{Theorem}

\paragraph{On the limit system.}
Theorem \ref{Joli_Theoreme} identifies the limit dynamics of a family of solutions
of the system \eqref{Eq:Euler}-\eqref{Eq:Newton}, when some of the solids  shrink to points,
as a  system compound of the Euler-type system \eqref{Eq:DefUStar}-\eqref{Eq:EvolOmegaStar}   for the fluid,
the Newton's laws \eqref{Eq:Evolh(i)Star} for the solids that have a fixed radius and 
the point vortex systems  \eqref{Eq:Evolh(ii)Star}-\eqref{Eq:Evolh(iii)Star} for the limit point particles.  
The interest of \eqref{Eq:DefPStar} is to give a meaning for the trace of the limit fluid pressure $\pi^\star$ on the boundary of the solids that have a fixed radius; this gives a sense to the right hand sides in \eqref{Eq:Evolh(i)Star}.
Regarding the solids with a vanishing radius the limit equation is not the same in case (ii)  and in case (iii),
as we can see in \eqref{Eq:Evolh(ii)Star}-\eqref{Eq:Evolh(iii)Star}.
A common feature is that the limit equation is independent of the shape of the rigid body which has shrunk.\footnote{However let us recall that we assume that the  solids ${\mathcal S}_{1}$, \dots, ${\mathcal S}_{N}$ are not discs. 
 The case  of a disk is peculiar as  several degeneracies appear in this case. 
 We refer to \cite{GMS} for a complete treatment of this case  for a single small body of type   (ii)  or  (iii)  immersed in a irrotational incompressible perfect fluid occupying the full plane or  a bounded  plane domain; in particular it is shown that the case of a homogeneous disk is rather simple whereas the case of a non-homogeneous disk requires appropriate modifications.}
 
In case (ii) the rigid body reduces at  the limit  in a point-mass particle which satisfies the second order differential equation  \eqref{Eq:Evolh(ii)Star}. 
This type of systems has already been discussed  by Friedrichs in  \cite[Chapter 3]{Fr}, see also  \cite{GRKO}. 
The force in the right hand side of  \eqref{Eq:Evolh(ii)Star} extends  the classical Kutta-Joukowski force, as it is a gyroscopic force orthogonally proportional to its relative velocity and proportional to the circulation around the body. 
The Kutta-Joukowski-type lift force was originally studied in the case of a single body in a irrotational unbounded flow at the beginning of the 20th century in the course of the first mathematical investigations of aeronautics; see for example \cite{Lamb}. 
 
In case (iii) the rigid body reduces at  the limit  in a massless point particle  which satisfies the first order differential equation \eqref{Eq:Evolh(iii)Star}, which can be seen as a classical point vortex equation, its vortex strength being given by the circulation around the rigid body. 
Historically the point vortex system, which dates back to Helmholtz, Kirchhoff, Kelvin and Poincar\'e,  has been seen as a simplification of the 
the 2D incompressible Euler equations when the vorticity of the fluid is concentrated in a finite number of points, see for instance \cite{MP}.
The key feature of the derivation of the point vortex equations from the 2D incompressible Euler equations 
is that the self-interaction has to be discarded. 
Theorem \ref{Joli_Theoreme} proves that such equations can also be obtained as the limit of the dynamics of rigid bodies of type (iii).
The desingularization of the   background fluid velocity $u^\star$ mentioned in  \eqref{DefUellKappa} precisely corresponds to the cancellation of the self-interaction. 
 
On the other hand the genuine fluid vorticity $\omega^\star$ is convected by  the background fluid velocity $ u^\star$, according to \eqref{Eq:EvolOmegaStar}. 
A precise decomposition of the velocity field $u^\star$ obtained in the limit will be given below, see \eqref{Eq:DecompUstar}. 
Systems mixing an evolution equation for absolutely continuous vorticity such as  \eqref{Eq:EvolOmegaStar} and some evolution equations for point vortices such as  \eqref{Eq:Evolh(iii)Star} have been coined as  vortex-wave systems  by Marchioro and Pulvirenti  in the early 90s, see \cite{MP}.
\paragraph{On the lifespan, on the convergences, and on the uniqueness.}
Observe that the existence of a common lifetime for a subsequence $\overline{\boldsymbol{\varepsilon}} \rightarrow 0^+$ is a part of the result, as Theorem~\ref{Th:Yudovich} does not provide any quantitative information on the existence times $T^{\boldsymbol{\varepsilon}}$ before collisions. 
 
Let us also stress that the convergences in \eqref{Eq:CVh} are different depending on whether the rigid body has a positive mass in the limit or not. 
Indeed the weaker convergence obtained in Case (iii) is associated with the degeneracy of the solid dynamics into a first order equation. 
Except for some well-prepared initial data the convergence is indeed limited to the weak-$\star$ topology of $W^{1,\infty}(0,T)$. 
We refer here to \cite{benyo} for partial results regarding multi-scale features of the time-evolution of some toy models of the limit system above which attempts to give more insight on this issue. 
The issue of the uniqueness of the solution to the limit system and the associated issue of the convergence of the whole sequence, not only a subsequence, is a delicate matter. We refer to 
\cite{MP1,MP,LM-old} for some positive results concerning the vortex-wave system with massless point vortices (the system occupying the whole plane). In the case of several massive point vortices, we refer to the recent work \cite{LM} which gives results when the initial vorticity is bounded, compactly supported and locally constant in a neighborhood of the point vortices. A key ingredient in all these uniqueness results is that the point vortices stay away one from another and remain distant from the support of the vorticity (or at least, that the vorticity remains constant in their neighborhood.) \par
In the particular cases where uniqueness holds and the point vortices and the vorticity remain distant, we can improve a bit the statement of Theorem~\ref{Joli_Theoreme} into the following one.
\begin{Theorem} \label{VarianteDuJoliTheoreme}
Suppose the assumptions of Theorem~\ref{Joli_Theoreme} to be satisfied, and suppose moreover that for this data the limit system \eqref{Eq:DefUStar}-\eqref{DefUellKappa} admits a unique solution in $[0,T^\star)$ (of class $W_{loc}^{2,\infty}([0,T^*))$ for the solids and the massive point vortices, $W_{loc}^{1,\infty}([0,T^*))$ for the massless point vortices, and $C^{0}([0,T^*); L^\infty(\Omega)-w\star)$ for the vorticity) for which for all $t \in [0,T^\star)$, the point vortices and the large solids do not meet one another and do not meet the support of vorticity nor the outer boundary. Then the maximal existence times $T^{\boldsymbol{\varepsilon}}$ satisfy $\liminf_{\overline{\boldsymbol{\varepsilon}} \rightarrow 0} T^{\boldsymbol{\varepsilon}} \geq T^\star$, and the convergences \eqref{Eq:CVu}-\eqref{Eq:CVtheta} hold on any time interval $[0,T] \subset [0,T^\star)$ and are valid without restriction to a subsequence.
\end{Theorem}
\paragraph{On the relationships with earlier results.}
Theorem \ref{Joli_Theoreme} extends results obtained in the earlier works
\cite{GLS}, which deals with the case of a single small body of type (ii)  immersed in an incompressible perfect fluid occupying the rest of the plane, 
\cite{GLS2}, which deals with the case of a single small body of type (iii)  immersed in an incompressible perfect fluid occupying the rest of the plane, and
\cite{GMS}, which deals with   the case of a single small body of type  (ii)  or (iii)  immersed in an irrotational incompressible perfect fluid occupying a bounded  plane domain. 
In particular we consider for the first time the case of several small rigid bodies, for which the strategies of the previous papers cannot be adapted straightforwardly, 
despite the results  recently obtained  in \cite{GLMS} in the case of several rigid bodies of type  (i).
Indeed the main difficulty is to understand the influence of solids between themselves, and to analyze it to understand how the coupling at leading order disappears in the limit. This is made more difficult by the fact that each solid possesses its own scale.

\paragraph{On the relationships with the case of the Navier-Stokes equations.}
Let us mention that the Euler system is a rough modeling for a fluid in a neighborhood of  rigid boundaries as  even a slight amount of viscosity may drastically change the behavior of the fluid close to the boundary, due to boundary layers, and sometimes even in the bulk of the fluid when the boundary layers detach from the boundary. 
While the Navier-Stokes equations certainly represent a better choice in terms of modeling, it is certainly useful
to first understand the case of the Euler equations. 
In this direction let us mention that Gallay has proven in \cite{Gallay} that the point vortex system can
also be obtained as vanishing viscosity limits of concentrated smooth vortices
driven by the incompressible Navier-Stokes equations, see also the recent extension to vortex-wave systems in \cite{NN}.

%
%
%
%
%
%
%
%
%
%
\section{Preliminaries}
In this section, we introduce some notations and basic tools that are needed in the sequel. Then we describe briefly the proof and the organization of the rest of the paper. \par
%
%
%
%
%
%
%
%
\subsection{Solid variables and configuration spaces}
Below we introduce notations for the solid velocities and for the admissible configurations of the location of the solids and of the support of the vorticity.
\paragraph{Solid velocities.}
The solid velocities will be denoted as follows:
\begin{equation} \label{Eq:PPChapeau}
p_{\kappa} = (h'_{\kappa},\vartheta'_{\kappa})^T, \ \ \widehat{p}_{\kappa} = (h'_{\kappa},\varepsilon_{\kappa} \vartheta'_{\kappa})^T, \ \
{\mathbf p} = (p_{1},\dots,p_{N}) \text{ and } \ \widehat{\mathbf p} = (\widehat{p}_{1},\dots,\widehat{p}_{N}).
\index{Solid velocity!E1@$p_{\kappa}$: linear and angular velocity of the $\kappa$-th solid}
\index{Solid velocity!E2@$\mathbf{p}$: linear and angular velocity of all solids at once}
\index{Solid velocity!E3@\unexpanded{\unexpanded{\ensuremath{\widehat{p}_{\kappa}}}}: linear and scaled angular velocity of the $\kappa$-th solid}
\index{Solid velocity!E4@\unexpanded{\unexpanded{\ensuremath{\widehat{\mathbf{p}}}}}: linear and scaled angular velocity of all solids at once}
\end{equation}
For $i \in \{1,2,3\}$, $p_{\kappa,i}$ denotes the $i$-th coordinate of $p_{\kappa}$. In terms of these coordinates, \eqref{Eq:SolidVelocity} reads as follows
\begin{equation} \label{Eq:VitSolideXi}
v_{{\mathcal S},\kappa}(t,x) = \sum_{i=1}^3 p_{\kappa,i} \xi_{\kappa,i},
\end{equation}
with $\xi_{\kappa,i}=e_{i}$ for $i=1,2$ and $\xi_{\kappa,3} = (x-h_{\kappa})^\perp$ on $\partial {\mathcal S}_\kappa$
(this anticipates the notation \eqref{def-xi-j}). Above $e_1$ and $e_2$ are the unit vectors of the canonical basis. \par

\paragraph{Admissible configurations.} We introduce notations for the spaces of configuration of the solids which can also possibly incorporate  the configuration for the vorticity. Given $\delta > 0$, we let 
\begin{multline} \label{MathfrakQdelta}
{\mathcal Q}_{\delta} := \{(\overline{\boldsymbol\varepsilon}, {\bf q}) \in (0,1)^{N_s} \times \mathbb R^{3N}  \ :\\
\forall \nu,\mu \in \{ 1, \dots, N \} \, \text{ s.t. } \nu \neq \mu, \  d({\mathcal S}^\varepsilon_\mu({\bf q}),{\mathcal S}^\varepsilon_\nu({\bf q})) > 2\delta
\text{ and } d(\mathcal S^\varepsilon_\mu({\bf q}),\partial\Omega)> 2\delta \}.
\index{T@Domains and admissible configurations!Q5b@$\mathcal{Q}_{\delta}$: bundle of shrinking bodies positions at distance $>2\delta$ one from another}
\end{multline}
\begin{multline} \label{MathfrakQWdelta}
\mathfrak{Q}_{\delta} := \{(\overline{\boldsymbol\varepsilon}, {\bf q},\omega) \in (0,1)^{N_s} \times \mathbb R^{3N} \times L^{\infty}(\Omega)  \ : (\overline{\boldsymbol\varepsilon}, {\bf q}) \in {\mathcal Q}_{\delta} \text{ and } \\
\forall \mu \in \{ 1, \dots, N \}, \ d(\mathcal S^\varepsilon_\mu({\bf q}),\mbox{Supp}(\omega))> 2\delta
  \}.
\index{T@Domains and admissible configurations!Q5c@$\mathfrak{Q}_{\delta}$: bundle of shrinking bodies positions and vorticity at distance $>2\delta$ one from another}
\end{multline}
Given $\varepsilon_{0}>0$, we refine the above sets by limiting the size of small solids as follows
\begin{equation} \label{Qeps}
\mathcal{Q}^{\varepsilon_{0}}_{\delta} := \{(\overline{\boldsymbol\varepsilon}, {\bf q}) \in \mathcal{Q}_{\delta} \ / \ \overline{\boldsymbol\varepsilon} < \varepsilon_{0} \} \ \text{ and } \
\mathfrak{Q}^{\varepsilon_{0}}_{\delta} := \{(\overline{\boldsymbol\varepsilon}, {\bf q},\omega) \in \mathfrak{Q}_{\delta} \ / \ \overline{\boldsymbol\varepsilon} < \varepsilon_{0} \} ,
\index{T@Domains and admissible configurations!Q6a@$\mathcal{Q}^{\varepsilon_{0}}_{\delta}$:  $\mathcal{Q}_{\delta}$ with small solids of size less than $\varepsilon_{0}$}
\index{T@Domains and admissible configurations!Q6b@$\mathfrak{Q}^{\varepsilon_{0}}_{\delta}$: $\mathfrak{Q}_{\delta}$ with small solids of size less than $\varepsilon_{0}$}
\end{equation}
where as before $\overline{\boldsymbol\varepsilon} < \varepsilon_{0}$ expresses that $\varepsilon_{i} < \varepsilon_{0}$ for all $i \in {\mathcal P}_{s}$. \par
\paragraph{$\boldsymbol\nu$-neighborhoods in $\Omega$.}
In many situations, it will be helpful to consider some neighborhoods of the solids or of their boundaries; we therefore denote for $A \subset \Omega$ and $\nu >0$:
\begin{equation} \label{Def:NuVoisinage}
{\mathcal V}_{\nu}(A):= \{ x \in \Omega \ / \ d(x,A) < \nu \}.
\index{X@Miscellaneous!M0@${\mathcal V}_{\nu}(A)$: $\nu$-neighborhood of $A \subset \Omega$}
\end{equation}
For instance the above conditions for ${\mathcal Q}_{\delta}$ can be rephrased in the form ${\mathcal V}_{\delta}({\mathcal S}^\varepsilon_\mu({\bf q})) \cap {\mathcal V}_{\delta}({\mathcal S}^\varepsilon_\nu({\bf q})) = \emptyset$ and so on.

%
%
%
%
%
%
%
\subsection{Potentials and decomposition of the fluid velocity}
\label{Subsec:Potentials}
Below we first recall the definition of the so-called Kirchhoff potentials and the associated notion of added inertia. Then we introduce the stream functions for the circulation terms, the hydrodynamic Biot-Savart  operator and we finally conclude by recalling the standard decomposition of the velocity field in terms of vorticity, solid velocities and circulations. 
\paragraph{The Kirchhoff potentials.}
First, for $\kappa \in \{1, \ldots , N \}$ and $j\in \{1,\dots,5\}$ 
we introduce the function $\xi_{ \kappa,j} ({\bf q},\cdot): \partial\mathcal{F}({\bf q}) \rightarrow \R^2$ as follows:
\begin{align} \nonumber
&\text{on } \ \partial\mathcal{F}({\bf q}) \setminus \partial \mathcal{S}_{\kappa}, \ \xi_{ \kappa,j} ({\bf q},\cdot) :=0, \\
\label{def-xi-j}
&\text{on } \ \partial \mathcal{S}_{\kappa}, \ \ 
\left\{ \begin{array}{l}
\xi_{ \kappa,j} ({\bf q},x) := e_{j} \text{ for } j=1,2, \medskip \\
\xi_{ \kappa,3} ({\bf q},x) := (x-h_\kappa)^\perp, \medskip \\
\xi_{ \kappa,4} ({\bf q},x) := (-x_{1}+h_{\kappa,1}, x_{2}-h_{\kappa,2}) \ 
\text{ and } \
\xi_{ \kappa,5} ({\bf q},x) := (x_{2}-h_{\kappa,2}, x_{1}-h_{\kappa,1})  .
\end{array}  \right. 
\index{X@Miscellaneous!M1@$\xi_{\kappa,i}$: affine vector field centered on the $\kappa$-th solid}
\end{align}
We denote by  $$K_{ \kappa,j} ({\bf q},\cdot) := n \cdot \xi_{ \kappa,j} ({\bf q},\cdot)$$ the normal trace
of $\xi_{ \kappa,j} $ on $\partial \mathcal{F}({\bf q})$, where $n$ denotes the unit normal vector pointing outside ${\mathcal F}({\bf q})$.
We introduce the Kirchhoff potentials $\varphi_{ \kappa,j}({\bf q},\cdot)$, as the unique (up to an additive constant) solutions in $\mathcal F({\bf q})$ of the following Neumann problems:
\begin{subequations} \label{Kir}
\begin{alignat}{3}  \label{Kir1}
\Delta \varphi_{ \kappa,j} &= 0 & \quad & \text{ in } \mathcal F({\bf q}),\\  \label{Kir2}
\frac{\partial \varphi_{ \kappa,j}}{\partial n} ({\bf q},\cdot)&= K_{\kappa,j} ({\bf q},\cdot) & \quad & \text{ on }\partial\mathcal{F}({\bf q}).
\index{U@Potentials!P1@$\varphi_{\kappa,j}$: Kirchhoff potentials in ${\mathcal F}$}
\end{alignat}
\end{subequations}
We fix the additive constant by requiring (for instance) that
\begin{equation*}
\int_{\partial {\mathcal S}_{\kappa}({\bf q})} \varphi_{ \kappa,j} \, ds =0.
\end{equation*}
In the same spirit, we define the {\it standalone Kirchhoff potentials} as the solutions in $\R^2 \setminus {\mathcal S}_{\kappa}({\bf q})$ of the following Neumann problem:
\begin{subequations} \label{KirSA}
\begin{alignat}{3}
\label{Kir1SA}
\Delta \widehat\varphi_{ \kappa,j} &= 0 & \quad & \text{ in } \mathcal \R^2 \setminus \mathcal S_{\kappa}({\bf q}),\\
\label{Kir2SA}
\frac{\partial \widehat\varphi_{ \kappa,j}}{\partial n} ({\bf q},\cdot)&= K_{\kappa,j} ({\bf q},\cdot) & \quad & \text{ on }\partial\mathcal{S}_{\kappa}({\bf q}), \\
\label{Kir3SA}
\nabla \widehat\varphi_{ \kappa,j}(x) &\longrightarrow 0 & \quad & \text{ as } |x| \rightarrow +\infty, \\
\label{Kir4SA}
\int_{\partial {\mathcal S}_{\kappa}} \widehat\varphi_{ \kappa,j}(x) & \, ds(x) = 0. & &\quad   \ 
\index{U@Potentials!P2@\unexpanded{\unexpanded{\ensuremath{\widehat{\varphi}_{\kappa,j}}}}: standalone Kirchhoff potentials in $\R^2\setminus{\mathcal S}_{\kappa}$}
\end{alignat}
\end{subequations}
We underline that this potential is defined {as if ${\mathcal S}_{\kappa}$ were alone in the plane}, and consequently merely depends on the position $q_{\kappa}$. \par
We also define the {\it final Kirchhoff potentials} corresponding to the domain $\widecheck{\mathcal F}({\bf q}_{(i)})$ where small solids have disappeared as to satisfy
\begin{subequations} \label{KirMacro}
\begin{alignat}{3}
\label{Kir1Macro}
\Delta \widecheck\varphi_{ \kappa,j} &= 0 & \quad & \text{ in } \widecheck{\mathcal F}({\bf q}_{(i)}), \\
\label{Kir2Macro}
\frac{\partial \widecheck\varphi_{ \kappa,j}}{\partial n} ({\bf q},\cdot)&= K_{\kappa,j} ({\bf q},\cdot) & \quad & \text{ on }\partial \widecheck{\mathcal F}({\bf q}_{(i)}).
\index{U@Potentials!P3@\unexpanded{\unexpanded{\ensuremath{\widecheck{\varphi}_{\kappa,j}}}}: final Kirchhoff potentials in $\widecheck{\mathcal F}({\bf q}_{(i)})$}
\end{alignat}
\end{subequations} \par

 \paragraph{Inertia matrices.} We first define the (diagonal) $3N \times 3N$ matrix of genuine inertia by ${\mathcal M}_{g} = ({\mathcal M}_{g,\kappa,i,\kappa',i'})_{1 \leq i,i' \leq 3}$ with
\begin{equation} \label{Eq:TrueInertia}
{\mathcal M}_{g,\kappa,i,\kappa',i'} = \delta_{\kappa,\kappa'} \delta_{i,i'} (\delta_{i \in \{1,2\} } m_{\kappa} + \delta_{i,3} J_{\kappa} ).
\end{equation}
The  $3N \times 3N$ matrix of {\it added inertia} is defined by ${\mathcal M}_{a} = ({\mathcal M}_{a,\kappa,i,\kappa',i'})$ with
\begin{equation} \label{Eq:AddMassMatrix}
{\mathcal M}_{a,\kappa,i,\kappa',i'} ({\bf q}) = \int_{{\mathcal F}({\bf q})} \nabla \varphi_{ \kappa,i} ({\bf q},\cdot) \cdot \nabla \varphi_{\kappa',i'} ({\bf q},\cdot) \, dx.
\end{equation}
This allows to define the {\it total mass matrix} ${\mathcal M} ({\bf q})$ by
\begin{equation} \label{Eq:MasseTotale}
{\mathcal M} ({\bf q})= {\mathcal M}_{g} + {\mathcal M}_{a} ({\bf q}).
\end{equation}
We also define the {\it $\kappa$-th added inertia matrix} as the $3 \times 3$ matrix defined by
\begin{equation} \label{Eq:AddMassMatrixKappa}
({\mathcal M}_{a,\kappa})_{i,j}  ({\bf q}) = \int_{{\mathcal F}({\bf q})} \nabla \varphi_{ \kappa,i} ({\bf q},\cdot) \cdot \nabla \varphi_{\kappa,j} ({\bf q},\cdot) \, dx,
\end{equation}
and the {\it $\kappa$-th standalone added inertia matrix} as the $3 \times 3$ matrix defined by
\begin{equation} \label{Eq:AddMassMatrixKappaStandalone}
(\widehat{\mathcal M}_{a,\kappa})_{i,j} (\vartheta_{\kappa}) = \int_{\R^2 \setminus {\mathcal S}_\kappa ({\bf q})} \nabla \widehat{\varphi}_{ \kappa,i}  \cdot \nabla \widehat{\varphi}_{\kappa,j}   \, dx.
\end{equation}
Finally, when the small solids have disappeared, we also consider the $3N_{(i)} \times 3N_{(i)}$ {\it final added mass matrix} $\widecheck{\mathcal M}_{a} ({\mathbf q}_{(i)})= ({\mathcal M}_{a,\kappa,i,\kappa',i'}) ({\mathbf q}_{(i)})$ defined by
\begin{equation} \label{Eq:AddedMassMacro}
\widecheck{\mathcal M}_{a,\kappa,i,\kappa',i'}  ({\mathbf q}_{(i)}) = \int_{\widecheck{\mathcal F}({\mathbf q}_{(i)})} \nabla \widecheck{\varphi}_{\kappa,i} ({\mathbf q}_{(i)},\cdot) \cdot \nabla \widecheck{\varphi}_{\kappa',i'} ({\mathbf q}_{(i)},\cdot) \, dx.
\end{equation}
\begin{Remark} \label{Rem:PasBoules}
All those added mass matrices are Gram matrices, and consequently symmetric and positive semi-definite. Moreover, an elementary consequence of our assumption that the solids $\mathcal{S}_1, \dots, \mathcal{S}_N$ are not balls is that they are symmetric positive definite matrices, as Gram matrices of independent families of vectors. 
This will be of particular interest for the standalone added mass matrices $\widehat{\mathcal M}_{a,1}, \dots, \widehat{\mathcal M}_{a,N}$.
In the case of balls, these matrices are singular. In that case, mass-vanishing small solids require a different treatment (see \cite{GMS}).
\end{Remark}
\paragraph{Stream functions for the circulation terms.}
To take into account the circulations of velocity around the solids, we introduce for each $\kappa \in \{ 1, \dots, N\}$ the stream function  ${\psi}_{\kappa}= {\psi}_{\kappa}({\bf q},\cdot)$ defined on ${\mathcal F}({\bf q})$ of the harmonic vector field which has circulation $\delta_{\nu \kappa}$ around $\partial \mathcal{S}_{\nu}({\bf q})$ for $\nu=1,\dots,N$. 
More precisely, for every ${\bf q}$, there exist unique constants $C_{\kappa \nu}({\bf q}) \in \mathbb R$ such that the unique solution 
${\psi}_{\kappa}({\bf q},\cdot)$ of the Dirichlet problem:
\begin{subequations} 
\label{def_stream_F}
\begin{alignat}{3}
\Delta {\psi}_{\kappa}({\bf q},\cdot) & =0 & \quad & \text{ in } \mathcal F({\bf q}) \\
{\psi}_{\kappa}({\bf q},\cdot) & = C_{\kappa \nu}({\bf q}) & \quad & \text{ on } \partial \mathcal S_{\nu}({\bf q}), \ \nu=1,\dots, N, \\
{\psi}_{\kappa}({\bf q},\cdot) & = 0 & \quad & \text{ on } \partial \Omega,
\index{U@Potentials!P4@\unexpanded{\unexpanded{\ensuremath{{\psi}_{\kappa}}}}: circulation potential in ${\mathcal F}({\bf q})$}
\end{alignat}
satisfies
\begin{equation} \label{circ-norma_F}
\int_{\partial\mathcal S_{\nu}({\bf q})} \frac{\partial {\psi}_{\kappa}}{\partial n} ({\bf q},\cdot)  \, ds=- \delta_{\nu \kappa}, \ \nu=1,\dots,N.
\end{equation}
\end{subequations}
These functions ${\psi}_{\kappa}$ have their {\it standalone} counterparts, the stream functions $\widehat{\psi}_{\kappa}= \widehat{\psi}_{\kappa}({\bf q},\cdot)$ defined on $\R^2 \setminus \mathcal S_{\kappa}({\bf q})$ of the harmonic vector field which has circulation $1$ around $\partial \mathcal{S}_{\kappa}({\bf q})$. 
They are defined as follows: for every ${\bf q}$, there exists a unique constant $C_{\kappa}({\bf q}) \in \mathbb R$ such that the unique solution 
$\widehat{\psi}_{\kappa}({\bf q},\cdot)$ of the Dirichlet problem:
\begin{subequations} 
\label{def_stream}
\begin{alignat}{3}
\Delta \widehat{\psi}_{\kappa}({\bf q},\cdot) & =0 & \quad & \text{ in } \R^2 \setminus \mathcal S_{\kappa}({\bf q}) \\
\widehat{\psi}_{\kappa}({\bf q},\cdot) & = C_{\kappa}({\bf q}) & \quad & \text{ on } \partial \mathcal S_{\kappa}({\bf q}), \\
\nabla \widehat{\psi}_{\kappa}({\bf q},x) & \rightarrow 0 \text{ as } |x| \rightarrow +\infty,
\index{U@Potentials!P5@\unexpanded{\unexpanded{\ensuremath{\widehat{\psi}_{\kappa}}}}: standalone circulation potential in $\R^2 \setminus {\mathcal S}_{\kappa}$}
\end{alignat}
satisfies
\begin{equation} \label{circ-norma}
\int_{\partial\mathcal S_{\kappa}({\bf q})} \frac{\partial \widehat{\psi}_{\kappa}}{\partial n} ({\bf q},\cdot)  \, ds=-1.
\end{equation}
\end{subequations}
This allows to introduce the following vector depending merely on ${\mathcal S}_{\kappa}^\varepsilon$, that is on $\varepsilon_{\kappa}$ and $q_{\kappa}$:
\begin{equation} \label{Eq:Zeta}
\zeta^\varepsilon_{\kappa}(q_{\kappa}) = - \int_{\partial {\mathcal S}_{\kappa}} (x - h_{\kappa}) \frac{\partial \widehat{\psi}_{\kappa}}{\partial n} (q_{\kappa},x) \,  \, ds(x)
= R(\vartheta_{\kappa}) \zeta^\varepsilon_{\kappa}(q_{\kappa,0}) = \varepsilon_{\kappa} R(\vartheta_{\kappa}) \zeta^1_{\kappa}(q_{\kappa,0}).
\end{equation}
To simplify the notations, we denote $\zeta^1_{\kappa,0}:=\zeta^1_{\kappa}(q_{\kappa,0})$. This is referred to as the {\it conformal center} of solid. \par
Finally, as for the Kirchhoff potentials, we can introduce the {\it final stream functions for the circulation} $\widecheck{\psi}_{\kappa}({\bf q}_{(i)})$, $\kappa=1,\dots,N_{(i)}$, defined in $\widecheck{\mathcal F}({\bf q}_{(i)})$. Here $\widecheck{\psi}_{\kappa}({\bf q}_{(i)})$ is the stream function of the harmonic vector field which has circulation $\delta_{\nu \kappa}$ around $\partial \mathcal{S}_{\nu}({\bf q})$ for $\nu=1,\dots,N_{(i)}$. It can be obtained as follows: for every ${\bf q}_{(i)}$, there exist unique constants $\widecheck{C}_{\kappa \nu}({\bf q}_{(i)}) \in \mathbb R$ such that the unique solution 
$\widecheck{\psi}_{\kappa}({\bf q}_{(i)},\cdot)$ of the Dirichlet problem:
\begin{subequations} 
\label{def_stream_Fcheck}
\begin{alignat}{3}
\Delta \widecheck{\psi}_{\kappa}({\bf q}_{(i)},\cdot) & =0 & \quad & \text{ in } \widecheck{\mathcal F}({\bf q}_{(i)}), \\
\widecheck{\psi}_{\kappa}({\bf q}_{(i)},\cdot) & = \widecheck{C}_{\kappa \nu}({\bf q}_{(i)}) & \quad & \text{ on } \partial \mathcal S_{\nu}(q_{\nu}), \ \nu=1,\dots, N_{(i)}, \\
\widecheck{\psi}_{\kappa}({\bf q}_{(i)},\cdot) & = 0 & \quad & \text{ on } \partial \Omega,
\index{U@Potentials!P6@\unexpanded{\unexpanded{\ensuremath{\widecheck{\psi}_{\kappa}}}}: final circulation potential in $\widecheck{\mathcal F}(q_{(i)})$}
\end{alignat}
satisfies
\begin{equation} \label{circ-norma_Fcheck}
\int_{\partial\mathcal S_{\nu}({\bf q}_{(i)})} \frac{\partial \widecheck{\psi}_{\kappa}}{\partial n} ({\bf q}_{(i)},\cdot)  \, ds=- \delta_{\nu \kappa}, \ \nu=1,\dots,N_{(i)}.
\end{equation}
\end{subequations}
\paragraph{Biot-Savart kernel.}
Following \cite{Lin1,Lin2} we introduce two hydrodynamic Biot-Savart operators as follows. 
Given $\omega \in L^\infty({\mathcal F})$, we define the velocities $K[\omega]$ and $\widecheck{K}[\omega]$ as the solutions of
\begin{equation} \label{Eq:DefK}
\left\{ \begin{array}{l}
\div K[\omega] = 0 \ \text{ in } \ {\mathcal F}({\bf q}), \medskip \\
\curl K[\omega] = \omega \ \text{ in } \ {\mathcal F}({\bf q}), \medskip \\
K[\omega]\cdot n = 0 \ \text{ on } \ \partial {\mathcal F}({\bf q}), \medskip \\
\displaystyle \oint_{\partial {\mathcal S}_{\nu}} K[\omega] \cdot \tau \, ds = 0  \ \text{ for } \ \nu =1, \dots, N,
\end{array} \right.
\index{U@Potentials!P7@\unexpanded{\unexpanded{\ensuremath{{K}[\omega]}}}: Biot-Savart kernel in ${\mathcal F}({\bf q})$}
\end{equation}
and
\begin{equation} \label{Eq:DefKCheck}
\left\{ \begin{array}{l}
\div \widecheck{K}[\omega] = 0 \ \text{ in } \ \widecheck{\mathcal F}({\bf q}), \medskip \\
\curl \widecheck{K}[\omega] = \omega \ \text{ in } \ \widecheck{\mathcal F}({\bf q}), \medskip \\
\widecheck{K}[\omega]\cdot n = 0 \ \text{ on } \ \partial \widecheck{\mathcal F}({\bf q}), \medskip \\
\displaystyle \oint_{\partial {\mathcal S}_{\nu}} K[\omega] \cdot \tau \, ds = 0  \ \text{ for } \ \nu =1, \dots, N_{(i)}.
\end{array} \right.
\index{U@Potentials!P8@\unexpanded{\unexpanded{\ensuremath{\widecheck{K}[\omega]}}}: final Biot-Savart kernel in $\widecheck{\mathcal F}({\bf q}_{(i)})$}
\end{equation}
These are the standard and the final Biot-Savart operators, respectively.
\paragraph{Standard decomposition of the velocity field.}
These potentials allow to decompose the velocity field $u$ in several terms. Since it is the unique solution to the following $\div$/$\curl$ system:
\begin{equation} \label{Eq:DivCurlU}
\left\{ \begin{array}{l}
\div u = 0 \ \text{ in } \ {\mathcal F}({\bf q}), \medskip \\
\curl u = \omega \ \text{ in } \ {\mathcal F}({\bf q}), \medskip \\
u\cdot n = (h'_{\nu} + \vartheta'_{\nu}(x-h_{\nu})^\perp) \cdot n  \ \text{ on } \ \partial {\mathcal S}_{\nu}   \ \text{ for } \ \nu =1, \dots, N, \medskip \\
u\cdot n = 0 \ \text{ on } \ \partial \Omega, \medskip \\
\displaystyle \oint_{\partial {\mathcal S}_{\nu}} u \cdot \tau \, ds = \gamma_{\nu}  \ \text{ for } \ \nu =1, \dots, N,
\end{array} \right.
\end{equation}
we have the standard decomposition of the velocity field $u$:
\begin{equation} \label{Eq:DecompUeps}
u = \sum_{ \substack{ {\nu \in \{1, \dots,N\}} \\ {i \in \{1,2,3\}} } } p_{\nu,i} \nabla \varphi_{\nu,i} + \sum_{\nu \in \{1,\dots,N\}} \gamma_{\nu} \nabla^{\perp} {\psi}_{\nu} + K[\omega] \ \text{ in } \ {\mathcal F}({\bf q}).
\end{equation}
We introduce the following notation for the first term in the decomposition: we let $u^{pot}$ be the potential part of the fluid velocity
\begin{equation} \label{Eq:DefUPot}
u^{pot} :=\sum_{ \substack{ {\nu \in \{1, \dots,N\}} \\ {i \in \{1,2,3\}} } } p_{\nu,i} \nabla \varphi_{\nu,i}.
\index{Velocity fields!U1@$u^{pot}$: potential part of the velocity field}
\end{equation}
Note that the velocity field $u^\star$ obtained in the limit (see \eqref{Eq:DefUStar}) can be decomposed as in \eqref{Eq:DecompUeps} with the ``final'' quantities:
\begin{equation} \label{Eq:DecompUstar}
u^\star
= \sum_{ \substack{ {\nu \in \{1, \dots,N_{(i)}\}} \\ {i \in \{1,2,3\}} } } p^\star_{\nu,i} \nabla \widecheck{\varphi}_{\nu,i}
+ \sum_{\nu \in \{1,\dots,N_{(i)}\}} \gamma_{\nu} \nabla^{\perp} \widecheck{\psi}_{\nu}
+ \widecheck{K} \left[\omega^\star + \sum_{\nu \in {\mathcal P}_{s}} \gamma_{\nu} \delta_{h^\star_{\nu}}\right]  \ \text{ in } \ \widecheck{\mathcal F}({\bf q}^\star_{(i)}),
\end{equation}
where $p^\star_{\nu}:=(h^\star_{\nu}, \vartheta^\star_{\nu})$ for $\nu=1,\dots,N_{(i)}$.
%
%

%
%
%
%
%
%
%
\subsection{Brief description of the proof and organization of the paper}
\label{secsec}
Let us now give a rough idea of the proof.
One of the main difficulties to pass to the limit is to obtain uniform estimates as the sizes of the small solids go to zero. A standard energy estimate proves insufficient since the energy is not bounded as the size of small solids diminish (notice that the energy of a point vortex is infinite). The hardest case is the one of small and massless solids, for which the kinetic energy gives the weakest information. We explain first the main ideas to obtain uniform estimates in the case of a single solid ($N=1$), and then we explain some additional arguments needed in the case of several solids. \par
\ \par
\noindent
{\it Case of a single solid.} The starting point consists in decomposing the velocity field using the potentials described above. In particular, one extracts the singularity due to the fixed velocity circulation along the solid by decomposing $u^\varepsilon$ in the form
\begin{equation} \label{Eq:DecompBase}
u^\varepsilon(t,x) = \gamma_{1} \nabla^\perp \widehat{\psi}_{1}(q_{1}(t),x) + u^{reg}(t,x),
\end{equation}
where $u^{reg}$ is the ``regular part'' of the velocity. Then we inject this decomposition in \eqref{Eq:Newton}, which we can rewrite
\begin{equation}
\label{Eq:Deq}
{\mathcal M}_{g} p'_{1,i} = - \int_{\partial {\mathcal F}} (\partial_{t} u + (u \cdot \nabla) u ) \cdot \nabla \Phi_{1,i} \, dx.
\end{equation}
The fact that we use the {\it standalone} circulation stream function in \eqref{Eq:DecompBase} allows to get rid of the most singular  terms arising in the right hand side of 
\eqref{Eq:Deq}
when using the decomposition 
\eqref{Eq:DecompBase}. 
This is due to the following properties
\begin{equation*}
\partial_{t} \nabla^\perp \widehat{\psi}_{1} + \nabla (v_{{\mathcal S},1} \cdot \nabla^\perp \widehat{\psi}_{1}) =0
\text{ and }
\int_{\partial {\mathcal S}_{1}} |\nabla \widehat{\psi}_1 |^2 K_{1,i} \, ds = 0,
\end{equation*}
which will be proved in a more general setting in \eqref{Eq:QuadraticPsi}, and which allow to treat the terms containing $\partial_{t} \nabla^\perp \widehat{\psi_{1}}$ and $|\nabla^\perp \widehat{\psi_{1}}|^2$.
Then the most singular  remaining term is {\it linear} in $\nabla^\perp \widehat{\psi}_{1}$. Studying this term, we see that, in order to have a chance to perform an energy estimate in which this term does not give a too strong contribution (we will say that this term is {\it gyroscopic} or more precisely {\it weakly gyroscopic}), it is necessary to consider a {\it modulated variable} 
\begin{equation*}
\tilde{p}=p-\text{modulation}(\varepsilon,q,p,u^\varepsilon).
\end{equation*}
This modulation is imposed by the system, and one must incorporate it in the other terms of the equation and show that they do not contribute too strongly to the time evolution of the {\it modulated energy} associated with $\tilde{p}$. This will give a {\it normal form} of the equation. To obtain this normal form, it is needed to decompose $u^{reg}$ in \eqref{Eq:DecompBase} in a potential part $u^{pot}$ (only due to the movement of the solid) and an ``exterior'' part $u^{ext}$, this exterior part being actually the source of the modulation. The terms that arise when taking $u^{pot}$, $u^{ext}$ and the modulation into account will either be proven to contribute mildly to the modulated energy or be incorporated in the estimate as added inertia terms. \par
\ \par
\noindent
{\it Case of several solids.} 
When several solids are present, a new serious difficulty appears: if 
we write a normal form such as described above for each small solid then the equations are coupled by  terms associated with other solids which may include up to second derivatives in time.
 Because of this difficulty the strategy used in our previous papers \cite{GLS,GLS2,GMS,GLMS}  seems to fail. 
 To overcome this difficulty we use again  normal forms of the ODEs driving the motion of the solids but in a two-steps process. 
First we use a normal form for the system coupling the time-evolution of all the solids to obtain a rough estimate of the acceleration of the bodies. 
Then we turn to  normal forms that are specific to each small solid, with an appropriate modulation related to the influence of the other solids and of the fluid vorticity.
This specific normal form allows in particular to take into  account the specific scaling associated with each solid.
The previous rough estimate of the acceleration is  used here to prove that the coupling due to the acceleration of the other solids  is weaker than expected in the limit.  
Then, thanks to these individual normal forms, we obtain  precise uniform \textit{a priori} estimates of the velocities of the bodies.

\ \par
After uniform estimates are obtained, we use compactness arguments to pass to the limit. The normal forms obtained above play a central role to describe the dynamics in the limit of the small solids.  
For what concerns the large solids, we must study in particular the convergence of the pressure near their boundary. \par
\ \par
\noindent
{\bf Organization of the sequel of the paper.}
A central tool to develop the arguments above is a careful description of the potentials used in the  decomposition \eqref{Eq:DecompUeps} of the velocity field.
Indeed we analyze their behavior as the size of some of the solids go to zero, and of their derivative with respect to position. 
We use an extension  of the reflection method for a div/curl system with prescribed circulations, see Section~\ref{Sec:Expansions}. In Section~\ref{Sec:FAPE} we prove the first \textit{a priori} estimates on the system. This encompasses in particular  vorticity estimates,  (not yet modulated) energy estimates and 
the above-mentioned rough acceleration estimates. Then in Section~\ref{Sec:Modulations} we describe the modulations, and explain in particular how they are determined and estimated. Then in Section~\ref{Sec:NormalForm} we establish our normal forms. This allows to obtain the modulated energy estimates in Section~\ref{Sec:MEE}. Finally in Section~\ref{Sec:PTTL} we pass to the limit.

%
%
%
%
%
%
%
%
%
%
%
%
%
%
\section{Estimates on the potentials}
\label{Sec:Expansions}
In this section, we show how the various potentials appearing in the decomposition \eqref{Eq:DecompUeps}  of the velocity (including the Kirchhoff potentials $\varphi_{\kappa,i}$, the circulation stream functions $\psi_{\kappa}$ and the stream function associated with the Biot-Savart kernel $K[\omega]$) can be approximated and estimated by using in particular their standalone counterparts in $ \R^2 \setminus \mathcal S_{\kappa}$ 
or their final counterparts in $\widecheck{\mathcal F}$. \par
\ \par
\noindent
{\it Convention on the higher-order H\"older spaces.} Throughout this section, we will take the following convention for the $C^{k,\alpha}$-seminorms, $k \geq 1$, $\alpha \in (0,1)$, when considered on a curve. The $0$-th order H\"older seminorms $| \cdot |_{\alpha}$ are the standard ones, and for a open set ${\mathcal O}$ in $\R^2$, we also consider the same seminorms $| \cdot |_{C^{k,\alpha}({\mathcal O})}$ as usual. For a smooth curve $\gamma$ on the plane and $k \geq 1$, we set for $f \in C^{k,\alpha}(\gamma)$:
\begin{gather} \label{Eq:NormesHolder}
| f |_{C^{k,\alpha}(\gamma)} := \inf \Big\{ | u |_{C^{k,\alpha}({\mathcal O})}, \ u \text{ is an extension of } f \text{ to some neighborhood } {\mathcal O} \text{ of } \gamma
\Big\} \\
\nonumber
\| f \|_{C^{k,\alpha}(\gamma)} := \| f \|_{\Lip(\gamma)} + | f |_{C^{k,\alpha}(\gamma)} .
\end{gather}
For a fixed curve $\gamma$, this is equivalent to the usual norm $\| f\|_{\infty} + | \partial_{\tau}^k f |_{\alpha}$ (due to the existence of continuous extension operators), but the constants in this equivalence of norms are not uniform as a curve shrinks (due to curvature terms in $\partial_{\tau}^k f$). \par
\ \par
To study the above mentioned potentials we begin the section by considering an auxiliary general problem.
%
%
%
%
%
\subsection{An auxiliary Dirichlet problem}
In this subsection we consider a general problem of Dirichlet type that will be helpful to study all the functions used in the decomposition \eqref{Eq:DecompUeps} and their behavior as $\overline{\boldsymbol{\varepsilon}}$ goes to $0$. The general idea is that the Dirichlet boundary conditions will be merely satisfied up to an additive constant on each component of the boundary, but in return we impose a zero-flux condition on these components. \par
To be more specific, we consider the general situation of a domain $\Omega$ in which are embedded $N$ solids ${\mathcal S}_{1}$, \dots, ${\mathcal S}_{N}$, such as described before.
The fluid domain is then ${\mathcal F}:=\Omega \setminus ({\mathcal S}_{1} \cup \dots \cup {\mathcal S}_{N})$. Note that the results of this subsection will be applied not only to ${\mathcal F}^\varepsilon$ such as described in the introduction, but also in other domains (such as $\widecheck{{\mathcal F}}$ or a domain in which one of the small solids has been removed).  \par
We consider $N$ functions $\alpha_{\kappa} \in C^{\infty}(\partial{\mathcal S}_{\kappa};\R)$, $\kappa=1,\dots,N$, and a function $\alpha_{\Omega} \in C^{\infty}(\partial \Omega;\R),$ and study the following problem 
\begin{equation} \label{Eq:HAlpha}
\left\{ \begin{array}{l}
\Delta \mathfrak{H}[\alpha_{1},\dots,\alpha_{N};\alpha_{\Omega}] = 0 \ \text{ in } \ {\mathcal F}, \\
\mathfrak{H}[\alpha_{1},\dots,\alpha_{N};\alpha_{\Omega}] = \alpha_{\Omega} \ \text{ on } \ \partial \Omega, \\
\mathfrak{H}[\alpha_{1},\dots,\alpha_{N};\alpha_{\Omega}] = \alpha_{\kappa} + c_{\kappa}  \ \text{ on } \ \partial {\mathcal S}_{\kappa} \ \text{ for } \ \kappa \in \{1,\dots,N\}, \\
\int_{\partial {\mathcal S}_{\kappa}} \partial_{n} \mathfrak{H}[\alpha_{1},\dots,\alpha_{N};\alpha_{\Omega}](x) \, ds (x)=0 \text{ for } \kappa  \in \{1,\dots,N\}.
\end{array} \right.
\end{equation}
where the unknowns are the function $\mathfrak{H}[\alpha_{1},\dots,\alpha_{N};\alpha_{\Omega}]$ defined in ${\mathcal F}$ and the constants $c_1,\dots,c_N$. \par
%
%
%
\subsubsection{Existence of solutions for problem \eqref{Eq:HAlpha}}
\label{Subsec:PotManySolids}
\paragraph{A general existence result.}
The existence of solutions to problem \eqref{Eq:HAlpha} is granted by the following statement. For the moment, all solids are considered of fixed size.
\begin{Lemma} \label{Lem:EtrangeDirichlet}
Given $N$ functions $\alpha_{\kappa} \in C^{\infty}(\partial{\mathcal S}_{\kappa};\R)$, $\kappa=1,\dots,N$, and a function $\alpha_{\Omega} \in C^{\infty}(\partial \Omega;\R),$ there exist a unique function $\mathfrak{H}[\alpha_{1},\dots,\alpha_{N};\alpha_{\Omega}]$ and unique constants $c_{1}$,\dots,$c_{N}$
solution to System \eqref{Eq:HAlpha}.
\end{Lemma} 
\begin{proof}[Proof of Lemma~\ref{Lem:EtrangeDirichlet}]
We first introduce the solution $\widetilde{\mathfrak{H}}[\alpha_{1},\dots,\alpha_{N};\alpha_{\Omega}]$ of the standard Dirichlet problem 
\begin{equation} \nonumber 
\left\{ \begin{array}{l}
\Delta \widetilde{\mathfrak{H}}[\alpha_{1},\dots,\alpha_{N};\alpha_{\Omega}] = 0 \ \text{ in } \ {\mathcal F}, \\
\widetilde{\mathfrak{H}}[\alpha_{1},\dots,\alpha_{N};\alpha_{\Omega}] = \alpha_{\Omega} \ \text{ on } \ \partial \Omega, \\
\widetilde{\mathfrak{H}}[\alpha_{1},\dots,\alpha_{N};\alpha_{\Omega}] = \alpha_{\kappa} \ \text{ on } \ \partial {\mathcal S}_{\kappa} \ \text{ for } \ \kappa \in \{1,\dots,N\}.
\end{array} \right.
\end{equation}
Then we correct this solution by means of the following ones: for $\kappa \in \{1, \dots,N\}$ one defines $\mathfrak{h}_{\kappa}$ as the unique solution to
\begin{equation} \nonumber 
\left\{ \begin{array}{l}
\Delta \mathfrak{h}_{\kappa} = 0 \ \text{ in } \ {\mathcal F}, \\
\mathfrak{h}_{\kappa} = 0 \ \text{ on } \ \partial \Omega, \\
\mathfrak{h}_{\kappa} =1 \ \text{ on } \ \partial {\mathcal S}_{\kappa}, \\
\mathfrak{h}_{\kappa} =0 \ \text{ on } \ \partial {\mathcal S}_{\nu} \ \text{ for } \ \nu \neq \kappa.
\end{array} \right.
\end{equation}
Obviously, this family is linearly independent (it is connected to the first De Rham cohomology space of ${\mathcal F}$). Then it remains to prove that the linear mapping from $\mbox{Span}\{\mathfrak{h}_{1}, \dots, \mathfrak{h}_{N} \}$ to $\R^N$, defined by
\begin{equation} \label{Eq:DefN}
\mathfrak{N} : 
\mathfrak{h} \mapsto \left( \int_{\partial {\mathcal S}_{1}} \partial_{n} \mathfrak{h}(x) \, ds (x) , \dots, \int_{\partial {\mathcal S}_{N}} \partial_{n} \mathfrak{h}(x) \, ds (x),
\right)
\end{equation}
is an isomorphism. This is easy, since when $\mathfrak{h}$ belongs to its kernel, one has
\begin{equation*}
\int_{{\mathcal F}} |\nabla \mathfrak{h}|^2 \, dx = \int_{\partial {\mathcal F}} \mathfrak{h} \partial_{n} \mathfrak{h} \, ds(x)=0. 
\end{equation*}
Hence since $\mathfrak{h}=0$ on $\partial \Omega$, we deduce $\mathfrak{h}=0$ in ${\mathcal F}$. 
\end{proof}
%
%
%
%
%
\paragraph{Uniform estimates for fixed sizes.}
In the sequel, a case of particular interest is the case of the  ``final'' fluid domain where all small solids have been removed (hence the fluid domain is larger). 
Therefore we consider a domain $\Omega$ in which are embedded $N_{(i)}$ solids ${\mathcal S}_{1}$, \dots, ${\mathcal S}_{N_{(i)}}$ of fixed size, each of them being obtained by a rigid movement from a fixed shape, such as described before (in particular we still use the notation ${\mathcal S}_{i}(q_{i})$). The fluid domain is then $\widecheck{\mathcal F}:=\Omega \setminus ({\mathcal S}_{1} \cup \dots \cup {\mathcal S}_{N_{(i)}})$. 
We obtain a sort of maximum principle for $\mathfrak{H}[\alpha_{1},\dots,\alpha_{N_{(i)}};\alpha_{\Omega}]$ as long as the solids remain a distance at least $\delta>0$ one from another and from the outer boundary. 
\begin{Lemma} \label{Lem:MMH}
Let $\delta>0$. There exists a constant $C>0$ depending merely on $\delta$, $\Omega$, and the shapes of ${\mathcal S}_{1}, \dots, {\mathcal S}_{N_{(i)}}$ such that for any
\begin{multline*}
{\bf q}_{(i)}=(q_{1},\dots,q_{N_{(i)}}) \in {\mathcal{Q}}_{(i),\delta} := \Big\{ (q_{1},\dots,q_{N_{(i)}}) \in \R^{3N_{(i)}} \ \Big/ \
\forall i \in \{1, \dots, N_{(i)} \}, \ \ \mbox{dist}({\mathcal S}_{i}(q_{i}), \partial \Omega ) > 2 \delta \\
\text{ and } \forall j \in \{1, \dots, N_{(i)} \} \text{ with } i \neq j, \ \
 \mbox{dist}({\mathcal S}_{i}(q_{i}), {\mathcal S}_{j}(q_{j}) ) > 2\delta \Big\},
\end{multline*}
for any functions $\alpha_{\lambda} \in C^{\infty}(\partial{\mathcal S}_{\lambda};\R)$, $\lambda=1,\dots,N_{(i)}$ and any function $\alpha_{\Omega} \in C^{\infty}(\partial \Omega;\R)$,
one has
\begin{equation} \label{Eq:MMH}
\| \mathfrak{H}[\alpha_{1},\dots,\alpha_{N_{(i)}};\alpha_{\Omega}] \|_{L^{\infty}(\widecheck{\mathcal F})} \leq C
\| (\alpha_{1},\dots,\alpha_{N_{(i)}};\alpha_{\Omega}) \|_{L^{\infty}(\partial \widecheck{\mathcal F})} .
\end{equation}
In particular, $\mathfrak{H}[\alpha_{1},\dots,\alpha_{N_{(i)}};\alpha_{\Omega}]$ can be defined for any functions $\alpha_{\lambda} \in C^{0}(\partial{\mathcal S}_{\lambda};\R)$, $\lambda=1,\dots,N_{(i)}$ and any function $\alpha_{\Omega} \in C^{0}(\partial \Omega;\R)$.
\end{Lemma}
Before getting to the proof of Lemma~\ref{Lem:MMH} we state the following uniform Schauder estimates, see e.g. \cite[p. 98]{GT}.
\begin{Lemma}
\label{Lem:USE}
Let $\delta>0$. There exists a uniform constant $C>0$ such that for all ${\bf q}_{(i)} \in {\mathcal{Q}}_{(i),\delta}$ the following Schauder estimate holds for $u \in C^{2,\frac{1}{2}}(\widecheck{{\mathcal F}}({\bf q}_{(i)}))$:
\begin{equation*}
\| u \|_{C^{2,\frac{1}{2}}\left(\widecheck{{\mathcal F}}({\bf q}_{(i)})\right)} \leq C \left(\| \Delta u \|_{C^{\frac{1}{2}}\left(\widecheck{{\mathcal F}}({\bf q}_{(i)})\right)} + \| u \|_{C^{2,\frac{1}{2}}\left(\partial \widecheck{{\mathcal F}}({\bf q}_{(i)})\right)}\right).
\end{equation*}
\end{Lemma}
\begin{proof}[Proof of Lemma~\ref{Lem:USE}]
First one establishes the result locally by using smooth diffeomorphisms close to the identity from ${\mathcal F}({\bf q}_{(i)})$ to ${\mathcal F}(\widetilde{\bf q}_{(i)})$ when $\widetilde{{\bf q}}_{(i)}$ is close to ${\bf q}_{(i)}$. Using elliptic regularity for smooth operators with coefficients close to those of the Laplacian, this yields the result in the neighborhood of ${\bf q}_{(i)}$. One concludes by compactness of ${\mathcal{Q}}_{(i),\delta}$. We omit the details. 
\end{proof}
We now prove Lemma~\ref{Lem:MMH}.
\begin{proof}[Proof of Lemma~\ref{Lem:MMH}]
We consider $\alpha_{\lambda} \in C^{\infty}(\partial{\mathcal S}_{\lambda};\R)$, $\lambda=1,\dots,N_{(i)}$ and $\alpha_{\Omega} \in C^{\infty}(\partial \Omega;\R)$ and prove \eqref{Eq:MMH}; the conclusion that $\mathfrak{H}$ can be extended to continuous functions follows then immediately by density. 
We examine the proof of Lemma~\ref{Lem:EtrangeDirichlet}: we see that $\widetilde{\mathfrak{H}}[\alpha_{1},\dots,\alpha_{N};\alpha_{\Omega}]$ satisfies the maximum principle, and hence \eqref{Eq:MMH}. It remains to prove that the correction in $\mbox{Span}\{\mathfrak{h}_{1}, \dots, \mathfrak{h}_{N} \}$ can be estimated in the same way. \par
It follows from Lemma~\ref{Lem:USE} that the functions $\mathfrak{h}_{\lambda}$ are uniformly bounded in $C^{2,\frac{1}{2}}(\widecheck{{\mathcal F}})$. This involves in particular that the integrals
\begin{equation*}
\int_{\partial {\mathcal S}_{\lambda}} \partial_{n} \widetilde{\mathfrak{H}}[\alpha_{1},\dots,\alpha_{N};\alpha_{\Omega}] \, ds(x)
= \int_{\partial \widecheck{{\mathcal F}}} \widetilde{\mathfrak{H}}[\alpha_{1},\dots,\alpha_{N};\alpha_{\Omega}] \partial_{n} \mathfrak{h}_{\lambda} \, ds(x),
\ \lambda=1,\dots,N,
\end{equation*}
can be bounded uniformly in terms of $\| (\alpha_{1},\dots,\alpha_{N_{(i)}};\alpha_{\Omega}) \|_{L^{\infty}(\partial \widecheck{\mathcal F})}$.
It remains to prove that the isomorphism $\mathfrak{N}$ defined in \eqref{Eq:DefN} is uniformly invertible for ${\bf q}_{(i)} \in  {\mathcal{Q}}_{(i),\delta}$. 
Let $\mathfrak{h}$ in $\mbox{Span}\{\mathfrak{h}_{1}, \dots, \mathfrak{h}_{N_{(i)}} \}$, say
$\mathfrak{h} = \sum_{\lambda =1}^{N_{(i)}} \rho_{\lambda} \mathfrak{h}_{\lambda}.$
We observe that for some positive constant $C$:
\begin{equation} \label{Eq:Normehi}
\sum_{\lambda \in {\mathcal P}_{(i)}} | \rho_{\lambda} |  \leq C \| \mathfrak{h} \|_{H^{1/2}(\partial \widecheck{{\mathcal F}})},
\end{equation}
since the functions in $\mbox{Span}\{\mathfrak{h}_{1}, \dots, \mathfrak{h}_{N} \}$ are constant on $\partial \widecheck{{\mathcal F}}$.
Now we have
\begin{equation*}
\int_{\widecheck{\mathcal F}} |\nabla \mathfrak{h}|^2 \, dx = \int_{\partial \widecheck{\mathcal F}} \mathfrak{h} \partial_{n} \mathfrak{h} \, ds(x)
\leq \sum_{\lambda \in {\mathcal P}_{(i)}} | \rho_{\lambda} | \left| \int_{\partial {\mathcal S}_{\lambda}} \partial_{n} \mathfrak{h}  \, ds(x) \right|
\leq C \| \mathfrak{h} \|_{H^{1/2}(\partial \widecheck{{\mathcal F}})}  \sum_{\lambda \in {\mathcal P}_{(i)}} \left| \int_{\partial {\mathcal S}_{\lambda}} \partial_{n} \mathfrak{h}  \, ds(x) \right| ,
\end{equation*}
where we have used that $\mathfrak{h}=\rho_\lambda$ on $\partial {\mathcal S}_\lambda$. 
Moreover,  by the trace inequality (which is uniform in ${\mathcal Q}_{(i),\delta}$ by straightforward localization arguments), 
$$\| \mathfrak{h} \|_{H^{1/2}(\partial \widecheck{{\mathcal F}})} \leq C \| \mathfrak{h} \|_{H^1(\widecheck{{\mathcal F}})} ,$$ 
and,  since for $\mathfrak{h}$ in $\mbox{Span}\{\mathfrak{h}_{1}, \dots, \mathfrak{h}_{N} \}$ we have $\mathfrak{h}=0$ on $\partial \Omega$,  by 
Poincar\'e's inequality (which is also uniform in ${\bf q}_{(i)}$, since it merely depends on the diameter of the domain), 
$$\| \mathfrak{h} \|^2_{H^1(\widecheck{{\mathcal F}})} \leq C \int_{\widecheck{\mathcal F}} |\nabla \mathfrak{h}|^2 \, dx .$$
Gathering the  inequalities above we deduce that
\begin{equation*}
\| \mathfrak{h} \|_{H^{1/2}(\partial \widecheck{{\mathcal F}})} \leq C \sum_{\lambda =1}^{N_{(i)}} 
\left| \int_{\partial {\mathcal S}_{\lambda}} \partial_{n} \mathfrak{h}  \, ds(x) \right|.
\end{equation*}
The conclusion follows by using again \eqref{Eq:Normehi}.
\end{proof}
%
%
%
%
%
%
%
\subsubsection{A potential for a standalone solid}
Now we consider the situation where the single solid ${\mathcal S}_{\kappa}$, rather than being embedded in $\Omega$ together with other solids ${\mathcal S}_{\nu}$, $\nu \neq \kappa$, is alone in the plane. This will play a central role in the description of the asymptotic behavior of the general potentials as some solids shrink to points. \par
To be more specific, we consider the solid ${\mathcal S}_{\kappa}$ obtained by a rigid movement and a homothety of scale $\varepsilon_{\kappa}$ with respect to its counterpart of size $1$ at initial position:
\begin{equation*}
{\mathcal S}_{\kappa}^{\varepsilon}={\mathcal S}_{\kappa}^{\varepsilon}(h_{\kappa},\vartheta_{\kappa}) = h_{\kappa} + \varepsilon_{\kappa} R(\vartheta_{\kappa})({\mathcal S}_{\kappa,0}^1 - h_{\kappa,0}),
\end{equation*}
and we study the above outer Dirichlet problem on $\R^2 \setminus {\mathcal S}^{\varepsilon}_{\kappa}$. Precisely we show the following.
\begin{Proposition} \label{Pro:DirichletSAM}
Let $\varepsilon_{\kappa}>0$, and let $\alpha \in C^{\infty}(\partial {\mathcal S}_{\kappa}^\varepsilon; \R)$.
Then there exist a unique constant $\widehat{c}_\kappa[\alpha]$ and a unique function $\widehat{\mathfrak{f}}_{\kappa}^\varepsilon[\alpha] \in C^{\infty}(\R^2 \setminus {\mathcal S}_{\kappa}^\varepsilon)$ solution to the system
\begin{equation} \label{Eq:PhiAlphachapeau}
\left\{ \begin{array}{l}
\Delta \widehat{\mathfrak{f}}_\kappa^\varepsilon[\alpha] = 0 \ \text{ in } \ \R^2 \setminus {\mathcal S}_\kappa^{\varepsilon}, \\ \medskip
\widehat{\mathfrak{f}}_\kappa^\varepsilon[\alpha](x) = \alpha + \widehat{c}_\kappa[\alpha] \ \text{ on } \ \partial {\mathcal S}_\kappa^{\varepsilon}, \\ \medskip
\widehat{\mathfrak{f}}_\kappa^\varepsilon[\alpha](x) \longrightarrow 0 \ \text{ as } \ |x| \longrightarrow +\infty .
\end{array} \right.
\end{equation}
Moreover one has the following estimates, where the constant $C$ merely depends on ${\mathcal S}_{\kappa,0}^1$ and $k \in \N \setminus \{0,1\}$
(hence is independent of $\varepsilon_{\kappa}$):
\begin{gather} 
\label{Eq:PhiHatMP}
\| \widehat{\mathfrak{f}}_\kappa^\varepsilon[\alpha] \|_{L^\infty(\R^2 \setminus {\mathcal S}_\kappa^\varepsilon)} \leq 2 \| \alpha \|_{L^\infty(\partial {\mathcal S}_\kappa^\varepsilon)}
\ \text{ and } \ 
| \widehat{c}_\kappa[\alpha] | \leq  \| \alpha \|_{L^\infty(\partial {\mathcal S}_\kappa^\varepsilon)} , \\
\label{Eq:NablaPhiHatGlobal}
\varepsilon_{\kappa} \| \nabla \widehat{\mathfrak{f}}_\kappa^\varepsilon[\alpha] \|_{L^{\infty}(\R^2 \setminus {\mathcal S}_\kappa^\varepsilon)} 
+ \varepsilon_{\kappa}^{k+\frac{1}{2}} | \widehat{\mathfrak{f}}_\kappa^\varepsilon[\alpha] |_{C^{k,\frac{1}{2}}(\R^2 \setminus {\mathcal S}_\kappa^\varepsilon)} 
\leq C \left(
\| \alpha \|_{L^{\infty}(\partial {\mathcal S}_\kappa^\varepsilon)} 
+ \varepsilon_{\kappa}^{k+\frac{1}{2}}| \alpha |_{C^{k,\frac{1}{2}}(\partial {\mathcal S}_\kappa^\varepsilon)}  \right), 
\end{gather}
and
\begin{multline}
\label{Eq:EstNablaPhiHat}
\forall x \text{ s.t. } |x-h_{\kappa}| \geq C \, \varepsilon_\kappa, \ \ 
| \widehat{\mathfrak{f}}_\kappa^\varepsilon[\alpha](x) | \leq C \frac{\varepsilon_\kappa}{|x - h_{\kappa}|} \| \alpha \|_{L^\infty(\partial {\mathcal S}_\kappa^\varepsilon)} \\
 \text{ and }
| \nabla \widehat{\mathfrak{f}}_\kappa^\varepsilon[\alpha](x) | \leq C \frac{\varepsilon_\kappa}{|x - h_{\kappa}|^2} \| \alpha \|_{L^\infty(\partial {\mathcal S}_\kappa^\varepsilon)} .
\end{multline}
\end{Proposition}
\begin{Remark}
Notice that Estimate~\eqref{Eq:EstNablaPhiHat} and the divergence theorem involve that
\begin{equation} \label{Eq:FluxNulPhiAlphaChapeau}
\int_{\partial \mathcal{S}_{\kappa}^\varepsilon} \partial_{n} \widehat{\mathfrak{f}}_\kappa^\varepsilon[\alpha] \, ds =0.
\end{equation}
\end{Remark}
\begin{proof}[Proof of Proposition~\ref{Pro:DirichletSAM}]
We proceed in two steps.\par
\ \par
\noindent
{\bf Step 1.} We first consider the case when $\varepsilon_\kappa = 1$. 
Since the above estimates are invariant by translation and rotation, without loss of generality, we can suppose that $\vartheta_{\kappa}=0$ and that
$0$ is in the interior of $\mathcal{S}_{\kappa}^1.$
Identifying $\R^2$ and $\C$, we use the inversion $z \mapsto 1/z$ with respect to $0$. Denoting the Riemann sphere by $\widehat{\C}$,
we set
$\Omega' := \left\{ 1/z, \, z \in \widehat{\C} \setminus {\mathcal S}_\kappa^1 \right\}$
(which is a regular bounded domain since $0$ is in the interior of ${\mathcal S}_{\kappa}^1$),
and consider the Dirichlet problem:
\begin{equation} \label{Eq:DirichletJordan}
\Delta \theta = 0 \text{ in } \Omega'   \text{ and }
\theta(z) = \alpha \left( 1/z\right) \text{ for } z \in \partial \Omega'.
\end{equation}
Notice that $0 \in \overset{\circ}{\Omega'}$ because it is the image of the point at infinity by the inversion $z \mapsto 1/z$.
Then we can set for $z \in \mathcal{S}_{\kappa}^1$:
\begin{equation} \label{Eq:DefPhialpha}
\widehat{\mathfrak{f}}^1[\alpha](z) = \theta \left( 1/z \right) - \theta(0)
\ \text{ and } \
\widehat{c}_\kappa[\alpha] = -\theta(0).
\end{equation}
By conformality of the inversion $z \mapsto 1/z$, this function satisfies \eqref{Eq:PhiAlphachapeau}. Conversely, starting from a solution to \eqref{Eq:PhiAlphachapeau}, we can invert and obtain a solution to \eqref{Eq:DirichletJordan} up to an additive constant on the boundary condition, which we remove. This proves the uniqueness of the solution to \eqref{Eq:PhiAlphachapeau}. \par
Now \eqref{Eq:PhiHatMP} is a direct consequence of \eqref{Eq:DefPhialpha} and of the maximum principle. 
Estimate \eqref{Eq:NablaPhiHatGlobal} is also a consequence of \eqref{Eq:DefPhialpha}: we make use of Schauder's estimates in $\Omega'$, then we invert using that $d(\R^2 \setminus {\mathcal S}_\kappa^1,0)>0$. 
Let us now focus on \eqref{Eq:EstNablaPhiHat}. The function
\begin{equation} \label{Eq:Defeta}
\eta(z) := \partial_{z} \theta(z) = \partial_{x} \theta (z) - i \partial_{y} \theta (z) 
\end{equation}
is holomorphic in $\Omega'$. We call $a_{k}(\eta)$, $k$ in $\N$, the coefficients of its power series expansion at $0$, so that
\begin{equation} \label{RG}
\eta(z) = \sum_{k \geq 0} a_{k}(\eta) z^k .
\end{equation}
We introduce $r>0$ such that the circle $S(0,r)$ lies inside $\Omega'$ at positive distance from $\partial \Omega'$. Using interior elliptic estimates (see e.g. \cite[Theorem 2.10, p. 23]{GT}), we see that $\| \eta \|_{C^0(S(0;r))} \leq C \| \alpha \|_{L^\infty(\partial {\mathcal S}_\kappa^1)}$ for some constant $C>0$ merely depending on ${\mathcal S}_\kappa^1$. 
Then, by using the Cauchy integral formula on $S(0,r)$, we deduce that there exists 
 $C_{{\mathcal S}}>0$ depending only on ${\mathcal S}_\kappa^1$ such that 
$|a_{k}(\eta)| \leq C_{{\mathcal S}}^k \| \alpha \|_{L^\infty(\partial {\mathcal S}_\kappa^1)}$  for all  $k \in \N.$
Now, by \eqref{Eq:DefPhialpha}, \eqref{Eq:Defeta} and \eqref{RG}, 
\begin{equation*}
\partial_{z} \widehat{\mathfrak{f}}^1[\alpha](z)
 =  - \frac{1}{z^2} \sum_{k \geq 0} \frac{a_{k}(\eta)}{z^k} .
\end{equation*}
Thus $| \nabla \widehat{\mathfrak{f}}^1[\alpha](x) | \leq C_{{\mathcal S}}  |z|^{-2} \| \alpha \|_{L^\infty(\partial {\mathcal S}_\kappa^1)} $ 
for $|z|$ large enough, for instance $|z-h_{\kappa}| \geq 2 C_{{\mathcal S}}$. But for $|z-h_{\kappa}|$ large enough (depending on $\mathcal{S}_{\kappa}^1$ only) we have that $|z-h_{\kappa}| \leq 2 |z|$. 
Hence we deduce the second inequality in \eqref{Eq:EstNablaPhiHat}, and then the first one by integration from infinity. \par
\ \par
\noindent
{\bf Step 2.} Obtaining the estimates for arbitrary $\varepsilon_\kappa>0$ is just a matter of rescaling. We call $\widehat{\mathfrak{f}}_\kappa^1$ the potential obtained above in the exterior domain $\R^2 \setminus {\mathcal S}_\kappa^1$  and $\widehat{\mathfrak{f}}_{\kappa}^\varepsilon$ the corresponding potential in $\R^2 \setminus {\mathcal S}_{\kappa}^\varepsilon$. Given $\alpha$ in $C^{\infty}(\partial {\mathcal S}_{\kappa}^\varepsilon; \R)$ we set $\alpha^\varepsilon(x) = \alpha(\varepsilon_{\kappa} x)$ defined on $\partial {\mathcal S}_{\kappa}^1$.
Then clearly
\begin{equation} \nonumber 
\forall x \in \R^2 \setminus {\mathcal S}^\varepsilon_\kappa, \ \ 
\widehat{\mathfrak{f}}_\kappa^\varepsilon[\alpha](x) = \widehat{\mathfrak{f}}_{\kappa}^1[\alpha^\varepsilon](x/\varepsilon_{\kappa}), \ \ 
\nabla \widehat{\mathfrak{f}}_\kappa^\varepsilon[\alpha](x) = \frac{1}{\varepsilon_\kappa} \nabla \widehat{\mathfrak{f}}_{\kappa}^1[\alpha^\varepsilon](x/\varepsilon_{\kappa}).
\end{equation}
The estimates \eqref{Eq:PhiHatMP}--\eqref{Eq:EstNablaPhiHat} follow; Estimate \eqref{Eq:NablaPhiHatGlobal} in particular is just the rescaled Schauder estimate (note that the seminorms defined in \eqref{Eq:NormesHolder} scale in the same way as H\"older seminorms on open sets).
\end{proof}
%
%
%
%
%
\subsubsection{A construction of the potential in the presence of small solids}
\label{Subsec:CPPSM}
Now we consider again the situation of a domain $\Omega$ in which are embedded $N$ solids, among which $N_{(i)}$ stay of fixed size and $N_s$ are shrinking. The only constraints that we will use is
$\mbox{dist}(\partial {\mathcal S}_\kappa, \partial {\mathcal S}_\nu) \geq \delta$ for $\kappa \neq \nu$ and $\mbox{dist}(\partial {\mathcal S}_\kappa, \partial \Omega) \geq \delta$ for all $\kappa$
where $\delta>0$ is fixed. The constants that follow will merely depend on $\delta$, $\Omega$ and on the shape of the {\bf unscaled} solids ${\mathcal S}_{\kappa}^1$ at size $1$. In particular they are independent of $\varepsilon_{N_{(i)}+1}, \cdots, \varepsilon_{N}$ (as long as they are small enough) and of the exact positions of the solids (as long as the above constraints are satisfied). \par
In this context we give a particular construction of $\mathfrak{H}[\alpha_{1},\dots,\alpha_{N};\alpha_{\Omega}]$, inspired by the method of successive reflections (see e.g. \cite{LLS} and references therein). The solution $\mathfrak{H}[\alpha_{1},\dots,\alpha_{N};\alpha_{\Omega}]$ will be obtained by means of the inversion of an operator on 
\begin{equation*}
(\eta_{1},\dots,\eta_{N},\eta_{\Omega}) \in E_{\partial {\mathcal F}} := C^0(\partial {\mathcal S}_{1}) \times \dots \times C^0(\partial {\mathcal S}_{N}) \times C^0(\partial \Omega),
\end{equation*}
which will be a perturbation of the identity by a contractive map. \par
\ \par
Let us describe this contractive map.
We first recall that $\widecheck{\mathcal F}$ refers to the larger fluid domain where the small solids have been removed, see \eqref{Eq:AugmentedDomain}. Correspondingly, $\partial \widecheck{{\mathcal F}} = \partial {\mathcal S}_{1} \cup \dots \cup \partial {\mathcal S}_{N_{(i)}} \cup \partial \Omega$.
Now given $(\eta_{1},\dots,\eta_{N},\eta_{\Omega}) \in E_{\partial {\mathcal F}}$ we first introduce $\widecheck{\mathfrak g} = \widecheck{\mathfrak g}[\eta_{1},\dots,\eta_{N_{(i)}};\eta_{\Omega}]$  and $\widecheck{c}_{\lambda}=\widecheck{c}_{\lambda}[\eta_{1},\dots,\eta_{N_{(i)}};\eta_{\Omega}]$ as the solution in $\widecheck{\mathcal F}$ of the Dirichlet problem
\begin{equation} \label{Eq:CorrectionInterieure}
\left\{ \begin{array}{l}
 - \Delta \widecheck{\mathfrak g} = 0 \text{ in } \ \widecheck{\mathcal F}, \\
 \widecheck{\mathfrak g} = \eta_{\Omega} \ \text{ on } \ \partial \Omega, \\
 \widecheck{\mathfrak g} = \eta_{\lambda} + \widecheck{c}_{\lambda} \text{ on } \partial {\mathcal S}_{\lambda}, \ \ \forall \lambda=1,\dots,N_{(i)}, \\
 \int_{\partial {\mathcal S}_{\lambda}} \partial_{n} \widecheck{\mathfrak g} \, ds (x)=0 , \ \ \forall \lambda=1,\dots,N_{(i)}.
\end{array} \right.
\end{equation}
This problem has a solution as described in Lemma~\ref{Lem:MMH}. 
Note in particular that Lemma~\ref{Lem:MMH} brings the following estimate:
\begin{equation} \label{Eq:CheckLambda}
\| \widecheck{\mathfrak g} \|_{L^\infty(\widecheck{\mathcal F})} \leq C \| (\eta_{1}, \dots, \eta_{N_{(i)}},\eta_{\Omega}) \|_{L^\infty(\partial {\mathcal S}_{1} \times \dots \times \partial {\mathcal S}_{N_{(i)}} \times \partial \Omega)}.
\end{equation}
 \par
Next we introduce the function $\mathfrak{m} = \mathfrak{m}[\eta_{1},\dots,\eta_{N};\eta_{\Omega}]$ in ${\mathcal F}$ by
\begin{equation} \label{Eq:DefFrakM}
\mathfrak{m} := \widecheck{\mathfrak g} + \sum_{\lambda \in {\mathcal P}_{s}} \widehat{\mathfrak{f}}_{\lambda}[\eta_{\lambda} - \widecheck{\mathfrak g}_{|\partial {\mathcal S}_{\lambda}}]
 \ \text{ with } \ \widecheck{\mathfrak g}=\widecheck{\mathfrak g}[\eta_{1},\dots,\eta_{N_{(i)}};\eta_{\Omega}] ,
\end{equation}
where as in \eqref{Eq:PS}, we have denoted ${\mathcal P}_{s} = \{N_{(i)}+1,\dots,N\}$ the set of indices for shrinking solids. 
Note that $\mathfrak{m}$ is the unique solution to the following Dirichlet problem of type \eqref{Eq:HAlpha} (for some constants $c_{1}$, \dots, $c_{N}$):
\begin{equation} \label{Eq:SysFrakM}
\left\{ \begin{array}{l}
 - \Delta \mathfrak{m} = 0 \text{ in } \ {\mathcal F}, \\ \medskip
 \mathfrak{m} = \eta_{\Omega} + \sum_{\lambda \in {\mathcal P}_{s}} \widehat{\mathfrak{f}}_{\lambda}[\eta_{\lambda}-\widecheck{\mathfrak g}_{|\partial {\mathcal S}_{\lambda}}] 
\ \text{ on } \ \partial \Omega, \\ \medskip
 \mathfrak{m} = \eta_{\nu} + \sum_{\lambda \in {\mathcal P}_{s}} \widehat{\mathfrak{f}}_{\lambda}[\eta_{\lambda}-\widecheck{\mathfrak g}_{|\partial {\mathcal S}_{\lambda}}] + c_{\nu}
\ \text{ on } \ \partial {\mathcal S}_{\nu} \text{ for } \nu \in {\mathcal P}_{(i)}, \\ \medskip
 \mathfrak{m} = \eta_{\nu} + \sum_{\lambda \in {\mathcal P}_{s} \setminus \{ \nu \}} \widehat{\mathfrak{f}}_{\lambda}[\eta_{\lambda}-\widecheck{\mathfrak g}_{|\partial {\mathcal S}_{\lambda}}] +c_{\nu}
\ \text{ on } \ \partial {\mathcal S}_{\nu} \text{ for } \nu \in {\mathcal P}_{s}, \\ \medskip
\int_{\partial {\mathcal S}_{\nu}} \partial_{n} \mathfrak{m} \, ds (x)=0 , \ \forall \nu=1,\dots,N,
\end{array} \right.
\end{equation}
where for the last equation we have used \eqref{Eq:PhiAlphachapeau}, \eqref{Eq:FluxNulPhiAlphaChapeau}, \eqref{Eq:CorrectionInterieure} and the divergence theorem. \par
Our goal is to prove that one can put the solution $\mathfrak{H}[\alpha_{1},\dots,\alpha_{N};\alpha_{\Omega}]$ of \eqref{Eq:HAlpha} in the form $\mathfrak{m}[\eta_{1},\dots,\eta_{N};\eta_{\Omega}]$ with $\eta_{1},\dots,\eta_{N},\eta_{\Omega}$ determined from $\alpha_{1},\dots,\alpha_{N},\alpha_{\Omega}$. 
For that we define the operator ${\mathcal T}:E_{\partial {\mathcal F}} \rightarrow E_{\partial {\mathcal F}}$ by
\begin{equation} \label{Eq:DefT}
{\mathcal T}[\eta_{1},\dots,\eta_{N};\eta_{\Omega}] :=
\left\{ \begin{array}{l}
\sum_{\lambda \in {\mathcal P}_{s}} \widehat{\mathfrak{f}}_{\lambda}[\eta_{\lambda}-\widecheck{\mathfrak g}_{|\partial {\mathcal S}_{\lambda}}]
 \text{ on } \ \partial \widecheck{{\mathcal F}} = \partial {\mathcal S}_{1} \cup \dots \cup \partial {\mathcal S}_{N_{(i)}} \cup \partial \Omega, \medskip \\ 
\sum_{\lambda \in {\mathcal P}_{s} \setminus \{ \nu \}} \widehat{\mathfrak{f}}_{\lambda}[\eta_{\lambda}-\widecheck{\mathfrak g}_{|\partial {\mathcal S}_{\lambda}}]
\ \text{ on } \ \partial {\mathcal S}_{\nu}, \ \text{ for } \ \nu \in {\mathcal P}_{s} , 
\end{array} \right.
\end{equation}
where again $\widecheck{\mathfrak g}=\widecheck{\mathfrak g}[\eta_{1},\dots,\eta_{N_{(i)}};\eta_{\Omega}]$. 
Then
\begin{equation} \label{Eq:DefAlternativeT}
\mathfrak{m}[\eta_{1},\dots,\eta_{N};\eta_{\Omega}] =
\left\{ \begin{array}{l}
(\mbox{Id} + {\mathcal T}) [\eta_{1},\dots,\eta_{N};\eta_{\Omega}]  \ \text{ on } \ \partial \Omega , \\
(\mbox{Id} + {\mathcal T}) [\eta_{1},\dots,\eta_{N};\eta_{\Omega}] +c_{\nu} \ \text{ on } \ \partial {\mathcal S}_{\nu} , \ \nu= 1, \dots, N.
\end{array} \right.
\end{equation}
Now we have the following lemma, where we recall that $\overline{\boldsymbol{\varepsilon}}=(\varepsilon_{N_{(i)}+1},\dots,\varepsilon_{N})$.
\begin{Lemma} \label{Lem:Contraction}
There exists $\varepsilon_{0}>0$  depending only on $\delta$, $\Omega$ and on the shape of the unscaled solids ${\mathcal S}_{\lambda}^1$ such that if $\overline{\boldsymbol{\varepsilon}} \leq \varepsilon_{0}$, then ${\mathcal T}$ is a $\frac{1}{2}$-contraction.
\end{Lemma}
\begin{proof}[Proof of Lemma~\ref{Lem:Contraction}]
The main argument is that the value of ${\mathcal T}[\eta_{1},\dots,\eta_{N};\eta_{\Omega}]$ on a connected component of the boundary, say $\partial {\mathcal S}_{\nu}$, is actually given by a sum of restrictions on  $\partial {\mathcal S}_{\nu}$ of potentials generated on other connected components of the boundary (and the same holds for $\partial \Omega$).
We first see that by Lemma~\ref{Lem:MMH}, $\widecheck{\mathfrak{g}}$  satisfies \eqref{Eq:CheckLambda}. Then we use \eqref{Eq:EstNablaPhiHat}: for $\nu \neq \lambda$, this allows to estimate $\widehat{\mathfrak{f}}_{\lambda}[\eta_{\lambda}-\widecheck{\mathfrak g}_{|\partial {\mathcal S}_{\lambda}}]$ 
on the $\delta$-neighborhood  ${\mathcal V}_{\delta}(\partial {\mathcal S}_{\nu})$ of $\partial {\mathcal S}_{\nu}$ (see \eqref{Def:NuVoisinage}) by
\begin{equation}
\label{lalala}
\| \widehat{\mathfrak{f}}_{\lambda}[\eta_{\lambda}-\widecheck{\mathfrak g}_{|\partial {\mathcal S}_{\lambda}}] \|_{L^{\infty}({\mathcal V}_{\delta}(\partial {\mathcal S}_{\nu}))}
 \leq  C  \varepsilon_{\lambda} \big( \| (\eta_{1}, \dots, \eta_{N_{(i)}},\eta_{\Omega}) \|_{L^\infty(\partial {\mathcal S}_{1} \times \dots \times \partial {\mathcal S}_{N_{(i)}} \times \partial \Omega)} 
+  \| \eta_{\lambda} \|_{L^\infty(\partial {\mathcal S}_{\lambda}) } \big) ,
\end{equation}
and the same holds for ${\mathcal V}_{\delta}(\partial \Omega)$. \par
\noindent
By the definition \eqref{Eq:DefT} of ${\mathcal T}$, we deduce that on $\partial \widecheck{{\mathcal F}} = \partial {\mathcal S}_{1} \cup \dots \cup \partial {\mathcal S}_{N_{(i)}} \cup \partial \Omega$, 
\begin{eqnarray*}
\| {\mathcal T}[\eta_{1},\dots,\eta_{N};\eta_{\Omega}] \|_{L^\infty(\partial \widecheck{{\mathcal F}})} &\leq& C \sum_{\lambda \in {\mathcal P}_{s}}
 \varepsilon_{\lambda} \big(
\| \eta_{\lambda}\|_{L^\infty(\partial {\mathcal S}_{\lambda})}
+ \| (\eta_{1}, \dots, \eta_{N_{(i)}},\eta_{\Omega}) \|_{L^\infty(\partial {\mathcal S}_{1} \times \dots \times \partial {\mathcal S}_{N_{(i)}} \times \partial \Omega)} \big) \\
&\leq& C \left(\sum_{\lambda \in {\mathcal P}_{s} }  \varepsilon_{\lambda} \right)
\| (\eta_{1}, \dots, \eta_{N},\eta_{\Omega}) \|_{L^\infty(\partial {\mathcal S}_{1} \times \dots \times \partial {\mathcal S}_{N} \times \partial \Omega)} ,
\end{eqnarray*}
while on $\partial {\mathcal S}_{\nu}$ for $\nu \in {\mathcal P}_{s}$, we get
\begin{eqnarray*}
\| {\mathcal T}[\eta_{1},\dots,\eta_{N};\eta_{\Omega}] \|_{L^\infty(\partial {\mathcal S}_{\nu})} &\leq& C \sum_{\lambda \in {\mathcal P}_{s} \setminus \{ \nu \}}
 \varepsilon_{\lambda} \big(
\| \eta_{\lambda}\|_{L^\infty(\partial {\mathcal S}_{\lambda})}
+ \| (\eta_{1}, \dots, \eta_{N_{(i)}},\eta_{\Omega}) \|_{L^\infty(\partial {\mathcal S}_{1} \times \dots \times \partial {\mathcal S}_{N_{(i)}} \times \partial \Omega)} \big) \\
&\leq& C \left(\sum_{\lambda \in {\mathcal P}_{s} \setminus \{ \nu \}}  \varepsilon_{\lambda} \right)
\| (\eta_{1}, \dots, \eta_{N},\eta_{\Omega}) \|_{L^\infty(\partial {\mathcal S}_{1} \times \dots \times \partial {\mathcal S}_{N} \times \partial \Omega)} .
\end{eqnarray*}
Hence the operator ${\mathcal T}$ is a $\frac{1}{2}$-contraction if $\overline{\boldsymbol{\varepsilon}}$ is small enough.
\end{proof}
\noindent
Now we consider such an $\overline{\boldsymbol{\varepsilon}}$. From Lemma~\ref{Lem:Contraction} we infer that
$\mbox{Id} + {\mathcal T}$ is invertible.
We deduce the following lemma.
\begin{Lemma} \label{Lem:ApresInversion}
Given $(\alpha_{1},\dots,\alpha_{N};\alpha_{\Omega})$ in $E_{\partial {\mathcal F}}$ we introduce
\begin{equation} \label{Eq:BetaAlpha}
(\beta_{1},\dots,\beta_{N},\beta_{\Omega}) := (\mbox{Id} + {\mathcal T})^{-1}(\alpha_{1},\dots,\alpha_{N},\alpha_{\Omega}).
\end{equation}
Then
\begin{equation} \nonumber 
\mathfrak{H}[\alpha_{1},\dots,\alpha_{N};\alpha_{\Omega}] = \mathfrak{m}[\beta_{1},\dots,\beta_{N};\beta_{\Omega}].
\end{equation}
\end{Lemma}
\begin{proof}[Proof of Lemma~\ref{Lem:ApresInversion}]
From  \eqref{Eq:SysFrakM}, \eqref{Eq:DefAlternativeT} and \eqref{Eq:BetaAlpha}, we see that $\mathfrak{m}[\beta_{1},\dots,\beta_{N};\beta_{\Omega}]$ is the unique solution to \eqref{Eq:HAlpha} corresponding to the boundary data $(\alpha_{1},\dots,\alpha_{N};\alpha_{\Omega})$. \par
\end{proof}
We finish this paragraph by noticing the fact that ${\mathcal T}$ has important regularizing properties. Recall that $\delta$ was introduced at the beginning of Subsection~\ref{Subsec:CPPSM}.
\begin{Lemma} \label{Lem:TTropCool}
Given $\delta >0$, 
there exists $\varepsilon_{0}>0$ such that for all $\overline{\boldsymbol{\varepsilon}}$ with $\overline{\boldsymbol{\varepsilon}} \leq \varepsilon_{0}$, 
for all $k \in \N$, there exists a positive constant $C$ merely depending on $k$, $\delta$, $\Omega$ and on the unscaled solids ${\mathcal S}_{\lambda}^1$ such that for any $(\eta_{1},\dots,\eta_{N};\eta_{\Omega}) \in E_{\partial {\mathcal F}}$, one has
\begin{equation} \nonumber 
\| {\mathcal T}(\eta_{1},\dots,\eta_{N};\eta_{\Omega}) \|_{C^{k,\frac{1}{2}}(\partial {\mathcal F})}  \leq C \| (\eta_{1},\dots,\eta_{N},\eta_{\Omega}) \|_{\infty}.
\end{equation}
\end{Lemma}
\begin{proof}[Proof of Lemma~\ref{Lem:TTropCool}]
We introduce for each $\nu \in \{1,\dots,N\}$ a neighborhood ${\mathcal U}_{\nu}$ of $\partial {\mathcal S}_{\nu}$ of size ${\mathcal O}(\delta)$,
and hence independent of $\varepsilon_{\nu}$.
More precisely, for $\nu \in {\mathcal P}_{(i)}$, we let ${\mathcal U}_{\nu}={\mathcal V}_{\delta/2}({\mathcal S}_{\nu})$ (where we recall the notation \eqref{Def:NuVoisinage}).
For $\nu \in {\mathcal P}_{s}$, we let ${\mathcal U}_{\nu}= B(h_{\nu},\delta/2)$ and we notice that for suitably small $\overline{\boldsymbol{\varepsilon}}$, one has ${\mathcal S}_{\nu} \subset B(h_{\nu},\delta/8)$.
We also introduce some neighborhood ${\mathcal U}'_{\nu}$ of ${\mathcal S}_{\nu}$ depending only on $\delta$ and satisfying $\overline{{\mathcal U}'_{\nu}} \subset {\mathcal U}_{\nu}$:
for instance for $\nu \in {\mathcal P}_{(i)}$, we consider ${\mathcal U}'_{\nu}={\mathcal V}_{\delta/4}({\mathcal S}_{\nu})$ and for $\nu \in {\mathcal P}_{s}$, we let ${\mathcal U}'_{\nu}= B(h_{\nu},\delta/4)$. 
In the same way we introduce the $\delta/2$-neighborhood (respectively $\delta/4$-neighborhood ) ${\mathcal U}_{0}$ (resp. ${\mathcal U}'_{0}$)
of $\partial \Omega$. 
Then by interior elliptic regularity estimates we find a positive constant $C=C(k,{\mathcal U}_{\nu}, {\mathcal U}'_{\nu})$ such that for any harmonic function $f$ on ${\mathcal U}_{\nu}$ one has
\begin{equation*}
\| f \|_{C^{k,\frac{1}{2}}({\mathcal U}'_{\nu})}  \leq C \| f \|_{L^{\infty}({\mathcal U}_{\nu})}.
\end{equation*}
We apply it to $\widehat{\mathfrak{f}}_{\lambda}[\eta_{\lambda}-\widecheck{\mathfrak g}_{|\partial {\mathcal S}_{\lambda}}]$ for $\lambda \neq \nu$ to get a H\"older estimate on ${\mathcal U}'_{\nu}$ and restrict it to $\partial {\mathcal S}_{\nu}$ and
$\partial \Omega$
(which is trivial with the convention \eqref{Eq:NormesHolder}).
Finally we use \eqref{Eq:DefT} and \eqref{lalala}. This ends the proof of Lemma~\ref{Lem:TTropCool}.
\end{proof}
%
%
%
%
%
%
%
%
\subsubsection{Asymptotic behavior for problem \eqref{Eq:HAlpha}}
In this paragraph we study the behavior of the solutions \eqref{Eq:HAlpha} as some of the embedded solids shrink to points.
Let $\overline{\boldsymbol{\varepsilon}}$ satisfy the assumptions of Lemma~\ref{Lem:Contraction}.
We consider a particular case of $\mathfrak{H}[\alpha_{1}, \dots,\alpha_{N};\alpha_{\Omega}],$ when all $\alpha_{\kappa}$ but one are zero and $\alpha_{\Omega}=0$ as well. Let $\kappa \in \{1, \dots, N\}$ and $\alpha_{\kappa} \in C^0 (\partial {\mathcal S}_{\kappa};\R)$. We denote
\begin{equation} \label{Eq:PhiAlpha}
\mathfrak{f}_{\kappa}[\alpha_{\kappa}] := \mathfrak{H}[0,\dots,0, \alpha_{\kappa}, 0,\dots,0;0],
\end{equation}
where $\alpha_{\kappa}$ is on the $\kappa$-th position. 
The first result of this section, concerning the case when the $\kappa$-th solid is small, is the following one. We recall the notation ${\mathcal P}_{s}$ for the set of indices for shrinking solids, see \eqref{Eq:PS}, and the notation \eqref{Def:NuVoisinage} for a $\nu$-neighborhood.
\begin{Proposition} \label{Pro:DirichletPetits}
Let $\delta>0$. There exists $\varepsilon_{0}>0$ such that the following holds. 
There exists a constant $C>0$ depending only on $\delta$, $\Omega$, $k \geq 2$ and the reference solids ${\mathcal S}_{\lambda}^1$, $\lambda=1,\dots,N$,  such that for any
$\overline{\boldsymbol\varepsilon}$ such that $\overline{{\boldsymbol\varepsilon}} \leq \varepsilon_{0}$, for any $\kappa \in {\mathcal P}_{s}$, 
for any ${\bf q} \in {\mathcal{Q}}_{\delta}$, for any $\alpha^\varepsilon \in C^{\infty}(\partial {\mathcal S}^\varepsilon_{\kappa}; \R)$, one has
\begin{gather} \label{Eq:EstPhiHatPhi}
\| \nabla \mathfrak{f}_{\kappa}[\alpha^\varepsilon] - \nabla \widehat{\mathfrak{f}}_\kappa[\alpha^\varepsilon] \|_{L^{\infty}({\mathcal F}^\varepsilon)}
\leq C \varepsilon_{\kappa} \| \alpha^\varepsilon \|_{L^\infty(\partial {\mathcal S}^\varepsilon_{\kappa})}, \\
\nonumber
\big| \mathfrak{f}_{\kappa}[\alpha^\varepsilon] \big|_{C^{k,\frac{1}{2}}({\mathcal V}_{\delta}(\partial \widecheck{\mathcal F}))}
+ \sum_{\lambda \in {\mathcal P}_{s} \setminus \{ \kappa \}} \varepsilon_{\lambda}^{k-\frac{1}{2}}  \big| \mathfrak{f}_{\kappa}[\alpha^\varepsilon] \big|_{C^{k,\frac{1}{2}}({\mathcal V}_{\delta}(\partial{\mathcal S}_{\lambda}^\varepsilon))} \hskip 5cm \\ \hskip 5cm 
\label{Eq:EstPhiHatPhi2}
+ \varepsilon_{\kappa}^{k-\frac{1}{2}}  \big| \mathfrak{f}_{\kappa}[\alpha^\varepsilon] - \widehat{\mathfrak{f}}_\kappa[\alpha^\varepsilon] \big|_{C^{k,\frac{1}{2}}({\mathcal V}_{\delta}(\partial{\mathcal S}_{\kappa}^\varepsilon))}
\leq C \varepsilon_{\kappa} \| \alpha^\varepsilon \|_{L^\infty(\partial {\mathcal S}^\varepsilon_{\kappa})},
\end{gather}
where
$\mathfrak{f}_{\kappa}[\alpha^\varepsilon] \in C^{\infty}(\overline{\mathcal F}^\varepsilon({\bf q}))$ is the unique solution given by \eqref{Eq:PhiAlpha},
$\widehat{\mathfrak{f}}_\kappa[\alpha^\varepsilon] \in C^{\infty}(\R^2 \setminus {\mathcal S}^\varepsilon_\kappa)$ is the unique solution to \eqref{Eq:PhiAlphachapeau}.
\end{Proposition}
Let us highlight that there is no H\"older norm in the right hand side of \eqref{Eq:EstPhiHatPhi2}, as opposed to \eqref{Eq:NablaPhiHatGlobal}. 
\begin{proof}[Proof of Proposition~\ref{Pro:DirichletPetits}] 
First, we fix $\varepsilon_{0}$ so that  Lemma~\ref{Lem:Contraction} and Lemma~\ref{Lem:TTropCool} apply.
We let the $(N+1)$-tuple ${\bf A}$ be 
\begin{equation*}
{\bf A}:=(0,\dots,0, \alpha, 0,\dots,0,0) ,
\end{equation*}
where $\alpha$ is on the $\kappa$-th position and we introduce
\begin{equation} \label{Eq:DefB}
{\bf B}=({\beta}_{1},\dots,{\beta}_{N},{\beta}_{\Omega}):=(I + {\mathcal T})^{-1}({\bf A}).
\end{equation}
Then according to Lemma~\ref{Lem:ApresInversion} we have
\begin{equation} \label{Eq:NotreSolution}
\mathfrak{f}_{\kappa}[\alpha] =\mathfrak{m}[{\bf B}] \ \text{ in } \ \overline{\mathcal F}.
\end{equation}
Now relying on \eqref{Eq:DefFrakM}, we arrive at the formula
\begin{equation} \label{Eq:DiffFFhat}
\mathfrak{f}_{\kappa}[\alpha] - \widehat{\mathfrak{f}}_\kappa[\alpha] = \widecheck{\mathfrak{g}}_{\beta} + \sum_{\lambda \in {\mathcal P}_{s}} \widehat{\mathfrak{f}}_{\lambda}\left[  \widetilde{\beta}_{\lambda} \right] \ \text{ in } \ \overline{\mathcal F}, 
\end{equation}
with
\begin{equation} 
\label{numetoile}
\widecheck{\mathfrak g}_{\beta}:=\widecheck{\mathfrak g}[\beta_{1},\dots,\beta_{N_{(i)}};\beta_{\Omega}] 
\ \text{ and for }  \lambda \in {\mathcal P}_{s} , \ 
\widetilde{\beta}_{\lambda} := 
\left\{ \begin{array}{ll}
\beta_{\lambda} - \widecheck{\mathfrak g}_{\beta|\partial {\mathcal S}_{\lambda}}  \ &\text{ when } \ \lambda \neq \kappa, \\
\beta_{\lambda} - \widecheck{\mathfrak g}_{\beta|\partial {\mathcal S}_{\lambda}} - \alpha \ &\text{ when } \ \lambda = \kappa.
\end{array} \right.
\end{equation}
Our goal is to estimate the right-hand side of \eqref{Eq:DiffFFhat}. A first step is to estimate ${\bf B} - {\bf A}$.
To that purpose we first notice that
\begin{equation} \label{Eq:B-A}
{\bf B} - {\bf A} = - {\mathcal T}({\bf B}) = - {\mathcal T} \circ (I + {\mathcal T})^{-1}({\bf A}) .
\end{equation}
Due to Lemma~\ref{Lem:Contraction}, we have $\| (I + {\mathcal T})^{-1} \|_{{\mathcal L}(C^{0}(\partial {\mathcal F}))} \leq 2$, so in particular we deduce
\begin{equation} \label{Eq:AppliContraction}
\| {\bf B} - {\bf A} \|_{L^{\infty}(\partial {\mathcal F})} 
\leq \| {\mathcal T}({\bf A}) \|_{L^{\infty}(\partial {\mathcal F})}.
\end{equation}
Now when computing ${\mathcal T}({\bf A})$ with \eqref{Eq:DefT}, we see that the function $\widecheck{\mathfrak{g}}$ involved in \eqref{Eq:DefT}  
 and the constants $\widecheck{c}_{\lambda}$ from \eqref{Eq:CorrectionInterieure} are zero because the only non-trivial boundary data $\alpha$ is located on a small solid ${\mathcal S}_{\kappa}$, $\kappa \in {\mathcal P}_{s}$. 
Hence \eqref{Eq:DefT} gives
\begin{equation} \label{Eq:FirstIteration}
{\mathcal T}({\bf A}) =
\left\{ \begin{array}{l}
\widehat{\mathfrak{f}}_{\kappa}[\alpha]  \ \text{ on } \ \partial \Omega \ \text{ and on } \ \partial {\mathcal S}_{\lambda} \ \text{ for } \ \lambda \in \{1, \dots, N \} \setminus \{\kappa \}, \\
0 \text{ on } \ \partial {\mathcal S}_{\kappa}.
\end{array} \right.
\end{equation}
We deduce from \eqref{Eq:AppliContraction}, \eqref{Eq:FirstIteration}, \eqref{Eq:EstNablaPhiHat} and the separation between the connected components of the boundary, 
that
\begin{equation} \label{Eq:EstB}
{\bf B} = {\bf A} + {\mathcal O}(\varepsilon_{\kappa} \| \alpha \|_{L^{\infty}(\partial {\mathcal S}_{\kappa})}) \ \text{ in } \ L^\infty(\partial {\mathcal F}).
\end{equation}
Now we obtain higher order estimates. 
By \eqref{Eq:FirstIteration}, \eqref{Eq:EstNablaPhiHat} and interior elliptic regularity estimates, 
$\| {\mathcal T}(\mathbf{A}) \|_{C^{k,\frac{1}{2}}(\partial {\mathcal F})}  \leq C \varepsilon_{\kappa} \| \alpha \|_{L^{\infty}(\partial {\mathcal S}_{\kappa})}.$
By \eqref{Eq:EstB} and Lemma~\ref{Lem:TTropCool}, 
$\| {\mathcal T}({\bf B}-\mathbf{A}) \|_{C^{k,\frac{1}{2}}(\partial {\mathcal F})}  \leq C \varepsilon_{\kappa} \| \alpha \|_{L^{\infty}(\partial {\mathcal S}_{\kappa})}.$
We deduce
\begin{equation} \label{Eq:EstTB}
\| {\mathcal T}(\mathbf{B}) \|_{C^{k,\frac{1}{2}}(\partial {\mathcal F})}  \leq C \varepsilon_{\kappa} \| \alpha \|_{L^{\infty}(\partial {\mathcal S}_{\kappa})},
\end{equation}
which together with \eqref{Eq:B-A} gives
\begin{equation} \label{Eq:EstB2}
{\bf B} = {\bf A} + {\mathcal O}(\varepsilon_{\kappa} \| \alpha \|_{L^{\infty}(\partial {\mathcal S}_{\kappa})}) \ \text{ in } \ C^{k,\frac{1}{2}}(\partial {\mathcal F}).
\end{equation}
Now the terms in the right-hand side of \eqref{Eq:DiffFFhat} can be estimated as follows.
By \eqref{Eq:EstB2}, the fact that $\mathbf{A}_{i}=0$ for $i=1, \dots, N_{(i)}$, uniform Schauder estimates in $\widecheck{{\mathcal F}}$ (Lemma~\ref{Lem:USE}) and \eqref{numetoile}, 
\begin{equation} \label{Eq:EstCheckLambda}
\| \widecheck{\mathfrak g}_{\beta} \|_{C^{k,\frac{1}{2}}(\widecheck{{\mathcal F}})} \leq  C \varepsilon_{\kappa} \| \alpha \|_{L^{\infty}(\partial {\mathcal S}_{\kappa})}.
\end{equation}
Let us now turn to the estimate of $\widehat{\mathfrak{f}}_{\lambda} \big[ \widetilde{\beta}_{\lambda} \big]$, $\lambda \in {\mathcal P}_{s}$. From ${\bf B} - {\bf A} = ({\beta}_{1} - \delta_{\kappa,1} \alpha ,\dots,{\beta}_{N}-\delta_{\kappa,N} \alpha ,{\beta}_{\Omega})$, \eqref{Eq:EstB2}, \eqref{Eq:EstCheckLambda} and \eqref{numetoile},  we infer that for all $\lambda$ in ${\mathcal P}_{s}$,
$\| \widetilde{\beta}_{\lambda} \|_{C^{k,\frac{1}{2}}(\partial {\mathcal S}_{\lambda})} \leq C \varepsilon_{\kappa} \| \alpha \|_{L^{\infty}(\partial {\mathcal S}_{\kappa})}.$
Recalling the convention \eqref{Eq:NormesHolder} we deduce that
\begin{equation*}
\left\| \, \widetilde{\beta}_{\lambda}
- \frac{1}{|\partial {\mathcal S}_{\lambda}|} \int_{\partial {\mathcal S}_{\lambda}} \widetilde{\beta}_{\lambda} \, \right\|_{L^{\infty}(\partial {\mathcal S}_{\lambda})}
\leq C \varepsilon_{\kappa} \varepsilon_{\lambda} \| \alpha \|_{L^{\infty}(\partial {\mathcal S}_{\kappa})}.
\end{equation*}
Using \eqref{Eq:NablaPhiHatGlobal} on the solid ${\mathcal S}_{\lambda}$ and the fact that the operators $\widehat{\mathfrak{f}}_{\lambda}$ do not see constants we deduce
\begin{equation} \label{Eq:EstAutresPot}
\forall \lambda \in {\mathcal P}_{s}, \ \ \ 
\big\| \, \nabla \widehat{\mathfrak{f}}_{\lambda} \big[  \widetilde{\beta}_{\lambda} \big]  \, \big\|_{L^\infty(\R^2 \setminus {\mathcal S}_{\lambda})}
+ \varepsilon_\lambda^{k-\frac{1}{2}} \big| \, \widehat{\mathfrak{f}}_{\lambda} \big[ \widetilde{\beta}_{\lambda} \big]  \, \big|_{C^{k,\frac{1}{2}}(\R^2 \setminus {\mathcal S}_{\lambda})}  
\leq C \varepsilon_{\kappa} \| \alpha \|_{L^{\infty}(\partial {\mathcal S}_{\kappa})}.
\end{equation}
Then interior regularity for Laplace equation involves that in the $\delta$-neighborhood ${\mathcal V}_{\delta}(\partial {\mathcal F}^\varepsilon \setminus \partial {\mathcal S}_{\lambda})$ of $\partial {\mathcal F}^\varepsilon \setminus \partial {\mathcal S}_{\lambda}$, 
\begin{equation} \label{Eq:EstAutresPot2}
\forall \lambda \in {\mathcal P}_{s}, \ \ \ 
\big| \, \widehat{\mathfrak{f}}_{\lambda} \big[ \widetilde{\beta}_{\lambda} \big]  \, \big|_{C^{k,\frac{1}{2}}({\mathcal V}_{\delta}(\partial {\mathcal F}^\varepsilon \setminus \partial {\mathcal S}_{\lambda}))}  
\leq C \varepsilon_{\kappa} \| \alpha \|_{L^{\infty}(\partial {\mathcal S}_{\kappa})}.
\end{equation}
Now \eqref{Eq:DiffFFhat}, \eqref{Eq:EstCheckLambda}, \eqref{Eq:EstAutresPot} and \eqref{Eq:EstAutresPot2} give \eqref{Eq:EstPhiHatPhi} and 
\begin{equation*}
\big|\mathfrak{f}_{\kappa}[\alpha^\varepsilon] - \widehat{\mathfrak{f}}_\kappa[\alpha^\varepsilon] \big|_{C^{k,\frac{1}{2}}({\mathcal V}_{\delta}(\partial \widecheck{\mathcal F}))}
+
\sum_{\lambda \in {\mathcal P}_{s}} \varepsilon_{\lambda}^{k-\frac{1}{2}}  \big|\mathfrak{f}_{\kappa}[\alpha^\varepsilon] - \widehat{\mathfrak{f}}_\kappa[\alpha^\varepsilon] \big|_{C^{k,\frac{1}{2}}({\mathcal V}_{\delta}(\partial{\mathcal S}_{\lambda}^\varepsilon))}
 \leq C \varepsilon_{\kappa} \| \alpha^\varepsilon \|_{L^\infty(\partial {\mathcal S}^\varepsilon_{\kappa})}.
\end{equation*}
Now we estimate $\big|\widehat{\mathfrak{f}}_\kappa[\alpha^\varepsilon] \big|_{C^{k,\frac{1}{2}}({\mathcal V}_{\delta}(\partial{\mathcal F}^\varepsilon \setminus \partial {\mathcal S}_{\kappa}^\varepsilon))}$ with \eqref{Eq:EstNablaPhiHat} and interior regularity estimate for the Laplace equation to arrive at \eqref{Eq:EstPhiHatPhi2}. This ends the proof of Proposition \ref{Pro:DirichletPetits}.
\end{proof}
\ \par
There is a corresponding result in the situation where the non trivial boundary data is not given on a small solid, but rather on solids of fixed size and on the outer boundary $\partial \Omega$.
\begin{Proposition} \label{Pro:DirichletGros}
Let $\delta>0$ and $k \geq 2$. There exist two positive constants $C$ and $\varepsilon_{0}$ depending only on $\delta$, $\Omega$ and the reference solids ${\mathcal S}_{\lambda}^1$, $\lambda=1,\dots,N$ ($C$ depending moreover on $k$), such that for any $\overline{\boldsymbol\varepsilon}$ with $\overline{{\boldsymbol\varepsilon}} \leq \varepsilon_{0}$, the following holds.
Fix ${\bf q} \in {\mathcal{Q}}_{\delta}$ and consider for each $\kappa \in \{1,\dots,N_{(i)} \}$ a function $\alpha_{\kappa} \in C^{0}(\partial {\mathcal S}_{\kappa}; \R)$, and let $\alpha_{\Omega} \in C^{0}(\partial \Omega; \R)$.
Let
\begin{equation} \label{Eq:DefHalpha}
\mathfrak{H}_{\alpha}:= \mathfrak{H}[\alpha_{1},\dots,\alpha_{N_{(i)}}, 0,\dots,0;\alpha_{\Omega}] \in C^{0}(\overline{\mathcal F}^\varepsilon({\bf q})),
\end{equation}
and $\widecheck{\mathfrak{g}}_\alpha  :=\widecheck{\mathfrak{g}}[\alpha_{1},\dots,\alpha_{N_{(i)}};\alpha_{\Omega}] $ in $C^{\infty}(\widecheck{{\mathcal F}}({\bf q}_{(i)}))$ 
where $\widecheck{\mathfrak{g}}$ is given by \eqref{Eq:CorrectionInterieure}. Then
\begin{equation} \label{Eq:EstPhiCheckg}
\left\| \, \nabla \mathfrak{H}_{\alpha} - \nabla \widecheck{\mathfrak{g}}_\alpha
+ \sum_{\lambda \in {\mathcal P}_{s}} \nabla \widehat{\mathfrak{f}}_{\lambda} \big[ \widecheck{\mathfrak{g}}_{\alpha|\partial {\mathcal S}^\varepsilon_{\lambda}} \big]   \right\|_{L^{\infty}({\mathcal F}^\varepsilon)}
\leq C  |\overline{\boldsymbol\varepsilon}|
\Big( \| \alpha_{\Omega} \|_{L^\infty(\partial \Omega)} +  \sum_{\kappa \in {\mathcal P}_{(i)}} \| \alpha_{\kappa} \|_{L^\infty(\partial {\mathcal S}_{\kappa})} \Big),
\end{equation}
where $\vert\overline{{\boldsymbol\varepsilon}}\vert $ is defined in  \eqref{Eq:TailleTotalePetitsSolides}, and
\begin{equation} \label{Eq:PhiGrosHolder}
\big| \mathfrak{H}_{\alpha} - \widecheck{g}_{\alpha} \big|_{C^{k,\frac{1}{2}}({\mathcal V}_{\delta}(\partial\widecheck{\mathcal F}))}
+
\sum_{\nu \in {\mathcal P}_{s}} \varepsilon_{\nu}^{k-\frac{1}{2}}  \big| \mathfrak{H}_{\alpha} \big|_{C^{k,\frac{1}{2}}({\mathcal V}_{\delta}(\partial{\mathcal S}_{\nu}^\varepsilon))}
\leq C
\Big( \| \alpha_{\Omega} \|_{L^\infty(\partial \Omega)} +  \sum_{\kappa \in {\mathcal P}_{(i)}} \| \alpha_{\kappa} \|_{L^\infty(\partial {\mathcal S}_{\kappa})} \Big).
\end{equation}
Moreover,  uniformly for $\alpha_{1},\dots,\alpha_{N_{(i)}}$ and $\alpha_{\Omega}$ in a bounded set of $C^0$ and for in $q \in {\mathcal Q}_{\delta}$, one has for all $\lambda \in {\mathcal P}_{s}$, as $\varepsilon_{\lambda} \rightarrow 0^+$, 
\begin{multline} \label{Eq:EstCorrectLambda}
\left\| \nabla \widehat{\mathfrak{f}}_{\lambda} \big[ \widecheck{\mathfrak{g}}_{\alpha|\partial {\mathcal S}^\varepsilon_{\lambda}} \big]\right\| _{L^\infty(\R^2 \setminus {\mathcal S}^\varepsilon_{\lambda})}
\text{ is bounded}, \ \ 
\left\| \nabla \widehat{\mathfrak{f}}_{\lambda} \big[ \widecheck{\mathfrak{g}}_{\alpha|\partial {\mathcal S}^\varepsilon_{\lambda}} \big] \right\|_{L^p(\R^2 \setminus {\mathcal S}^\varepsilon_{\lambda})} \longrightarrow 0 \ 
\text{ for }  p <+\infty \\
 \text{ and } \ \left\| \nabla \widehat{\mathfrak{f}}_{\lambda} \big[ \widecheck{\mathfrak{g}}_{\alpha|\partial {\mathcal S}^\varepsilon_{\lambda}} \big] \right\|_{C^k(\{ x \in \overline{\Omega} / d(x,{\mathcal S}^\varepsilon_{\lambda}) \geq c \} ) } \longrightarrow 0 \text{ for any } c>0 \text{ and } k \in \N.
\end{multline}
\end{Proposition}
\begin{proof}[Proof of Proposition~\ref{Pro:DirichletGros}]
We proceed as in the proof as Proposition~\ref{Pro:DirichletPetits}. We introduce
\begin{equation*}
{\bf A}:=(\alpha_{1},\dots,\alpha_{N_{(i)}},0,\dots,0,\alpha_{\Omega}) ,
\end{equation*}
and define ${\bf B}=(\beta_{1},\dots,\beta_{N},\beta_{\Omega})$ again by \eqref{Eq:DefB}. Then Lemma~\ref{Lem:ApresInversion} states that
$\mathfrak{H}_{\alpha} =\mathfrak{m}[{\bf B}] \ \text{ in } \ \overline{\mathcal F}^\varepsilon$.
Here, instead of \eqref{Eq:DiffFFhat}, \eqref{Eq:DefFrakM} allows to write 
\begin{equation} \nonumber 
\mathfrak{H}_{\alpha} = \widecheck{\mathfrak{g}}_{\beta} + \sum_{\lambda \in {\mathcal P}_{s}} \widehat{\mathfrak{f}}_{\lambda}\left[ \beta_{\lambda} - \widecheck{\mathfrak g}_{\beta|\partial {\mathcal S}_{\lambda}} \right]
\ \text{ with } \ 
\widecheck{\mathfrak g}_{\beta}:=\widecheck{\mathfrak g}[\beta_{1},\dots,\beta_{N_{(i)}};\beta_{\Omega}] .
\end{equation}
Consequently 
\begin{equation} \label{Eq:DiffFFhat3}
\mathfrak{H}_{\alpha} 
- \widecheck{\mathfrak{g}}_{\alpha}
+ \sum_{\lambda \in {\mathcal P}_{s}} \widehat{\mathfrak{f}}_{\lambda} \big[ \widecheck{\mathfrak{g}}_{\alpha|\partial {\mathcal S}_{\lambda}} \big]
= \widecheck{\mathfrak{g}}_{\beta} - \widecheck{\mathfrak{g}}_{\alpha} 
+ \sum_{\lambda \in {\mathcal P}_{s}} \widehat{\mathfrak{f}}_{\lambda}\left[ {\beta}_{\lambda}  \right]
+ \sum_{\lambda \in {\mathcal P}_{s}} \widehat{\mathfrak{f}}_{\lambda} \left[ \widecheck{\mathfrak{g}}_{\alpha|\partial {\mathcal S}_{\lambda}} - \widecheck{\mathfrak g}_{\beta|\partial {\mathcal S}_{\lambda}}  \right].
\end{equation}
To establish \eqref{Eq:EstPhiCheckg}, we estimate the right-hand side of \eqref{Eq:DiffFFhat3}, starting with an estimate of ${\bf B}-{\bf A}$.
Instead of \eqref{Eq:FirstIteration}, we now obtain from \eqref{Eq:DefT} that
\begin{equation*}
{\mathcal T}({\bf A}) =
- \sum_{\lambda \in {\mathcal P}_{s}} \widehat{\mathfrak{f}}_{\lambda} \big[ \, \widecheck{\mathfrak{g}}_{\alpha|\partial {\mathcal S}_{\lambda}} \big]
\text{ on } \partial \widecheck{\mathcal F} \text{ and } 
{\mathcal T}({\bf A}) = - \sum_{\lambda \in {\mathcal P}_{s} \setminus \{ \nu \}} \widehat{\mathfrak{f}}_{\lambda} \big[ \, \widecheck{\mathfrak{g}}_{\alpha|\partial {\mathcal S}_{\lambda}} \big]
 \text{ on } \partial {\mathcal S}_{\nu} \text{ for } \nu \in {\mathcal P}_{s}.
\end{equation*}
Again, ${\mathcal T}({\bf A})$ on $\partial {\mathcal S}_{\nu}$ is obtained as traces of harmonic functions generated by non-homogeneous data on boundaries of solids different from ${\mathcal S}_{\nu}$. Now Lemma~\ref{Lem:MMH} involves that
\begin{equation} \label{SQUARE}
\| \widecheck{\mathfrak g}_{\alpha} \|_{L^\infty(\widecheck{\mathcal F})} \leq C \| \mathbf{A} \|_{L^{\infty}(\partial {\mathcal F})},
\end{equation}
where with a slight abuse of notation we have set
$\| \mathbf{A} \|_{L^{\infty}(\partial {\mathcal F})} := \| \alpha_{\Omega} \|_{L^\infty(\partial \Omega)} +  \sum_{\kappa \in {\mathcal P}_{(i)}} \| \alpha_{\kappa} \|_{L^\infty(\partial {\mathcal S}_{\kappa})}.$
%
%
By \eqref{Eq:EstNablaPhiHat} and interior regularity estimates, 
\begin{equation*}
\| {\mathcal T}({\bf A}) \|_{C^{k,\frac{1}{2}}(\partial {\mathcal F})}  \leq C \vert\overline{{\boldsymbol\varepsilon}}\vert 
\| \mathbf{A} \|_{L^{\infty}(\partial {\mathcal F})}.
\end{equation*}
Using \eqref{Eq:AppliContraction} we therefore obtain
\begin{equation*}
{\bf B} - {\bf A} = {\mathcal O}(\vert\overline{{\boldsymbol\varepsilon}}\vert)\| \mathbf{A} \|_{L^{\infty}(\partial {\mathcal F})} \text{ in } L^\infty(\partial {\mathcal F}),
\end{equation*}
in place of \eqref{Eq:EstB}.
Using Lemma~\ref{Lem:TTropCool}, we deduce
\begin{equation*}
\| {\mathcal T}({\bf B} - {\bf A}) \|_{C^{k,\frac{1}{2}}(\partial {\mathcal F})}  \leq C \vert\overline{{\boldsymbol\varepsilon}}\vert \| \mathbf{A} \|_{L^{\infty}(\partial {\mathcal F})}.
\end{equation*}
We arrive at
\begin{equation*}
\| {\mathcal T}({\bf B}) \|_{C^{k,\frac{1}{2}}(\partial {\mathcal F})}  \leq C \vert\overline{{\boldsymbol\varepsilon}}\vert \| \mathbf{A} \|_{L^{\infty}(\partial {\mathcal F})},
\end{equation*}
which replaces \eqref{Eq:EstTB}. Since ${\bf B}={\bf A} - {\mathcal T} ({\bf B})$, 
\begin{equation} \label{Eq:EstB3}
{\bf B} = {\bf A} + {\mathcal O}(\vert\overline{{\boldsymbol\varepsilon}}\vert) \| \mathbf{A} \|_{L^{\infty}(\partial {\mathcal F})} \ \text{ in } \ C^{k,\frac{1}{2}}(\partial {\mathcal F}).
\end{equation}
Then we deduce estimates on the right-hand side of \eqref{Eq:DiffFFhat3}. First by \eqref{Eq:EstB3} and the uniform Schauder elliptic estimates in $\widecheck{\mathcal F}$ for $\delta$-admissible configurations (Lemma~\ref{Lem:USE}), 
\begin{equation} \label{Eq:EstGB}
\| \widecheck{\mathfrak g}_{\beta} - \widecheck{\mathfrak g}_{\alpha}  \|_{C^{k,\frac{1}{2}}(\widecheck{{\mathcal F}})} \leq  C \vert\overline{{\boldsymbol\varepsilon}}\vert \| \mathbf{A} \|_{L^{\infty}(\partial {\mathcal F})}.
\end{equation}
Next, for $\lambda \in {\mathcal P}_{s}$, by \eqref{Eq:EstB3} and the fact that ${\bf A}_{\lambda}=0$ for $\lambda \in {\mathcal P}_{s}$, 
$\left\| {\beta}_{\lambda} \right\|_{C^{k,\frac{1}{2}}(\partial {\mathcal S}_{\lambda})} \leq C \vert\overline{{\boldsymbol\varepsilon}}\vert \| \mathbf{A} \|_{L^{\infty}(\partial {\mathcal F})}$,
and consequently
\begin{equation*}
\left\| {\beta}_{\lambda} - \frac{1}{|\partial {\mathcal S}_{\lambda}|} \int_{\partial {\mathcal S}_{\lambda}} {\beta}_{\lambda} \right\|_{L^{\infty}(\partial {\mathcal S}_{\lambda})} 
\leq C \vert\overline{{\boldsymbol\varepsilon}}\vert \varepsilon_{\lambda} \| \mathbf{A} \|_{L^{\infty}(\partial {\mathcal F})}.
\end{equation*}
All the same from \eqref{Eq:EstGB} we deduce
\begin{equation*}
\left\| \widecheck{\mathfrak g}_{\beta} - \widecheck{\mathfrak g}_{\alpha} - \frac{1}{|\partial {\mathcal S}_{\lambda}|} \int_{\partial {\mathcal S}_{\lambda}} (\widecheck{\mathfrak g}_{\beta} - \widecheck{\mathfrak g}_{\alpha} )\right\|_{L^{\infty}(\partial {\mathcal S}_{\lambda})} 
\leq C \vert\overline{{\boldsymbol\varepsilon}}\vert \varepsilon_{\lambda} \| \mathbf{A} \|_{L^{\infty}(\partial {\mathcal F})}.
\end{equation*}
Hence with \eqref{Eq:NablaPhiHatGlobal} and the fact that the operators $\widehat{\mathfrak{f}}_{\mu}$ do not see constants we deduce that for all $\lambda \in {\mathcal P}_{s}$,
\begin{gather}
\label{Est:Autresf}
\left\| \nabla \widehat{\mathfrak{f}}_{\lambda} [{\beta}_{\lambda}] \right\|_{L^\infty(\R^2 \setminus {\mathcal S}_{\lambda})} 
+ \left\| \nabla \widehat{\mathfrak{f}}_{\lambda} \big[ \widecheck{\mathfrak g}_{\beta|\partial {\mathcal S}_{\lambda}} - \widecheck{\mathfrak g}_{\alpha|\partial {\mathcal S}_{\lambda}} \big] \right\|_{L^\infty(\R^2 \setminus {\mathcal S}_{\lambda})}
\leq C \vert\overline{{\boldsymbol\varepsilon}}\vert \| \mathbf{A} \|_{L^{\infty}(\partial {\mathcal F})} , \\
\label{Est:Autresf2}
\varepsilon_{\lambda}^{k-\frac{1}{2}} \left( \left|  \widehat{\mathfrak{f}}_{\lambda} [{\beta}_{\lambda}] \right\|_{C^{k,\frac{1}{2}}(\R^2 \setminus {\mathcal S}_{\lambda})} 
+ \left|  \widehat{\mathfrak{f}}_{\lambda} \big[ \widecheck{\mathfrak g}_{\beta|\partial {\mathcal S}_{\lambda}} - \widecheck{\mathfrak g}_{\alpha|\partial {\mathcal S}_{\lambda}} \big] \right|_{C^{k,\frac{1}{2}}(\R^2 \setminus {\mathcal S}_{\lambda})} \right)
\leq C \vert\overline{{\boldsymbol\varepsilon}}\vert \| \mathbf{A} \|_{L^{\infty}(\partial {\mathcal F})} .
\end{gather}
Putting together \eqref{Eq:DiffFFhat3}, \eqref{Eq:EstGB} and \eqref{Est:Autresf} we obtain \eqref{Eq:EstPhiCheckg}. \par
Now to get \eqref{Eq:PhiGrosHolder},  we estimate the right-hand side of \eqref{Eq:DiffFFhat3} in $C^{k,\frac{1}{2}}({\mathcal V}_{\delta}(\partial \widehat{\mathcal F}^\varepsilon))$ and in $C^{k,\frac{1}{2}}({\mathcal V}_{\delta}(\partial {\mathcal S}_{\nu}))$ for $\nu \in {\mathcal P}_{s}$. For the first term in \eqref{Eq:DiffFFhat3} we simply use \eqref{Eq:EstGB}. 
We now focus on the two remaining sums. First, we can estimate them in $C^{k,\frac{1}{2}}({\mathcal V}_{\delta}(\partial \widehat{\mathcal F}^\varepsilon))$ thanks to \eqref{Est:Autresf} and local elliptic estimates.
Let us now fix in $\nu \in {\mathcal P}_{s}$ and estimate these two remaining sums of \eqref{Eq:DiffFFhat3} in $C^{k,\frac{1}{2}}({\mathcal V}_{\delta}(\partial {\mathcal S}_{\nu}))$. We first use \eqref{Est:Autresf} and interior elliptic regularity to deduce that
\begin{equation*}
\sum_{\lambda \in {\mathcal P}_{s} \setminus \{\nu\} } \left|  \widehat{\mathfrak{f}}_{\lambda} \big[ \, {\beta}_{\lambda} \big]
\right|_{C^{k,\frac{1}{2}}({\mathcal V}_{\delta}(\partial {\mathcal S}_{\nu}))} +
\sum_{\lambda \in {\mathcal P}_{s} \setminus \{ \nu \}} \left| \widehat{\mathfrak{f}}_{\lambda} \big[ \, \widecheck{\mathfrak g}_{\beta|\partial {\mathcal S}_{\lambda}} - \widecheck{\mathfrak g}_{\alpha|\partial {\mathcal S}_{\lambda}} \big] \right|_{C^{k,\frac{1}{2}}({\mathcal V}_{\delta}(\partial {\mathcal S}_{\nu}))}
\leq C \vert\overline{{\boldsymbol\varepsilon}}\vert \| \mathbf{A} \|_{L^{\infty}(\partial {\mathcal F})} .
\end{equation*}
For the remaining terms corresponding to $\lambda=\nu$, we use \eqref{Est:Autresf2}. Altogether, putting these estimates in \eqref{Eq:DiffFFhat3} we obtain the uniform estimate 
\begin{multline*}
\left| \mathfrak{H}_{\alpha} - \widecheck{\mathfrak{g}}_\alpha + \sum_{\lambda \in {\mathcal P}_{s}} \widehat{\mathfrak{f}}_{\lambda} \big[ \widecheck{\mathfrak{g}}_{\alpha|\partial {\mathcal S}_{\lambda}} \big] \right|_{C^{k,\frac{1}{2}}({\mathcal V}_{\delta}(\partial \widecheck{\mathcal F}))} \\
+
\sum_{\nu \in {\mathcal P}_{s}}
\varepsilon_{\nu}^{k - \frac{1}{2}} \left| \mathfrak{H}_{\alpha} - \widecheck{\mathfrak{g}}_\alpha + \sum_{\lambda \in {\mathcal P}_{s}} \widehat{\mathfrak{f}}_{\lambda} \big[ \widecheck{\mathfrak{g}}_{\alpha|\partial {\mathcal S}_{\lambda}} \big] \right|_{C^{k,\frac{1}{2}}({\mathcal V}_{\delta}(\partial {\mathcal S}_{\nu}))} \leq C \vert\overline{{\boldsymbol\varepsilon}}\vert \| \mathbf{A} \|_{L^{\infty}(\partial {\mathcal F})}.
\end{multline*}
Now using \eqref{SQUARE} and interior regularity estimate, we have uniformly in $q \in {\mathcal Q}_{\delta}$:
\begin{equation} \label{Eq:GcheckSurLesPetits}
\forall \lambda \in {\mathcal P}_{s}, \ \ 
\| \widecheck{\mathfrak{g}}_\alpha \|_{C^{k,\frac{1}{2}}({\mathcal V}_{\delta}(\partial {\mathcal S}_{\lambda}))} \leq C \| \mathbf{A} \|_{L^{\infty}(\partial {\mathcal F})}. 
\end{equation}
This implies in particular $\| \widecheck{\mathfrak{g}}_\alpha - \widecheck{\mathfrak{g}}_\alpha(h_{\lambda}) \|_{L^{\infty}(\partial {\mathcal S}_{\lambda})} \leq C \| \mathbf{A} \|_{L^{\infty}(\partial {\mathcal F})} \varepsilon_{\lambda}$ for $\lambda \in {\mathcal P}_{s}$.  Hence using Proposition~\ref{Pro:DirichletSAM} and the invariance of $\widehat{\mathfrak{f}}_{\lambda}$ with respect to additive constant, we obtain a uniform estimate 
\begin{equation*}
\forall \lambda \in {\mathcal P}_{s}, \ \ 
\varepsilon_{\lambda}^{k - \frac{1}{2}} \left\| \widehat{\mathfrak{f}}_{\lambda} \big[ \widecheck{\mathfrak{g}}_{\alpha|\partial {\mathcal S}_{\lambda}} \big] \right\|_{C^{k,\frac{1}{2}}(\R^2 \setminus {\mathcal S}_{\lambda})} \leq C \| \mathbf{A} \|_{L^{\infty}(\partial {\mathcal F})}.
\end{equation*}
Moreover by \eqref{Eq:EstNablaPhiHat} and interior regularity estimates, one has
\begin{equation*}
\forall \lambda \in {\mathcal P}_{s}, \ \ 
 \left\| \widehat{\mathfrak{f}}_{\lambda} \big[ \widecheck{\mathfrak{g}}_{\alpha|\partial {\mathcal S}_{\lambda}} \big] \right\|_{C^{k,\frac{1}{2}}({\mathcal V}_{\delta}(\partial {\mathcal F}^\varepsilon \setminus \partial {\mathcal S}_{\lambda}^\varepsilon))} \leq C \| \mathbf{A} \|_{L^{\infty}(\partial {\mathcal F})}.
\end{equation*}
This gives \eqref{Eq:PhiGrosHolder}. \par
%
%
We now turn to \eqref{Eq:EstCorrectLambda}. Since \eqref{Eq:EstCorrectLambda} corresponds to a phenomenon that we will meet at different stages of the paper, we encapsulate it in a lemma which establishes the smallness of some correctors on small solids.
\begin{Lemma}
 \label{Lem:SmallCorrectors}
Let $\lambda \in {\mathcal P}_{s}$. Let $\varepsilon_{n} \in (0,1)^\N$, $\varepsilon_{n} \rightarrow 0$. Let $(g_{n})$ a sequence of functions $g_{n}: \partial  {\mathcal S}_{\lambda}^{\varepsilon_{n}} \rightarrow \R$ such that, with our convention on the H\"older spaces,
$\| g_{n} \|_{C^{k,\frac{1}{2}}(\partial {\mathcal S}_{\lambda}^{\varepsilon_{n}})} \leq C$.
Then, as $n \rightarrow +\infty$, 
$\nabla \widehat{f}^\varepsilon_{\lambda}[g_{n}]$ is bounded in $L^\infty(\R^2 \setminus {\mathcal S}_{\lambda}^{\varepsilon_{n}})$, 
$\left\| \nabla \widehat{f}^\varepsilon_{\lambda}[g_{n}] \right\|_{C^k(\{ x \in \overline{\Omega} / d(x,{\mathcal S}_{\lambda}^{\varepsilon_{n}}) \geq c \} ) } \rightarrow 0 $ for any $c>0$ and $ k \in \N$,
 and 
$\nabla \widehat{f}^\varepsilon_{\lambda}[g_{n}] \rightarrow 0 $ in $L^p(\Omega \setminus {\mathcal S}_{\lambda}^{\varepsilon_{n}}), \ p < +\infty$.
\end{Lemma}
\begin{proof}[Proof of Lemma~\ref{Lem:SmallCorrectors}]
We first observe  that up to an additional constant on $\partial {\mathcal S}^{\varepsilon_{n}}_{\lambda}$, one has $\| g_{n} \|_{L^{\infty}( \partial {\mathcal S}_{\lambda}^{\varepsilon_{n}})} \leq C \varepsilon_{n}$.
Then the boundedness of $\nabla \widehat{f}_{\lambda}[g_{n}]$ in $L^\infty(\R^2 \setminus {\mathcal S}^{\varepsilon_{n}}_{\lambda})$
 is a consequence of \eqref{Eq:NablaPhiHatGlobal}. Moreover the second part of the lemma follows from \eqref{Eq:EstNablaPhiHat}. The third assertion is a consequence of the first two.
\end{proof}
Now \eqref{Eq:EstCorrectLambda} is a direct consequence of Lemma~\ref{Lem:SmallCorrectors} and of \eqref{Eq:GcheckSurLesPetits}.  This ends the proof of Proposition~\ref{Pro:DirichletGros}.
\end{proof}
\begin{Remark} \label{Rem:GrosDoncPetit}
Note that, since $\varepsilon_{\kappa}=1$ for $\kappa \in {\mathcal P}_{(i)}$, Estimate \eqref{Eq:EstPhiHatPhi} is also valid in this case. Indeed due to  \eqref{Eq:EstPhiCheckg}-\eqref{Eq:EstCorrectLambda} and \eqref{Eq:NablaPhiHatGlobal}
we see that
that $\nabla \mathfrak{f}_{\kappa}[\alpha^\varepsilon]$ and $\nabla \widehat{\mathfrak{f}}_\kappa[\alpha^\varepsilon]$ are both of size ${\mathcal O}(\| \alpha^\varepsilon \|)$.
\end{Remark}
%

%
\subsubsection{Shape derivatives of  potentials solving Dirichlet problems}
\label{SSSec:ShapeDeriv}
In this paragraph, we estimate the shape derivatives of  potentials solving Dirichlet problems. This will be useful to estimate the time-derivative of some velocity fields in forthcoming paragraphs.
We refer to \cite{HP,SZ} for general references on shape differentiation.  

Let us first recall a way to write these shape derivatives. 
We consider a reference configuration ${\bf  \overline{q}} $ in ${\mathcal Q}$. 
Given $\mu \in \{1,\dots,N\}$, $m \in \{1,2,3\}$ and $p_{\mu}^* = (\ell_{\mu}^*,\omega_{\mu}^*) \in \R^3$, we define $h_{\mu}(t)=\overline{h}_{\mu} + t \ell^*_{\mu}$ and consider in $\R^2$ a smooth time-dependent vector field such that
$\xi_\mu^*(t,x)=\ell_\mu^* + \omega_{\mu}^*(x-h_{\mu}(t))^\perp$ in a neighborhood of $\partial {\mathcal S}_{\mu}({\bf  \overline{q}})$ and $\xi_\mu^*(x)=0$ in a neighborhood of $\partial {\mathcal F}({\bf  \overline{q}}) \setminus \partial {\mathcal S}_{\mu}({\bf  \overline{q}})$. We associate then the corresponding flow $(s,x) \mapsto T^*_\mu(s,x)$ (for $s$ small and $x \in {\mathcal F}({\bf  \overline{q}})$) that satisfies
\begin{equation*}
\frac{\partial T^*_{\mu}}{\partial s}(s,x)=\xi^*_\mu(s,T^*_\mu(s,x)), \ \ T^*_\mu(0,x)=x.
\end{equation*}
For small $s$, $T^*_{\mu}(s,\cdot)$ sends ${\mathcal F}(\overline{q})$ into ${\mathcal F}(\overline{q}+ s {\bf p}^*_{\mu})$, where we denote by ${\bf p}_{\mu}^* \in \R^{3N}$ the vector given by ${\bf p}_{\mu}^* =(\delta_{\kappa \mu} p_{\mu}^*)_{\kappa=1\dots N}$. Then the shape derivative of a potential $\varphi = \varphi({\bf q},x)$ (defined and regular on $\bigcup_{{\bf q} \in {\mathcal Q}} \{{\bf q} \} \times \overline{{\mathcal F}}({\bf q})$) with respect to $q_{\mu}$ is then obtained as
\begin{equation*}
\frac{\partial \varphi}{\partial q_{\mu}}({\bf q},x) \cdot p^*_{\mu} = \frac{d}{ds} \varphi({\bf  \overline{q}} + s {\bf p}_{\mu}^*, x )_{\big|s=0} 
= \frac{d}{ds} \varphi({\bf  \overline{q}} + s {\bf p}_{\mu}^*, T^*_{\mu}(s,x))_{\big|s=0}
- \frac{\partial \varphi}{\partial x}({\bf  \overline{q}},x) \cdot \xi^*_{\mu}(0,x).
\end{equation*}
This is actually independent of the choice of the family of diffeomorphisms $T^*_{\mu}(s,\cdot):{\mathcal F}(\overline{q}) \rightarrow {\mathcal F}(\overline{q}+ s {\bf p}^*_{\mu})$ as long as $T^*_{\mu}(0,\cdot)=\mbox{Id}$, $\partial_{s} T^*_{\mu}(0,\cdot)=\xi^*_{\mu}(0,x)$ on $\partial {\mathcal S}_{\mu}({\bf  \overline{q}})$ and $\partial_{s} T^*_{\mu}(0,\cdot)=0$ on $\partial {\mathcal F}({\bf  \overline{q}}) \setminus \partial {\mathcal S}_{\mu}({\bf  \overline{q}})$. We set 
\begin{equation*}
\displaystyle \frac{\partial \varphi}{\partial q_{\mu,m}}:= \frac{\partial \varphi}{\partial q_{\mu}} \cdot e_m,
\end{equation*}
where $(e_1 , e_2 ,e_3)$ is the canonical basis of $\R^3$.  \par
%
%
%
\begin{Lemma} \label{Lem:SD}
Consider a regular family of functions $(\Phi({\bf q},\cdot))_{{\bf q} \in {\mathcal Q}}$, with $\Phi({\bf q},\cdot):{\mathcal F}({\bf q}) \rightarrow \R$ satisfying
$-\Delta \Phi({\bf q},\cdot) =0 $ in $ {\mathcal F}({\bf q})$ and 
$\Phi({\bf q},\cdot) = \alpha({\bf q},\cdot) $ on $ \partial {\mathcal F}({\bf q})$,
where $\alpha$ is a smooth function on $\bigcup_{{\bf q} \in {\mathcal Q}} \{{\bf q} \} \times \partial {\mathcal F}({\bf q})$.
Then for $\mu \in \{1,\dots,N\}$ and $m \in \{1,2,3\}$ the shape derivative $\frac{\partial \Phi({\bf q},\cdot)}{\partial q_{\mu,m}}$ is the solution to the system
\begin{equation} \label{Eq:SysSD}
\left\{ \begin{array}{l}
\displaystyle -\Delta \frac{\partial \Phi({\bf q},\cdot)}{\partial q_{\mu,m}} = 0 \ \text{ in } \ {\mathcal F}({\bf q}), \medskip \\
\displaystyle \frac{\partial \Phi({\bf q},\cdot)}{\partial q_{\mu,m}} 
= \frac{\partial \alpha({\bf q},\cdot)}{\partial q_{\mu,m}} + \left(\frac{\partial \alpha({\bf q},\cdot)}{\partial x} - \frac{\partial \Phi({\bf q},\cdot)}{\partial x}\right) \cdot n \,K_{\mu,m} \ \text{ on } \  \partial {\mathcal F}({\bf q}).
\end{array} \right.
\end{equation}
\end{Lemma}
\begin{Remark}
Note that the material derivative $\frac{\partial \alpha({\bf q},\cdot)}{\partial q_{\mu,m}} + \frac{\partial \alpha({\bf q},\cdot)}{\partial x} \cdot n K_{\mu,m}$
is well-defined for functions $\alpha$ defined on the boundary $\bigcup_{{\bf q} \in {\mathcal Q}} \{{\bf q} \} \times \partial {\mathcal F}({\bf q})$ in the $({\bf q},x)$ plane, because the vector $(\delta_{\mu,m},(\xi_{\mu,m} \cdot n)n)$ is tangent to it, where $\delta_{\mu,m} $ denotes the vector in $\R^{3N}$ for which only the coordinate corresponding to $(\mu,m)$ is nonzero and is equal to $1$. 
 Alternatively, we may smoothly extend $\alpha$ in $\bigcup_{{\bf q} \in {\mathcal Q}} \{{\bf q} \} \times \overline{{\mathcal F}({\bf q})}$ and define the partial derivatives with respect to $q_{\mu,m}$ and $x$ independently.
\end{Remark}
\begin{proof}[Proof of Lemma~\ref{Lem:SD}]
That $\frac{\partial \Phi({\bf q},\cdot)}{\partial q_{\mu,m}}$ is harmonic in ${\mathcal F}({\bf q})$ is just a matter of commuting derivatives. 
For what concerns the boundary condition, we use that $\Phi({\bf q},x) = \alpha({\bf q},x)$ on $\partial {\mathcal F}({\bf q})$ to infer that for any ${p}_{\mu}^* \in \R^3$,
$\Phi(\overline{q}+s{\bf p}_{\mu}^*,T^*_{\mu}(s,x)) = \alpha \big( \overline{q} + s {\bf p}_{\mu}^*, T^*_{\mu}(s,x) \big)$ for small $s $ and $x \in \partial {\mathcal F}(\overline{q})$,
where as before ${\bf p}_{\mu}^* =(\delta_{\kappa \mu} p_{\mu}^*)_{\kappa=1\dots N}$.
Differentiating with respect to $s$, we deduce
\begin{equation*}
\frac{\partial \Phi({\bf q},\cdot)}{\partial q_{\mu}} \cdot p_{\mu}^* + \frac{\partial \Phi({\bf q},\cdot)}{\partial x} \cdot \xi_{\mu}^* = \frac{\partial \alpha}{\partial q_{\mu}} \cdot p_{\mu}^* + \frac{\partial \alpha}{\partial x} \cdot \xi_{\mu}^*  \ \text{ on } \  \partial {\mathcal F}({\bf q}).
\end{equation*}
It follows that
\begin{equation} \label{Eq:BCSDbis}
\frac{\partial \Phi({\bf q},\cdot)}{\partial q_{\mu}} \cdot p_{\mu}^*  = \frac{\partial \alpha}{\partial q_{\mu}} \cdot p_{\mu}^* + \frac{\partial [\alpha - \Phi({\bf q},\cdot)]}{\partial x}  \cdot \xi_{\mu}^*  \ \text{ on } \  \partial {\mathcal F}({\bf q}). 
\end{equation}
It remains to notice that since $\Phi({\bf q},x) = \alpha({\bf q},x)$ on the boundary, the gradient of $\alpha(\cdot) - \Phi({\bf q},\cdot)$ with respect to $x$ on the boundary is normal. With
$\xi^*_\mu \cdot  n = \sum_{m=1}^3 p^*_{\mu,m} K_{\mu,m}$, we reach the conclusion.
\end{proof}
The equivalent of Lemma~\ref{Lem:SD} holds for the variant of the Dirichlet problem that we considered above.
\begin{Corollary} \label{Cor:SD}
Consider a smooth function $\alpha$ on $\bigcup_{{\bf q} \in {\mathcal Q}}  \{{\bf q} \} \times \partial {\mathcal F}({\bf q})$ and 
a regular family of functions $(\widetilde{\Phi}({\bf q},\cdot))_{{\bf q} \in {\mathcal Q}}$, with $\widetilde{\Phi}({\bf q},\cdot):{\mathcal F}({\bf q}) \rightarrow \R$ and a regular family of constants $(c_{1}({\bf q}), \dots, c_{N}({\bf q}))_{{\bf q} \in {\mathcal Q}}$ which are solution to
\begin{equation*}
\left\{ \begin{array}{l}
-\Delta \widetilde{\Phi}({\bf q},\cdot) =0 \text{ in } {\mathcal F}({\bf q}), \\
\widetilde{\Phi}({\bf q},\cdot) = \alpha({\bf q},\cdot) + c_{\lambda}({\bf q}) \text{ on } \partial {\mathcal S}_{\lambda}({\bf q}), \ \ \forall \lambda \in \{1,\dots,N\},  \\
\widetilde{\Phi}({\bf q},\cdot) = \alpha({\bf q},\cdot)  \text{ on } \partial \Omega, \\
\int_{\partial {\mathcal S}_{\lambda}} \partial_{n} \widetilde{\Phi}({\bf q},x) \, ds = 0, \ \ \forall \lambda \in \{1,\dots,N\} .
\end{array} \right.
\end{equation*}
Then for $\mu \in \{1,\dots,N\}$ and $m \in \{1,2,3\}$ the shape derivative $\displaystyle\frac{\partial \widetilde{\Phi}({\bf q},\cdot)}{\partial q_{\mu,m}}$ is the solution to the system
\begin{equation} \label{Eq:SD}
\left\{ \begin{array}{l}
\displaystyle -\Delta \left(\frac{\partial \widetilde{\Phi}({\bf q},\cdot)}{\partial q_{\mu,m}}\right) = 0 \text{ in } {\mathcal F}({\bf q}), \medskip \\
\displaystyle \frac{\partial \widetilde{\Phi}({\bf q},\cdot)}{\partial q_{\mu,m}} = \frac{\partial \alpha({\bf q},\cdot)}{\partial q_{\mu,m}} + \left(\frac{\partial \alpha({\bf q},\cdot)}{\partial x} - \frac{\partial \widetilde{\Phi}({\bf q},\cdot)}{\partial x}\right) \cdot n \,K_{\mu,m} + c'_{\lambda}({\bf q}) \text{ on } \partial {\mathcal S}_{\lambda}({\bf q}), \ \ \forall \lambda \in \{1,\dots,N\}, \medskip \\
\displaystyle \frac{\partial \widetilde{\Phi}({\bf q},\cdot)}{\partial q_{\mu,m}} = \frac{\partial \alpha({\bf q},\cdot)}{\partial q_{\mu,m}} \text{ on } \partial \Omega, \medskip \\
\displaystyle \int_{\partial {\mathcal S}_{\lambda}} \partial_{n} \left(\frac{\partial \widetilde{\Phi}({\bf q},\cdot)}{\partial q_{\mu,m}}\right) \, ds = 0, \ \ \forall \lambda \in \{1,\dots,N\},
\end{array} \right.
\end{equation}
for some constants $c'_{1}({\bf q}),\dots,c'_{N}({\bf q})$.
\end{Corollary}
\begin{proof}[Proof of Corollary~\ref{Cor:SD}]
We check the validity of the various equations in \eqref{Eq:SD}. As for Lemma~\ref{Lem:SD}, the first equation is obtained by commuting derivatives with respect to $x$ and ${\bf q}$.
To obtain the second equation, we observe that the shape derivative of a constant function with respect to $x$ on $\partial {\mathcal S}_{\lambda}$ (for each ${\bf q}$) is a constant function on $\partial {\mathcal S}_{\lambda}$. Let us highlight that the regularity with respect to ${\bf q}$ is a consequence of the construction and of the regularity for the usual Dirichlet problem. 
The third equation is trivial. Finally we see that the flux of $\frac{\partial \widetilde{\Phi}({\bf q},\cdot)}{\partial q_{\mu,m}}$  across $\partial {\mathcal S}_{\lambda}$ for $\lambda \neq \mu$ and across $\partial \Omega$ is zero, since these components of the boundary are fixed and the flux of $\widetilde{\Phi}({\bf q},\cdot)$ across them is zero for all ${\bf q}$. Considering that $\frac{\partial \widetilde{\Phi}({\bf q},\cdot)}{\partial q_{\mu,m}}$ is harmonic and using the divergence theorem, it follows that the flux across $\partial {\mathcal S}_\mu$ of $\frac{\partial \widetilde{\Phi}({\bf q},\cdot)}{\partial q_{\mu,m}}$ is zero as well. \par
\end{proof}
\begin{Remark} \label{Rem:BCSD}
In both \eqref{Eq:SysSD} and \eqref{Eq:SD}, we may write
\begin{equation} \nonumber 
\frac{\partial \alpha({\bf q},\cdot)}{\partial q_{\mu,m}} + \left(\frac{\partial \alpha({\bf q},\cdot)}{\partial x} - \frac{\partial \Phi({\bf q},\cdot)}{\partial x}\right) \cdot n \,K_{\mu,m} =
\frac{\partial \alpha({\bf q},\cdot)}{\partial q_{\mu,m}} + \left(\frac{\partial \alpha({\bf q},\cdot)}{\partial x} - \frac{\partial \Phi({\bf q},\cdot)}{\partial x}\right) \cdot \xi_{\mu,m}.
\end{equation}
This is just a matter of stoping the computation at \eqref{Eq:BCSDbis}, or of keeping in mind that, since $\alpha({\bf q},\cdot) - \Phi({\bf q},\cdot)$ is constant on the boundary, its tangential derivative is zero.
\end{Remark}
%

%
%
%
%
%
\subsubsection{Transposing to the Neumann problem}
Let us now describe how the analysis of the paragraphs above can be transposed to the Neumann problem. Given $\kappa \in \{1,\dots,N\}$, $q \in {\mathcal{Q}}_{\delta}$ and $\beta \in C^{\infty}(\partial {\mathcal S}_{\kappa}; \R)$ such that
\begin{equation} \label{Eq:NeumannIntNulle}
\int_{\partial {\mathcal S}_{\kappa}} \beta(x) \, ds(x)=0,
\end{equation}
we consider  the solution $\mathfrak{f}^{\mathcal N}_{\kappa}[\beta] \in C^{\infty}({\mathcal F}({\bf q}))$ (unique up to an additive constant) of the Neumann problem
\begin{equation} \label{Eq:PhiBeta}
\left\{ \begin{array}{l}
\Delta \mathfrak{f}^{\mathcal N}_{\kappa}[\beta] = 0 \ \text{ in } \ {\mathcal F}({\bf q}), \\
\partial_{n} \mathfrak{f}^{\mathcal N}_{\kappa}[\beta] = 0 \ \text{ on } \ \partial {\mathcal F}({\bf q}) \setminus \partial {\mathcal S}_{\kappa}, \\
\partial_{n} \mathfrak{f}^{\mathcal N}_{\kappa}[\beta] = \beta  \ \text{ on } \ \partial {\mathcal S}_{\kappa},
\end{array} \right.
\end{equation}
and $\widehat{\mathfrak{f}^{\mathcal N}_{\kappa}}[\beta] \in C^{\infty}(\R^2 \setminus {\mathcal S}_{\kappa})$ be the solution (unique up to an additive constant) of the {\it standalone} Neumann problem 
\begin{equation} \nonumber 
\left\{ \begin{array}{l}
\Delta \widehat{\mathfrak{f}}^{\mathcal N}_{\kappa}[\beta] = 0 \ \text{ in } \ \R^2 \setminus {\mathcal S}_{\kappa}, \\
\nabla \widehat{\mathfrak{f}}^{\mathcal N}_{\kappa}[\beta](x) \longrightarrow 0 \ \text{ as } \ |x| \longrightarrow +\infty , \\
\partial_{n} \widehat{\mathfrak{f}}^{\mathcal N}_{\kappa}[\beta] = \beta  \ \text{ on } \ \partial {\mathcal S}_{\kappa}.
\end{array} \right.
\end{equation}
Condition \eqref{Eq:NeumannIntNulle} allows to write the function $\beta$ as
\begin{equation*}
\beta = \partial_{\tau} {\mathcal B}.
\end{equation*}
Then the following result is elementary to check.
\begin{Lemma} \label{Lem:Neumann}
One has the correspondence
$\nabla {\mathfrak{f}^{\mathcal N}_{\kappa}}[\beta] = \nabla^\perp \mathfrak{f}_{\kappa}[{\mathcal B}]$ and 
$\nabla \widehat{\mathfrak{f}}^{\mathcal N}_{\kappa}[\beta] = \nabla^\perp \widehat{\mathfrak{f}}_{\kappa}[{\mathcal B}].$
In particular, one can apply Proposition \ref{Pro:DirichletSAM} to $\widehat{\mathfrak{f}}^{\mathcal N}_{\kappa}[\beta]$ and Propositions~\ref{Pro:DirichletPetits} and \ref{Pro:DirichletGros}  and ${\mathfrak{f}^{\mathcal N}_{\kappa}}[\beta]$ with $\| {\mathcal B} \|_{L^\infty(\partial {\mathcal S}^\varepsilon_{\kappa})} ={\mathcal O} (\varepsilon_{\kappa} \| \beta \|_{L^\infty(\partial {\mathcal S}^\varepsilon_{\kappa})})$ in the right-hand side in place of $\| \alpha^\varepsilon\|_{L^\infty(\partial {\mathcal S}^\varepsilon_{\kappa})}$ .
\end{Lemma}
Of course, in the same way, we can consider the Neumann counterpart of $\mathfrak{H}$ defined in \eqref{Eq:HAlpha}, say $\mathfrak{H}^{\mathcal N}[\beta_{1},\dots,\beta_{N};\beta_{\Omega}]$, and in the same way obtain the correspondence with $\mathfrak{H}[{\mathcal B}_{1},\dots,{\mathcal B}_{N};{\mathcal B}_{\Omega}]$ where ${\mathcal B}_{1}$, \dots, ${\mathcal B}_{N}$ and ${\mathcal B}_{\Omega}$ are primitives of $\beta_{1}$, \dots, $\beta_{N}$ and $\beta_{\Omega}$ on $\partial {\mathcal S}_{1}$, \dots, $\partial {\mathcal S}_{N}$ and $\partial \Omega$, respectively. \par
In the sequel we will use mainly the case of the Neumann problem.
%
%
%
%
%
%
%
\subsection{Estimates of the Kirchhoff potentials}
\label{Subsec:EKP}
In this paragraph we apply the above results in the case of the Kirchhoff potentials defined in  \eqref{Kir} and study their shape derivatives as well.
\subsubsection{The Kirchhoff potentials}
We first recall several properties of the standalone Kirchhoff potentials proved for instance in  \cite{GLS2}.
\begin{Lemma} \label{Lem:StandaloneKirchhoff}
The standalone Kirchhoff potentials $\widehat{\varphi}_{\kappa,k}^\varepsilon$, $\kappa \in \{1,\dots,N\}$, $k \in \{1,\cdots,5\}$, have the following properties:
\begin{eqnarray}
\nonumber
&\bullet& \text{for fixed } q_{\kappa}, \ 
\widehat{\varphi}_{\kappa,k}^\varepsilon(x-h_{\kappa}) = \varepsilon_{\kappa}^{1+\delta_{k\geq3}} \widehat{\varphi}_{\kappa,k}^1 \left(\frac{x-h_{\kappa}}{\varepsilon_{\kappa}}\right) \\
\label{Eq:ScalePhi}
& \ &  \hskip 5.5cm \text{ and }\  \nabla \widehat{\varphi}_{\kappa,k}^\varepsilon(x- h_{\kappa}) = \varepsilon_{\kappa}^{\delta_{k\geq3}} \nabla \widehat{\varphi}_{\kappa,k}^1 \left(\frac{x-h_{\kappa}}{\varepsilon_{\kappa}}\right), \\
\label{Eq:BehaviourPhi1}
&\bullet&
\nabla \widehat{\varphi}^\varepsilon_{\kappa,k}(x) = {\mathcal O}\left(\frac{\varepsilon_{\kappa}^{2+\delta_{k\geq3}}}{|x-h_{\kappa}|^2}\right) \text{ at infinity,} \\
\label{Eq:EstStandaloneKirchhoff}
&\bullet&  \varepsilon_{\kappa}^{-\delta_{k\geq3}} \nabla \widehat{\varphi}_{\kappa,k}^\varepsilon \text{ is bounded in }  \R^2 \setminus {\mathcal S}^\varepsilon_{\kappa}
\ \text{ and } \  
\widehat{\varphi}_{\kappa,k}^\varepsilon = {\mathcal O}(\varepsilon_{\kappa}^{1+\delta_{k\geq3}}) \text{ on } \partial {\mathcal S}_{\kappa}^\varepsilon.
\end{eqnarray}
\end{Lemma}
\begin{Remark}
It is elementary to check that given $q_{\kappa}$, we recover the $\kappa$-th standalone Kirchhoff potentials at $q_{\kappa}$ from their equivalent at the basic position through
\begin{gather*}
\begin{pmatrix} \widehat{\varphi}_{\kappa,1}(q_{\kappa},h_{\kappa} + R(\vartheta_{\kappa})x) \\ \widehat{\varphi}_{\kappa,2}(q_{\kappa},h_{\kappa} + R(\vartheta_{\kappa})x) \end{pmatrix}
= R(-\vartheta_{\kappa}) \begin{pmatrix} \widehat{\varphi}_{\kappa,1}(0,x) \\ \widehat{\varphi}_{\kappa,2}(0,x) \end{pmatrix}
,\quad  \widehat{\varphi}_{\kappa,3}(q_{\kappa},h_{\kappa} + R(\vartheta_{\kappa})x) =\widehat{\varphi}_{\kappa,3}(0,x) ,
\\ \text{ and } \  \begin{pmatrix} \widehat{\varphi}_{\kappa,4}(q_{\kappa},h_{\kappa} + R(\vartheta_{\kappa})x) \\ \widehat{\varphi}_{\kappa,5}(q_{\kappa},h_{\kappa} + R(\vartheta_{\kappa})x) \end{pmatrix}
= R(2\vartheta_{\kappa}) \begin{pmatrix} \widehat{\varphi}_{\kappa,4}(0,x) \\ \widehat{\varphi}_{\kappa,5}(0,x)
\end{pmatrix} .
\end{gather*}
Consequently, all the estimates on the standalone Kirchhoff potentials are independent of the position $q_{\kappa}$.
\end{Remark}
We have the following first statement regarding the behavior of the Kirchhoff potentials $\varphi_{\kappa,k}$ in ${\mathcal F}^\varepsilon$ for small values of $\varepsilon_{\kappa}$.
\begin{Proposition}
\label{Pro:ExpKirchhoff}
For $\delta>0$, there exists $\varepsilon_{0}>0$ depending only on $\delta$, $\Omega$ and the shape of the reference solids ${\mathcal S}_{\lambda}^1$, $\lambda=1,\dots,N$, such that for any $\overline{\boldsymbol\varepsilon}$ with $\overline{{\boldsymbol\varepsilon}} \leq \varepsilon_{0}$, the following holds. 
Let $\kappa \in \{1,\dots,N\}$, $k \in \{1,\cdots,5\}$ and $\ell \in \N \setminus \{0,1\}$. For some constant $C>0$ independent of $\overline{{\boldsymbol\varepsilon}}$,
the following holds 
uniformly for ${\bf q} \in {\mathcal Q}_{\delta}$:
\begin{gather}
\label{Eq:ExpKirchhoff1}
\| \nabla \varphi_{\kappa,k} - \nabla \widehat{\varphi}_{\kappa,k} \|_{L^\infty({\mathcal F}^\varepsilon({\bf q}))} \leq C \varepsilon_{\kappa}^{2+\delta_{k\geq3}}, \\
\label{Eq:ExpKirchhoff1a}
\big| \varphi_{\kappa,k} \big|_{C^{\ell,\frac{1}{2}}({\mathcal V}_{\delta}(\partial \widecheck{\mathcal F}))} 
+ \sum_{\lambda \in {\mathcal P}_{s} \setminus \{\kappa\}} \varepsilon_{\lambda}^{\ell-\frac{1}{2}} \big| \varphi_{\kappa,k} \big|_{C^{\ell,\frac{1}{2}}({\mathcal V}_{\delta}(\partial {\mathcal S}_{\lambda}))} + \, \varepsilon_{\kappa}^{\ell-\frac{1}{2}} \big| \varphi_{\kappa,k} - \widehat{\varphi}_{\kappa,k} \big|_{C^{\ell,\frac{1}{2}}({\mathcal V}_{\delta}(\partial {\mathcal S}_{\kappa}))} 
\leq C \varepsilon_{\kappa}^{2+\delta_{k\geq3}}, \\
\label{Eq:ExpKirchhoff1Bis}
\| \nabla \varphi_{\kappa,k} \|_{L^{\infty}({\mathcal F}^\varepsilon({\bf q}))} \leq C \varepsilon_{\kappa}^{\delta_{k\geq3}} \text{ and } 
\nabla \varphi_{\kappa,k}(x) = {\mathcal O} \left(\frac{\varepsilon_{\kappa}^{2+\delta_{k\geq3}}}{|x-h_{\kappa}|^2}\right) \text{ for } x \in {\mathcal F}^\varepsilon(q)
\text{ s.t. } |x-h_{\kappa}| \geq C  \varepsilon_{\kappa} ,
\end{gather}
and one has, up to an additional constant on each connected component of the boundary,
\begin{equation} \label{Eq:ExpKirchhoff2}
\varphi_{\kappa,k} = \left\{ \begin{array}{l}
\displaystyle  {\mathcal O}(\varepsilon_{\kappa}^{2+\delta_{k\geq3}}) \ \text{ on } \ \partial \Omega, \medskip \\
\displaystyle {\mathcal O}(\varepsilon_{\kappa}^{2+\delta_{k\geq3}} \varepsilon_{\mu}) \ \text{ on } \ \partial {\mathcal S}_{\mu} \ \text{ if } \ \mu \neq \kappa, \medskip \\
\displaystyle \widehat{\varphi}_{\kappa,k} + {\mathcal O}(\varepsilon_{\kappa}^{3+\delta_{k\geq3}})
= {\mathcal O}(\varepsilon_{\kappa}^{1+\delta_{k\geq3}}) \ \text{ on } \ \partial {\mathcal S}_{\kappa}.
\end{array} \right. 
\end{equation}
\end{Proposition}
\begin{proof}[Proof of Proposition~\ref{Pro:ExpKirchhoff}]
We use Lemma~\ref{Lem:Neumann} with $\beta = K_{\kappa,k}$, hence we may apply to it Proposition~\ref{Pro:DirichletPetits} if $\kappa \in  {\mathcal P}_{(i)} $ and Remark~\ref{Rem:GrosDoncPetit} otherwise.
Since $\| K_{\kappa,k} \|_{L^\infty(\partial {\mathcal S}_{\kappa})} = {\mathcal O}(\varepsilon_{\kappa}^{\delta_{k\geq3}})$, we obtain from \eqref{Eq:EstPhiHatPhi} and \eqref{Eq:EstPhiHatPhi2} that \eqref{Eq:ExpKirchhoff1} and \eqref{Eq:ExpKirchhoff1a} hold.
To obtain \eqref{Eq:ExpKirchhoff1Bis} we use \eqref{Eq:ExpKirchhoff1} together with \eqref{Eq:ScalePhi} and \eqref{Eq:BehaviourPhi1}.
For what concerns \eqref{Eq:ExpKirchhoff2}, it suffices then integrate $\nabla \varphi_{\kappa,k} - \nabla \widehat{\varphi}_{\kappa,k}$ on $\partial {\mathcal S}_{\mu}$ taking  into account \eqref{Eq:ExpKirchhoff1} and \eqref{Eq:BehaviourPhi1} when $\mu \neq \kappa$.
\end{proof}
\begin{Remark} \label{Rem:NormalisationKirchhoff}
The Kirchhoff potentials $\varphi_{\kappa,k}$ are defined up to a single additional constant (while the aforementioned additional constants in \eqref{Eq:ExpKirchhoff2} many differ from one connected component of the boundary to the other). We can however normalize this global additional constant so that
\begin{equation} \label{Eq:ExpKirchhoffNormalises}
{\varphi_{\kappa,k}} = {\mathcal O}(\varepsilon_{\kappa}^{1+\delta_{k\geq3}}) \ \text{ on } \ \partial {\mathcal S}_{\kappa} 
\ \text{ and } \ 
{\varphi_{\kappa,k}} = {\mathcal O}(\varepsilon_{\kappa}^{2+\delta_{k\geq3}}) \ \text{ on } \ \partial {\mathcal F} \setminus {\mathcal S}_{\kappa}.
\end{equation}
It suffices for instance to take $\widehat{\varphi_{\kappa,k}}(X) ={\varphi_{\kappa,k}}(X)$ for some point $X \in \partial \Omega$ (and integrate starting from this point).
\end{Remark}
In the case of Kirchhoff potentials corresponding to a solid of fixed size, we have the following more accurate result.
\begin{Proposition}
 \label{Pro:ExpKirchhoff2}
Let $\delta>0$. There exists $\varepsilon_{0}>0$ such that for all $\overline{\boldsymbol\varepsilon}$ with $\overline{{\boldsymbol\varepsilon}} \leq \varepsilon_{0}$ the following holds.
Let $\kappa \in  {\mathcal P}_{(i)} $, $k \in \{1,2,3\}$. Let $\ell \in \N \setminus \{0,1\}$. Then for some constant $C>0$ independent of $\overline{{\boldsymbol\varepsilon}}$, the following holds uniformly for ${\bf q} \in {\mathcal Q}_{\delta}$:
\begin{gather}
\label{Eq:ExpKirchhoff3}
\left\| \nabla \varphi_{\kappa,k} - \nabla \widecheck{\varphi}_{\kappa,k} 
+ \sum_{\lambda \in {\mathcal P}_{s}} \nabla \widehat{\mathfrak{f}}_{\lambda} \big[ \widecheck{\varphi}_{\kappa,k}  \big]
\right\|_{L^\infty({\mathcal F}({\bf q}))} \leq C \vert\overline{{\boldsymbol\varepsilon}}\vert , \\
\label{Eq:ExpKirchhoff4}
\big| \varphi_{\kappa,k} - \widecheck{\varphi}_{\kappa,k}  \big|_{C^{\ell,\frac{1}{2}}({\mathcal V}_{\delta}(\partial \widecheck{\mathcal F}))}
+
\sum_{\nu \in {\mathcal P}_{s}} \varepsilon_{\nu}^{\ell-\frac{1}{2}}  \big| \varphi_{\kappa,k} \big|_{C^{\ell,\frac{1}{2}}({\mathcal V}_{\delta}(\partial{\mathcal S}_{\nu}^\varepsilon))}
\leq C.
\end{gather}
and the terms $\nabla \widehat{\mathfrak{f}}_{\lambda} \big[ \widecheck{\varphi}_{\kappa,k}  \big]$ are 
 bounded in $L^\infty(\R^2 \setminus {\mathcal S}_{\lambda})$, converge to $0$ in $C^\ell(\{ x \in \overline{\Omega} / d(x,{\mathcal S}_{\lambda}) \geq c \}$ for all $c>0$ and $\ell \in \N$ and in $L^p(\Omega \setminus {\mathcal S}_{\lambda})$, $p < +\infty$.
\end{Proposition}
\begin{proof}[Proof of Proposition~\ref{Pro:ExpKirchhoff2}]
We let $\varepsilon_{0}$ as in Lemma~\ref{Lem:Contraction} and we reason as for Proposition~\ref{Pro:ExpKirchhoff}, using the correspondence between Dirichlet and Neumann problems (Lemma~\ref{Lem:Neumann}) and Proposition~\ref{Pro:DirichletGros}.
\end{proof}
\ \par
This has the following corollary on the added mass matrix. Recall that the added mass matrices where defined in \eqref{Eq:AddMassMatrix}--\eqref{Eq:AddMassMatrixKappaStandalone}. 
\begin{Corollary}
 \label{Cor:ExpAddedMass}
Let $\delta>0$. There exist constants $C>0$ and $\varepsilon_{0}>0$ such that for all $\kappa,\kappa' \in \{1,\dots,N\}$ and all $i,i' \in \{1,2,3\}$, as long as $(\boldsymbol\varepsilon,{\bf q}) \in \mathcal{Q}_\delta^{\varepsilon_{0}}$, 
\begin{equation} \label{Eq:AM-AMSA}
\left| {\mathcal M}_{a,\kappa,i,\kappa',i'} - \delta_{\kappa,\kappa'} \widehat{{\mathcal M}}_{a,\kappa,i,i'} \right| \leq C \varepsilon_{\kappa}^{2 + \delta_{3i}} \varepsilon_{\kappa'}^{2 + \delta_{3i'}}.
\end{equation}
Moreover one has,  uniformly for ${\bf q} \in \mathcal{Q}_\delta$, 
\begin{equation} \label{Eq:LimiteAM}
{\mathcal M}_{a,\kappa,i,\kappa',i'} \longrightarrow \delta_{\kappa \in {\mathcal P}_{(i)}} \delta_{\kappa' \in {\mathcal P}_{(i)}} \widecheck{{\mathcal M}}_{a,\kappa,\kappa',i,i'}  \ \text{ as } \ \overline{\boldsymbol\varepsilon} \longrightarrow 0.
\end{equation}
\end{Corollary}
\begin{proof}[Proof of Corollary~\ref{Cor:ExpAddedMass}]
We first write
\begin{equation} \label{Eq:CoefsMA}
{\mathcal M}_{a,\kappa,i,\kappa',i'} = \int_{\partial {\mathcal S}_{\kappa'}} \varphi_{\kappa,i} K_{\kappa',i'} \, ds,
\end{equation}
and notice that this formula is insensitive to a constant added to $\varphi_{\kappa,i}$.
Estimate \eqref{Eq:AM-AMSA} is then a direct consequence of \eqref{Eq:ExpKirchhoff2}.
The convergence \eqref{Eq:LimiteAM} follows in the same way from Proposition~\ref{Pro:ExpKirchhoff2}.
\end{proof}
\begin{Remark} \label{Rem:MA/MASAM}
Notice that both \eqref{Eq:AM-AMSA} and \eqref{Eq:LimiteAM} prove the convergence to $0$ of ${\mathcal M}_{a,\kappa,i,\kappa',i'}$ when $\kappa$ or $\kappa'$ belongs to ${\mathcal P}_{s}$. When both indices $\kappa$ and $\kappa'$ belong to ${\mathcal P}_{(i)}$, \eqref{Eq:AM-AMSA} merely proves that it remains bounded. Notice also that, as a consequence of \eqref{Eq:ScalePhi}, $\widehat{{\mathcal M}}_{a,\kappa,i,i'}$ satisfies the scale relation
%
$\widehat{{\mathcal M}}^\varepsilon_{a,\kappa,i,i'} = \varepsilon_{\kappa}^{2 + \delta_{3i} + \delta_{3i'}} \widehat{{\mathcal M}}^1_{a,\kappa,i,i'} $.
%
%
\end{Remark}
\subsubsection{Shape derivatives of the Kirchhoff potentials}
\label{SSSec:ShapeNEU}
In this paragraph, we estimate the shape derivatives of the Kirchhoff potentials.
An expression of the shape derivative of the Kirchhoff potentials was already obtained in \cite{GLMS}. 
Here we give a slightly different proof for this expression by relying on the results of Section \ref{SSSec:ShapeDeriv} (and extend it for indices $4$ and $5$). 
Precisely we consider the shape derivative $\displaystyle\frac{\partial\varphi_{\lambda,\ell}}{\partial q_{\mu,m}}({\bf q},\cdot)$ of the 
 Kirchhoff potentials $\varphi_{\lambda,\ell}$ for $\lambda \in \{1,\dots,N\}$ and $\ell \in \{1,\dots,5\}$ 
with respect to the variable $q_{\mu,m}$, for $\mu=1,\ldots,N$, $m=1,2,3$.
\begin{Lemma} \label{Lem:ShapeDerivativeKirchhoff}
For $\lambda=1,\ldots,N$, $\ell\in \{1,\dots,5\}$, $\mu=1,\ldots,N$, $m=1,2,3$, the function $\displaystyle\frac{\partial\varphi_{\lambda,\ell}}{\partial q_{\mu,m}}({\bf q},\cdot)$ is harmonic in $\mathcal F({\bf q})$ and satisfies:
\begin{gather} \label{Eq:DerivKirchBC1}
\frac{\partial}{\partial n} \left( \frac{\partial \varphi_{\lambda,\ell}}{\partial q_{\mu,m}} \right) ({\bf q},\cdot)= 0
\ \text{ on } \ 
\partial\mathcal F({\bf q}) \setminus \partial {\mathcal S}_{\mu}, \\
\label{Eq:DerivKirchBC2}
\frac{\partial}{\partial n} \left( \frac{\partial \varphi_{\lambda,\ell}}{\partial q_{\mu,m}} \right) ({\bf q},\cdot)=
\frac{\partial}{\partial \tau} \left[ \left( \frac{\partial \varphi_{\lambda,\ell}}{\partial \tau} 
- (\xi_{\lambda,\ell} \cdot \tau) \right) (\xi_{\mu,m} \cdot n) \right] +\delta_{\ell \geq 3} \delta_{m \in \{1,2\}} \partial_{\tau}
\left(\xi_{\lambda,\ell}^\perp \cdot e_{m} \right)
\ \text{ on } \ \partial {\mathcal S}_{\mu} .
\end{gather}
\end{Lemma}
We recall that the notation $\xi_{\lambda,\ell}$ is defined in 
\eqref{def-xi-j}. 
\begin{proof}[Proof of Lemma~\ref{Lem:ShapeDerivativeKirchhoff}]
As previously, we translate the Neumann problem defining the Kirchhoff potential $\varphi_{\lambda,\ell}$ into a Dirichlet problem (or in other words, we consider the harmonic conjugate of $\varphi_{\lambda,\ell}$). Hence we introduce the function $\varphi_{\lambda,\ell}^*$ and the constants $c_{1},\cdots,c_N$  that satisfy
\begin{equation} \nonumber \label{Eq:PhiEtoile}
\left\{ \begin{array}{l}
\displaystyle -\Delta \varphi_{\lambda,\ell}^* = 0 \text{ in } {\mathcal F}({\bf q}), \medskip \\
\displaystyle \varphi_{\lambda,\ell}^* = {\mathcal J}_{\lambda,\ell} + c_{\lambda} \text{ on } \partial {\mathcal S}_{\lambda}({\bf q}),  \medskip \\
\displaystyle \varphi_{\lambda,\ell}^* = c_{\kappa} \text{ on } \partial {\mathcal S}_{\kappa}({\bf q}), \ \ \forall \kappa \neq \lambda, \medskip \\
\displaystyle \varphi_{\lambda,\ell}^* = 0 \text{ on } \partial \Omega, \medskip \\
\displaystyle \int_{\partial {\mathcal S}_{\kappa}} \partial_{n} \varphi_{\lambda,\ell}^* \, ds = 0, \ \ \forall \kappa \in \{1,\dots,N\},
\end{array} \right.
\end{equation}
where ${\mathcal J}_{\lambda,\ell}$ is defined as a primitive of $K_{\lambda,\ell}$ on ${\mathcal S}_{\lambda}$. Namely we take ${\mathcal J}_{\lambda,\ell} = 
0$ on $\partial {\mathcal F} \setminus \partial {\mathcal S}_{\lambda}$, and on $\partial {\mathcal S}_{\lambda}$,
\begin{multline} \label{Eq:DefJ}
{\mathcal J}_{\lambda,\ell} = -x_{2}  \text{ if } \ell = 1, \ \ 
{\mathcal J}_{\lambda,\ell} = x_{1} \text{ if } \ell = 2, \ \ 
{\mathcal J}_{\lambda,\ell} =  \frac{|x-h_{\lambda}|^2}{2} \text{ if } \ell = 3, \\ 
{\mathcal J}_{\lambda,\ell} =  (x_{1}-h_{\lambda,1})(x_{2}-h_{\lambda,2}) \text{ if } \ell = 4 \text{ and } 
{\mathcal J}_{\lambda,\ell} =  \frac{(x_{1}-h_{\lambda,1})^2 - (x_{2}-h_{\lambda,2})^2}{2} \text{ if } \ell = 5.
\end{multline}
We extend ${\mathcal J}_{\lambda,\ell}$ in the neighborhood of these boundaries by the same formulas. In particular, one has the relation
\begin{equation} \label{Eq:DefJ2}
\nabla {\mathcal J}_{\lambda,\ell} = - \xi_{\lambda,\ell}^\perp  \text{ in the neighborhood of } \partial {\mathcal F}. 
\end{equation}
Then  $\nabla \varphi_{\lambda,\ell} = \nabla^\perp \varphi^*_{\lambda,\ell}$  in ${\mathcal F}({\bf q})$,
and thus
%
$\nabla \left(\frac{\partial \varphi_{\lambda,\ell}}{\partial q_{\mu,m}}\right) = \nabla^\perp \left(\frac{\partial \varphi^*_{\lambda,\ell}}{\partial q_{\mu,m}}\right) \text{ in } {\mathcal F}({\bf q})$.
%
By Corollary~\ref{Cor:SD}, we find 
\begin{align*}
\frac{\partial \varphi^*_{\lambda,\ell}}{\partial q_{\mu,m}} 
&= \delta_{\kappa \lambda} \delta_{\mu \lambda} \frac{\partial {\mathcal J}_{\lambda,\ell}}{\partial q_{\lambda,m}}
+ (\delta_{\kappa \lambda} \nabla {\mathcal J}_{\lambda,\ell} \cdot n - \partial_{n} \varphi^*_{\lambda,\ell}) \,K_{\mu,m} + c'_{\kappa}
\ \text{ on } \ \partial {\mathcal S}_{\kappa}({\bf q}), \ \kappa \in \{1,\dots,N\}, \\
\frac{\partial \varphi^*_{\lambda,\ell}}{\partial q_{\mu,m}} &= 0 \ \text{ on } \ \partial \Omega.
\end{align*}
We compute  $\frac{\partial {\mathcal J}_{\lambda,\ell}}{\partial q_{\lambda,m}}$ as follows:
\begin{equation*}
\frac{\partial {\mathcal J}_{\lambda,\ell}}{\partial q_{\lambda,m}}
= \delta_{\ell \geq 3} \delta_{m \in \{1,2\}} \nabla {\mathcal J}_{\lambda,\ell} \cdot e_{m} 
\ \text{ on } \ \partial {\mathcal S}_{\lambda}.
\end{equation*}
Since
$\partial_{\tau} \varphi_{\lambda,\ell} = -\partial_{n} \varphi^*_{\lambda,\ell}$
and 
$\displaystyle \partial_{n} \left(\frac{\partial \varphi_{\lambda,\ell}}{\partial q_{\mu,m}}\right) 
= \partial_{\tau} \left(\frac{\partial \varphi^*_{\lambda,\ell}}{\partial q_{\mu,m}}\right)$,
using \eqref{Eq:DefJ2} 
we obtain \eqref{Eq:DerivKirchBC2}.
\end{proof}
This allows us to prove the following estimates on the shape derivatives of the Kirchhoff potentials.
\begin{Proposition}
\label{Pro:ExpShapeDerivatives}
Let $\delta>0$. There is  $\varepsilon_{0}>0 $  such that for all $\overline{\boldsymbol\varepsilon}$ such that 
 $\overline{{\boldsymbol\varepsilon}} \leq \varepsilon_{0}$, 
 for  $\lambda, \mu, \kappa \in \{1,\dots,N\}$, for $\ell \in \{1,2,3\}$ and $m \in \{1,2,3,4,5\}$, uniformly for ${\bf q} \in \mathcal{Q}_\delta$, one has 
\begin{gather} \label{Eq:EstDerivKirchhoff1}
\frac{\partial \varphi_{\lambda,\ell}}{\partial q_{\mu,m}}
=  {\mathcal O}(\varepsilon_{\lambda}^{\delta_{\ell \geq 3} + 2 \delta_{\lambda \neq \mu}} \varepsilon_{\mu}^{\delta_{m3} + 2 \delta_{\mu \neq \kappa}})
\ \text{ on } \ \partial {\mathcal S}_\kappa \text{ (up to an additive constant)}, \\
\label{Eq:EstDerivKirchhoff1,5}
\left\| \nabla \frac{\partial \varphi_{\lambda,\ell}}{\partial q_{\mu,m}} \right\|_{L^\infty({\mathcal F})}
=  {\mathcal O}(\varepsilon_{\lambda}^{\delta_{\ell \geq 3} + 2 \delta_{\lambda \neq \mu}} \varepsilon_{\mu}^{-1+\delta_{m3}}), \\
\label{Eq:EstDerivKirchhoff2}
\nabla \frac{\partial \varphi_{\lambda,\ell}}{\partial q_{\mu,m}}(x)
=  {\mathcal O}(\varepsilon_{\lambda}^{\delta_{\ell \geq 3} + 2 \delta_{\lambda \neq \mu}} \varepsilon_{\mu}^{1+\delta_{m3}})  \ \ \text{ for } x \text{ such that } d(x,{\mathcal S}_{\mu}) \geq \delta.
\end{gather}
\end{Proposition}
\begin{proof}[Proof of Proposition~\ref{Pro:ExpShapeDerivatives}]
We proceed in three steps. \par
\ \par

\noindent
{\bf Step 1.} By Lemma~\ref{Lem:ShapeDerivativeKirchhoff}, 
\begin{equation} \label{Eq:CSD1}
\frac{\partial \varphi_{\lambda,\ell}}{\partial q_{\mu,m}} = \mathfrak{f}^{\mathcal N}_{\mu}[\partial_\tau {\mathcal B}] ,
\end{equation}
where we recall that $\mathfrak{f}^{\mathcal N}_{\mu}$ was defined in \eqref{Eq:PhiBeta} and where ${\mathcal B}$ is given on $\partial {\mathcal S}_{\mu}$ by a primitive of the data \eqref{Eq:DerivKirchBC2}  on ${\mathcal S}_{\mu}$: 
\begin{equation*}
{\mathcal B} = 
\left( \frac{\partial \varphi_{\lambda,\ell}}{\partial \tau}  - (\xi_{\lambda,\ell} \cdot \tau) \right) (\xi_{\mu,m} \cdot n) + \delta_{\ell \geq 3} \delta_{m \in \{1,2\}} \xi_{\lambda,\ell}^\perp \cdot e_{m} \ \text{ on } \ \partial {\mathcal S}_{\mu},
\end{equation*}
where we recall the convention \eqref{def-xi-j} on $\xi_{\lambda,\ell}$ (in particular, this is $0$ away from ${\mathcal S}_{\lambda}$). \par
\ \par
\noindent
{\bf Step 2.} Now we evaluate ${\mathcal B}$ on $\partial {\mathcal S}_{\mu}$. For $\lambda \neq \mu$, Proposition~\ref{Pro:ExpKirchhoff} gives directly
\begin{equation*}
\varepsilon_{\mu}^{j-\frac{1}{2}} | \varphi_{\lambda,\ell} |_{C^{j,\frac{1}{2}}({\mathcal V}_{\delta}(\partial {\mathcal S}_{\mu}))}  \leq C \varepsilon_{\lambda}^{2+\delta_{\ell \geq 3}},
\end{equation*}
In the case $\mu=\lambda$,
by Proposition~\ref{Pro:ExpKirchhoff}, for $j \geq 2$, one has
$
\varepsilon_{\lambda}^{j-\frac{1}{2}} | \varphi_{\lambda,\ell} - \widehat{\varphi}_{\lambda,\ell} |_{C^{j,\frac{1}{2}}({\mathcal V}_{\delta}(\partial {\mathcal S}_{\lambda}))} 
 \leq C \varepsilon_{\lambda}^{2+\delta_{\ell \geq 3}}.
$
Moreover from Proposition~\ref{Pro:DirichletSAM}, using the scale relation \eqref{Eq:ScalePhi}, we see that
$\varepsilon_{\lambda}^{j+\frac{1}{2}} | \widehat{\varphi}_{\lambda,\ell} |_{C^{j,\frac{1}{2}}({\mathcal F})} 
\leq C \varepsilon_{\lambda}^{1+ \delta_{\ell \geq 3}}$.
We deduce that 
\begin{equation*}
\varepsilon_{\lambda}^{j-\frac{1}{2}} | \varphi_{\lambda,\ell} |_{C^{j,\frac{1}{2}}({\mathcal V}_{\delta}(\partial {\mathcal S}_{\lambda}))} 
\leq C \varepsilon_{\lambda}^{\delta_{\ell \geq 3}}.
\end{equation*}
On the other hand, for all $\mu$ (including $\lambda$), the tangent $\tau$ on $\partial {\mathcal S}_{\mu}$ satisfies itself 
 $\varepsilon_{\mu}^{j+\frac{1}{2}} | \tau |_{C^{j,\frac{1}{2}}(\partial {\mathcal S}_{\mu})} \leq C$ (this is a scaling argument consistent with \eqref{Eq:NormesHolder}). 
For what concerns the $L^\infty$ norm, it follows from Propositions~\ref{Pro:ExpKirchhoff} and \ref{Pro:ExpKirchhoff2} that
$\| \nabla \varphi_{\lambda, \ell} \|_{L^\infty(\partial {\mathcal S}_{\mu})} = {\mathcal O}(\varepsilon_{\lambda}^{2 \delta_{\lambda \neq \mu} + \delta_{\ell3}})$. We deduce with the Leibniz rule that for all $\mu \in \{1,\dots,N\}$
\begin{equation*}
\varepsilon_{\mu}^{\frac{5}{2}} | \partial_{\tau} \varphi_{\lambda,\ell} |_{C^{2,\frac{1}{2}}(\partial {\mathcal S}_{\mu})}
\leq C \varepsilon_{\lambda}^{2 \delta_{\mu \neq \lambda} + \delta_{\ell \geq 3}}.
\end{equation*}
It follows then that
\begin{equation}
\label{starfoula}
\|{\mathcal B}\|_{L^\infty(\partial {\mathcal S}_{\mu})} 
+ \varepsilon_\mu^{\frac{5}{2}} |{\mathcal B}|_{C^{2,\frac{1}{2}}(\partial {\mathcal S}_{\mu})} 
= {\mathcal O}(\varepsilon_{\lambda}^{2 \delta_{\mu \neq \lambda} + \delta_{\ell \geq 3}} \varepsilon_{\mu}^{\delta_{m3}} ).
\end{equation}
\par
\ \par
\noindent
{\bf Step 3.} Now we deduce estimates on $\mathfrak{f}^{\mathcal N}_{\mu}[\partial_\tau {\mathcal B}]$ as follows: we apply Lemma~\ref{Lem:Neumann}, Proposition~\ref{Pro:DirichletPetits} and Remark~\ref{Rem:GrosDoncPetit} to  $\mathfrak{f}^{\mathcal N}_{\mu}[\partial_\tau {\mathcal B}]$ to obtain that
 for  $ \mu \in \{1,\dots,N\}$, 
\begin{equation} \label{Eq:ApproxNFN}
\nabla \mathfrak{f}^{\mathcal N}_{\mu}[\partial_\tau {\mathcal B}]  
=\nabla \widehat{\mathfrak{f}}^{\mathcal N}_{\mu}[\partial_\tau {\mathcal B}] 
+{\mathcal O} \left(\varepsilon_{\mu} \| {\mathcal B} \|_{L^\infty(\partial {\mathcal S}_{\mu})}\right) \text{ in } L^\infty({\mathcal F}(q)).
\end{equation}
To estimate $\nabla \widehat{\mathfrak{f}}^{\mathcal N}_{\mu}[\partial_\tau {\mathcal B}] $, we use Proposition~\ref{Pro:DirichletSAM} and \eqref{starfoula}. 
Hence \eqref{Eq:EstDerivKirchhoff2} is a consequence of  \eqref{Eq:ApproxNFN} and \eqref{Eq:EstNablaPhiHat}, and \eqref{Eq:EstDerivKirchhoff1,5} follows from \eqref{Eq:NablaPhiHatGlobal}.
We deduce \eqref{Eq:EstDerivKirchhoff1} by integrating \eqref{Eq:EstDerivKirchhoff1,5} (if $\kappa=\mu$) and \eqref{Eq:EstDerivKirchhoff2} (otherwise) over $\partial {\mathcal S}_\kappa$. The estimate on $\partial \Omega$ is performed in the same way.
\end{proof}
%
%
%
%
%
%
%
%
\subsection{Estimates on the circulation stream function}
In this section we study the circulation stream functions ${\psi}_{\kappa}^\varepsilon$, for $\kappa=1,\cdots,N$, introduced in \eqref{def_stream_F}.

We first recall several elementary properties of the standalone circulation stream functions   $\widehat{\psi}_{\kappa}^\varepsilon$, for $\kappa=1,\cdots,N$, introduced in \eqref{def_stream}.
We refer for instance to \cite{GLS} for a proof. 

\begin{Lemma} \label{Lem:StandaloneCSF}
For $\varepsilon_{\kappa}=1$, 
\begin{equation} \label{Eq:SEPsi}
\widehat{\psi}_{\kappa}^1 \big( (h_{\kappa},\vartheta_{\kappa}),x \big) 
= \widehat{\psi}_{\kappa}^1 \big( (0,0), R(-\vartheta_{\kappa})(x-h_{\kappa}) \big),
\end{equation}
 for fixed $q_{\kappa}$,
\begin{equation} \label{Eq:ScalePsi}
\nabla \widehat{\psi}_{\kappa}^\varepsilon(x- h_{\kappa}) = \frac{1}{\varepsilon_{\kappa}} \nabla \widehat{\psi}_{\kappa}^1 \left(\frac{x-h_{\kappa}}{\varepsilon_{\kappa}}\right),
\end{equation}
the function $\partial_{1} \widehat{\psi}_{\kappa} - i \partial_{2} \widehat{\psi}_{\kappa}$ admits the following  Laurent series expansion for $C$ such that 
${\mathcal S}_{\kappa}^1 \subset B(0,C)$, 
\begin{equation} \label{Eq:PsiLS}
\partial_{1} \widehat{\psi}_{\kappa} - i \partial_{2} \widehat{\psi}_{\kappa}= \frac{1}{2i\pi z} + \sum_{k \geq 2} \frac{a_{k}}{z^k}  \ \text{ for }  \ 
z= x_{1}-h_{1,\kappa} + i(x_{2}-h_{2,\kappa})  \  \text{ and }  \  
 |z| \geq C.
\end{equation}
\end{Lemma}
Note in particular that \eqref{Eq:ScalePsi}-\eqref{Eq:PsiLS} involve
\begin{equation} \label{Eq:BehaviourPsi1}
\nabla^{\perp} \widehat{\psi}^\varepsilon_{\kappa} (x)= \frac{(x-h_{\kappa})^\perp}{2\pi |x-h_{\kappa}|^2} +  {\mathcal O}\left(\frac{\varepsilon_{\kappa}}{|x-h_{\kappa}|^2}\right) \text{ for } |x - h_{\kappa}| \geq C \varepsilon_{\kappa},
\end{equation}
and consequently
\begin{equation} \label{Eq:BehaviourPsi2}
(x-h_{\kappa})^\perp \cdot \nabla^{\perp} \widehat{\psi}^\varepsilon_{\kappa} (x)= \frac{1}{2\pi} +  {\mathcal O}\left(\varepsilon_{\kappa}\right) \text{ for } |x - h_{\kappa}| \geq {\mathcal O}(1).
\end{equation}
The ${\mathcal O}\left(\varepsilon_{\kappa}\right)$ above can be taken in any norm, because this functions is harmonic, since
\begin{equation} \label{Eq:VitRotHarm}
(x-h_{\kappa})^\perp \cdot \nabla^{\perp} \widehat{\psi}^\varepsilon_{\kappa} (x)
= \mbox{Re}[i(z-h_{\kappa})(\partial_{1} \widehat{\psi}_{\kappa} - i \partial_{2} \widehat{\psi}_{\kappa})].
\end{equation}
We are now in position to study ${\psi}_{\kappa}^\varepsilon$. 
\subsubsection{Estimates on the reflected circulation stream function}
\label{SSS:CSF}
For  $\kappa=1,\dots,N$, we consider in the difference between the circulation stream function $\psi_{\kappa}$ and its standalone version $\widehat{\psi}_{\kappa}$, that is 
\begin{equation} \label{Eq:DefPsiKappar}
\psi_{\kappa}^r := \psi_{\kappa} - \widehat{\psi}_{\kappa} . 
\end{equation}
By \eqref{def_stream_F} and 
\eqref{def_stream}  
 there are some constants $c_{\lambda}$,  for $\lambda=1,\dots,N$, such that 
\begin{equation} \label{Eq:psir}
\left\{ \begin{array}{l}
\Delta \psi_{\kappa}^r = 0 \ \text{ in } \ {\mathcal F}, \medskip \\
\psi_{\kappa}^r = c_{\kappa} \ \text{ on } \ \partial {\mathcal S}_{\kappa} , \medskip \\
\psi_{\kappa}^r = - \widehat{\psi}_{\kappa} + c_{\nu} \ \text{ on } \ \partial {\mathcal S}_{\nu} , \ \forall \nu \neq \kappa, \medskip \\
\psi_{\kappa}^r = - \widehat{\psi}_{\kappa}  \ \text{ on } \ \partial \Omega , \medskip \\
\int_{\partial {\mathcal S}_{\nu}} \partial_{n} \psi_{\kappa}^r = 0, \ \text{ for all } \  \nu=1,\dots,N.
\end{array} \right.
\end{equation}
Thus $\psi_{\kappa}^r$ can be considered as  a ``reflected'' circulation stream function: one can view it as the part of $\psi_{\kappa}$ due to the response of the domain to the standalone stream function $\widehat{\psi}_{\kappa}$. 
We have the following estimates on $\psi_{\kappa}^r$. 
\begin{Lemma} \label{Lem:WBorne}
Let $\delta>0$. There exists $\varepsilon_{0}>0$ such that the following holds.
Let $\kappa \in \{1,\dots,N\}$ and $k \in \N$. There exists $C>0$ such that for any $\overline{\boldsymbol\varepsilon}$ such that $\overline{{\boldsymbol\varepsilon}} \leq \varepsilon_{0}$ and any ${\bf q} \in {\mathcal{Q}}_{\delta}$, one has
\begin{gather}
\label{Eq:WBorne}
\| \nabla \psi_{\kappa}^r \|_{L^{\infty}({\mathcal F})}  \leq  C, \\
\label{Eq:WHolder}
\forall \lambda \in \{1,\dots,N\},  \ \ \ \varepsilon_{\lambda}^{k-\frac{1}{2}} \| \psi_{\kappa}^r \|_{C^{k,\frac{1}{2}}({\mathcal V}_{\delta}(\partial {\mathcal S}_{\lambda}))}  \leq  C.
\end{gather}
\end{Lemma}
\begin{proof}[Proof of Lemma~\ref{Lem:WBorne}]
We let
\begin{multline} \label{Eq:DefDeNUplets}
{\bf A}:= (\widehat{\psi}_{\kappa|\partial {\mathcal S}_{1}},\dots,\widehat{\psi}_{\kappa|\partial {\mathcal S}_{N_{(i)}}}, 0,\dots,0,\widehat{\psi}_{\kappa|\partial \Omega})
\ \text{ and } \ 
\widecheck{\bf A}:= (\widehat{\psi}_{\kappa|\partial {\mathcal S}_{1}},\dots,\widehat{\psi}_{\kappa|\partial {\mathcal S}_{N_{(i)}}},\widehat{\psi}_{\kappa|\partial \Omega}), \\
\text{ where moreover we replace the } \kappa\text{-th element } \widehat{\psi}_{\kappa|\partial {\mathcal S}_{\kappa}} \text{ with } 0 \text{ whenever } \kappa \in {\mathcal P}_{(i)}.
\end{multline}
With Propositions~\ref{Pro:DirichletPetits} and \ref{Pro:DirichletGros} in mind, we rewrite $\psi_{\kappa}^r$ as
\begin{equation} \label{Eq:ExpressionPsir}
\psi_{\kappa}^r = - \mathfrak{H}[{\bf A}]
 - \sum_{\nu \in {\mathcal P}_{s} \setminus \{ \kappa \}} \mathfrak{f}_{\nu}[\widehat{\psi}_{\kappa|\partial {\mathcal S}_{\nu}}] .
\end{equation}
Due to Lemma~\ref{Lem:StandaloneCSF}, $\nabla \widehat{\psi}_{\kappa}$ is bounded on $\{x \ / d(x,\partial {\mathcal S}_{\kappa}) \geq \delta \}$, and hence so is $\widehat{\psi}_{\kappa}$. Thanks to interior elliptic estimates we may even obtain that
\begin{equation} \label{Eq:EstPsiAilleurs}
\varepsilon_{\nu}^{-1}\| \widehat{\psi}_{\kappa} - \widehat{\psi}_{\kappa}(h_{\nu}) \|_{L^{\infty} (\partial {\mathcal S}_{\nu}) }
+ | \widehat{\psi}_{\kappa} |_{C^{k,\frac{1}{2}}(\partial {\mathcal S}_{\nu})} \text{ is bounded for }\nu \neq \kappa.
\end{equation}
With uniform Schauder estimates in $\widecheck{{\mathcal F}}$ (Lemma~\ref{Lem:USE}), this involves that $\| \widecheck{\mathfrak{g}}[\widecheck{\bf A}] \|_{C^{k,\frac{1}{2}}(\widecheck{\mathcal{F}})}$ is bounded. With Proposition~\ref{Pro:DirichletGros} we deduce that $\mathfrak{H}[{\bf A}]$ gives a bounded contribution to \eqref{Eq:WBorne} and \eqref{Eq:WHolder}. \par
For what concerns the second term in \eqref{Eq:ExpressionPsir}, we use Proposition~\ref{Pro:DirichletPetits} and \eqref{Eq:EstPsiAilleurs}. It remains then to estimate the corresponding combination of standalone potentials $\widehat{\mathfrak f}_{\nu}[\widehat{\psi}_{\kappa|\partial {\mathcal S}_{\nu}}]$ for $\nu \in {\mathcal P}_{s} \setminus \{ \kappa \}$.  The conclusion follows from Proposition~\ref{Pro:DirichletSAM}.
\end{proof}
%
%
%
%
In addition to these uniform estimates, one may describe the limit of these circulation vector fields. For that we rely on the decomposition
\begin{equation} \label{Eq:DecompNatPsi}
\nabla^\perp \psi_{\kappa} = \nabla^\perp \widehat{\psi}_{\kappa}+ \nabla^\perp \psi_{\kappa}^r ,    
\end{equation}
and introduce two particular velocity vector fields that appear in the limit. For $\kappa \in {\mathcal P}_{s}$, we denote 
\begin{equation} \label{Eq:DefH}
H_{\kappa}(x):= \frac{(x-h_{\kappa})^\perp}{2\pi|x-h_{\kappa}|^2} ,
\end{equation}
and 
for $\kappa \in {\mathcal P}_{s}$, the potential $\widecheck{\psi}_{\kappa}^r$ as
the solution (up to an additive constant) of 
\begin{equation} \label{Eq:DefUKRKappa}
\left\{ \begin{array}{l}
\Delta  \widecheck{\psi}_{\kappa}^r = 0 \ \text{ in } \ \widecheck{{\mathcal F}}({\boldsymbol{q}}_{(i)}), \medskip \\
\nabla^\perp \widecheck{\psi}_{\kappa}^r  (x)\cdot n(x) = -H_{\kappa}(x) \cdot n(x)
\ \text{ on } \ \partial \Omega \cup \bigcup_{\nu \in  {\mathcal P}_{(i)} } \partial {\mathcal S}_{\nu}, \medskip \\
\displaystyle \oint_{\partial {\mathcal S}_{\nu}}  \nabla^\perp \widecheck{\psi}_{\kappa}^r \cdot \tau \, ds = 0  \ \text{ for } \ \nu \in  {\mathcal P}_{(i)}  .
\end{array} \right.
\end{equation}
It is straightforward to see that for any $\kappa \in {\mathcal P}_{s}$, 
\begin{equation} \label{Eq:Kdelta}
H_{\kappa} + \nabla^\perp \widecheck{\psi}_{\kappa}^r  = \widecheck{K}[\delta_{h_{\kappa}}] \ \text{ in } \ \widecheck{{\mathcal F}}({\bf q}_{(i)}).
\end{equation}
Then we have the following convergences, where all vector fields are put to $0$ inside the solids.
\begin{Proposition} \label{Pro:LimCirculation}
Let $\delta>0$. Uniformly for ${\bf q} \in {\mathcal Q}_{\delta}$, one has as $\overline{\boldsymbol{\varepsilon}} \rightarrow 0$ for any $k\in\N$, $p<+\infty$ and any $c>0$:
\begin{gather}
\label{Eq:CVPsi2}
\forall \kappa \in {\mathcal P}_{(s)}, \ \ \nabla^\perp \widehat{\psi}_{\kappa} \longrightarrow H_{\kappa}(x)
\text{ in } L^p(\Omega) \text{ for } p \in [1,2)    
\text{ and in } C^k(\{x \in \Omega \ / \ |x-h_\kappa| \geq c \}), \\
\label{Eq:CVPsi3}
\forall \kappa \in {\mathcal P}_{(s)},\  \ \nabla^\perp {\psi}^r_{\kappa} \longrightarrow \nabla^\perp \widecheck{\psi}_{\kappa}^r
\text{ in } L^p(\Omega) 
\text{ and in }  L^\infty(\{x \in \Omega  \ / \  d(x,\medcup_{\nu \in {\mathcal P}_{s}} {\mathcal S}_{\nu}) \geq c \}) , \\
\label{Eq:CVPsi1}
\forall \kappa \in {\mathcal P}_{(i)},  \ \   \nabla^\perp \psi^\varepsilon_{\kappa} \longrightarrow \nabla^\perp \widecheck{\psi}_{\kappa} 
\text{ in } L^p(\Omega)
\text{ and in }  L^\infty(\{x \in \Omega  \ / \   d(x,\medcup_{\nu \in {\mathcal P}_{s}} {\mathcal S}_{\nu}) \geq c \}) .
\end{gather}
\end{Proposition}
\begin{proof}[Proof of Proposition~\ref{Pro:LimCirculation}]
We begin with the proof of \eqref{Eq:CVPsi2}.
Considering  $\kappa \in {\mathcal P}_{(s)}$ and $ p \in [1,2)$, 
we first cut the integral in two:
\begin{equation*}
\int_{\Omega \setminus {\mathcal S}_{\kappa}^\varepsilon} \left|\nabla^\perp \widehat{\psi}^\varepsilon_{\kappa} - H_{\kappa}(x) \right|^p  \, dx
=
\int_{B(h_{\kappa},C \varepsilon_{\kappa}) \setminus {\mathcal S}_{\kappa}^\varepsilon} \left|\nabla^\perp \widehat{\psi}^\varepsilon_{\kappa} - H_{\kappa}(x) \right|^p  \, dx
+ \int_{\Omega \setminus B(h_{\kappa},C \varepsilon_{\kappa})} \left|\nabla^\perp \widehat{\psi}^\varepsilon_{\kappa} - H_{\kappa}(x) \right|^p  \, dx ,
\end{equation*}
where $C$ is taken as in  \eqref{Eq:BehaviourPsi1}.
For the first integral, using Lemma~\ref{Lem:StandaloneCSF} and a change of variable, we get
\begin{equation*}
\int_{B(h_{\kappa},C \varepsilon_{\kappa}) \setminus {\mathcal S}_{\kappa}^\varepsilon} \left|\nabla^\perp \widehat{\psi}^\varepsilon_{\kappa} - H_{\kappa}(x) \right|^p  \, dx
= \varepsilon_{\kappa}^{2-p} \int_{B(h_{\kappa},C ) \setminus {\mathcal S}_{\kappa}^1} \left|\nabla^\perp \widehat{\psi}^1_{\kappa} -  H_{\kappa}(x) \right|^p \, dx \\
= {\mathcal O}(\varepsilon_{\kappa}^{2-p}).
\end{equation*}
Concerning the second integral, by \eqref{Eq:BehaviourPsi1}, for some $R>0$ such that $\Omega \subset B(h_{\kappa},R)$, 
\begin{equation*}
\int_{\Omega \setminus B(h_{\kappa},C \varepsilon_{\kappa})} \left|\nabla^\perp \widehat{\psi}^\varepsilon_{\kappa} - H_{\kappa}(x) \right|^p  \, dx
 \leq  C \varepsilon_{\kappa}^p \int_{B(0,R) \setminus B(h_{\kappa},C \varepsilon)} \frac{1}{|x-h_{\kappa}|^{2p}} \, dx = {\mathcal O}(\varepsilon_{\kappa}^{2-p}).
\end{equation*}
Since $p \in [1,2)$, the convergence \eqref{Eq:CVPsi2} in $L^p(\Omega)$ follows. The convergence in $L^\infty$ away from $h_{\kappa}$ is a direct consequence of \eqref{Eq:BehaviourPsi1} and interior regularity estimates for harmonic functions. \par
We now prove \eqref{Eq:CVPsi3}. Let $\kappa \in {\mathcal P}_{s}$. We use the same notations \eqref{Eq:DefDeNUplets} as in the proof of Lemma~\ref{Lem:WBorne} and rely on \eqref{Eq:ExpressionPsir}. 
Due to Lemma~\ref{Lem:SmallCorrectors}, each of the terms $\nabla^\perp \mathfrak{f}_{\nu}[\widehat{\psi}_{\kappa|\partial {\mathcal S}_{\nu}}]$, for $\nu \in {\mathcal P}_{s} \setminus \{\kappa \}$,   converges to $0$ in $L^p(\Omega)$, $p<+\infty$ and in $L^\infty(\{x \in \Omega / d(x,\medcup_{\nu \in {\mathcal P}_{s}} {\mathcal S}_{\nu}) \geq c \})$.
Now by Proposition \ref{Pro:DirichletGros}, 
\begin{equation*}
\left\| \, \nabla \mathfrak{H}({\bf A}) - \nabla \widecheck{\mathfrak{g}}[\widecheck{\bf A}]
+ \sum_{\lambda \in {\mathcal P}_{s}} \nabla \widehat{\mathfrak{f}}_{\lambda} \big[ \widecheck{\mathfrak{g}}[\widecheck{\bf A}] \big] \right\|_{L^{\infty}({\mathcal F})}
\leq C | \overline{{\boldsymbol\varepsilon}} | 
\Big( \| \widehat{\psi}_{\kappa|\partial {\Omega}} \|_{L^\infty(\partial \Omega)} +  \sum_{\nu \in {\mathcal P}_{(i)}} \| \widehat{\psi}_{\kappa|\partial {\mathcal S}_{\nu}} \|_{L^\infty(\partial {\mathcal S}_{\nu})} \Big).
\end{equation*}
We recall that $\vert\overline{{\boldsymbol\varepsilon}}\vert $ was defined in  \eqref{Eq:TailleTotalePetitsSolides}.
Using again Lemma~\ref{Lem:SmallCorrectors}, we see that each of the terms $\nabla \widehat{\mathfrak{f}}_{\lambda} \big[ \widecheck{\mathfrak{g}}[\widecheck{\bf A}]_{|\partial {\mathcal S}_{\lambda}} \big]$ above converges to $0$ in $L^p(\Omega)$ and in $L^\infty(\{x \in \Omega / d(x,\medcup_{\nu \in {\mathcal P}_{s}} {\mathcal S}_{\nu}) \geq c \})$. Now from \eqref{Eq:DefUKRKappa} and \eqref{Eq:CVPsi2}, using the uniform Schauder estimates (see Lemma~\ref{Lem:USE}), we see that
$\nabla^\perp \widecheck{g}[\widecheck{\bf A}]$ converges to $- \nabla^\perp \widecheck{\psi}_{\kappa}^r \text{ in } C^{k,\frac{1}{2}}(\widecheck{\mathcal F})$ for all $k$.
This proves \eqref{Eq:CVPsi3}. \par
The proof of \eqref{Eq:CVPsi1} is analogous. Let $\kappa \in {\mathcal P}_{(i)}$. Here \eqref{Eq:DefPsiKappar} and \eqref{Eq:ExpressionPsir} give
\begin{equation*}
\psi^\varepsilon_{\kappa} = \widehat{\psi}_{\kappa} - \mathfrak{H}({\bf A}) - \sum_{\nu \in {\mathcal P}_{s}} \mathfrak{f}_{\nu}[\widehat{\psi}_{\kappa|\partial {\mathcal S}_{\nu}}] ,
\end{equation*}
where ${\bf A}$ was defined in \eqref{Eq:DefDeNUplets}.
Again, due to Lemma~\ref{Lem:SmallCorrectors}, each of the terms $\nabla^\perp \mathfrak{f}_{\nu}[\widehat{\psi}_{\kappa|\partial {\mathcal S}_{\nu}}]$ above converges to $0$ in $L^p(\Omega)$ (for $p<+\infty$) and in $L^\infty(\{x \in \Omega / d(x,\medcup_{\nu \in {\mathcal P}_{s}} {\mathcal S}_{\nu}) \geq c \})$.
Moreover, Proposition~\ref{Pro:DirichletGros} gives us here that 
\begin{equation*}
\left\| \, \nabla \mathfrak{H}({\bf A}) -  \nabla \widecheck{\mathfrak{g}}[\widecheck{\bf A}]
+ \sum_{\lambda \in {\mathcal P}_{s}} \nabla \widehat{\mathfrak{f}}_{\lambda} \big[ \widecheck{\mathfrak{g}}[\widecheck{\bf A}] \big]  \right\|_{L^{\infty}({\mathcal F})}
\leq  C | \overline{{\boldsymbol\varepsilon}} | 
\Big( \| \widehat{\psi}_{\kappa|\partial {\Omega}} \|_{L^\infty(\partial \Omega)} +  \sum_{\nu \in {\mathcal P}_{(i)} \setminus \{\kappa\}} \| \widehat{\psi}_{\kappa|\partial {\mathcal S}_{\nu}} \|_{L^\infty(\partial {\mathcal S}_{\nu})} \Big).
\end{equation*}
Using again Lemma~\ref{Lem:SmallCorrectors}, we see that each of the terms $\nabla \widehat{\mathfrak{f}}_{\lambda} \big[ \widecheck{\mathfrak{g}}[\widecheck{\bf A}] \big]$ above converges to $0$ in $L^p(\Omega)$ and in $L^\infty(\{x \in \Omega / d(x,\medcup_{\nu \in {\mathcal P}_{s}} {\mathcal S}_{\nu}) \geq c \})$. It remains to observe that here
$\widehat{\psi}_{\kappa} - \widecheck{g}[\widecheck{\bf A}] = \widecheck{\psi}_{\kappa}$ in $\widecheck{\mathcal F}$
since both sides satisfy \eqref{def_stream_Fcheck}. This gives \eqref{Eq:CVPsi1}.
\end{proof}
%
%
%
%
%
%
%
%
\subsubsection{Shape derivatives of the reflected circulation stream function}
\label{Par:SDCSF}
Here we are interested in differentiating $\psi_{\kappa}^r$ with respect to $q_{\mu,m}$.
\begin{Lemma} \label{Lem:DWBorne}
Let $\delta>0$.
There exist $\varepsilon_{0}>0 $ and $C>0$ such that for all $\overline{\boldsymbol\varepsilon}$ such that  $\overline{{\boldsymbol\varepsilon}} \leq \varepsilon_{0}$,
for all $\kappa,\mu \in \{1,\dots,N\}$, $m \in \{1,2,3\}$, for all ${\bf q} \in {\mathcal{Q}}_{\delta}$, 
\begin{gather}
\label{Eq:DWBorne}
\left\| \nabla \frac{\partial \psi_{\kappa}^r}{\partial q_{\mu,m}} \right\|_{L^\infty({\mathcal F} \setminus {\mathcal V}_{\delta/2}(\partial {\mathcal S}_{\mu}) )} \leq  C \varepsilon_{\mu}^{\delta_{m3}}
\ \text{ and } \ 
\left\| \nabla \frac{\partial \psi_{\kappa}^r}{\partial q_{\mu,m}} \right\|_{L^\infty( {\mathcal V}_{\delta}(\partial {\mathcal S}_{\mu}) )} \leq  C \varepsilon_{\mu}^{-1+\delta_{m3}}, \\
\label{Eq:DWBorne2}
\left\| \nabla \frac{\partial \psi_{\kappa}^r}{\partial q_{\mu,m}} \right\|_{L^p({\mathcal F})} \leq  C \varepsilon_{\mu}^{\delta_{m3}} \ \text{ for } \ p<2.
\end{gather}
\end{Lemma}
\begin{proof}[Proof of Lemma~\ref{Lem:DWBorne}]
We proceed in two steps. \par
\ \par
\noindent
{\bf Step 1.} 
We rely on \eqref{Eq:psir} and use Corollary~\ref{Cor:SD} and Remark~\ref{Rem:BCSD} to write
\begin{equation} \label{Eq:DQPsiAuBord}
\frac{\partial \psi_{\kappa}^r}{\partial q_{\mu,m}} = 
\left\{ \begin{array}{l}
\displaystyle - \delta_{\lambda \neq \kappa} \delta_{\mu \kappa} \frac{\partial \widehat{\psi}_{\kappa}}{\partial q_{\mu,m}}
- \left( \delta_{\lambda \neq \kappa} \nabla \widehat{\psi}_{\kappa} + \nabla \psi_{\kappa}^r \right) \cdot \xi_{\mu,m} +c'_{\lambda} 
\ \text{ on } \ \partial {\mathcal S}_{\lambda} \quad \text{ for } \quad \lambda=1,\dots,N, \\
\displaystyle -\delta_{\mu \kappa} \frac{\partial \widehat{\psi}_{\kappa}}{\partial q_{\mu,m}} \text{ on } \partial \Omega.
\end{array} \right.
\end{equation}
We now study the various terms in the first line of \eqref{Eq:DQPsiAuBord}. 
Due to Lemma~\ref{Lem:StandaloneCSF} we have
\begin{equation} \label{Eq:pqpsi}
\frac{\partial \widehat{\psi}_{\kappa}}{\partial q_{\kappa,m}} = - \nabla \widehat{\psi}_{\kappa} \cdot \xi^*_{\kappa,m} \text{ in } \R^2 \setminus {\mathcal S}_{\kappa}
\text{ with }
\xi^*_{ \kappa,j} ({\bf q},x) := e_{j} \text{ for } j=1,2 \text{ and }
\xi^*_{ \kappa,3} ({\bf q},x) := (x-h_\kappa)^\perp \text{ in } \R^2.
\end{equation}
The term $\delta_{\lambda \neq \kappa} \delta_{\mu \kappa} \frac{\partial \widehat{\psi}_{\kappa}}{\partial q_{\kappa,m}}$ merely gives a contribution when $\mu=\kappa$ on all the connected components of the boundary but $\partial {\mathcal S}_{\mu} = \partial {\mathcal S}_{\kappa}$. Due to \eqref{Eq:BehaviourPsi1} and  \eqref{Eq:pqpsi}, this contribution satisfies, up to an additional constant,
\begin{equation*}
\left\|  \delta_{\mu \kappa} \frac{\partial \widehat{\psi}_{\kappa}}{\partial q_{\mu,m}} \right\|_{L^\infty({\mathcal F} \setminus {\mathcal V}_{\delta/2}(\partial {\mathcal S}_{\kappa}))} \leq C \varepsilon_{\mu}^{\delta_{m3}}.
\end{equation*}
Using inner regularity for the Laplace equation and \eqref{Eq:VitRotHarm}, the same holds in $C^{k,\frac{1}{2}}({\mathcal F} \setminus {\mathcal V}_{\delta}(\partial {\mathcal S}_{\kappa}))$. Hence, up to an additive constant, we deduce
\begin{equation*}
\left\| \delta_{\lambda \neq \kappa} \delta_{\mu \kappa} \frac{\partial \widehat{\psi}_{\kappa}}{\partial q_{\mu,m}} \right\|_{L^\infty(\partial {\mathcal S}_{\lambda})} \leq C\varepsilon_{\lambda} \varepsilon_{\mu}^{\delta_{m3}}.
\end{equation*}
Let us now turn to the second term, which merely gives a contribution on $\partial {\mathcal S}_{\lambda}$ when $\lambda=\mu \neq \kappa$ (recall \eqref{def-xi-j}). By Lemma~\ref{Lem:StandaloneCSF} and \eqref{Eq:BehaviourPsi1}, we see that the term $\delta_{\lambda \neq \kappa} \nabla \widehat{\psi}_{\kappa} \cdot \xi_{\mu,m}$ gives a contribution of order $\varepsilon_{\mu}^{\delta_{m3}}$ in $L^\infty$-norm and in $C^{k,\frac{1}{2}}$-norm on $\partial {\mathcal S}_{\mu}$. \par
Finally we consider the last term, which again only gives a contribution on $\partial {\mathcal S}_{\lambda}$ when $\lambda=\mu$. By Lemma~\ref{Lem:WBorne},  the term $\nabla \psi_{\kappa}^r \cdot \xi_{\mu,m}$ gives a contribution of size $\varepsilon_{\mu}^{\delta_{m3}}$ in $L^\infty$-norm and at worst of order $\mathcal{O}(\varepsilon_{\mu}^{\delta_{m3}} \, \varepsilon_\lambda^{-k-\frac{1}{2}})$ in $C^{k,\frac{1}{2}}$-norm on $\partial {\mathcal S}_{\mu}$. \par
\medskip
Gathering these estimates we obtain, up to an additive constant on each connected component ${\mathcal S}_{\lambda}$ of the boundary, that for $k \geq 1$, 
\begin{multline*}
\left\| \frac{\partial \psi_{\kappa}^r}{\partial q_{\mu,m}} \right\|_{L^{\infty}(\partial {\mathcal S}_{\lambda})} \leq C \varepsilon_{\mu}^{\delta_{m3}} \varepsilon_{\lambda}^{\delta_{\lambda \neq \mu}}, \ \
\varepsilon_\lambda^{k+\frac{1}{2}} \left\| \frac{\partial \psi_{\kappa}^r}{\partial q_{\mu,m}} \right\|_{C^{k,\frac{1}{2}}(\partial {\mathcal S}_{\lambda})} \leq C \varepsilon_{\mu}^{\delta_{m3}} \varepsilon_{\lambda}^{\delta_{\lambda \neq \mu}}, \\ 
\left\| \frac{\partial \psi_{\kappa}^r}{\partial q_{\mu,m}} \right\|_{L^{\infty}(\partial \Omega)} \leq C \varepsilon_{\mu}^{\delta_{m3}} 
\ \text{ and } \ 
\left\| \frac{\partial \psi_{\kappa}^r}{\partial q_{\mu,m}} \right\|_{C^{k,\frac{1}{2}}(\partial \Omega)} \leq C \varepsilon_{\mu}^{\delta_{m3}}.
\end{multline*}
\par
\ \par
\noindent
{\bf Step 2.} 
As before we write
\begin{equation*}
\frac{\partial \psi_{\kappa}^r}{\partial q_{\mu,m}} = \mathfrak{H} \left[ \frac{\partial \psi_{\kappa}^r}{\partial q_{\mu,m}}_{|\partial {\mathcal S}_{1}}, \dots , \frac{\partial \psi_{\kappa}^r}{\partial q_{\mu,m}}_{|\partial {\mathcal S}_{N_{(i)}}},0,\ldots,0; \frac{\partial \psi_{\kappa}^r}{\partial q_{\mu,m}}_{|\partial \Omega} \right]
+ \sum_{\lambda \in {\mathcal P}_{s}} 
\mathfrak{f}_{\lambda} \left[ \frac{\partial \psi_{\kappa}^r}{\partial q_{\mu,m}}_{|\partial {\mathcal S}_{\lambda}} \right].
\end{equation*}
The $\mathfrak{H}$ term is bounded in $W^{1,\infty}(\widecheck{{\mathcal F}})$ due Proposition~\ref{Pro:DirichletGros}, the above estimates and the uniform Schauder estimates in $\widecheck{\mathcal F}$ (Lemma~\ref{Lem:USE}). The  $\mathfrak{f}_{\lambda}$ terms can be replaced by their standalone counterpart $\widehat{\mathfrak{f}}_{\lambda}$ thanks to Proposition~\ref{Pro:DirichletPetits}. These $\widehat{\mathfrak{f}}_{\lambda}$ are estimated by Proposition~\ref{Pro:DirichletSAM} which gives the estimates in \eqref{Eq:DWBorne}. \par
Concerning \eqref{Eq:DWBorne2}, by the above considerations, we only need to discuss the contribution of the $\widehat{\mathfrak{f}}_{\lambda}$ terms. Mixing \eqref{Eq:NablaPhiHatGlobal} and \eqref{Eq:EstNablaPhiHat}, and distinguishing $x \in B(h_{\lambda},C \varepsilon) \setminus {\mathcal S}_{\lambda}^\varepsilon$ and $x \in {\mathcal F} \setminus B(h_{\lambda},C \varepsilon_{\lambda})$, we see that
\begin{equation} \label{Eq:EstDesFChapeau}
\forall x \in {\mathcal F}, \ \ 
\left| \nabla \widehat{\mathfrak{f}}_{\lambda} \Big[ \frac{\partial \psi_{\kappa}^r}{\partial q_{\mu,m}}_{|\partial {\mathcal S}_{\lambda}} \Big](x) \right|
\leq C \frac{\varepsilon_{\lambda}}{|x- h_{\lambda}|^2} 
\left( \left\| \frac{\partial \psi_{\kappa}^r}{\partial q_{\mu,m}} \right\|_{L^{\infty}(\partial {\mathcal S}_{\lambda})}
+ \varepsilon_\lambda^{k+\frac{1}{2}} \left\| \frac{\partial \psi_{\kappa}^r}{\partial q_{\mu,m}} \right\|_{C^{k,\frac{1}{2}}(\partial {\mathcal S}_{\lambda})} \right).
\end{equation}
Now we put the above inequality to the power $p$ and integrate. We can inject $\Omega$ in some ball $B(h_{\lambda},R)$ with $R>0$ fixed so that we write ${\mathcal F} \subset B(h_{\lambda},R) \setminus B(h_{\lambda},C' \varepsilon_{\lambda})$ for some positive $C'$. The result follows.
\end{proof}
%
%
%
%
%
\subsubsection{Reflected circulation stream function of a phantom solid}
\label{Par:Phantom}
In this paragraph, we extend the above estimates on the reflected circulation stream function $\psi_{\kappa}^r$ to a slight variant. This variant will play an important role in the definition of the modulation and in the passage to the limit, in particular for what concerns the desingularization \eqref{DefUellKappa}. \par
For $\kappa \in {\mathcal P}_{s}$ we first introduce the following ``$\kappa$-augmented'' fluid domain as follows:
\begin{equation} \label{Eq:DomaineKappaAugmente}
\widecheck{{\mathcal F}}_{\kappa}({\bf q}) := {\mathcal F}({\bf q}) \cup {\mathcal S}_{\kappa}({\bf q}).
\end{equation}
Note in particular that
\begin{equation*}
\partial \widecheck{{\mathcal F}}_{\kappa}({\bf q}) = \partial {\mathcal F}({\bf q}) \setminus \partial {\mathcal S}_{\kappa}({\bf q}) = \partial \Omega \cup \bigcup_{\nu \in \{1,\dots,N\} \setminus \{\kappa\}} \partial {\mathcal S}_{\nu}.
\end{equation*}
Now we introduce $\psi^{r,\not \kappa}_{\kappa}$ as the solution in $\widecheck{{\mathcal F}}_{\kappa}({\bf q})$  (together with constants $c_{\lambda}$, $\lambda \in \{1,\dots,N\} \setminus \{ \kappa \}$) to the system:
\begin{equation} \label{Eq:PsirefKappa}
\left\{ \begin{array}{l}
\Delta \psi^{r,\not \kappa}_{\kappa} = 0 \ \text{ in } \ \widecheck{{\mathcal F}}_{\kappa}({\bf q}), \medskip \\ 
\psi^{r,\not \kappa}_{\kappa} = - \widehat{\psi}_{\kappa} + c_{\lambda} \ \text{ on } \ \partial {\mathcal S}_{\lambda}({\bf q}), \ \text{ for } \ \lambda \in \{1,\dots,N\} \setminus \{\kappa\},\medskip \\
\psi^{r,\not \kappa}_{\kappa} = - \widehat{\psi}_{\kappa} \ \text{ on } \ \partial \Omega, \medskip \\
\displaystyle \int_{\partial {\mathcal S}_{\nu}({\bf q})} \partial_{n} \psi^{r,\not \kappa}_{\kappa} \, ds  = 0 \ \text{ for } \ \nu \in \{1,\dots,N\} \setminus \{ \kappa \}.
\end{array} \right.
\end{equation}
The only difference indeed between $\psi_{\kappa}^r$ and  $\psi_{\kappa}^{r,\not \kappa}$ is that the constraint $\psi_{\kappa}^r = c_{\kappa}$ on $\partial {\mathcal S}_{\kappa}$ in \eqref{Eq:psir} has disappeared in \eqref{Eq:PsirefKappa}, and that the domain is $\widecheck{{\mathcal F}}_{\kappa}$ rather than ${\mathcal F}$. 
Adapting the arguments above we obtain the following result.
\begin{Lemma} \label{Lem:Phantom}
Let $\delta>0$. There exists $\varepsilon_{0}>0$ such that the following holds.
Let $\kappa \in {\mathcal P}_{s}$ and $k \in \N$. There exists $C>0$ such that for any $\overline{\boldsymbol\varepsilon}$ such that $\overline{{\boldsymbol\varepsilon}} \leq \varepsilon_{0}$ and any ${\bf q} \in {\mathcal{Q}}_{\delta}$, one has
\begin{gather}
\label{Eq:PhantomBorne}
\| \nabla \psi_{\kappa}^{r,\not \kappa} \|_{L^{\infty}({\mathcal F})}  \leq  C, \ \text{ and  }\ \forall \lambda \in \{1,\dots,N\},  \ \ \ \varepsilon_{\lambda}^{\delta_{\kappa \neq \lambda}(k-\frac{1}{2})} \| \psi_{\kappa}^{r,\not \kappa} \|_{C^{k,\frac{1}{2}}({\mathcal V}_{\delta}(\partial {\mathcal S}_{\lambda}))}  \leq  C, \\
\label{Eq:DPhantomBorne}
\left\| \nabla \frac{\partial \psi_{\kappa}^{r,\not \kappa}}{\partial q_{\mu,m}} \right\|_{L^\infty({\mathcal F} \setminus {\mathcal V}_{\delta/2}(\partial {\mathcal S}_{\mu}) )} \leq  C \varepsilon_{\mu}^{\delta_{m3}}
\ \text{ and } \ 
\left\| \nabla \frac{\partial \psi_{\kappa}^{r,\not \kappa}}{\partial q_{\mu,m}} \right\|_{L^\infty({\mathcal V}_{\delta}(\partial {\mathcal S}_{\mu}) )} \leq  C \varepsilon_{\mu}^{-1+\delta_{m3} + \delta_{\mu \kappa}}.
\end{gather}
Moreover, uniformly for ${\bf q} \in {\mathcal Q}_{\delta}$, one has as $\overline{\boldsymbol{\varepsilon}} \rightarrow 0$ for any $k\in\N$, $p<+\infty$ and any $c>0$:
\begin{equation} \label{Eq:CVPhantom}
\nabla^\perp {\psi}^{r,\not \kappa}_{\kappa} \longrightarrow \nabla^\perp \widecheck{\psi}_{\kappa}^r
\text{ in } L^p(\Omega) 
\text{ and in }  L^\infty(\{x \in \Omega / d(x,\medcup_{\nu \in {\mathcal P}_{s} \setminus \{ \kappa \}} {\mathcal S}_{\nu}) \geq c \}) .
\end{equation}
\end{Lemma}
\begin{proof}[Proof of Lemma~\ref{Lem:Phantom}]
This is a mere adaptation of Lemmas~\ref{Lem:WBorne} and \ref{Lem:DWBorne} and of \eqref{Eq:CVPsi3}. Hence we only stress the variations in the proofs.

To get \eqref{Eq:PhantomBorne}, the main point is that \eqref{Eq:ExpressionPsir} has to be replaced by 
\begin{equation} \label{Eq:ExprPsiKappar}
\psi_{\kappa}^{r,\not \kappa} = - \mathfrak{H}^{\not \kappa}({\bf A}) - \sum_{\nu \in {\mathcal P}_{s} \setminus \{\kappa \}} \mathfrak{f}^{\not \kappa}_{\nu}[\widehat{\psi}_{\kappa|\partial {\mathcal S}_{\nu}}] ,
\end{equation}
where the potentials $\mathfrak{H}^{\not \kappa}$ and $\mathfrak{f}^{\not \kappa}$ correspond to the domain $\widecheck{{\mathcal F}}_{\kappa}$ rather than ${\mathcal F}$, and where we define the $N$-tuple 
${\bf A}:= (\widehat{\psi}_{\kappa|\partial {\mathcal S}_{1}},\dots, \widehat{\psi}_{\kappa|\partial {\mathcal S}_{N_{(i)}}},0,\dots,0, \widehat{\psi}_{\kappa|\partial {\Omega}})$, where $N$ corresponds to $N-1$ solids plus $\Omega$ (because there is no boundary $\partial {\mathcal S}_\kappa$).
Then the same argument as in  Lemma~\ref{Lem:WBorne} applies to obtain  \eqref{Eq:PhantomBorne}, using Propositions~\ref{Pro:DirichletPetits} and \ref{Pro:DirichletGros} in the domain with $N-1$ solids $\widecheck{{\mathcal F}}_{\kappa}$. \par
Concerning the estimate \eqref{Eq:DPhantomBorne} of the shape derivative, when $\mu \neq \kappa$, it suffices to make the slight correction to the boundary condition \eqref{Eq:DQPsiAuBord}:
\begin{equation*}
\frac{\partial \psi_{\kappa}^{r,\not \kappa}}{\partial q_{\mu,m}} = 
\left\{ \begin{array}{l}
\displaystyle - \delta_{\mu \kappa} \frac{\partial \widehat{\psi}_{\kappa}}{\partial q_{\kappa,m}}
- \left( \nabla \psi_{\kappa}^{r,\not \kappa} + \nabla \widehat{\psi}_{\kappa} \right) \cdot \xi_{\mu,m} +c'_{\lambda} 
\text{ on } \partial {\mathcal S}_{\lambda} \ \text{ for } \ \lambda \in \{1,\dots,N\} \setminus \{ \kappa \}, \\
\displaystyle -\delta_{\mu \kappa} \frac{\partial \widehat{\psi}_{\kappa}}{\partial q_{\kappa,m}} \text{ on } \partial \Omega.
\end{array} \right.
\end{equation*}
Then the same reasoning as in Lemma~\ref{Lem:DWBorne} applies. Importantly enough, $\partial {\mathcal S}_{\kappa}$ is now in the bulk of the domain $\widecheck{{\mathcal F}}_{\kappa}$ so that the standalone potentials (see \eqref{Eq:EstDesFChapeau}) give a bounded contribution to $\nabla \frac{\partial \psi_{\kappa}^{r,\not \kappa}}{\partial q_{\mu,m}}$ in the neighborhood of ${\mathcal S}_{\kappa}$. \par
When $\mu=\kappa$, the situation is a bit different, because $\frac{\partial}{\partial q_{\kappa}}$ is no longer a shape derivative (the domain $\widecheck{{\mathcal F}}_{\kappa}$ does not depend on $q_{\kappa}$) but a simple derivative with respect to a parameter. The boundary condition becomes 
$
\frac{\partial \psi_{\kappa}^{r,\not \kappa}}{\partial q_{\kappa,m}} =  - \frac{\partial \widehat{\psi}_{\kappa}}{\partial q_{\kappa,m}}
\ \text{ on } \ \partial \widecheck{{\mathcal F}}_{\kappa},
$
and the boundedness of $\varepsilon_{\kappa}^{-\delta_{m3}} \nabla \frac{\partial \psi_{\kappa}^{r,\not \kappa}}{\partial q_{\kappa,m}}$ (here in the whole $\widecheck{{\mathcal F}}_{\kappa}$) follows as before. \par
Finally, to prove \eqref{Eq:CVPhantom}, we rely again on \eqref{Eq:ExprPsiKappar} and reason as for \eqref{Eq:CVPsi1}. We approximate $\nabla \mathfrak{H}^{\not \kappa}({\bf A})$ by $\nabla \widecheck{\mathfrak{g}}[\widecheck{\bf A}]$ with the same $\widecheck{\mathfrak{g}}$  and the same $\widecheck{\bf A}:= (\widehat{\psi}_{\kappa|\partial {\mathcal S}_{1}},\dots, \widehat{\psi}_{\kappa|\partial {\mathcal S}_{N_{(i)}}}, \widehat{\psi}_{\kappa|\partial {\Omega}})$ as in the proof of \eqref{Eq:CVPsi1} (since $\kappa \in {\mathcal P}_{s}$). Hence we obtain the same limit.
\end{proof}
%
%
%
%
%
%
%
%
%
%
%
%
%
%
\subsection{Estimates of the Biot-Savart kernel}
\subsubsection{Biot-Savart kernel}
The following will be useful for both the \textit{a priori} estimates and the passage to the limit.
We consider $\omega \in L^{\infty}_{c}({\mathcal F})$ and compare the generated velocity $K[\omega]$ in ${\mathcal F}$ (in the domain with all solids) and the generated velocity $\widecheck{K}[\omega]$ in $\widecheck{\mathcal F}$ (in the larger domain with only solids of family (i)) as defined in \eqref{Eq:DefK} and \eqref{Eq:DefKCheck}. In particular we prove that these velocity fields are bounded independently of $\overline{\boldsymbol\varepsilon}$. 
Precisely we have the following result.
\begin{Lemma} \label{Lem:BiotSavart}
Let $\delta>0$. There exists $\varepsilon_{0} >0$ such that the following holds. For any $p \in (2,+\infty]$, there exist $C>0$  such that for any $(\boldsymbol\varepsilon,{\bf q},\omega) \in \mathfrak{Q}_\delta^{\varepsilon_{0}}$, one has
\begin{equation} \label{Eq:BiotSavartBorne}
\left\| K[\omega] \right\|_{L^{\infty}({\mathcal F}^\varepsilon)} \leq  C \| \omega \|_{L^{p}({\mathcal F}^\varepsilon)}
\ \text{ and } \ 
\varepsilon_{\lambda}^{k-\frac{1}{2}} \left| K[\omega] \right|_{C^{k-1,\frac{1}{2}}(\partial {\mathcal S}_{\lambda}^\varepsilon)} \leq  C \| \omega \|_{L^{p}({\mathcal F}^\varepsilon)},
\ \forall \lambda =1, \dots,N.
\end{equation}
In the same way, there exists $\varepsilon_{0}>0$ and for each $p \in (1,+\infty)$, there exist $C>0$  such that for any $(\boldsymbol\varepsilon,{\bf q}) \in \mathcal{Q}_\delta^{\varepsilon_{0}}$, any $f \in L^p({\mathcal F}^\varepsilon({\bf q});\R^2)$ such that $\dist(\Supp(f), \partial {\mathcal F}^\varepsilon({\bf q})) \geq \delta,$
\begin{equation} \label{Eq:BiotSavartBorne2}
\left\| K[\div(f)] \right\|_{L^{p}({\mathcal F}^\varepsilon)}  \leq  C \| f \|_{L^{p}({\mathcal F}^\varepsilon)} .
\end{equation}
Finally, uniformly for $({\bf q},\omega)$ such that $(\boldsymbol\varepsilon,{\bf q},\omega) \in \mathfrak{Q}_\delta^{\varepsilon_{0}}$ when $\overline{\boldsymbol{\varepsilon}}$ is small and $\omega$ is bounded in $L^\infty$,
\begin{multline} \label{Eq:BiotSavartLimite}
\left\| K[\omega] - \widecheck{K}[\omega] \right\|_{L^p({\mathcal F}({\bf q}))} \rightarrow 0 
\ \text{for } p\in (2,+\infty)
\\ 
\text{ and } 
\left\| K[\omega] - \widecheck{K}[\omega] \right\|_{L^\infty \left(\{ x \in {\widecheck{\mathcal F}}/d \left(x,\cup_{\nu \in {\mathcal P}_{s}} {\mathcal S}_{\nu}\right) \geq c\} \right)} \rightarrow 0
\text{ as } \overline{\boldsymbol\varepsilon} \rightarrow 0 .
\end{multline}
\end{Lemma}
\begin{Remark}
Actually, our proof only needs $\omega$ or $f$ to be supported away from the small solids.
\end{Remark}
\begin{proof}[Proof of Lemma~\ref{Lem:BiotSavart}]
For $\delta>0$, we let $\varepsilon_{0}$ as in Lemma~\ref{Lem:Contraction}.
Clearly, the difference $\widetilde{R}[\omega] := K[\omega] - \widecheck{K}[\omega]$ satisfies
\begin{equation} \nonumber
\left\{ \begin{array}{l}
\div \widetilde{R}[\omega] = \curl \widetilde{R}[\omega] = 0 \ \text{ in } \ {\mathcal F}({\bf q}), \medskip \\
\widetilde{R}[\omega]\cdot n = 0 \ \text{ on } \ \partial \widecheck{\mathcal F}({\bf q}), \medskip \\
\widetilde{R}[\omega]\cdot n = -\widecheck{K}[\omega] \cdot n \ \text{ on } \ \partial {\mathcal F} \setminus \partial \widecheck{\mathcal F}({\bf q}), \medskip \\
\displaystyle \oint_{\partial {\mathcal S}_{\nu}} \widetilde{R}[\omega] \cdot \tau \, ds = 0  \ \text{ for } \ \nu =1, \dots, N.
\end{array} \right.
\end{equation}
In particular one can write $\widetilde{R}[\omega]=\nabla^\perp \widetilde{\varphi}[\omega]$ where 
\begin{equation} \label{ontoile}
\widetilde{\varphi}[\omega] = -\sum_{\kappa \in {\mathcal P}_{s}} \mathfrak{f}_{\kappa} \left[\widecheck{\Psi}[\omega]_{|\partial {\mathcal S}_{\kappa}}\right] ,
\end{equation}
and where $\widecheck{\Psi}[\omega]$ is a stream function for $\widecheck{K}[\omega]$, that is, $\widecheck{K}[\omega] =\nabla^\perp \widecheck{\Psi}[\omega]$. \par
We first estimate $\widecheck{K}[\omega]$. As for Lemma~\ref{Lem:USE}, we have uniform Calder\'on-Zygmund estimates (see e.g. \cite[Lemma 9.17]{GT}) in $\widecheck{{\mathcal F}}$ as long as ${\bf q}_{(i)} \in {\mathcal{Q}}_{(i),\delta}$. It follows that for each $p \in (1,+\infty)$, one has a uniform constant $C>0$ such that
\begin{equation} \label{Eq:EstKSobolev}
\left\| \widecheck{K}[\omega] \right\|_{W^{1,p}(\widecheck{\mathcal F})} \leq  C \| \omega \|_{L^{p}({\mathcal F}^\varepsilon)}.
\end{equation}
Then we invoke Sobolev embedding for $p>2$ to get the bound 
\begin{equation} \label{Eq:SobK}
\| \widecheck{K}[\omega] \|_{L^{\infty}(\widecheck{\mathcal F})} \leq C \| \omega \|_{\infty}.
\end{equation}
This embedding is also uniform in $\widecheck{{\mathcal F}}$ as long as  ${\bf q}_{(i)} \in {\mathcal{Q}}_{(i),\delta}$: it suffices to use an extension operator inside each solid and use the embedding in $\Omega$. We notice that since $\omega$ is distant from the solids, by inner regularity for the Laplace equation, we have

\begin{equation} \label{Eq:Ck12Psi}
\| \widecheck{K}[\omega] \|_{C^{k,\frac{1}{2}}({\mathcal V}_{\delta}(\partial {\mathcal F}))} \leq C \| \omega  \|_{\infty} 
\ \text{ and } \ 
\| \widecheck{\Psi}[\omega] \|_{C^{k,\frac{1}{2}}({\mathcal V}_{\delta}(\partial {\mathcal F}))} \leq C \| \omega  \|_{\infty} .
\end{equation}
Now we apply Proposition~\ref{Pro:DirichletPetits} and Proposition~\ref{Pro:DirichletSAM} to each term in the right-hand-side of \eqref{ontoile}. This gives
\begin{equation*} 
\left\| \tilde{R}[\omega] \right\|_{L^{\infty}({\mathcal F}^\varepsilon)} \leq  C \| \omega \|_{L^{p}({\mathcal F}^\varepsilon)}
\ \text{ and } \ 
\varepsilon_{\lambda}^{k-\frac{1}{2}} \left| \tilde{R}[\omega] \right|_{C^{k-1,\frac{1}{2}}(\partial {\mathcal S}_{\lambda}^\varepsilon)} \leq  C \| \omega \|_{L^{p}({\mathcal F}^\varepsilon)},
\ \forall \lambda =1, \dots,N.
\end{equation*}
We consequently deduce \eqref{Eq:BiotSavartBorne} with \eqref{Eq:SobK}. \par
The convergence \eqref{Eq:BiotSavartLimite} of $\widetilde{R}[\omega]$ to $0$ as $\overline{\boldsymbol\varepsilon} \rightarrow 0$ is proven as \eqref{Eq:EstCorrectLambda}: it is a consequence of Lemma~\ref{Lem:SmallCorrectors} and \eqref{Eq:Ck12Psi}. \par
Finally \eqref{Eq:BiotSavartBorne2} is proven in the same way, albeit in a weaker context. 
Denoting by $K_{\R^2} $ the Biot-Savart operator in the full plane, such that 
 \begin{equation} \nonumber 
 \left\{ \begin{array}{l}
 \div K_{\R^2}[\omega] = 0 \ \text{ in } \ \R^2, \medskip \\
 \curl K_{\R^2}[\omega] = \omega \ \text{ in } \ \R^2, \medskip \\
 K_{\R^2}[\omega](x) \longrightarrow 0 \ \text{ as } \ x \longrightarrow +\infty,
 \end{array} \right.
 \end{equation}
we recall that  $K_{\R^2} \circ \div = \nabla^\perp \Delta_{\R^2}^{-1} \div$ is a Calder\'on-Zygmund operator which sends $L^p(\R^2)$ into itself (for $p \in (1,+\infty)$). It remains to check that the correction to obtain $K[\div(f)]$ is also estimated uniformly in $L^p({\mathcal F}({\bf q}))$. Thanks to the constraint on the support of $f$, it is again a consequence of interior elliptic estimates and of Propositions~\ref{Pro:DirichletPetits}, \ref{Pro:DirichletGros} and \ref{Pro:DirichletSAM}.
\end{proof}
\subsubsection{Shape derivatives of the Biot-Savart kernel}
\label{Par:SDBSK}
In this paragraph, for fixed $\omega$, we estimate the shape derivative $\displaystyle \frac{\partial K[\omega]}{\partial q_{\mu,m}}$.
\begin{Lemma} \label{Lem:DBSBorne}
Let $\delta>0$, $\mu \in \{1,\dots,N\}$, $m \in \{1,2,3\}$. There exists $C>0$ and $\varepsilon_{0}>0$ such that for any $({\boldsymbol\varepsilon},{\bf q},\omega) \in \mathfrak{Q}_{\delta}$, 
\begin{gather*} 
\left\| \frac{\partial K[\omega]}{\partial q_{\mu,m}} \right\|_{L^\infty({\mathcal F} \setminus {\mathcal V}_{\delta/2}(\partial {\mathcal S}_{\mu}) )} \leq  C \varepsilon_{\mu}^{\delta_{m3}} \left\| \omega \right\|_{L^{\infty}({\mathcal F}({\bf q}))}
, \ \ \ 
\left\| \frac{\partial K[\omega]}{\partial q_{\mu,m}} \right\|_{L^\infty({\mathcal V}_{\delta}(\partial {\mathcal S}_{\mu}) )} \leq  C \varepsilon_{\mu}^{-1+\delta_{m3}} \left\| \omega \right\|_{L^{\infty}({\mathcal F}({\bf q}))} , \\
\text{ and } \ 
\left\| \frac{\partial K[\omega]}{\partial q_{\mu,m}} \right\|_{L^p({\mathcal F})} \leq  C \varepsilon_{\mu}^{\delta_{m3}} \left\| \omega \right\|_{L^{\infty}({\mathcal F}({\bf q}))} 
\ \text{ for } \ p<2.
\end{gather*}
\end{Lemma}
\begin{proof}[Proof of Lemma~\ref{Lem:DBSBorne}]
Here we first introduce $K_{\Omega}[\omega]$ that satisfies
\begin{equation} \label{Eq:DefKOmega}
\left\{ \begin{array}{l}
\div K_{\Omega}[\omega] = 0 \ \text{ in } \ \Omega, \medskip \\
\curl K_{\Omega}[\omega] = \omega \ \text{ in } \ \Omega, \medskip \\
K_{\Omega}[\omega] \cdot n = 0 \ \text{ on } \ \partial \Omega.
\end{array} \right.
\end{equation}
(Recall that we suppose $\Omega$ simply connected to simplify.) Note that $K_{\Omega}[\omega]$ (whose shape derivative is obviously zero) can be put in the form
$K_{\Omega}[\omega] = \nabla^\perp \Psi_{\Omega}[\omega]$ 
 with 
$\Psi_{\Omega}[\omega] = \Delta_{\Omega}^{-1} \omega$,
where $\Delta_{\Omega}^{-1}$ is the usual inverse of the Laplacian with homogeneous Dirichlet boundary conditions on $\partial \Omega$.
Now the difference $R[\omega] = K[\omega] - K_{\Omega}[\omega]$ satisfies
\begin{equation} \label{Eq:DefR-BS}
\left\{ \begin{array}{l}
\div R[\omega] = \curl R[\omega] = 0 \ \text{ in } \ {\mathcal F}({\bf q}), \medskip \\
R[\omega]\cdot n = -K_{\Omega}[\omega] \cdot n \ \text{ on } \ \partial {\mathcal S}_{\nu}  \ \text{ for } \ \nu =1, \dots, N, \medskip \\
R[\omega]\cdot n = 0 \ \text{ on } \ \partial \Omega, \medskip \\
\displaystyle \oint_{\partial {\mathcal S}_{\nu}} R[\omega] \cdot \tau \, ds = 0  \ \text{ for } \ \nu =1, \dots, N.
\end{array} \right.
\end{equation}
It follows that $R[\omega]$ can be put in the form $R[\omega]= \nabla^\perp \eta[\omega]$ with
\begin{equation*}
\left\{ \begin{array}{l}
\Delta \eta[\omega] =0 \ \text{ in } \ {\mathcal F}({\bf q}), \medskip \\
\eta[\omega] = -\Psi_{\Omega}[\omega] + c_{\nu} \ \text{ on } \ \partial {\mathcal S}_{\nu}, \ \text{ for all } \ \nu=1,\dots,N, \medskip \\
\eta[\omega] = 0 \ \text{ on } \ \partial \Omega, \medskip \\
\displaystyle \int_{\partial {\mathcal S}_{\nu}} \partial_{n} \eta[\omega] \, ds = 0  \ \text{ for } \ \nu =1, \dots, N.
\end{array} \right.
\end{equation*}
%
%
%
Consequently, using Corollary~\ref{Cor:SD} we find that for some constants $c'_{\lambda}$,  $\lambda \in \{1,\dots,N\}$,  one has
$$\frac{\partial \eta[\omega]}{\partial q_{\mu,m}} = 0  \ \text{ on } \ \partial \Omega  \  \text{  and } \
\frac{\partial \eta[\omega]}{\partial q_{\mu,m}} = c'_\lambda \ \text{ on }  \ \partial \mathcal{S}_\lambda \ \text{ for } \ \lambda \neq \mu ,$$
while  on $\partial {\mathcal S}_{\mu}({\bf q})$, one has
\begin{eqnarray*}
\frac{\partial \eta[\omega]}{\partial q_{\mu,m}} &=&  \left( - \partial_{n} \Psi_{\Omega}[\omega] - \partial_{n} \eta[\omega] \right)  \, K_{\mu,m} + c'_{\mu} \\
&=& (K_{\Omega}[\omega] +R[\omega]) \cdot \tau  \, K_{\mu,m} + c'_{\mu} = (K[\omega] \cdot \tau ) \, K_{\mu,m} + c'_{\mu}. 
\end{eqnarray*}
Using Lemma~\ref{Lem:BiotSavart}, we can bound this boundary condition as in the proofs of Proposition~\ref{Pro:ExpShapeDerivatives} or Lemma~\ref{Lem:DWBorne}, so that we obtain for some uniform constant $C>0$
\begin{equation*}
\left\|\frac{\partial \eta[\omega]}{\partial q_{\mu,m}}\right\|_{L^\infty(\partial {\mathcal S}_\mu)} 
+ \varepsilon_\mu^{\frac{5}{2}} \left|\frac{\partial \eta[\omega]}{\partial q_{\mu,m}} \right|_{C^{2,\frac{1}{2}}(\partial {\mathcal S}_{\mu})} 
\leq C \varepsilon_{\mu}^{\delta_{m3}} \left\| \omega \right\|_{L^{\infty}({\mathcal F}^\varepsilon({\bf q}))} .
\end{equation*}
Then we use that 
$\frac{\partial \eta[\omega]}{\partial q_{\mu,m}} = \mathfrak{f}_{\mu} \left[ \left(\frac{\partial \eta[\omega]}{\partial q_{\mu,m}}\right)_{|\partial {\mathcal S}_{\mu}} \right] $  in $ {\mathcal F}^\varepsilon({\bf q})$, Propositions~\ref{Pro:DirichletPetits} and \ref{Pro:DirichletGros} to approximate it by the functions $\widehat{\mathfrak{f}}_{\lambda}$ and $ \widecheck{g}$, and we estimate the latter by Proposition~\ref{Pro:DirichletSAM}. The estimate in $L^p$ norm is exactly the same as \eqref{Eq:DWBorne2}. We omit the details. 
\end{proof}
%
%
%
%
%
%
%
%
%
%
%
%
%
%
\section{First \textit{a priori} estimates}
\label{Sec:FAPE}
In this section we establish several a priori estimates on the system: on the fluid vorticity, on a renormalized  energy of the system, giving a first bound of the solid velocities  (which will be improved later on by modulated energy estimates), and  a rough estimate of the solid accelerations. As we will see, these accelerations estimates are not straightforward and rely on a global normal form for the solids equations.
They will help uncouple a bit the equations and obtain {\it individual normal forms} for the solid equations in Section~\ref{Sec:NormalForm}.
\subsection{Vorticity estimates}
\begin{Lemma} \label{Lem:Vorticite}
For a solution to System \eqref{Eq:Euler}-\eqref{Eq:Newton} and $p \in [1,+\infty]$, $\| \omega \|_{p}$ is conserved over time and given 
$\delta >0$, $\| K[\omega] \|_{\infty}$ is bounded independently of $t$ and $\boldsymbol\varepsilon$ {as long as $(\varepsilon,{\bf q},\omega) \in \mathfrak{Q}_{\delta}$}.
\end{Lemma}
\begin{proof}
The first statement is due to
\begin{equation} \label{Eq:Vorticite}
\partial_{t} \omega^\varepsilon + (u^\varepsilon \cdot \nabla) \omega^\varepsilon =0 \ \text{ in } \ {\mathcal F}^\varepsilon,
\end{equation}
and Liouville's theorem. The second follows then from Lemma~\ref{Lem:BiotSavart}.
\end{proof}

%
%
%
%
\subsection{Energy estimates}

This subsection is devoted to a sort of energy estimate, which gives a first bound of $\widehat{p}_{\kappa}$ (recall the definition in \eqref{Eq:PPChapeau}).

\begin{Proposition}
 \label{Pro:APEE}
Let $\delta>0$. There exist $C>0$ and $\varepsilon_{0}>0$ such that as long as $(\boldsymbol\varepsilon,{\bf q},\omega) \in \mathfrak{Q}_\delta^{\varepsilon_{0}}$, the solutions $(u^\varepsilon,h^\varepsilon,\vartheta^\varepsilon)$ of the system satisfy
\begin{equation} \label{Eq:APEEnergy}
\forall \kappa \in \{1,\dots,N\}, \ \ |\varepsilon_{\kappa}^{\delta_{\kappa \in {\mathcal P}_{(iii)}}} \widehat{p}_{\kappa}| \leq C.
\end{equation}
\end{Proposition}
Let us mention that this estimate will be improved in the sequel, by considering modulated energy estimates. 
\begin{proof}[Proof of Proposition~\ref{Pro:APEE}]
We first consider the total energy of the system: 
\begin{equation} \label{Eq:Energy} 
{\mathcal E}(t) :=  \frac{1}{2} \sum_{\kappa \in \{1, \dots, N\} } (m_{\kappa} |h'_{\kappa}|^2 + J_{\kappa} | \vartheta'_{\kappa}| ^2)
+ \frac{1}{2} \int_{{\mathcal F}(t)} |u|^2 \, dx .
\end{equation}
For a solution to \eqref{Eq:Euler}-\eqref{Eq:Newton}, this energy  ${\mathcal E}(t)$ is conserved over time.
This is proven by multiplying \eqref{Eq:Euler} by $u$, the equations in \eqref{Eq:Newton} by $h_{\kappa}'$ and $\vartheta_{\kappa}'$, respectively, summing and integrating by parts.
Now the conservation of ${\mathcal E}(t)$ is insufficient to reach Proposition~\ref{Pro:APEE} directly because the energy is not bounded as $\overline{\boldsymbol\varepsilon}$ goes to $0$. This is due to of the circulation part of the fluid velocity (see the second term in the decomposition \eqref{Eq:DecompUeps}) corresponding to small solids.
To circumvent this difficulty we will rather use the following quantity:
\begin{equation}
\label{CARRE}
\frac{1}{2}\sum_{\kappa \in \{1, \dots, N\} } (m_{\kappa} |h'_{\kappa}|^2 + J_{\kappa} | \vartheta'_{\kappa}| ^2)
+ \frac{1}{2} \int_{{\mathcal F}(t)} | u^{pot}  |^2 \, dx   ,
\end{equation}
where the potential part of the fluid velocity $u^{pot}$ was defined in \eqref{Eq:DefUPot}.
Since, by \eqref{Eq:MasseTotale},
\begin{equation*}
\frac{1}{2}\sum_{\kappa \in \{1, \dots, N\} } (m_{\kappa} |h'_{\kappa}|^2 + J_{\kappa} | \vartheta'_{\kappa}| ^2)
+ \frac{1}{2} \int_{{\mathcal F}(t)} | u^{pot} |^2 \, dx  =
\displaystyle \frac{1}{2} {\mathcal M} \, {\bf p} \, \cdot \, {\bf p}  ,
\end{equation*}
in order to prove Proposition~\ref{Pro:APEE}, it is sufficient to show that the quantity above is bounded independently of $t$ and $\boldsymbol\varepsilon$.
Indeed, once this is obtained, one uses ${\mathcal M}_{g} \leq {\mathcal M}$ to get a bound on $\widehat{p}_{\kappa}$ for $\kappa \in {\mathcal P}_{(i)} \cup {\mathcal P}_{(ii)}$, and one uses ${\mathcal M}_{a} \leq {\mathcal M}$ together with Corollary~\ref{Cor:ExpAddedMass} and Remarks~\ref{Rem:PasBoules} and \ref{Rem:MA/MASAM} to deduce a bound on $\varepsilon_{\kappa} \widehat{p}_{\kappa}$ when $\kappa$ is in ${\mathcal P}_{(iii)}$. \par
To prove that the quantity in \eqref{CARRE} is bounded 
 we rely on the decomposition \eqref{Eq:DecompUeps} of the fluid velocity. We call $u ^{c}$ the circulation part of the fluid velocity, that is second term in the right-hand side of \eqref{Eq:DecompUeps}:
\begin{equation} \nonumber 
u ^{c}:= \sum_{\kappa \in \{1, \dots, N\} } \gamma_{\kappa} \nabla^{\perp} \psi_{\kappa}(q(t), \cdot).
\end{equation}
Since $K[\omega]$ is orthogonal to $u^{pot}$ in $L^2({\mathcal F}(q))$ (as follows from an integration by parts), we can decompose the energy \eqref{Eq:Energy} as
\begin{align*}
{\mathcal E}(t)
& = \frac{1}{2}\sum_{\kappa \in \{1, \dots, N\} } (m_{\kappa} |h'_{\kappa}|^2 + J_{\kappa} | \vartheta'_{\kappa}| ^2)
+ \frac{1}{2} \int_{{\mathcal F}(t)} |u^{pot} (t,\cdot)|^2 \, dx  \\
& +  \frac{1}{2} \int_{{\mathcal F}(t)} | K[\omega] |^2 \, dx 
+ \int_{{\mathcal F}(t)} u ^{c}(t,\cdot) \cdot (K[\omega] + u^{pot} ) (t,\cdot) \, dx 
+ \frac{1}{2} \int_{{\mathcal F}(t)} |u_c|^2 \, dx .
\end{align*}
Proposition~\ref{Pro:APEE} then follows from the assumptions on the initial data, Lemma~\ref{Lem:Vorticite}, the conservation of ${\mathcal E}(t)$ and  the following lemma.  
%
\begin{Lemma} \label{Lem:RenormalisationEnergie}
For $\delta>0$, there exists $\varepsilon_{0}>0$ such that the following properties hold as long as $(\varepsilon,{\bf q},\omega) \in {\mathfrak Q}^{\varepsilon_{0}}_{\delta}$:
\begin{gather}
\label{Eq:DiffPartiesSingulieres}
\int_{{\mathcal F}(t)} |u ^{c}(t,\cdot)|^2 \, dx - \int_{{\mathcal F}(0)} |u ^{c}(0,\cdot)|^2 \, dx  \text{ is bounded independently of } t \text{ and } \boldsymbol\varepsilon, \\
\label{Eq:DecompOrth}
\int_{{\mathcal F}(t)} u ^{c}(t,\cdot) \cdot  (K[\omega] + u^{pot} ) (t,\cdot) \, dx
\text{ is bounded independently of } t \text{ and } \boldsymbol\varepsilon .
\end{gather}
\end{Lemma}
\begin{proof}[Proof of Lemma~\ref{Lem:RenormalisationEnergie}]
We first notice that the vector fields $\nabla^\perp \psi_{\nu}$ are orthogonal one to another in $L^2$  as follows from an integration by parts.
Hence to prove \eqref{Eq:DiffPartiesSingulieres} it suffices to prove that the circulation stream functions $\psi_{\nu}$ satisfy 
\begin{equation*}
\int_{{\mathcal F}(t)} |\nabla \psi_{\nu}({\bf q}(t),\cdot)|^2 \, dx - \int_{{\mathcal F}(0)} |\nabla \psi_{\nu}({\bf q}(0),\cdot)|^2 \, dx
\text{ is bounded independently of } t \text{ and } \boldsymbol\varepsilon.
\end{equation*}
We use Lemma~\ref{Lem:WBorne}; consequently it suffices to prove that for all $\nu$, the standalone circulation stream function $\widehat{\psi}_{\nu}$ satisfies 
\begin{equation} \label{Eq:DiffPartiesSingulieresBis}
\int_{{\mathcal F}(t)} |\nabla \widehat{\psi}_{\nu}({\bf q}(t),\cdot)|^2 \, dx - \int_{{\mathcal F}(0)} |\nabla \widehat{\psi}_{\nu}({\bf q}(0),\cdot)|^2 \, dx
\text{ is bounded independently of } t \text{ and } \boldsymbol\varepsilon.
\end{equation}
Now using Lemma~\ref{Lem:StandaloneCSF}, we see that 
\begin{equation*}
\int_{{\mathcal F}(t)} |\nabla \widehat{\psi}_{\nu}({\bf q}(t),\cdot)|^2 \, dx = \int_{R(-\vartheta_\nu) ( {\mathcal F}(t) - h_\nu(t)) + h_{\nu,0} } |\nabla \widehat{\psi}_{\nu}({\bf q}(0),\cdot)|^2 \, dx.
\end{equation*}
Then
\begin{equation*}
\left| \int_{{\mathcal F}(t)} |\nabla \widehat{\psi}_{\nu}({\bf q}(t),\cdot)|^2 \, dx
- \int_{{\mathcal F}(0)} |\nabla \widehat{\psi}_{\nu}({\bf q}(0),\cdot)|^2 \, dx \right|
\leq  \int_{\Delta_\nu} |\nabla \widehat{\psi}_{\nu}({\bf q}(0),\cdot)|^2 \, dx ,
\end{equation*}
where $\Delta_\nu$ is the symmetric difference $\big( R(-\vartheta_\nu) ( {\mathcal F}(t) - h_\nu(t)) + h_{\nu,0} \big) \, \triangle \, {\mathcal F}(0)$.
Since $(\varepsilon,{\bf q},\omega) \in {\mathfrak Q}^{\varepsilon_{0}}_{\delta}$ and ${\mathcal F}(t) \subset \Omega$, there is $R>0$ independent of $\varepsilon$ such that 
$\Delta_\nu \subset  B(h_{\nu,0} , R) \setminus B(h_{\nu,0} , \delta)$. Hence using \eqref{Eq:BehaviourPsi1}, we arrive at \eqref{Eq:DiffPartiesSingulieresBis} and hence at \eqref{Eq:DiffPartiesSingulieres}.

To get \eqref{Eq:DecompOrth} we first integrate by parts:
\begin{equation*}
\int_{{\mathcal F}(t)} \nabla^{\perp} \psi_{\nu} \cdot  (K[\omega] + u^{pot} )  \, dx = - \int_{{\mathcal F}(t)} \psi_{\nu} \omega \, dx + \int_{\partial {\mathcal F}(t)} \psi_{\nu}  (K[\omega] + u^{pot} )  \cdot \tau \, ds(x).
\end{equation*}
The part of the second integral on $\partial \Omega$ vanishes  due to \eqref{def_stream_F}, and the parts of the second integral  
on each $\partial \mathcal{S}_\lambda$, $\lambda=1, \dots, N$, vanish as well 
because $\psi_{\nu}$ is constant on each connected component of the boundary and $K[\omega] + u^{pot}$ has zero-circulation on each $\partial \mathcal{S}_\lambda$. Now the first term is bounded independently of $t$ and $\boldsymbol\varepsilon$, because $ \psi_{\nu}$ is bounded on the support of $\omega$: this can be seen by integrating $\nabla  \psi_{\nu}$ from some point in $\partial \Omega$ and using Lemmas \ref{Lem:StandaloneCSF} and \ref{Lem:WBorne} and to the remoteness of $ \mathcal{S}_\lambda$ to the support of $\omega$ (due to $(\varepsilon,{\bf q},\omega) \in {\mathfrak Q}^{\varepsilon_{0}}_{\delta}$). 
\end{proof}
Hence the proof of Proposition~\ref{Pro:APEE} is complete.
\end{proof}
\subsection{Rough estimate for the acceleration of the bodies}
The goal of this subsection is to prove the following statement.

\begin{Proposition}
 \label{Pro:Acceleration}
Let $\delta>0$. There exists $C>0$ and ${\varepsilon}_0 >0$ such that the solutions $(u^\varepsilon,h^\varepsilon,\vartheta^\varepsilon)$ of the system satisfy, as long as $(\varepsilon,{\bf q},\omega) \in {\mathfrak Q}^{\varepsilon_{0}}_{\delta}$, 
\begin{equation} \label{Eq:AccelerationEst}
\forall \kappa \in \{1, \dots, N\}, \ \ | \varepsilon_{\kappa}^{2 \delta_{\kappa \in {\mathcal P}_{(iii)}}} \widehat{p}\, '_{\kappa}| \leq C(1 + |\widehat{\boldsymbol{p}}^{\varepsilon}|) .
\end{equation}
\end{Proposition}
The rest of the subsection is devoted to the proof of Proposition \ref{Pro:Acceleration}.
\subsubsection{A decomposition of the velocity}
The proof of Proposition \ref{Pro:Acceleration} relies on the following decomposition of the velocity.
\begin{Definition}
We decompose the velocity field $u^\varepsilon$ as follows:
\begin{equation} \label{Eq:DecompGlob}
u^\varepsilon = u^{pot}  + \sum_{\nu \in \{1,\dots,N\}} \gamma_{\nu} \nabla^{\perp} \widehat{\psi}_{\nu} + u^{ext},
\index{Velocity fields!U3@$u^{ext}$: exterior part of the velocity field}
\end{equation}
where the potential part of the velocity $u^{pot}$ was defined in \eqref{Eq:DefUPot}. We will call $u^{ext}$ the exterior part of the velocity field.
\end{Definition}
Notice the difference between \eqref{Eq:DecompGlob} and the standard decomposition \eqref{Eq:DecompUeps}, in that the circulation potentials considered here are standalone, following the strategy hinted in Section \ref{secsec}, and developed below, see in particular the treatment of the term  $T_4$  in \eqref{Eq:QuadraticPsi}. 

Comparing the standard decomposition \eqref{Eq:DecompUeps} of $u^\varepsilon$ and \eqref{Eq:DecompGlob}, we see with \eqref{Eq:DefPsiKappar} that
\begin{equation} \label{Eq:ReecritureUext}
u^{ext} = K[\omega] +  \sum_{\nu \in \{1,\dots,N\}} \gamma_{\nu} \nabla^{\perp} \psi_{\nu}^r.
\end{equation}
An important property of the decomposition \eqref{Eq:DecompGlob} is given by the following lemma, concerning the field $u^{ext}$ associated with a solution to System \eqref{Eq:Euler}--\eqref{Eq:Newton}.
\begin{Lemma} \label{Lem:UextGlob}
Given $\delta>0$, there exist some constants $\varepsilon_{0}$ and $C>0$ such that, for a solution to the system, as long as $(\boldsymbol\varepsilon,{\bf q},\omega) \in \mathfrak{Q}_\delta^{\varepsilon_{0}}$, one has for $u^{ext}$ considered as a function of $(t,x)$:
\begin{gather}
\label{Eq:EstUextGlob}
\| u^{ext} \|_{L^\infty({\mathcal F}({\bf q}))} \leq C , \\
\label{Eq:EstUextGlob2}
\| \partial_{t} u^{ext} \|_{L^\infty({\mathcal V}_{\delta}(\partial {\mathcal F}) \setminus \bigcup_{\nu \in {\mathcal P}_{s}} {\mathcal V}_{\delta/2}(\partial {\mathcal S}_{\nu}) )} \leq C (1 + |\widehat{{\bf p}}^{\varepsilon}|),  \\
\label{Eq:EstUextGlob3}
\| \partial_{t} u^{ext} \|_{ L^\infty({\mathcal V}_{\delta}(\partial {\mathcal S}_{\nu})) } \leq C \varepsilon_{\nu}^{-1} (1 + |\widehat{{\bf p}}^{\varepsilon}|), \ \ \forall \nu \in {\mathcal P}_{s}.
\end{gather}
\end{Lemma}
\begin{proof}[Proof of Lemma~\ref{Lem:UextGlob}] \par
First, \eqref{Eq:EstUextGlob} follows from directly from \eqref{Eq:ReecritureUext} and Lemmas~\ref{Lem:WBorne} and \ref{Lem:BiotSavart}. 
For what concerns \eqref{Eq:EstUextGlob2}-\eqref{Eq:EstUextGlob3}, we start with 
\begin{equation} \label{Eq:ptuext}
\partial_{t} u^{ext} = K[\partial_{t} \omega^\varepsilon] + \sum_{\substack{{\mu \in \{1,\dots,N\}} \\ {m \in \{1,2,3\}}}} \frac{\partial}{\partial q_{\mu,m}} \left[ K[\omega^\varepsilon] +  \sum_{\nu \in \{1,\dots,N\}} \gamma_{\nu} \nabla^{\perp} \psi_{\nu}^r \right] \cdot p_{\mu,m}.
\end{equation}
The shape derivatives of $K[\omega^\varepsilon]$ and $\nabla^{\perp} \psi_{\nu}^r$ with respect to $q_{\mu,m}$ are estimated separately
in $L^\infty({\mathcal V}_{\delta}(\partial {\mathcal S}_{\mu}({\bf q})))$ 
and in $L^\infty({\mathcal F}({\bf q}) \setminus {\mathcal V}_{\delta/2}(\partial {\mathcal S}_{\mu}({\bf q})))$ 
by using Lemma \ref{Lem:DBSBorne} and Lemma~\ref{Lem:DWBorne} respectively. 
Observing that $\varepsilon_{\mu}^{\delta_{m3}} |p_{\mu,m}| = |\widehat{p}_{\mu,m}|$, 
it follows that the second term in \eqref{Eq:ptuext} gives a contribution as in \eqref{Eq:EstUextGlob2}-\eqref{Eq:EstUextGlob3}. \par
It remains to study
\begin{equation} \label{Eq:RelK}
K[\partial_{t} \omega^\varepsilon] = -K[\div(u^\varepsilon \omega^\varepsilon)].
\end{equation}
We estimate $u^\varepsilon \omega^\varepsilon$ using  the decomposition \eqref{Eq:DecompUeps}. Using that $(\boldsymbol\varepsilon,{\bf q},\omega) \in \mathfrak{Q}_\delta^{\varepsilon_{0}}$, the energy estimates and \eqref{Eq:ExpKirchhoff1Bis}, we deduce that 
$\| u^{pot} \, \omega^\varepsilon \|_{L^\infty({\mathcal F}({\bf q}))} \leq C$.
Using that $(\boldsymbol\varepsilon,{\bf q},\omega) \in \mathfrak{Q}_\delta^{\varepsilon_{0}}$ and Lemmas~\ref{Lem:StandaloneCSF} and \ref{Lem:WBorne}, we also find  that $\omega^\varepsilon \sum_{\nu \in \{1,\dots,N\}} \gamma_{\nu} \nabla^{\perp} \widehat{\psi}_{\nu}$ is bounded in $L^\infty({\mathcal F}(q))$. 
%
With \eqref{Eq:EstUextGlob}, we finally deduce that
\begin{equation} \label{Num41}
\| u^\varepsilon \omega^\varepsilon \|_{L^\infty({\mathcal F}({\bf q}))} \leq C.
\end{equation}
With \eqref{Eq:BiotSavartBorne2} and \eqref{Eq:DecompGlob}, this gives
\begin{equation} \label{Eq:Num42}
\| K[\div(u^\varepsilon \omega^\varepsilon)]\|_{L^p({\mathcal F}(q))} \leq C, \ \text{ for } \ p \in (1, +\infty). 
\end{equation}
By using the support of vorticity and local elliptic estimates near the boundaries one concludes that $ K[\div(u^\varepsilon \omega^\varepsilon)] $ is bounded in $L^\infty({\mathcal V}_{\delta}(\partial {\mathcal F}({\bf q})))$, and \eqref{Eq:EstUextGlob2}-\eqref{Eq:EstUextGlob3} follow.
\end{proof}
\subsubsection{Proof of the acceleration estimates}
We are now in position to prove Proposition~\ref{Pro:Acceleration}.
%
%
%
%
\begin{proof}[Proof of Proposition~\ref{Pro:Acceleration}] 
We cut the proof in several steps. \par
\paragraph{Step 1.}
By \eqref{Eq:Euler}, \eqref{Kir} and an integration by parts
we write the solid equation \eqref{Eq:Newton} as
\begin{equation} \label{Eq:Solide2}
( {\mathcal M}_{g} (\mathbf{p}^{\varepsilon})' )_{\kappa,j} = - \int_{{\mathcal F}({\bf q})} (\partial_{t} u^\varepsilon + (u^\varepsilon \cdot \nabla) u^\varepsilon) \cdot \nabla \varphi_{\kappa,j} \, dx ,
\end{equation}
where we recall  the notation \eqref{Eq:TrueInertia}.
Next we inject the decomposition  \eqref{Eq:DecompGlob} of $u^\varepsilon$. 
In the right-hand side, we extract from $\partial_{t} u^\varepsilon$ the part corresponding to 
\begin{equation} \label{Eq:dtupot}
\partial_{t} u^{pot} = \sum_{ \substack{ {\mu \in \{1, \dots,N\}} \\ {m \in \{1,2,3\}} } } \big[ p'_{\mu,m} \nabla \varphi_{\mu,m} + p_{\mu,m} (\nabla \varphi_{\mu,m})' \big].
\end{equation}
When injected in \eqref{Eq:Solide2}, the first term in \eqref{Eq:dtupot} gives the added mass term $-({\mathcal M}_{a} ({\bf p}^{\varepsilon})')_{\kappa,j}$ (recall the notation \eqref{Eq:AddMassMatrix}) which we put on the left-hand side, while the second one gives shape-derivatives terms, see the term $T_1$ below.
For the term $(u^\varepsilon \cdot \nabla) u^\varepsilon$ in \eqref{Eq:Solide2} we use
\begin{equation} \label{Eq:UNablaU}
(u^\varepsilon \cdot \nabla) u^\varepsilon = \frac{\nabla | u^\varepsilon| ^2}{2} + \omega^\varepsilon (u^{\varepsilon})^\perp . 
\end{equation}
When injected in \eqref{Eq:Solide2}, the first term in the right hand side of \eqref{Eq:UNablaU} can be integrated by parts  to arrive at 
\begin{equation*} 
- \frac{1}{2} \int_{\partial {\mathcal S}_{\kappa}({\bf q})} | u^\varepsilon| ^2 K_{\kappa,j} \, ds .
\end{equation*}
Then we develop the square 
$$
| u^\varepsilon| ^2 =   \left| u^{pot} + \sum_{\nu \in \{1,\dots,N\}}\gamma_{\nu} \nabla^{\perp} \widehat{\psi}_{\nu} + u^{ext} \right|^2 ,
$$
by separating between 
$$
\gamma_{\kappa} \nabla^{\perp} \widehat{\psi}_{\kappa}  \,  \text{ and }  \,  
 u^{pot} + u^{ext} + \sum_{\nu \in \{1,\dots,N\} \setminus  \{\kappa\} }\gamma_{\nu} \nabla^{\perp} \widehat{\psi}_{\nu} ,
$$
to arrive at 
%
 \begin{equation}
\label{Eq:RoDev}
\left({\mathcal M}_{g} ({\bf p}^{\varepsilon})'+ {\mathcal M}_{a} ({\bf p}^{\varepsilon})'\right)_{\kappa,j} 
= T_1 + \ldots + T_7 ,
\end{equation}
where
 \begin{align*}
T_1 &:= - \sum_{\substack{{\lambda,\mu \in \{1, \dots, N\}} \\ {\ell,m=1,2,3} }} \int_{{\mathcal F}({\bf q})} p_{\lambda,\ell} p_{\mu,m} \frac{\partial \nabla \varphi_{\lambda,\ell}}{\partial q_{\mu,m}} \cdot \nabla \varphi_{\kappa,j} \, dx , \\
T_2 &:= - \sum_{\nu \in \{1,\dots,N\} }\gamma_{\nu} \int_{{\mathcal F}({\bf q})} \partial_{t} \nabla^{\perp} \widehat{\psi}_{\nu} \cdot \nabla \varphi_{\kappa,j} \, dx  , \\
T_3 &:= - \int_{{\mathcal F}({\bf q})} \partial_{t} u^{ext} \cdot \nabla \varphi_{\kappa,j} \, dx , \\
T_4 &:= - \frac{1}{2} \int_{\partial {\mathcal S}_{\kappa}({\bf q})}  \left|    \gamma_{\kappa} \nabla^{\perp} \widehat{\psi}_{\kappa}   \right|^2 K_{\kappa,j} \, ds , \\
T_5 &:= -  \frac{1}{2} \int_{\partial {\mathcal S}_{\kappa}({\bf q})} \left| u^{pot} + u^{ext} +  \sum_{\nu \in \{1,\dots,N\} \setminus  \{\kappa\} } \gamma_{\nu} \nabla^{\perp} \widehat{\psi}_{\nu}  \right|^2 K_{\kappa,j} \, ds, \\
T_6 &:=  - \gamma_{\kappa}  \int_{\partial {\mathcal S}_{\kappa}({\bf q})}  \left(u^{pot} + u^{ext} +  \sum_{\nu \in \{1,\dots,N\} \setminus  \{\kappa\} }\gamma_{\nu} \nabla^{\perp} \widehat{\psi}_{\nu}\right)  \cdot \nabla^{\perp} \widehat{\psi}_{\kappa} \, ds, \\
T_7 &:=- \int_{{\mathcal F}({\bf q})}  \omega^{\varepsilon} u^{\varepsilon\perp} \cdot \nabla \varphi_{\kappa,j} \, dx.
\end{align*}
\paragraph{Step 2.}
We now estimate these seven terms. In this proof it will be convenient to take the convention of Remark~\ref{Rem:NormalisationKirchhoff} for the Kirchhoff potentials. \par
\ \par
\noindent
{\it Estimate of $T_1$.} We first integrate by parts:
\begin{equation} \label{Eq:AccTermeSD}
\int_{{\mathcal F}(t)} p_{\lambda,\ell} p_{\mu,m} \frac{\partial \nabla \varphi_{\lambda,\ell}}{\partial q_{\mu,m}} \cdot \nabla \varphi_{\kappa,j} \, dx
= p_{\lambda,\ell} p_{\mu,m}
\int_{\partial {\mathcal S}_{\kappa}}  \frac{\partial \varphi_{\lambda,\ell}}{\partial q_{\mu,m}} \, K_{\kappa,j} \, ds.
\end{equation}
To estimate the integral in the right-hand-side we rely on the estimates of the shape derivatives in Proposition~\ref{Pro:ExpShapeDerivatives}.
We distinguish several cases, according to the possible equalities between $\kappa$, $\lambda$ and $\mu$: \par
\begin{itemize}
\item  {First case: $\lambda=\mu$.}
Then either $\kappa=\lambda = \mu$ and this integral is ${\mathcal O}(\varepsilon_{\lambda}^{1+\delta_{\ell 3}}  \varepsilon_{\mu}^{\delta_{m3}}  \varepsilon_{\kappa}^{\delta_{j3}})$ (the additional power of $\varepsilon_{\lambda}$ comes from the integration on $\partial {\mathcal S}_{\kappa}= \partial {\mathcal S}_{\lambda}$),
  or $\kappa \neq \lambda = \mu$ and the integral is  ${\mathcal O}(\varepsilon_{\lambda}^{2+\delta_{\ell 3}}  \varepsilon_{\mu}^{\delta_{m3}}  \varepsilon_{\kappa}^{1+\delta_{j3}})$.
\item {Second case: $\lambda \neq \mu$.} %
Then either $\kappa\neq \mu$ and we see the integral is ${\mathcal O}(\varepsilon_{\lambda}^{2+\delta_{\ell 3}}  \varepsilon_{\mu}^{2+\delta_{m3}}  \varepsilon_{\kappa}^{1+\delta_{j3}})$,
 or $\kappa = \mu$  and the integral is  ${\mathcal O}(\varepsilon_{\lambda}^{2+\delta_{\ell 3}}  \varepsilon_{\mu}^{\delta_{m3}}  \varepsilon_{\kappa}^{1+\delta_{j3}})$.
\end{itemize}
We recall that $\varepsilon_{\mu}^{\delta_{m3}} |p_{\mu,m}| = |\widehat{p}_{\mu,m}|$.
Using the energy estimates provided by Proposition~\ref{Pro:APEE} (which give $\varepsilon_{\lambda}^{1+\delta_{\ell3}} p_{\lambda,\ell}$ bounded), we see that in all cases, the term in \eqref{Eq:AccTermeSD}  is at least estimated by ${\mathcal O}(|\widehat{p}_{\mu,m}| \varepsilon_{\kappa}^{\delta_{j3}})$ (the worst case being the first one where $\kappa=\lambda = \mu$). \par
\ \par
\noindent
{\it Estimate of $T_2$.} We first deduce from Lemma~\ref{Lem:StandaloneCSF} that 
\begin{equation*}
\partial_{t} \widehat{\psi}_{\nu} + v_{{\mathcal S},\nu} \cdot \nabla \widehat{\psi}_{\nu} =0
\ \text{ and } \ 
\partial_{t} \nabla \widehat{\psi}_{\nu} + (v_{{\mathcal S},\nu} \cdot \nabla) \nabla \widehat{\psi}_{\nu} = \vartheta_{k}' \nabla^\perp \widehat{\psi}_{\nu} ,
\end{equation*}
where we denote by $v_{{\mathcal S},\nu}$ the $\nu$-th solid vector field, see \eqref{Eq:SolidVelocity}.
Using the formulas 
\begin{equation*}
\nabla (a \cdot b) = (a \cdot \nabla) b + (b \cdot \nabla) a - a^{\perp} \curl (b) - b^{\perp} \curl (a), 
\end{equation*}
$\curl(x^{\perp})=2$ and $(a \cdot \nabla) x^{\perp} = a^\perp$, we find
\begin{equation} \label{Eq:HatPsiTourne}
\partial_{t} \nabla^{\perp} \widehat{\psi}_{\nu} + \nabla \left( v_{{\mathcal S},\nu} \cdot \nabla^{\perp} \widehat{\psi}_{\nu} \right) =0.
\end{equation}
By an integration by parts it follows that
\begin{equation*}
\nonumber
\int_{{\mathcal F}({\bf q})} \partial_{t} \nabla^{\perp} \widehat{\psi}_{\nu} \cdot \nabla \varphi_{\kappa,j} \, dx 
= - \int_{\partial {\mathcal S}_{\kappa}({\bf q})} v_{{\mathcal S},\nu} \cdot \nabla^{\perp} \widehat{\psi}_{\nu} \, K_{\kappa,j} \, ds.
\end{equation*}
Now when $\nu=\kappa$ it is straightforward to estimate this term by ${\mathcal O}(\varepsilon_{\kappa}^{\delta_{j3}}) |\widehat{\mathbf{p}}^{\varepsilon}|$ since $\nabla^{\perp} \widehat{\psi}^\varepsilon_{\kappa} = {\mathcal O}(1/\varepsilon_{\kappa})$ on $\partial {\mathcal S}_{\kappa}$. When $\nu \neq \kappa$, one can use the divergence theorem inside ${\mathcal S}_{\kappa}$:
\begin{equation} \label{Eq:PsiNu}
\int_{\partial {\mathcal S}_{\kappa}({\bf q})} v_{{\mathcal S},\nu} \cdot \nabla^{\perp} \widehat{\psi}_{\nu} \, K_{\kappa,j} \, ds
=
-   \int_{{\mathcal S}_{\kappa}({\bf q})} \div\left( (h'_{\nu} + \vartheta'_{\nu} (x-h_{\nu})^\perp \cdot \nabla^{\perp} \widehat{\psi}_{\nu}) \xi_{\kappa,j}\right) \, dx.
\end{equation}
Now on the one hand using \eqref{Eq:BehaviourPsi1} and interior regularity estimates for the Laplace equation, we obtain
\begin{equation*}
\int_{ {\mathcal S}_{\kappa}({\bf q})} \div( (h'_{\nu} \cdot \nabla^{\perp} \widehat{\psi}_{\nu}) \xi_{\kappa,j}) \, dx 
= \int_{{\mathcal S}_{\kappa}({\bf q})} \xi_{\kappa,j} \cdot \nabla (h'_{\nu} \cdot \nabla^{\perp} \widehat{\psi}_{\nu})  \, dx 
= {\mathcal O}(\varepsilon_{\kappa}^{2+ \delta_{j3}}) |h'_{\nu}|.
\end{equation*}
On the other hand, we use \eqref{Eq:BehaviourPsi2} and see that
\begin{equation*}
\int_{{\mathcal S}_{\kappa}({\bf q})} \div\left( (\vartheta'_{\nu} (x-h_{\nu})^\perp \cdot \nabla^{\perp} \widehat{\psi}_{\nu}) \xi_{\kappa,j}\right) \, ds
= {\mathcal O}(\varepsilon_{\kappa}^{2+ \delta_{j3}}) \varepsilon_{\nu} |\vartheta'_{\nu}|.
\end{equation*}
Altogether the term $T_{2}$ can be estimated by
\begin{equation} \label{Eq:EstT2}
T_{2} = {\mathcal O}(\varepsilon_{\kappa}^{2+\delta_{j3}}) |\widehat{p}_{\nu}|.
\end{equation}
\ \par
\noindent
{\it Estimate of $T_3$.}  We first integrate by parts to find
\begin{equation*}
\int_{{\mathcal F}({\bf q})} \partial_{t} u^{ext} \cdot \nabla \varphi_{\kappa,j} \, dx
=  \int_{\partial {\mathcal F}({\bf q})} \partial_{t} u^{ext} \cdot n \, \varphi_{\kappa,j} \, ds.
\end{equation*}
By Lemma~\ref{Lem:UextGlob} (using \eqref{Eq:EstUextGlob2} on $\partial \Omega$ and \eqref{Eq:EstUextGlob3} on the rest of the boundary), we have $\| \partial_{t} u^{ext} \|_{L^{1}(\partial {\mathcal F}({\bf q}))} = {\mathcal O}(1 + |\widehat{\bf p}^\varepsilon|)$.
We use \eqref{Eq:ExpKirchhoffNormalises} to estimate the Kirchhoff potential $\varphi_{\kappa,j}$ on the boundary and infer that
\begin{equation*}
T_{3} = {\mathcal O}(\varepsilon_{\kappa}^{1+\delta_{j3}}) (1 + |\widehat{{\bf p}}^{\varepsilon}|).
\end{equation*}
\ \par
\noindent
{\it Estimate of $T_4$.}  We have for any $j \in \{1,2,3\}$
\begin{equation} \label{Eq:QuadraticPsi}
\int_{\partial {\mathcal S}_{\kappa}} |\gamma_{\kappa} \nabla^{\perp} \widehat{\psi}_{\kappa}|^2 K_{\kappa,j} \, ds=0.
\end{equation}
This is a consequence of Blasius' lemma, see e.g. \cite[p. 511]{GLS}. This also a direct consequence of Lamb's lemma (see Lemma~\ref{Lem:Lamb} below).\par %
\ \par
\noindent
{\it Estimate of $T_5$.} 
Using Lemma~\ref{Lem:UextGlob}, Proposition~\ref{Pro:ExpKirchhoff} and \eqref{Eq:BehaviourPsi1}
we see that
\begin{equation} \label{Eq:Upot+Uext}
\left|u^{pot} + u^{ext} +  \sum_{\nu \in \{1,\dots,N\} \setminus  \{\kappa\} }\gamma_{\nu} \nabla^{\perp} \widehat{\psi}_{\nu} \right| \leq C \left(1+ |\widehat{p}_\kappa|
+ \sum_{\nu \neq \kappa} \varepsilon_\nu^2 |\widehat{p}_\nu|
\right)
\text{ on } \partial \mathcal{S}_\kappa.
\end{equation} 
Considering that $K_{\kappa,j}={\mathcal O}(\varepsilon_{\kappa}^{\delta_{j3}})$ and that we integrate over $\partial {\mathcal S}_{\kappa}$, using the energy estimates, we deduce that this term can be bounded by $C \varepsilon_{\kappa}^{\delta_{j3}} (1 + |\widehat{p}_{\kappa}|)$.\par 
\ \par
\noindent
{\it Estimate of $T_6$.} 
Using \eqref{Eq:Upot+Uext}, the energy estimates, $\nabla^{\perp} \widehat{\psi}^\varepsilon_{\kappa} = {\mathcal O}(1/\varepsilon_{\kappa})$ on $\partial {\mathcal S}_{\kappa}$ and again that $\partial {\mathcal S}_{\kappa}$ is of size ${\mathcal O}(\varepsilon_{\kappa})$, we see that this term is also estimated by $C \varepsilon_{\kappa}^{\delta_{j3}} (1 + |\widehat{{p}}_{\kappa}|)$. \par
\ \par
\noindent
{\it Estimate of $T_7$.} 
We use the decomposition \eqref{Eq:DecompGlob} of $u^{\varepsilon}$, the compactness of the support of $\omega^\varepsilon$ in ${\mathcal F}(q)$ due to  $(\boldsymbol\varepsilon,{\bf q},\omega) \in \mathfrak{Q}_\delta^{\varepsilon_{0}}$, the decay of the Kirchhoff potentials \eqref{Eq:ExpKirchhoff1Bis},  the energy estimates,
 \eqref{Eq:BehaviourPsi1}
and 
\eqref{Eq:EstUextGlob}
to conclude that this term is of order ${\mathcal O}(\varepsilon_{\kappa}^{1+\delta_{j3}})$. \par
\paragraph{Step 3.}
Gathering what precedes we have established, recalling \eqref{Eq:MasseTotale},
\begin{equation} \label{Eq:MainAccelerationEstimate}
\left| \left( {\mathcal M} {\bf p}' \right)_{\kappa,j} \right|
\leq C \varepsilon_{\kappa}^{\delta_{j3}} \left( 1 + |\widehat{\boldsymbol{p}}^{\varepsilon}| \right).
\end{equation}
Now define the ``homogeneous'' inertia matrix ${\mathcal M}^\circ$ as the total inertia matrix ${\mathcal M}$ where we divide each $(\kappa,j)$-th row and each $(\kappa,j)$-th column by $\varepsilon_{\kappa}^{\delta_{j3}}$. Then \eqref{Eq:MainAccelerationEstimate} translates now into
\begin{equation} \nonumber 
\left| \left( {\mathcal M}^\circ (\widehat{\boldsymbol{p}}^{\varepsilon})' \right)_{\kappa,j} \right|
\leq C \left( 1 + |\widehat{\boldsymbol{p}}^{\varepsilon}| \right).
\end{equation}
We now introduce the matrix ${\mathcal M}^\ast$ as the total homogeneous inertia matrix ${\mathcal M}^\circ$ where each $(\kappa,j)$-th column is divided
by $\varepsilon_{\kappa}^{\min(2,\alpha_{\kappa}) \delta_{\kappa \in {\mathcal P}_{(iii)}}}$, where we recall that $\alpha_{\kappa}$ was introduced in \eqref{Eq:Family_iii}. Calling $\widecheck{{\bf p}}$ the vector with $(\kappa,j)$-th coordinate $\varepsilon_{\kappa}^{\min(2,\alpha_{\kappa}) \delta_{\kappa \in {\mathcal P}_{(iii)}}} \widehat{p}_{\kappa,j}$, we hence have
\begin{equation*}
{\mathcal M}^\circ (\widehat{\boldsymbol{p}}^{\varepsilon})' = {\mathcal M}^\ast \widecheck{{\bf p}}'.
\end{equation*}
Hence to end the proof of Proposition~\ref{Pro:Acceleration}, it remains to prove that $({\mathcal M}^\ast)^{-1}$ is bounded independently of ${\boldsymbol\varepsilon}$ at least for small $\overline{{\boldsymbol\varepsilon}}$. 
Now  gathering the rows and columns of ${\mathcal M}^\ast$ according to families (i), (ii) and (iii), we have a block matrix:
\begin{equation*}
{\mathcal M}^\ast =
\left(
\begin{array}{c|c|c}
A_{(i)(i)} &  A_{(i)(ii)} & A_{(i)(iii)} \\
\hline
A_{(ii)(i)} &  A_{(ii)(ii)} & A_{(ii)(iii)} \\
\hline
A_{(iii)(i)} &  A_{(iii)(ii)} & A_{(iii)(iii)} 
\end{array}
\right).
\end{equation*}
%
%
%

%
%
%
%
Using Corollary~\ref{Cor:ExpAddedMass} we see that the entries of the added mass matrix ${\mathcal M}_{a}$ that correspond to different solids satisfy:
\begin{equation} \label{Eq:EstExtraDiag}
({\mathcal M}_{a})_{\lambda,\ell,\mu,m} = {\mathcal O}(\varepsilon_{\lambda}^{2+\delta_{\ell3}} \varepsilon_{\mu}^{2+\delta_{m3}}) \ \ \text{ for } \lambda \neq \mu, \ \ell,m=1,2,3.
\end{equation} 
Moreover, using Corollary~\ref{Cor:ExpAddedMass} and Remark~\ref{Rem:PasBoules}, we see that for $\lambda \in {\mathcal P}_{(iii)}$ and $\ell,m \in\{1,2,3\}$, 
\begin{equation*}
{\mathcal M}_{a,\lambda,\ell,\lambda,m} = \varepsilon_{\lambda}^{2 + \delta_{3 \ell} + \delta_{3m}} \widehat{{\mathcal M}}^1_{a,\lambda,\ell,m} + {\mathcal O}(\varepsilon_{\lambda}^{4 + \delta_{3 \ell} + \delta_{3m}}),
\end{equation*}
where $\widehat{{\mathcal M}}^1_{a,\lambda}$ is a fixed symmetric positive-definite matrix. \par
Relying on the genuine mass and \eqref{Eq:Family_i}-\eqref{Eq:Family_ii} for the first two families, and either on the genuine mass (when $\alpha_{\kappa} \leq 2$) or the added mass (when $\alpha_{\kappa} >2$) and \eqref{Eq:Family_iii} for the third family,  we deduce that the diagonal blocks $A_{(i)(i)}$, $A_{(ii)(ii)}$ and $A_{(iii)(iii)}$ are uniformly invertible. Moreover we also see that the blocks above the diagonal $A_{(i)(ii)}$, $A_{(i)(iii)}$ and $A_{(ii)(iii)}$ remain bounded. Hence by Cramer's rule the upper triangular block matrix
\begin{equation*}
{\mathcal M}^u :=
\left(
\begin{array}{c|c|c}
A_{(i)(i)} &  A_{(i)(ii)} & A_{(i)(iii)} \\
\hline
0 &  A_{(ii)(ii)} & A_{(ii)(iii)} \\
\hline
0 & 0 & A_{(iii)(iii)} 
\end{array}
\right),
\end{equation*}
whose determinant is $\det(A_{(i)(i)} ) \det(A_{(ii)(ii)} ) \det(A_{(iii)(iii)} )$,
is uniformly invertible. As can be seen from Neumann's series, when $\| {\mathcal M}^\ast - {\mathcal M}^u \| \leq \frac{1}{2 \| ({\mathcal M}^u)^{-1} \|}$ for some matrix norm, then ${\mathcal M}^\ast$ is invertible with $\| ({\mathcal M}^\ast)^{-1}\| \leq 2 \| ({\mathcal M}^u)^{-1} \|$. Since from \eqref{Eq:EstExtraDiag} the blocks under the diagonal $A_{(ii)(i)}$, $A_{(iii)(i)}$ and $A_{(iii)(ii)}$ converge to zero, we see that ${\mathcal M}^\ast$ is uniformly invertible for suitably small $\overline{{\boldsymbol\varepsilon}}$.
The result follows. 

\end{proof}
%
%

%
%
%
%
%
%
%
\section{Introduction of the modulations}
\label{Sec:Modulations}
%
%
%
%
%
%
%
In this section, we introduce the modulations that will play a central role in the normal forms of Section~\ref{Sec:NormalForm} and consequently in the modulated energy estimates of Section~\ref{Sec:MEE} and in the passage to the limit of Section~\ref{Sec:PTTL}.
\subsection{Decomposition of the fluid velocity focused on a small solid}

In this section, we merely consider $\kappa$ in $ {\mathcal P}_{s}$, because only the small solids will actually be concerned with the modulations.
To define the modulation, we first introduce a decomposition of the velocity field in the same spirit as \eqref{Eq:DecompGlob}, but here more focused on the $\kappa$-th solid.
\begin{Definition}
For each $\kappa $ in ${\mathcal P}_{s}$, we introduce the following decomposition 
\begin{equation} \label{Eq:DecompKappa}
u^\varepsilon = u^{pot}_\kappa  + \gamma_{\kappa} \nabla^{\perp} \widehat{\psi}_{\kappa} + u^{ext}_\kappa
\ \ \text{ with } \ \ 
u^{pot}_\kappa:= \sum_{i \in \{1,2,3\}} p_{\kappa,i} \nabla \varphi_{\kappa,i}.
\index{Velocity fields!U4@$u^{pot}_\kappa$, $u^{ext}_\kappa$: decomposition of $u^\varepsilon$ focused on ${\mathcal S}_{\kappa}$}
\end{equation}
We will refer to $u^{pot}_\kappa$ as potential part of the decomposition \eqref{Eq:DecompKappa}, $\gamma_{\kappa} \nabla^{\perp} \widehat{\psi}_{\kappa}$ as its circulation part,
and $u^{ext}_\kappa$ as the $\kappa$-th exterior field.
\end{Definition}
When comparing with the decomposition \eqref{Eq:DecompGlob}, we see that
\begin{equation} \label{Eq:UextUextKappa}
u^{ext}_\kappa = u^{ext} + \sum_{\nu \neq \kappa} \sum_{i=1}^3  p_{\nu,i} \nabla \varphi_{\nu,i} + \sum_{\nu \neq \kappa} \gamma_{\nu} \nabla^\perp \widehat{\psi}_{\nu}.
\end{equation}
The $\kappa$-th exterior field will play a central role in the definition of the modulation.
 In \eqref{Eq:DecompKappa}, the first two vector fields can be thought as ``attached'' to ${\mathcal S}_{\kappa}$ (to its velocity and to the constant circulation around it), while $u_{\kappa}^{ext}$ corresponds to the vector field to which ${\mathcal S}_{\kappa}$ ``is subjected'' from the exterior (which includes the reflections of $\nabla^\perp \widehat{\psi}_{\kappa}$ on $\partial \Omega$ and the other solids). \par
We first note that, due to \eqref{Eq:DecompKappa}, $u^{ext}_\kappa$ satisfies the following $\div$-$\curl$ system
\begin{equation} \label{Eq:SysUExtKappa}
\left\{ \begin{array}{l}
\div u^{ext}_\kappa = 0 \ \text{ in } \ {\mathcal F}({\bf q}), \medskip \\
\curl u^{ext}_\kappa = \omega^{\varepsilon} \ \text{ in } \ {\mathcal F}({\bf q}), \medskip \\
u^{ext}_\kappa\cdot n = - \gamma_{\kappa} \nabla^\perp \widehat{\psi}_{\kappa} \cdot n + \sum_{\nu \neq \kappa}  \sum_{i=1}^3  p_{\nu,i} \nabla \varphi_{\nu,i} \cdot n \ \text{ on } \ \partial {\mathcal F}({\bf q}), \medskip \\
\displaystyle \oint_{\partial {\mathcal S}_{\nu}} u^{ext}_\kappa \cdot \tau \, ds = \delta_{\nu \neq \kappa} \gamma_{\nu}  \ \text{ for } \ \nu =1, \dots, N.
\end{array} \right.
\end{equation}
Recall that $ \nabla^\perp \widehat{\psi}_{\kappa}$ is tangent to $\partial {\mathcal S}_{\kappa}$; it follows in particular that $u^{ext}_\kappa\cdot n = 0$ on $\partial {\mathcal S}_{\kappa}$. \par
We have the following estimate of the $\kappa$-th exterior field $u^{ext}_\kappa$.
\begin{Lemma} \label{Lem:EstExtField} 
Let $\delta>0$. There exists $\varepsilon_{0}>0$ and $C>0$ such that for all $\overline{\boldsymbol\varepsilon}$ with $\overline{\boldsymbol\varepsilon} \leq \varepsilon_{0}$,
  as long as {$(\boldsymbol\varepsilon,{\bf q},\omega) \in \mathfrak{Q}_{\delta}^{\varepsilon_{0}}$:}
\begin{equation} \nonumber 
\| u^{ext}_\kappa \|_{L^\infty(\partial {\mathcal S}_{\kappa})} \leq C.
\end{equation}
\end{Lemma}
\begin{proof}
Thanks to  Lemma~\ref{Lem:UextGlob}, we only have to estimate the two sums in the right-hand side of \eqref{Eq:UextUextKappa}.
For that purpose, we rely on the fact that that the sums are over $\nu \neq \kappa$.
Concerning the Kirchhoff potential parts we can use $\nabla \varphi_{\nu,i} = {\mathcal O}(\varepsilon_{\nu}^{2+\delta_{i3}})$ on $\partial {\mathcal S}_{\kappa}$
(Proposition~\ref{Pro:ExpKirchhoff}) and the energy estimates (Proposition~\ref{Pro:APEE}) to deduce that this term is bounded.
Concerning the circulation part, due \eqref{Eq:BehaviourPsi1} we have $\nabla^\perp \widehat{\psi}_{\nu} = {\mathcal O}(1)$ on $\partial {\mathcal S}_{\kappa}$
for $\nu \neq \kappa$, which also yields a bounded term. 
\end{proof}
%
%
%

%
%
%
%
%
%
%
%
%
\subsection{Approximation of the $\kappa$-th exterior field}
The goal of this paragraph is to show how $u^{ext}_\kappa$ can be approximated on $\partial {\mathcal S}_{\kappa}$ by a linear combination of four basic vector fields. For this we introduce the following notations.
Recalling \eqref{def-xi-j}, we denote for each $\kappa \in \{1, \dots,N\}$
\begin{equation} \nonumber 
{\mathcal K}_{\kappa}={\mathcal K}_{\kappa}({\bf q}):= \mbox{Span\,} \{ \xi_{\kappa,1}, \xi_{\kappa,2}, \xi_{\kappa,3},  \xi_{\kappa,4}, \xi_{\kappa,5} \} 
\ \text{ and } \ 
{\mathcal K}_{\kappa,s}={\mathcal K}_{\kappa,s}({\bf q}):= \mbox{Span\,} \{ \xi_{\kappa,1}, \xi_{\kappa,2}, \xi_{\kappa,4}, \xi_{\kappa,5} \}.
\index{X@Miscellaneous!M2@${\mathcal K}_{\kappa}$: vector space of affine vector fields}
\end{equation}
Note in particular that $\xi_{\kappa,3}$ is excluded from ${\mathcal K}_{\kappa,s}$.
Together with these spaces, we define the linear operator $\mbox{Kir}_{\kappa}$, defined on ${\mathcal K}_{\kappa}$, transforming an affine vector field in the corresponding linear combination of Kirchhoff vector fields; it is defined by
\begin{equation}
\label{alaph}
\mbox{Kir}_{\kappa}(\xi_{\kappa,i}) = \nabla \varphi_{\kappa,i} \ \text{ for all } i =1,2,3,4,5.
\index{X@Miscellaneous!M3@$\mbox{Kir}_{\kappa}$: transformation of affine vector fields in Kirchhoff vector fields}
\end{equation} 
This operator depends implicitly on ${\bf q}$ and $\varepsilon$. (Actually one may notice that ${\mathcal K}_{\kappa}$ and ${\mathcal K}_{s,\kappa}$ do not depend on ${\bf q}$ or $\varepsilon$; but the operators $\mbox{Kir}_{\kappa}$ do.)
Similarly we introduce 
\begin{equation}
\label{Kirchap}
\widehat{\mbox{Kir}}_{\kappa}(\xi_{\kappa,i}) = \nabla \widehat\varphi_{ \kappa,i}  \ \text{ for all } i =1,2,3,4,5.
\index{X@Miscellaneous!M3@$\widehat{\mbox{Kir}}_{\kappa}$: transformation of affine vector fields in standalone Kirchhoff vector fields}
\end{equation} 
It is an direct consequence of Proposition~\ref{Pro:ExpKirchhoff} that
\begin{equation} \label{Eq:DiffKir}
\left|{\mbox{Kir}}_{\kappa}(\xi_{\kappa,i})  - \widehat{\mbox{Kir}}_{\kappa}(\xi_{\kappa,i}) \right| \leq C \varepsilon_\kappa^{2+\delta_{i \geq 3}}
\ \text{ on } \ \partial {\mathcal S}_{\kappa}.
\end{equation}
%
%
%
Let us now describe a vector field $V_{\kappa} \in {\mathcal K}_{\kappa,s}$ that generates our approximation of $u^{ext}_\kappa$.
Having \eqref{Eq:SysUExtKappa} in mind, we first introduce the solution $\widecheck{u}^{k}=\widecheck{u}^{k}({\bf q},{\bf p},\omega,\cdot)$
in $\widecheck{{\mathcal F}}_{\kappa}({\bf q})$ (recall that this domain was introduced in \eqref{Eq:DomaineKappaAugmente})
of the following system: 
\begin{equation} \label{Eq:SysUcheck}
\left\{ \begin{array}{l}
\div \widecheck{u}_\kappa = 0 \ \text{ in } \ \widecheck{{\mathcal F}}_{\kappa}({\bf q}), \medskip \\ 
\curl \widecheck{u}_\kappa = \omega \ \text{ in } \ \widecheck{{\mathcal F}}_{\kappa}({\bf q}), \medskip \\
\widecheck{u}_\kappa \cdot n = - \gamma_\kappa \nabla^\perp \widehat{\psi}_{\kappa} \cdot n  +  \sum_{\nu \neq \kappa}  \sum_{i=1}^3   p_{\nu,i} \nabla \varphi_{\nu,i} \cdot n
 \ \text{ on } \ \partial \widecheck{{\mathcal F}}_{\kappa}({\bf q}), \medskip \\
\displaystyle \oint_{\partial {\mathcal S}_{\nu}({\bf q})} \widecheck{u}_\kappa \cdot \tau \, ds  = \gamma_{\nu} \ \text{ for } \ \nu \in \{1,\dots,N\} \setminus \{ \kappa \}.
\end{array} \right.
\index{Velocity fields!U5@$\widecheck{u}_\kappa$: approximation of the $\kappa$-th exterior field defined in $\widecheck{\mathcal{F}}_{\kappa}$}
\end{equation}
%
We start with the following lemma which estimates $\widecheck{u}_{\kappa}$ regardless of the fact that it comes from a solution to System \eqref{Eq:Euler}--\eqref{Eq:Newton}. We recall the notation \eqref{Def:NuVoisinage} for ${\mathcal V}_{\delta}(\partial {\mathcal S}_{\kappa})$.
\begin{Lemma} \label{Lem:CheckUKappa}
Given $\delta>0$ there exist constants $\varepsilon_{0}$ and $C>0$ such that as long as $(\boldsymbol\varepsilon,{\bf q},\omega) \in \mathfrak{Q}_\delta^{\varepsilon_{0}}$,
for all $\kappa \in {\mathcal P}_{s}$, all $\mu \in \{1,\dots,N\}$ and $m \in \{1,2,3\}$, one has:
\begin{equation} \nonumber 
\| \widecheck{u}_\kappa \|_{L^\infty({\mathcal V}_{\delta}(\partial {\mathcal S}_{\kappa}))} \leq C \left(1 + \| \omega \|_{\infty} + \sum_{\nu \neq \kappa} \varepsilon_{\nu}^2 |\widehat{p}_{\nu}|\right)
\text{ and }
\left\| \frac{\partial \widecheck{u}_\kappa}{\partial q_{\mu,m}} \right\|_{L^\infty({\mathcal V}_{\delta}(\partial {\mathcal S}_{\kappa}))} \leq C \varepsilon_{\mu}^{\delta_{m3}}
\left(1 + \| \omega \|_{\infty} + \sum_{\nu \neq \kappa} \varepsilon_{\nu} |\widehat{p}_{\nu}|\right).
\end{equation}
\end{Lemma}
\begin{proof}[Proof of Lemma~\ref{Lem:CheckUKappa}]
The proof is roughly the same as for Lemma~\ref{Lem:UextGlob} with the exception that we consider functions of $({\bf q},x)$ rather than $(t,x)$ and that the domain is no longer ${\mathcal F}({\bf q})$ but $\widecheck{{\mathcal F}}_{\kappa}({\bf q})$. This latter difference actually simplifies the proof because it avoids the singularity in the neighborhood of ${\mathcal S}_{\kappa}$. 
We call $\varphi_{\nu,i}^{\not \kappa}$ the various Kirchhoff potentials in $\widecheck{{\mathcal F}}_{\kappa}({\bf q})$, $\nu \in \{1,\dots,N\} \setminus \{\kappa \}$, $i \in \{1,2,3\}$, 
$K^{\not \kappa}$ the Biot-Savart operator in $\widecheck{{\mathcal F}}_{\kappa}({\bf q})$, and 
$\psi_{\nu}^{\not \kappa}$, for  $\nu \in \{1,\dots,N\} \setminus \{\kappa \}$, the various circulation stream functions in $\widecheck{{\mathcal F}}_{\kappa}({\bf q})$. 
We recall that for $\nu=\kappa$, $\psi^{r,\not \kappa}_{\kappa}$ was defined in \eqref{Eq:PsirefKappa}.
Correspondingly we see from \eqref{Eq:SysUcheck} and \eqref{Eq:PsirefKappa} that $\widecheck{u}_\kappa$ can be decomposed as follows:
\begin{equation} \label{Eq:DecompCheckUKappa}
\displaystyle \widecheck{u}_\kappa = \sum_{\nu \neq \kappa} p_{\nu} \nabla \varphi_{\nu}^{\not \kappa} +  \sum_{\nu \neq \kappa} \gamma_{\nu} \nabla^\perp \psi_{\nu}^{\not \kappa} + K^{\not \kappa}[\omega] + \gamma_{\kappa} \nabla^\perp \psi^{r,\not \kappa}_{\kappa}
\ \text{ in } \ \widecheck{{\mathcal F}}_{\kappa}({\bf q}).
\end{equation}
We observe that the statements of Section~\ref{Sec:Expansions} that were written in a general fluid domain ${\mathcal F}$ are valid in particular in the domain $\widecheck{{\mathcal F}}_{\kappa}({\bf q})$. This has the following consequences:
\begin{itemize}
\item[--] The estimates of Propositions~\ref{Pro:ExpKirchhoff} and \ref{Pro:ExpShapeDerivatives} are valid for the Kirchhoff potentials $\varphi_{\nu}^{\not \kappa}$,
\item[--] Decomposing the circulation stream functions $\psi_{\nu}^{\not \kappa}$, for $\nu \in \{1,\dots,N\} \setminus \{\kappa \}$, as in \eqref{Eq:DefPsiKappar} by introducing the potential $\psi_{\nu}^{\not \kappa,r}$ so that
\begin{equation} \label{Eq:DecompPsiNonKappa}
\psi_{\nu}^{\not \kappa}= \widehat{\psi}_{\nu} + \psi_{\nu}^{\not \kappa,r}
\ \text{ in } \ \widecheck{{\mathcal F}}_{\kappa}({\bf q}),
\end{equation}
the function $\psi_{\nu}^{\not \kappa,r}$ satisfies the estimates of Lemmas \ref{Lem:WBorne} and \ref{Lem:DWBorne},
\item[--] The estimates of Lemmas~\ref{Lem:BiotSavart} and \ref{Lem:DBSBorne} are valid for the Biot-Savart operator $K^{\not \kappa}$ in $\widecheck{{\mathcal F}}_{\kappa}({\bf q})$.
\end{itemize}
Finally we recall that the particular term $\nabla^\perp \psi^{r,\not \kappa}_{\kappa}$ was studied in Lemma~\ref{Lem:Phantom}. \par
Now we proceed as in Lemma~\ref{Lem:UextGlob}. Concerning the bound on $\| \widecheck{u}_\kappa \|_{L^\infty({\mathcal V}_{\delta}(\partial {\mathcal S}_{\kappa}))}$, we treat the various terms in the right-hand side of \eqref{Eq:DecompCheckUKappa} as follows:
\begin{itemize}
\item the terms $p_{\nu} \nabla \varphi_{\nu}^{\not \kappa}$ are of order $\varepsilon_{\nu}^2 \widehat{p}_{\nu}$ in ${\mathcal V}_{\delta}(\partial {\mathcal S}_{\kappa})$ by Proposition~\ref{Pro:ExpKirchhoff},
\item the terms $\nabla^\perp \psi_{\nu}^{\not \kappa}$ are bounded thanks to Lemma~\ref{Lem:WBorne} and the fact that ${\mathcal V}_{\delta}(\partial {\mathcal S}_{\kappa})$ is a distance ${\mathcal O}(1)$ from ${\mathcal S}_{\nu}$,
\item the term $K^{\not \kappa}[\omega]$ is bounded thanks to Lemma~\ref{Lem:BiotSavart},
\item the term $\nabla^\perp \psi^{r,\not \kappa}_{\kappa}$ is bounded thanks to Lemma \ref{Lem:Phantom}.
\end{itemize}
Concerning the  bound on the shape derivative $\partial_{q_{\mu,m}} \widecheck{u}_\kappa$, we proceed as follows, for $\mu \neq \kappa$: 
\begin{itemize}
\item the terms $p_{\nu} \nabla \partial_{q_{\mu,m}} \varphi_{\nu}^{\not \kappa}$ are estimated in ${\mathcal V}_{\delta}(\partial {\mathcal S}_{\kappa})$ by \eqref{Eq:EstDerivKirchhoff2} in Proposition~\ref{Pro:ExpShapeDerivatives},
\item for the terms $\nabla^\perp \partial_{q_{\mu,m}} \psi_{\nu}^{\not \kappa}$, $\nu \neq \kappa$, we use the decomposition \eqref{Eq:DecompPsiNonKappa}. For $\partial_{q_{\mu,m}} \nabla^\perp \widehat{\psi}_{\nu}$ (which vanishes unless $\mu=\nu$), we use \eqref{Eq:pqpsi}, \eqref{Eq:BehaviourPsi1}, \eqref{Eq:BehaviourPsi2} and the fact that ${\mathcal V}_{\delta}(\partial {\mathcal S}_{\kappa})$ is a distance ${\mathcal O}(1)$ from ${\mathcal S}_{\nu}$. For $\partial_{q_{\mu,m}} \nabla^\perp \psi_{\nu}^{\not \kappa,r}$ we use Lemma~\ref{Lem:DWBorne} (that is valid in $\widecheck{\mathcal{F}}_\kappa$) and again the fact that ${\mathcal V}_{\delta}(\partial {\mathcal S}_{\kappa})$ is a distance ${\mathcal O}(1)$ from $\partial \widecheck{{\mathcal F}}_{\kappa}$, 
\item the term $\partial_{q_{\mu,m}} K^{\not \kappa}[\omega]$ is estimated thanks to Lemma~\ref{Lem:DBSBorne}, using again the fact that ${\mathcal V}_{\delta}(\partial {\mathcal S}_{\kappa})$ is a distance ${\mathcal O}(1)$ from $\partial \widecheck{{\mathcal F}}_{\kappa}$, 
\item the term $\partial_{q_{\mu,m}} \nabla^\perp \psi^{r,\not \kappa}_{\kappa}$ is bounded by $C \varepsilon_{\mu}^{\delta_{m3}}$ in ${\mathcal V}_{\delta}(\partial {\mathcal S}_{\kappa})$ thanks to Lemma \ref{Lem:Phantom}.
\end{itemize}
Finally, when $\mu=\kappa$, only the last term in \eqref{Eq:DecompCheckUKappa} actually depends on $q_{\kappa}$. 
This dependence ---despite the fact that $\widecheck{u}_\kappa$ is defined in $\widecheck{{\mathcal F}}_{\kappa}$ is due to the boundary conditions in \eqref{Eq:PsirefKappa}.
The derivative of this term with respect to $q_{\kappa,m}$ is again estimated by $C \varepsilon_{\mu}^{\delta_{m3}}$ in ${\mathcal V}_{\delta}(\partial {\mathcal S}_{\kappa})$ thanks to Lemma \ref{Lem:Phantom}.

This concludes the proof of Lemma~\ref{Lem:CheckUKappa}.
\end{proof}
%
%
We remark that outside of the support of $\omega$, $\nabla \widecheck{u}_\kappa$ is a traceless $2 \times 2$ symmetric matrix; hence it is of the form
\begin{equation*}
\begin{pmatrix}
-a & b \\ b & a
\end{pmatrix}.
\end{equation*}
When $(\boldsymbol\varepsilon,{\bf q},\omega) \in \mathfrak{Q}_\delta$, $h_\kappa$ is outside of the support of $\omega$ for each $\kappa \in {\mathcal P}_{s}$; consequently we can set $(V_{\kappa,j})_{j=1,2,4,5}$ as follows
\begin{equation} \label{Eq:DefV}
\begin{pmatrix} V_{\kappa,1} \\ V_{\kappa,2} \end{pmatrix} := \widecheck{u}_{\kappa}(h_{\kappa})
\ \text{ and } \ 
\begin{pmatrix} -V_{\kappa,4} & V_{\kappa,5} \\ V_{\kappa,5} & V_{\kappa,4} \end{pmatrix}  := \nabla_{x} \widecheck{u}_{\kappa}(h_{\kappa}),
\end{equation}
where to lighten the notation we omitted the dependence on ${\bf q}, {\bf p}$, and $\omega$. Correspondingly we set
\begin{gather} \label{Eq:DecompV}
V_{\kappa} := \sum_{i \in \{1,2,4,5\}} V_{\kappa,i} \, \xi_{\kappa,i} = V_{\kappa} = \widecheck{u}_\kappa({\bf q},h_{\kappa}) + (x-h_{\kappa}) \cdot \nabla_{x} \widecheck{u}_\kappa({\bf q},h_{\kappa}) .
\end{gather}
%
%
%
%
%
We are now in position to state our approximation result.
\begin{Proposition} \label{Pro:ExistsApprox}
Let $\delta>0$. There exists $\varepsilon_{0}>0$ such that for each $\kappa \in {\mathcal P}_{s}$ and for $\overline{\boldsymbol\varepsilon} < \varepsilon_{0}$, the following holds. Consider the vector field $u^{ext}_\kappa$ introduced in the decomposition \eqref{Eq:DecompKappa} of the solution $u^\varepsilon$ of System \eqref{Eq:Euler}--\eqref{Eq:Newton} and $V_{\kappa}$ defined in \eqref{Eq:DecompV}. Then $V_{\kappa}$ belongs to $C^1([0,T]; {\mathcal K}_{\kappa,s})$ 
and there exists a family (parameterized by $\boldsymbol\varepsilon$) of functions  $u^{r}_{\kappa}$ in $C^1([0,T]; C^\infty(\overline{{\mathcal F}({\bf q}^{\varepsilon})}))$
such that, as long as $(\varepsilon,{\bf q},\omega) \in {\mathfrak{Q}}_{\delta}$,
\begin{equation} \label{Eq:AssumptionApproximation}
u^{ext}_\kappa = (\mbox{Id} - \mbox{Kir}_{\kappa}) V_{\kappa} + \varepsilon_{\kappa}^2 u^{r}_{\kappa} \ \text{ in } {\mathcal F},
\index{W@Modulations!M1@$V_{\kappa}$: linear approximation of the $\kappa$-th exterior field~$u^{ext}_\kappa$, generating modulations}
\end{equation}
and the following estimates are satisfied for some $C>0$ independent of $\boldsymbol\varepsilon$:
\begin{gather} \label{Eq:EstApprox1}
\| V_{\kappa}\|_{ C^0([0,T]) } +  \|  u^{r}_{\kappa} \|_{ C^0([0,T] ; C^0( \partial{\mathcal S}_{\kappa})) } \leq C,  \\
\label{Eq:EstApprox2B}
\| V_{\kappa}' \|_{ C^0([0,T]) }  + \varepsilon_{\kappa} \| \partial_{t} u^{r}_{\kappa} \|_{C^0([0,T]; C^0( \partial{\mathcal S}_{\kappa}))}
\leq C \left(1 + |\widehat{\mathbf{p}}^{\varepsilon}(t)| \right).
\end{gather}
\end{Proposition}
\begin{proof}[Proof of Proposition~\ref{Pro:ExistsApprox}]
We proceed in four steps. \par
\ \par
\noindent
{\bf Step 1.} 
We start with the estimates on $V_{\kappa}$. We denote
\begin{equation*}
\widecheck{\mathfrak{u}}_{\kappa}(t,x):= \widecheck{u}_{\kappa}({\bf q}^\varepsilon(t),{\bf p}^\varepsilon(t),\omega(t),x)  .
\end{equation*}
An estimate of $\widecheck{\mathfrak{u}}_{\kappa}$ in $L^\infty({\mathcal V}_{\delta}(\partial {\mathcal S}_{\kappa}))$
is obtained directly from Lemma~\ref{Lem:CheckUKappa} and energy estimates.
By the support of $\omega$, this yields the higher-order estimate
\begin{equation} \label{Eq:estchecku}
\| \widecheck{\mathfrak{u}}_{\kappa} \|_{C^{k,\frac{1}{2}}({\mathcal V}_{3 \delta/4 }(\partial {\mathcal S}_{\kappa}))} \leq C.
\end{equation}
Concerning the time-derivative of $\widecheck{\mathfrak{u}}_{\kappa}$, from \eqref{Eq:DecompCheckUKappa}, we have
\begin{equation*}
\partial_{t} \widecheck{\mathfrak{u}}_{\kappa}(t,x) = 
\sum_{\nu \neq \kappa} p'_{\nu} \nabla \varphi_{\nu}^{\not \kappa} + K^{\not \kappa}[\partial_{t}\omega] + 
\sum_{\substack{{\mu \in \{1,\dots,N\}} \\ {m \in\{1,2,3\}}}} p_{\mu,m} \frac{\partial \widecheck{u}_{\kappa}}{\partial q_{\mu,m}}.
\end{equation*}
To estimate the first term, we use the acceleration estimates \eqref{Eq:AccelerationEst}: since the contribution of $p'_{\nu}$ is through $p'_{\nu} \nabla {\varphi}^{\not \kappa}_{\nu}$,
due to \eqref{Eq:ExpKirchhoff1Bis}, it is of order ${\mathcal O}(\varepsilon_{\nu}^2 \widehat{p}\, '_{\nu})$ in ${\mathcal V}_{\delta}(\partial {\mathcal S}_{\kappa})$ and consequently bounded.
The term $K^{\not \kappa}[\partial_{t} \omega]$ is shown to be bounded in $L^p({\mathcal F})$ exactly as in \eqref{Eq:RelK} and \eqref{Eq:Num42}.
Due to the support of $\omega$, it is hence bounded in $L^\infty({\mathcal V}_{3\delta/4}(\partial {\mathcal S}_{\kappa}))$.
Finally, the last term is of order ${\mathcal O}(\widehat{p}_{\nu})$ thanks to Lemma~\ref{Lem:CheckUKappa} and energy estimates \eqref{Eq:APEEnergy}.
This proves that $\partial_{t}\widecheck{\mathfrak{u}}_{\kappa}$ is bounded in ${\mathcal V}_{3\delta/4}(\partial {\mathcal S}_{\kappa})$, so that by interior elliptic regularity:
\begin{equation} \label{Eq:ptchecku}
\| \partial_{t} \widecheck{\mathfrak{u}}_{\kappa} \|_{C^{k,\frac{1}{2}}({\mathcal V}_{\delta/2}(\partial {\mathcal S}_{\kappa}))} \leq C (1+ |\widehat{{\bf p}}|).
\end{equation}
The bounds on $V_{\kappa}$ in \eqref{Eq:EstApprox1}-\eqref{Eq:EstApprox2B} 
 follow, using \eqref{Eq:DefV}, \eqref{Eq:estchecku} and \eqref{Eq:ptchecku}. It remains to prove the bounds \eqref{Eq:EstApprox1}-\eqref{Eq:EstApprox2B} on $  u^{r}_{\kappa}  $. \par
\ \par
\noindent
{\bf Step 2.} Let us now relate the function  $u^{r}_{\kappa}$ defined by \eqref{Eq:AssumptionApproximation} to $\widecheck{u}_\kappa$.
First, we use \eqref{Eq:SysUExtKappa}, \eqref{Eq:SysUcheck} and the support of $\omega$ 
to infer that $u^{ext}_\kappa  - \widecheck{u}_\kappa$ satisfies
\begin{equation*}
\left\{ \begin{array}{l}
\div (u^{ext}_\kappa  - \widecheck{u}_\kappa) = 0 \ \text{ in } \ {\mathcal F}({\bf q}), \medskip \\
\curl (u^{ext}_\kappa  - \widecheck{u}_\kappa) = 0 \ \text{ in } \ {\mathcal F}({\bf q}), \medskip \\
(u^{ext}_\kappa  - \widecheck{u}_\kappa)\cdot n = 0 \ \text{ on } \ \partial {\mathcal F}({\bf q}) \setminus \partial {\mathcal S}_{\kappa} , \medskip \\
(u^{ext}_\kappa  - \widecheck{u}_\kappa)\cdot n = - \widecheck{u}_\kappa \cdot n \ \text{ on } \ \partial {\mathcal S}_{\kappa}, \medskip \\
\displaystyle \oint_{\partial {\mathcal S}_{\nu}} (u^{ext}_\kappa  - \widecheck{u}_\kappa) \cdot \tau \, ds = 0 \ \text{ for } \ \nu =1, \dots, N.
\end{array} \right.
\end{equation*}
Recalling the notation \eqref{Eq:PhiBeta}, this gives that
\begin{equation} \label{Eq:CompUextKUCheckK}
u^{ext}_\kappa  - \widecheck{u}_\kappa = - \nabla {\mathfrak{f}}^{\mathcal N}_{\kappa}[\widecheck{u}_{\kappa|{\partial \mathcal S}_{\kappa}} \cdot n],
\end{equation}
Then we use a Taylor expansion of $\widecheck{u}_\kappa$ in the neighborhood of ${\mathcal S}_{\kappa}$.
Using local elliptic regularity estimates on $\widecheck{u}_\kappa$ (which is harmonic in the $\delta$-neighborhood of ${\mathcal S}_{\kappa}$),
we may  estimate the second derivatives of $\widecheck{u}_\kappa$ in $L^\infty$ in some neighborhood of ${\mathcal S}_{\kappa}$ by its $L^\infty$ norm in a larger neighborhood
and hence by $C \| \widecheck{u}_\kappa \|_{\infty}$. It follows (recalling \eqref{Eq:DecompV}) that we may write in the $\delta/2$-neighborhood of ${\mathcal S}_{\kappa}$
\begin{equation} \label{Eq:Taylor}
\widecheck{u}_\kappa({\bf q},x)
 = V_{\kappa} + R_{\kappa}({\bf q},x)
 \ \text{ with } \
|R_{\kappa}({\bf q},x)| \leq  C \| \widecheck{u}_\kappa \|_{\infty} |x-h_{\kappa}|^2,
\end{equation}
where we omit temporarily the dependence of $\widecheck{u}_{\kappa}$ on ${\bf p}$ and $\omega$ to lighten the notations.

Recalling 
\eqref{Eq:PhiBeta} and \eqref{alaph}
 we observe that 
\begin{equation}
 \label{usopen}
\mbox{Kir}_{\kappa} V_{\kappa} = \nabla {\mathfrak{f}}^{\mathcal N}_{\kappa}[ V_{\kappa} \cdot n].
\end{equation}
Hence by \eqref{Eq:CompUextKUCheckK}, \eqref{Eq:Taylor} and  \eqref{usopen} we arrive at \eqref{Eq:AssumptionApproximation} with
\begin{equation}
 \label{paperp}
u^{r}_{\kappa}(t,x) := \varepsilon_{\kappa}^{-2} \left\{
R_{\kappa}({\bf q}(t),x) -\nabla {\mathfrak{f}}^{\mathcal N}_{\kappa} \big[ R_{\kappa}({\bf q}(t),x)_{|\partial {\mathcal S}_{\kappa}} \cdot n \big] \right\}.
\end{equation}
\ \\
{\bf Step 3.}
We now turn to the bound of $u^r_{\kappa}$ in  \eqref{Eq:EstApprox1}. We first notice that, due to \eqref{Eq:Taylor},  in the $\varepsilon_{\kappa}$-neighborhood of ${\mathcal S}_{\kappa}$,
\begin{equation} \label{Eq:Rkinfini}
\| R_{\kappa}({\bf q},\cdot) \|_{L^\infty({\mathcal V}_{\varepsilon_{\kappa}}(\partial {\mathcal S}_{\kappa}))} \leq C \varepsilon_{\kappa}^2.
\end{equation}
Since $R_{\kappa}({\bf q},x)$ is harmonic in $x$ in a neighborhood of ${\mathcal S}_{\kappa}$, using a scaling argument and local elliptic estimates, we also see that
\begin{equation} \label{Eq:RkHolder}
\varepsilon_{\kappa}^{k+\frac{1}{2}} \big| R_{\kappa}({\bf q},\cdot) \big|_{C^{k,\frac{1}{2}}(\partial {\mathcal S}_{\kappa})} \leq C \varepsilon_{\kappa}^2.
\end{equation}
Then we apply Lemma~\ref{Lem:Neumann} and Propositions~\ref{Pro:DirichletSAM} and \ref{Pro:DirichletPetits}, taking into account that the normal $n$ satisfies 
$\varepsilon_{\kappa}^{k+\frac{1}{2}} | n |_{C^{k,\frac{1}{2}}(\partial {\mathcal S}_{\kappa})} \leq C$. We obtain
\begin{equation} \label{Eq:EstFnkr}
\left\| \nabla {\mathfrak{f}}^{\mathcal N}_{\kappa} \big[ R_{\kappa}({\bf q},\cdot)_{|\partial {\mathcal S}_{\kappa}} \cdot n \big] \right\|_{L^\infty(\partial {\mathcal S}_{\kappa})} 
+ \varepsilon_{\kappa}^{k+\frac{1}{2}} \left| \nabla {\mathfrak{f}}^{\mathcal N}_{\kappa} \big[ R_{\kappa}({\bf q},\cdot)_{|\partial {\mathcal S}_{\kappa}} \cdot n \big] \right|_{C^{k,\frac{1}{2}}(\partial {\mathcal S}_{\kappa})} 
\leq C \varepsilon_{\kappa}^2 .
\end{equation}
In particular we see that $u^{r}_{\kappa}$ is bounded on $\partial {\mathcal S}_{\kappa}$ and satisfies \eqref{Eq:EstApprox1}. \par

\ \par
\noindent
{\bf Step 4.}
We finally estimate $\partial_{t} u_{\kappa}^r$. To that purpose we introduce the stream function $\widecheck{\eta}_{\kappa}$ of $\widecheck{u}_{\kappa}$, so that 
 $\widecheck{u}_{\kappa} = \nabla^\perp \widecheck{\eta}_{\kappa}$ and  we define
\begin{equation*} \label{Eq:DefAlphaR}
{\alpha}_{\kappa}^R= {\alpha}_{\kappa}^R({\bf q},x) := \widecheck{\eta}_{\kappa} - \sum_{i \in \{1,2,4,5\} } V_{\kappa,i} {\mathcal J}_{\kappa,i},
\end{equation*}
with ${\mathcal J}_{\kappa,i}$ defined in \eqref{Eq:DefJ}.
By \eqref{Eq:DefJ2}, Lemma~\ref{Lem:Neumann}, \eqref{Eq:DecompV} and \eqref{Eq:Taylor} we have
\begin{equation*} \label{Eq:RelR}
\nabla {\mathfrak{f}}^{\mathcal N}_{\kappa} \big[ R_{\kappa}({\bf q},\cdot)_{|\partial {\mathcal S}_{\kappa}} \cdot n \big] 
= \nabla^\perp {\mathfrak{f}}_{\kappa} \big[ {\alpha}_{\kappa}^R \big].
\end{equation*}
Hence \eqref{paperp} translates into: 
\begin{equation*}
u^{r}_{\kappa} = \varepsilon_{\kappa}^{-2} \left\{
R_{\kappa}({\bf q}(t), \cdot)
 - \nabla^\perp {\mathfrak{f}}_{\kappa} \big[ {\alpha}_{\kappa}^R \big]
 \right\}.
\end{equation*}
Thus 
\begin{equation} \label{Eq:ptur}
\partial_{t} u^{r}_{\kappa}(t,x) := \varepsilon_{\kappa}^{-2} \Bigg\{
\partial_{t} \mathfrak{R}_{\kappa}(t,x) - \nabla^\perp {\mathfrak{f}}_{\kappa} \big[ \partial_{t} \mathfrak{a}_{\kappa}(t,x)] - \sum_{\substack{{\mu \in \{1,\dots,N\}} \\ {m \in\{1,2,3\}}}} p_{\mu,m} \frac{\partial \nabla^\perp {\mathfrak{f}}_{\kappa} \big[ \mathfrak{a}_{\kappa}(t,\cdot)_{|\partial {\mathcal S}_{\kappa}} \big] }{\partial q_{\mu,m}} \Bigg\} ,
\end{equation}
where 
\begin{equation*}
\mathfrak{R}_{\kappa}(t,x):= R_{\kappa}({\bf q}^\varepsilon(t),{\bf p}^\varepsilon(t),\omega(t),x)
\text{ and }
\mathfrak{a}_{\kappa}(t,x):= \alpha^R_{\kappa}({\bf q}^\varepsilon(t),{\bf p}^\varepsilon(t),\omega(t),x) .
\end{equation*}
Relying on \eqref{Eq:DecompV} and \eqref{Eq:Taylor}, a computation gives
\begin{equation*}
\partial_t \mathfrak{R}_{\kappa} (t,x)=\frac{\partial \widecheck{\mathfrak{u}}_{\kappa}}{\partial t}(t,x) - \frac{\partial \widecheck{\mathfrak{u}}_{\kappa}}{\partial t} (t,h_{\kappa})
- (x-h_{\kappa}) \cdot \nabla \frac{\partial \widecheck{\mathfrak{u}}_{\kappa}}{\partial t}(t,h_{\kappa}) -  \nabla^2_{x} \widecheck{\mathfrak{u}}_{\kappa}(t,h_{\kappa}) \cdot h'_{\kappa} \otimes (x-h_{\kappa}).
\end{equation*}
With \eqref{Eq:estchecku} and \eqref{Eq:ptchecku}  we deduce
\begin{equation} \label{Eq:EstEK}
\| \partial_{t} \mathfrak{R}_{\kappa} \|_{L^\infty({\mathcal V}_{\varepsilon_{\kappa}}(\partial {\mathcal S}_{\kappa}))} \leq C \varepsilon_{\kappa} (1+ |\widehat{{\bf p}}|) .
\end{equation}
Since $R_{\kappa}({\bf q}, \cdot) = \nabla^\perp {\alpha}_{\kappa}^R$, it follows, using again interior elliptic regularity, that we may estimate the second term in \eqref{Eq:ptur} as follows:
\begin{equation} \label{Eq:Esta}
\| \partial_{t} \mathfrak{a}_{\kappa}(t,\cdot) - \partial_{t} \mathfrak{a}_{\kappa}(t,h_{\kappa}) \|_{L^\infty ({\mathcal S}_{\kappa})}
+ \varepsilon^{k+\frac{1}{2}}  | \partial_{t} \mathfrak{a}_{\kappa}(t,\cdot) |_{C^{k,\frac{1}{2}}({\mathcal S}_{\kappa})}
 \leq C \varepsilon^2_{\kappa} (1+ |\widehat{{\bf p}}|).
\end{equation}
With Propositions~\ref{Pro:DirichletPetits} and \ref{Pro:DirichletSAM}, this gives
\begin{equation*}
\left\| \nabla {\mathfrak{f}}_{\kappa} \big[  \partial_{t} \mathfrak{a}_{\kappa}(t,\cdot) \big] \right\|_{L^\infty(\partial {\mathcal S}_{\kappa})} 
\leq C \varepsilon_{\kappa}  (1+ |\widehat{{\bf p}}|).
\end{equation*}
Concerning the third term in \eqref{Eq:ptur}, we use Corollary~\ref{Cor:SD}, where here the function $\mathfrak{a}_{\kappa}(t,\cdot)$ is fixed. We find
\begin{equation*}
\displaystyle \frac{\partial {\mathfrak{f}}_{\kappa} \big[ \mathfrak{a}_{\kappa}(t,\cdot) \big]}{\partial q_{\mu,m}} 
= \left( \nabla {\mathfrak{a}}_{\kappa}^R - \nabla {\mathfrak{f}}_{\kappa} \big[ {\mathfrak{a}}_{\kappa}^R \big]\right) \cdot n \,K_{\mu,m} + c'_{\lambda} \text{ on } \partial {\mathcal S}_{\lambda} \text{ and }
\frac{\partial {\mathfrak{f}}_{\kappa} \big[ \mathfrak{a}_{\kappa}(t,\cdot) \big]}{\partial q_{\mu,m}} = 0 \text{ on } \partial \Omega.
\end{equation*}
With \eqref{Eq:Rkinfini}-\eqref{Eq:RkHolder}-\eqref{Eq:EstFnkr} and Propositions~\ref{Pro:DirichletSAM} and \ref{Pro:DirichletPetits}, we conclude that
\begin{equation} \nonumber
\left\| \nabla \frac{\partial {\mathfrak{f}}_{\kappa} \big[ \mathfrak{a}_{\kappa}(t,\cdot) \big]}{\partial q_{\mu,m}}  \right\|_{L^\infty(\partial {\mathcal S}_{\kappa})} 
\leq C \varepsilon_{\kappa} \varepsilon_{\mu}^{\delta_{m3}}.
\end{equation}
Injecting in \eqref{Eq:ptur} we find the last estimate of \eqref{Eq:EstApprox2B},  which concludes the proof of Proposition~\ref{Pro:ExistsApprox}. 
\end{proof}
%
%
%
%
%
\subsection{Definition of the modulations}
We conclude this section by introducing the first-order modulations $\alpha_{\kappa,i}$ and the second-order modulations $\beta_{\kappa,i}$, 
for $\kappa \in {\mathcal P}_{s}$ and $i=1,2$. We set
\begin{equation} \label{Eq:ParamModulation}
\text{ for } \ \kappa \in {\mathcal P}_{s}, \ \ 
\alpha_{\kappa,i} := V_{\kappa,i} \ \text{ for } i=1,2, 
\ \text{ and } \ 
\begin{pmatrix} \beta_{\kappa,1} \\ \beta_{\kappa,2}  \end{pmatrix} :=  \begin{pmatrix} - V_{\kappa,4} & V_{\kappa,5} \\ V_{\kappa,5} & V_{\kappa,4}  \end{pmatrix}
\zeta^\varepsilon_{\kappa}(q_{\kappa}).
\index{W@Modulations!M2@$\alpha_{\kappa,i}$: first-order modulations}
\index{W@Modulations!M3@$\beta_{\kappa,i}$: second-order modulations}
\end{equation}
We recall that $\zeta^\varepsilon_{\kappa} (q_{\kappa})$ is defined in \eqref{Eq:Zeta}.
We notice in passing that due to Proposition~\ref{Pro:ExistsApprox} and the scale relation in \eqref{Eq:Zeta}, the modulations can be estimated as follows:
\begin{equation} \label{Eq:BoundAlphaBeta}
| \alpha_{\kappa,i} | \leq C \ \text{ and } \ |\beta_{\kappa,i}| \leq C \varepsilon_{\kappa}.
\end{equation}
The first-order modulations will play a central role in the normal forms of Section~\ref{Sec:NormalForm} and hence in the modulated energy estimates of Section~\ref{Sec:MEE}, but also in the passage to the limit in Section~\ref{Sec:PTTL}.
The second-order modulations $\beta_{\kappa,1}$ and $\beta_{\kappa,2}$ disappear in the limit, but play an important role in the normal forms, in Subsection~\ref{Subsec:MGT} (see Lemma~\ref{Lem:Transfert45vers12}).
%
%
%
%
%
%
%
%
%
%
%
\section{Normal forms}
\label{Sec:NormalForm}
%
%
%
%
In this section, we present normal forms for the dynamics of small solids. It will be useful for both the modulated energy estimates (for solids of family $(iii)$) and the passage to the limit  (for solids of family $(ii)$ and $(iii)$).
\subsection{Statement of the normal form}
\begin{Proposition}
\label{Pro:MEE}
Let $\delta>0$. There exists $\varepsilon_{0}>0$ such that for $\overline{\boldsymbol\varepsilon} \leq \varepsilon_{0}$, the following holds. 
Consider the corresponding solutions $(u^\varepsilon,h^\varepsilon,\vartheta^\varepsilon)$ of the system, for each $\kappa \in {\mathcal P}_{s}$ the exterior field $u^{ext}_\kappa$ defined by \eqref{Eq:DecompKappa}, and $V_{\kappa}$ defined by \eqref{Eq:DefV} together with its coordinates $(V_{\kappa,i})_{i \in \{1,2,4,5\}}$ in the decomposition \eqref{Eq:DecompV}.
Introduce the modulated variable $\overline{\bf p}=(\overline{p}_{1}, \dots, \overline{p}_{N})$ as follows: for $i \in \{1,2,3\}$
\begin{equation} \label{Eq:VariableModulee}
\overline{p}_{\kappa,i} = {p}_{\kappa,i}   \ \text{ for } \ \kappa \in {\mathcal P}_{(i)}, \ \ 
\overline{p}_{\kappa,i} = {p}_{\kappa,i} - \delta_{i \in \{1,2\}} (\alpha_{\kappa,i} + \beta_{\kappa, i} ) \ \text{ for } \ \kappa \in {\mathcal P}_{s},
\index{Solid velocity!E5@$\overline{p}_{\kappa}$: modulated velocity of the $\kappa$-th solid}
\index{Solid velocity!E6@$\overline{\bf p}$: modulated velocity of all the solids at once}
\end{equation}
with $\alpha_{\kappa,i}$ and $\beta_{\kappa,i}$ given by  \eqref{Eq:ParamModulation}, and the time-dependent vector field $B_{\kappa} = (B_{\kappa,j} )_{j=1,2,3} $ given by
\begin{equation} \label{Eq:Gyro}
B_{\kappa,j}  := -\gamma_{\kappa} \sum_{k=1}^3 \overline{p}_{\kappa,k}  \int_{\partial \mathcal{S}_{\kappa}} \partial_{n} \widehat{\psi}_{\kappa} \, \xi_{\kappa,k}^{\perp} \cdot \xi_{\kappa,j} \, ds .
\end{equation}
Then as long as $(\boldsymbol\varepsilon,{\bf q}, \omega) \in \mathfrak{Q}_{\delta}^{\varepsilon_{0}}$, for each $\kappa \in {\mathcal P}_{s}$, one has
\begin{equation} \label{Eq:FormeNormale}
{\mathcal M}_{g,\kappa} {p}_{\kappa}' + 
{\mathcal M}_{a,\kappa} \overline{p}_{\kappa}' + \frac{1}{2} {\mathcal M}'_{a,\kappa} \overline{p}_{\kappa}  = A_{\kappa}(t) + B_{\kappa}(t) + 
C_{\kappa}(t) + D_{\kappa}(t),
\end{equation}
where the term $A_{\kappa}$ is weakly nonlinear in the sense that for some $K>0$ independent of ${\boldsymbol\varepsilon}$, for $j \in \{1,2,3\}$, 
\begin{equation} \label{Eq:WNL}
|A_{\kappa,j}(t)| \leq K \varepsilon_{\kappa}^{2+\delta_{j3}} \left( 1 + |\widehat{\bf p}(t)|  \right),
\end{equation}
the term $C_{\kappa}$ is gyroscopic of lower order in the sense that for all times, 
\begin{equation} \label{Eq:GyroF}
C_{\kappa}(t) \cdot \overline{p}_{\kappa}(t)= 0 , 
\end{equation}
and moreover for some $K>0$ independent of ${\boldsymbol\varepsilon}$, one has  for $j \in \{1,2,3\}$ 
\begin{equation} \label{Eq:EstGyro}
|C_{\kappa,j}(t)| \leq K \varepsilon_{\kappa}^{1+\delta_{j3}} \left( 1 + |\widehat{\boldsymbol p}(t)|^2 \right),
\end{equation}
and the term $D_{\kappa}$ is weakly gyroscopic in the sense that it satisfies for some $K>0$ independent of ${\boldsymbol\varepsilon}$,
\begin{equation} \label{Eq:WG}
\left|\int_{0}^t D_{\kappa}(\tau) \cdot \overline{p}_{\kappa}(\tau) \, d \tau \right| \leq 
K \varepsilon_{\kappa}^2 \left( 1 + t + \int_{0}^t |\widehat{p}_{\kappa}(\tau)|^2 \, d \tau \right),
\end{equation}
and moreover for some $K>0$ independent of ${\boldsymbol\varepsilon}$, one has  for $j \in \{1,2,3\}$ 
\begin{equation} \label{Eq:EstGyroW}
|D_{\kappa,j}(t)| \leq K \varepsilon_{\kappa}^{1+\delta_{j3}} . 
\end{equation}
\end{Proposition}
We recall  the notation \eqref{Eq:TrueInertia} and \eqref{Eq:AddMassMatrix}   for the matrices ${\mathcal M}_{g,\kappa} {p}_{\kappa}$ and
${\mathcal M}_{a,\kappa}$.

Let us highlight that $B_{\kappa}$ satisfies  
\begin{equation} \label{Eq:GyroFB}
B_{\kappa} \cdot \overline{p}_{\kappa}= 0 .
\end{equation}
 We will refer to this term as the main gyroscopic term. 
\begin{Remark}
Note the distinction between the modulated variable $\overline{p}_{\kappa}$ (for which $\overline{p}_{\kappa,3}=\vartheta_{\kappa}'$) on the left-hand side and the scaled variable $\widehat{p}_{\kappa}$ (with $\widehat{p}_{\kappa,3}=\varepsilon_{\kappa} \vartheta_{\kappa}'$) on the right-hand side.
\end{Remark}
The rest of the section is devoted to the proof of Proposition \ref{Pro:MEE}. 
%
%
%
%
\subsection{Starting point of the proof: rewriting the solid equation with various terms}
Given $\delta>0$, we first let $\varepsilon_{0}>0$ small enough so that all the statements of Sections~\ref{Sec:Expansions} to \ref{Sec:Modulations} apply. 
To prove the normal form \eqref{Eq:FormeNormale}, we will use a variant of the decomposition \eqref{Eq:DecompKappa}, which is better adapted to modulated variables.
\begin{Definition}
For each $\kappa \in {\mathcal P}_{s}$, we introduce the following decomposition 
\begin{equation} \label{Eq:DecompMod}
u^\varepsilon = \overline{u}_{\kappa}^{pot}  + \gamma_{\kappa} \nabla^{\perp} \widehat{\psi}_{\kappa} +  \overline{u}^{ext}_\kappa
\ \ \text{ with } \ \ 
\overline{u}_{\kappa}^{pot} := \sum_{j \in \{1,2,3\} } \overline{p}_{\kappa,j} \nabla \varphi_{\kappa,j} .
\index{Velocity fields!U6@$\overline{u}_\kappa^{pot}$, $\overline{u}^{ext}_\kappa$: decomposition of $u^\varepsilon$ focused on ${\mathcal S}_{\kappa}$ with modulated potential part}
\end{equation}
\end{Definition}
\noindent
In particular, comparing the decompositions \eqref{Eq:DecompKappa} and \eqref{Eq:DecompMod} we see that  
\begin{equation} \label{Eq:TildeUExt}
\overline{u}^{ext}_\kappa = u^{ext}_\kappa + \sum_{j=1}^2 (\alpha_{\kappa,j}+ \beta_{\kappa,j}) \nabla \varphi_{\kappa,j}.
\end{equation}
\begin{proof}[Proof of Proposition~\ref{Pro:MEE}]
We first observe that, by the first equation of \eqref{Eq:Euler} and by \eqref{Eq:UNablaU}, the fluid pressure $\pi^{\varepsilon} $ satisfies: 
\begin{equation}
\nabla \pi^{\varepsilon} = 
- \partial_{t} u^{\varepsilon} 
- \nabla \left(\frac{|u^{\varepsilon}|^2}{2}\right) 
-  \omega^{\varepsilon} u^{\varepsilon\perp}.
\end{equation}
Then by \eqref{Eq:Newton}, \eqref{Eq:TrueInertia}, \eqref{Kir} and an integration by parts we obtain that, for $\kappa \in {\mathcal P}_{s}$ and $j \in \{1,2,3\}$, 
\begin{equation} \label{Def:IJ}
({\mathcal M}_{g} p')_{\kappa,j}   = -I_{\kappa,j} -J_{\kappa,j} - L_{\kappa,j} , 
\end{equation}
 where 
\begin{multline} \label{SayL}
I_{\kappa,j}:= \int_{{\mathcal F}({\bf q})} \partial_{t} u^{\varepsilon} \cdot \nabla \varphi_{\kappa,j} \, dx , \quad 
J_{\kappa,j}:= \int_{{\mathcal F}({\bf q})} \nabla \left(\frac{|u^{\varepsilon}|^2}{2}\right) \cdot \nabla \varphi_{\kappa,j} \, dx,  \\
\text{ and } \  
L_{\kappa,j}:= \int_{{\mathcal F}({\bf q})} \omega^{\varepsilon} u^{\varepsilon\perp} \cdot \nabla \varphi_{\kappa,j} \, dx .
\end{multline}
By \eqref{Eq:DecompMod} 
\begin{gather} \label{saiI}
I_{\kappa,j}=I_{\kappa,j}^{1} + I_{\kappa,j}^{2} + I_{\kappa,j}^{3} ,
\end{gather}
 where 
 \begin{gather} 
\label{saiI1}
I_{\kappa,j}^{1}:= \gamma_{\kappa} \int_{{\mathcal F}({\bf q})} \partial_{t} \nabla^{\perp} \widehat{\psi}_{\kappa} \cdot \nabla \varphi_{\kappa,j} \, dx ,
 \\ \label{saiI2}
I_{\kappa,j}^{2}:= \int_{{\mathcal F}({\bf q})} \partial_{t} \overline{u}_{\kappa}^{pot} \cdot \nabla \varphi_{\kappa,j} \, dx 
\ \text{ and } \\
\label{DecompI}
I_{\kappa,j}^{3}:= \int_{{\mathcal F}({\bf q})} \partial_{t} \overline{u}^{ext}_\kappa \cdot \nabla \varphi_{\kappa,j} \, dx .
\end{gather}
Concerning $J_{\kappa,j}$, we integrate by parts to obtain
\begin{equation} \label{Ecr1:J}
J_{\kappa,j}=\int_{\partial {\mathcal S}_{\kappa}({\bf q})} \frac{|u^{\varepsilon}|^2}{2} K_{\kappa,j} \, ds.
\end{equation}
Given two vector fields $a$ and $b$ on $\partial {\mathcal S}_{\kappa}$ we define
\begin{gather}
\label{pazencor}
Q_{\kappa,j}(a,b) : = \int_{\partial {\mathcal S}_{\kappa}({\bf q})} a \cdot b \, K_{\kappa,j} \, dx \ \text{ and } \ 
Q_{\kappa,j}(a) : = Q_{\kappa,j}(a,a). 
\end{gather}
By \eqref{Eq:DecompMod}, we obtain, for $\kappa \in {\mathcal P}_{s}$ and $j \in \{1,2,3\}$, 
\begin{equation} 
\label{saiJ}
J_{\kappa,j}=J_{\kappa,j}^{1} + J_{\kappa,j}^{2} + J_{\kappa,j}^{3} + J_{\kappa,j}^{4} + J_{\kappa,j}^{5} +J_{\kappa,j}^{6} ,
\end{equation}
where
\begin{eqnarray}
\label{saiJ1}
J_{\kappa,j}^{1} &:=& \frac{1}{2} Q_{\kappa,j}(\gamma_{\kappa} \nabla^{\perp} \widehat{\psi}_{\kappa}), \II \\ 
\label{saiJ2}
J_{\kappa,j}^{2} &:=& \gamma_{\kappa} Q_{\kappa,j}(\nabla^{\perp} \widehat{\psi}_{\kappa}, \overline{u}_{\kappa}^{pot} + \overline{u}^{ext}_\kappa - v_{{\mathcal S},\kappa}), \II \\
\label{saiJ3}
J_{\kappa,j}^{3} &:=& \gamma_{\kappa} Q_{\kappa,j}(\nabla^{\perp} \widehat{\psi}_{\kappa}, v_{{\mathcal S},\kappa}) , \II \\ 
\label{saiJ4}
J_{\kappa,j}^{4} &:=& \frac{1}{2} Q_{\kappa,j}(\overline{u}_{\kappa}^{pot}), \II \\
\label{saiJ5}
J_{\kappa,j}^{5} &:=& \frac{1}{2} Q_{\kappa,j}(\overline{u}^{ext}_\kappa), \II \\
\label{saiJ6} 
J_{\kappa,j}^{6} &:=& Q_{\kappa,j}(\overline{u}_{\kappa}^{pot}, \overline{u}^{ext}_\kappa) , 
\end{eqnarray}
where we recall that $v_{{\mathcal S},\kappa}$ is the $\kappa$-th solid vector field, see \eqref{Eq:SolidVelocity}. 
%
%
In order to reach \eqref{Eq:FormeNormale}, the rest of the proof consists in combining \eqref{Def:IJ}, \eqref{saiI} and \eqref{saiJ}, and 
regrouping and treating the various terms above, for $\kappa \in {\mathcal P}_{s}$ and $j \in \{1,2,3\}$, in the following way:
\begin{equation} \label{Def:IJ2}
- (M_{g} p')_{\kappa,j}
= \underbrace{ L_{\kappa,j} }_{ \text{Lemma~\ref{Lem:L}} }
+  \underbrace{ J_{\kappa,j}^{1} }_{ \text{Lemma~\ref{Lem:J1}} }
+  \underbrace{ I_{\kappa,j}^{1}+J_{\kappa,j}^{3} }_{ \text{Lemma~\ref{Lem:I1J3}} }
+  \underbrace{ J_{\kappa,j}^{5} }_{ \text{Lemma~\ref{Lem:I3J5bis}} }
+  \underbrace{ I_{\kappa,j}^{3} }_{ \text{Lemma~\ref{Lem:I3J5}} }
+  \underbrace{ J_{\kappa,j}^{2} }_{ \text{Lemma~\ref{Lem:J2}} }
+  \underbrace{ I_{\kappa,j}^{2} + J_{\kappa,j}^{4} + J_{\kappa,j}^{6} }_{ \text{Lemma~\ref{Lem:I2J4J6}} }.
\end{equation}
For the rest of this section we fix $\kappa \in {\mathcal P}_{s}$ and $j \in \{1,2,3\}$. 
%
%
%
%
%
\subsection{Treatment of the simplest terms}
We start with the term $L_{\kappa,j}$ defined in \eqref{SayL}, recalling that a term is said weakly nonlinear when it satisfies the same inequality than  $A_{\kappa}$ in  \eqref{Eq:WNL}.
\begin{Lemma} \label{Lem:L}
The term $L_{\kappa,j}$ is weakly nonlinear. 
\end{Lemma}
\begin{proof}[Proof of Lemma \ref{Lem:L}]
This is an immediate consequence of \eqref{Num41} and
Proposition~\ref{Pro:ExpKirchhoff}, since, due to $(\boldsymbol\varepsilon,{\bf q},\omega) \in \mathfrak{Q}_\delta^{\varepsilon_{0}}$, the support of the vorticity is at distance more than $\delta$ from $\partial {\mathcal S}_{\kappa}$.
\end{proof}
For the term $J_{\kappa,j}^{1} $ defined in \eqref{saiJ1}, \eqref{Eq:QuadraticPsi} has established the following result. 
%
\begin{Lemma} \label{Lem:J1}
One has $J_{\kappa,j}^{1}=0$.
\end{Lemma}
Next we combine the $I_{\kappa,j}^{1}$ defined in \eqref{saiI1} and the term $J_{\kappa,j}^{3}$ defined in \eqref{saiJ3}. 
\begin{Lemma} \label{Lem:I1J3}
One has $I_{\kappa,j}^{1}+J_{\kappa,j}^{3}=0$.
\end{Lemma}
\begin{proof}[Proof of Lemma~\ref{Lem:I1J3}]
We have
\begin{eqnarray*}
I_{\kappa,j}^{1} + J_{\kappa,j}^{3} &=& \gamma_{\kappa} \int_{{\mathcal F}({\bf q})} \partial_{t} \nabla^{\perp} \widehat{\psi}_{\kappa} \cdot \nabla \varphi_{\kappa,j} \, dx 
+  \gamma_{\kappa} \int_{{\mathcal S}_{\kappa}({\bf q})} v_{{\mathcal S},\kappa} \cdot \nabla^{\perp} \widehat{\psi}_{\kappa} \, K_{\kappa,j} \, dx  \\
&=& \gamma_{\kappa} \int_{{\mathcal F}({\bf q})}
\Big[ \partial_{t} \nabla^{\perp} \widehat{\psi}_{\kappa} + \nabla \left( v_{{\mathcal S},\kappa} \cdot \nabla^{\perp} \widehat{\psi}_{\kappa} \right) \Big]
\cdot \nabla \varphi_{\kappa,j} \, dx .
\end{eqnarray*}
We conclude with \eqref{Eq:HatPsiTourne}.
\end{proof}
For the term $J_{\kappa,j}^{5} $ defined in \eqref{saiJ5}, we have the following result. 
\begin{Lemma} \label{Lem:I3J5bis}
The expression $J_{\kappa,j}^{5}$ is weakly nonlinear. 
\end{Lemma}
\begin{proof}[Proof of Lemma~\ref{Lem:I3J5bis}]
 By Proposition~\ref{Pro:ExistsApprox} and \eqref{Eq:TildeUExt}, 
\begin{equation}
\nonumber 
\overline{u}^{ext}_\kappa = (\mbox{Id} - \mbox{Kir}_{\kappa}) V_{\kappa} + \varepsilon_{\kappa}^2 u^{r}_{\kappa}
+ \sum_{k=1}^2 (\alpha_{\kappa,k} + \beta_{\kappa,k}) \nabla \varphi_{\kappa,k}   \text{ in } {\mathcal F}.
\end{equation}
Using \eqref{Eq:ParamModulation}, we obtain
\begin{equation} \label{ApproxTildeUeps}
\overline{u}^{ext}_\kappa = V_{\kappa}
+ \sum_{k=1}^2 \beta_{\kappa,k} \nabla \varphi_{\kappa,k} 
-  \sum_{k=4}^5 V_{\kappa,k} \nabla \varphi_{\kappa,k}
+ \varepsilon_{\kappa}^2 u^{r}_{\kappa}   \text{ in } {\mathcal F}.
\end{equation}
Using \eqref{Eq:BoundAlphaBeta},  \eqref{pazencor},
$\| \xi_{\kappa,k} \|_{L^{\infty}(\partial S_{\kappa})}= {\mathcal O}(\varepsilon_{\kappa}) $ for $ k=4,5$, 
 $|\partial {\mathcal S}_{\kappa}|= {\mathcal O}(\varepsilon_{\kappa})$ and  \eqref{Eq:ExpKirchhoff1Bis} 
we see that
\begin{equation*}
J_{\kappa,j}^{5} = Q_{\kappa,j}(V_{\kappa}) + {\mathcal O}(\varepsilon_{\kappa}^{2+\delta_{j3}}) . 
\end{equation*}
Now integrating by parts inside ${\mathcal S}_{\kappa}$, we obtain
\begin{equation*}
Q_{\kappa,j}(V_{\kappa}) =
\int_{{\mathcal S}_{\kappa}} \div(|V_{\kappa}|^2 \xi_{\kappa, j}) \, dx  
={\mathcal O}(\varepsilon_{\kappa}^{2+\delta_{j3}}),
\end{equation*}
which concludes  the proof of Lemma \ref{Lem:I3J5bis}.
\end{proof}
%
%
%
%
%
%
%
%
%
\subsection{Exterior acceleration term}
Here we deal with the exterior acceleration term $I_{\kappa,j}^{3}$ defined in \eqref{DecompI}.
\begin{Lemma} \label{Lem:I3J5}
The term $I_{\kappa,j}^{3}$  is weakly nonlinear.
\end{Lemma}
\begin{proof}[Proof of Lemma~\ref{Lem:I3J5}]
In this proof, by convenience, we will again take the convention of Remark~\ref{Rem:NormalisationKirchhoff} for the Kirchhoff potentials.
We start by integrating by parts and subdivide the boundary integral:
\begin{equation}
\label{LAST}
I_{\kappa,j}^{3} 
= \int_{\partial \Omega} \partial_{t} \overline{u}^{ext}_\kappa \cdot n \, \varphi_{\kappa,j} \, ds
+ \int_{\partial {\mathcal S}_{\kappa}} \partial_{t} \overline{u}^{ext}_\kappa \cdot n \, \varphi_{\kappa,j} \, ds.
+ \sum_{\nu\neq \kappa} \int_{\partial {\mathcal S}_{\nu}} \partial_{t} \overline{u}^{ext}_\kappa \cdot n \, \varphi_{\kappa,j} \, ds.
\end{equation}
\paragraph{Step 1.}
We first consider the second term in the right hand side of \eqref{LAST}. 
From \eqref{ApproxTildeUeps}, we see that
\begin{multline} \label{eq:ptuext}
\partial_{t} \overline{u}^{ext}_\kappa  = V'_{\kappa}
+ \sum_{k=1}^2 \beta'_{\kappa,k} \nabla \varphi_{\kappa,k} 
-  \sum_{k=4}^5 V'_{\kappa,k} \nabla \varphi_{\kappa,k}
+ \sum_{k=1}^2 \sum_{\substack{{\mu \in \{1,\dots,N\}} \\ {m \in \{1,2,3\}}}} \beta_{\kappa,k} p_{\mu,m} \frac{\partial \nabla \varphi_{\kappa,k}}{\partial q_{\mu,m}} \\
-  \sum_{k=4}^5 \sum_{\substack{{\mu \in \{1,\dots,N\}} \\ {m \in \{1,2,3\}}}}  V_{\kappa,k} p_{\mu,m} \frac{\partial \nabla \varphi_{\kappa,k}}{\partial q_{\mu,m}} 
+ \varepsilon_{\kappa}^2 \partial_{t} u^{r}_{\kappa}
\text{ on } \partial {\mathcal S}_{\kappa} .
\end{multline}
From Proposition~\ref{Pro:ExistsApprox} and Proposition~\ref{Pro:ExpKirchhoff}, we immediately see that the first and third terms in the right-hand side of \eqref{eq:ptuext} are 
of order ${\mathcal O}(1+|\widehat{\bf p}|)$. Moreover, from \eqref{Eq:Zeta} and \eqref{Eq:ParamModulation}, we see that
\begin{equation} \label{resser}
\begin{pmatrix} \beta_{\kappa,1} \\ \beta_{\kappa,2}  \end{pmatrix}' =  \begin{pmatrix} - V'_{\kappa,4} & V'_{\kappa,5} \\ V'_{\kappa,5} & V'_{\kappa,4}  \end{pmatrix}
\zeta^\varepsilon_{\kappa}(q_{\kappa})
+ \vartheta'_{\kappa} \begin{pmatrix} - V_{\kappa,4} & V_{\kappa,5} \\ V_{\kappa,5} & V_{\kappa,4}  \end{pmatrix}
(\zeta^\varepsilon_{\kappa}(q_{\kappa}))^\perp .
\end{equation}
Using Proposition~\ref{Pro:ExistsApprox} and \eqref{Eq:Zeta}  again, we see that this term is also of order ${\mathcal O}(1+|\widehat{\bf p}|)$. Concerning the last two terms in \eqref{eq:ptuext}, we use Proposition~\ref{Pro:ExpShapeDerivatives}, \eqref{Eq:EstApprox1} and \eqref{Eq:BoundAlphaBeta} to deduce that they are of order ${\mathcal O}(1+|\widehat{\bf p}|)$ as well. We conclude that
\begin{equation*}
\| \partial_{t} \overline{u}^{ext}_\kappa \|_{L^\infty(\partial {\mathcal S}_{\kappa})} \leq C (1+|\widehat{\bf p}|).
\end{equation*}
Using \eqref{Eq:ExpKirchhoffNormalises} and  that $|\partial {\mathcal S}_{\kappa}|={\mathcal O}(\varepsilon_{\kappa})$ we deduce that
\begin{equation}
\label{resser1}
\left|\int_{\partial {\mathcal S}_{\kappa}} \partial_{t} \overline{u}^{ext}_\kappa \cdot n \, \varphi_{\kappa,j} \, ds \right| \leq C \varepsilon_{\kappa}^{2+\delta_{j3}}(1+|\widehat{\bf p}|).
\end{equation}
\paragraph{Step 2.}
We now consider the integral  over $\partial \Omega$ that is the first term in the right hand side of \eqref{LAST}.
Recalling \eqref{Eq:SysUExtKappa} and \eqref{Eq:TildeUExt} we observe that  $\overline{u}^{ext}_\kappa\cdot n = - \gamma_{\kappa} \nabla^\perp \widehat{\psi}_{\kappa} \cdot n$ on $\partial \Omega$. 
Thus, on $\partial \Omega$,
$$
\partial_{t}  \overline{u}^{ext}_\kappa\cdot n = - \gamma_{\kappa} (\partial_{t}  \nabla^\perp \widehat{\psi}_{\kappa} ) \cdot n 
=   \gamma_{\kappa}  \nabla \left( v_{{\mathcal S},\kappa} \cdot \nabla^{\perp} \widehat{\psi}_{\kappa} \right)  \cdot n  ,
$$
thanks to \eqref{Eq:HatPsiTourne}. Therefore with \eqref{Eq:BehaviourPsi2}, we deduce
$\partial_{t}  \overline{u}^{ext}_\kappa\cdot n ={\mathcal O}(|\widehat{\bf p}|)$. 
On the other hand, by \eqref{Eq:ExpKirchhoffNormalises},  $\varphi_{\kappa,j} ={\mathcal O}( \varepsilon_{\kappa}^{2+\delta_{j3}})$ on $\partial \Omega$ and therefore 
\begin{equation}
\label{resser2}
\left| \int_{\partial \Omega} \partial_{t} \overline{u}^{ext}_\kappa \cdot n \, \varphi_{\kappa,j} \, ds \right| \leq C \varepsilon_{\kappa}^{2+\delta_{j3}}(1+|\widehat{\bf p}|).
\end{equation}
\paragraph{Step 3.}
Finally we address the integrals  in the right hand side of \eqref{LAST} 
which are over $\partial {\mathcal S}_{\nu}$ for $\nu \neq \kappa$. 
By \eqref{Eq:UextUextKappa} and \eqref{Eq:TildeUExt}, 
\begin{equation*}
\overline{u}^{ext}_\kappa =u^{ext} + \sum_{\lambda \neq \kappa} \sum_{i=1}^3  p_{\lambda,i} \nabla \varphi_{\lambda,i} + \sum_{\lambda \neq \kappa} \gamma_{\lambda} \nabla^\perp \widehat{\psi}_{\lambda}  + \sum_{i=1}^2 (\alpha_{\kappa,i}+ \beta_{\kappa,i}) \nabla \varphi_{\kappa,i},
\end{equation*}
so that 
 \begin{equation}
\label{USOP}
\int_{\partial  {\mathcal S}_{\nu}} \partial_{t} \overline{u}^{ext}_\kappa     \cdot n \, \varphi_{\kappa,j} \, ds 
= E_{\kappa,j}^{\nu,1} + \ldots + E_{\kappa,j}^{\nu,6} ,
\end{equation}
where
 \begin{align*}
E_{\kappa,j}^{\nu,1} &:= \int_{\partial {\mathcal S}_{\nu}} \partial_{t}u^{ext}  \cdot n \, \varphi_{\kappa,j} \, ds, \\
E_{\kappa,j}^{\nu,2}  &:=\sum_{\lambda \neq \kappa} \sum_{i=1}^3  p'_{\lambda,i} \int_{\partial {\mathcal S}_{\nu}}  \partial_n \varphi_{\lambda,i}  \, \varphi_{\kappa,j} \, ds , \\
E_{\kappa,j}^{\nu,3}  &:= \sum_{\substack{{\mu \in \{1,\dots,N\}} \\ {m \in \{1,2,3\}}}}    \sum_{\lambda \neq \kappa} \sum_{i=1}^3  p_{\lambda,i} p_{\mu,m} \int_{\partial {\mathcal S}_{\nu}}  \frac{\partial \nabla \varphi_{\lambda,i} }{\partial q_{\mu,m}} \cdot n \, \varphi_{\kappa,j} \, ds , \\
E_{\kappa,j}^{\nu,4}  &:= \sum_{\lambda \neq \kappa} \gamma_{\lambda} \int_{\partial {\mathcal S}_{\nu}} \partial_{t} \nabla^\perp \widehat{\psi}_{\lambda}   \cdot n \, \varphi_{\kappa,j} \, ds , \\
E_{\kappa,j}^{\nu,5} &:=  \sum_{\substack{{\mu \in \{1,\dots,N\}} \\ {m \in \{1,2,3\}}}}  \sum_{i=1}^2 (\alpha_{\kappa,i}+ \beta_{\kappa,i}) p_{\mu,m}  \int_{\partial {\mathcal S}_{\nu}}\frac{\partial \nabla \varphi_{\kappa,i}}{\partial q_{\mu,m}} \cdot n \, \varphi_{\kappa,j} \, ds , \\   
E_{\kappa,j}^{\nu,6} &:= \sum_{i=1}^2 (\alpha_{\kappa,i}+ \beta_{\kappa,i})'  \int_{\partial {\mathcal S}_{\nu}}  \nabla \varphi_{\kappa,i}  \cdot n \, \varphi_{\kappa,j} \, ds .
\end{align*}
\noindent
{\it Estimate of $E_{\kappa,j}^{\nu,1}$.} 
By \eqref{Eq:EstUextGlob2} and \eqref{Eq:EstUextGlob3},  $\| \partial_{t} u^{ext} \|_{ L^\infty(\partial {\mathcal S}_{\nu}) }  = {\mathcal O}( \varepsilon_{\nu}^{-1} (1 + |\widehat{{\bf p}}^{\varepsilon}|))$ and  by \eqref{Eq:ExpKirchhoffNormalises}, with $\nu \neq \kappa$,  $\|\varphi_{\kappa,j} \|_{ L^\infty(\partial {\mathcal S}_{\nu}) } = {\mathcal O}(\varepsilon_{\kappa}^{2+\delta_{j3}} )$  so that, by integration on $ \partial {\mathcal S}_{\nu}$, 
\begin{equation} \label{USOP1}
E_{\kappa,j}^{\nu,1} =   {\mathcal O}(\varepsilon_{\kappa}^{2+\delta_{j3}} (1 + |\widehat{{\bf p}}^{\varepsilon}|)) .
\end{equation}
\noindent
{\it Estimate of $E_{\kappa,j}^{\nu,2}$.} 
First by definition of the Kirchhoff potentials, see \eqref{def-xi-j} and \eqref{Kir}, 
\begin{equation} \label{annecy}
E_{\kappa,j}^{\nu,2}  = \sum_{i=1}^3  p'_{\nu,i} \int_{\partial {\mathcal S}_{\nu}}   \, \varphi_{\kappa,j}  K_{\nu,i}  \, ds .
\end{equation}
By Proposition \ref{Pro:Acceleration}, $| \widehat{p}\, '_{\nu}| =  {\mathcal O}(\varepsilon_{\nu}^{-2 \delta_{\nu \in {\mathcal P}_{(iii)}}} (1 + |\widehat{\boldsymbol{p}}^{\varepsilon}|))$, and by \eqref{Eq:AM-AMSA} and  \eqref{Eq:CoefsMA}, the integral in the right hand side of \eqref{annecy}  is ${\mathcal O}( \varepsilon_{\kappa}^{2 + \delta_{3j}} \varepsilon_{\nu}^{2 + \delta_{3i}})$, 
so that, since $\varepsilon_{\nu}^{\delta_{i3}} p_{\nu,i} = \widehat{p}_{\nu,i}$, 
\begin{equation} \label{USOP2}
E_{\kappa,j}^{\nu,2} =   {\mathcal O}( \varepsilon_{\kappa}^{2 + \delta_{3j}} (1 + |\widehat{\boldsymbol{p}}^{\varepsilon}|) )  .
\end{equation}
\noindent
{\it Estimate of $E_{\kappa,j}^{\nu,3}$.} 
By Lemma \ref{Lem:ShapeDerivativeKirchhoff}, 
\begin{align} \nonumber
E_{\kappa,j}^{\nu,3} &= \sum_{m=1}^3    \sum_{\lambda \neq \kappa} \sum_{i=1}^3  p_{\lambda,i} p_{\nu,m} 
\int_{\partial {\mathcal S}_{\nu}} \frac{\partial}{\partial \tau} \left[ \left( \frac{\partial \varphi_{\lambda,i}}{\partial \tau}  - (\xi_{\lambda,i} \cdot \tau) \right) (\xi_{\nu,m} \cdot n) \right]  \, \varphi_{\kappa,j} \, ds  \\ 
\label{train1}
&\quad +  \sum_{m=1}^2     p_{\nu,3} p_{\nu,m}  \int_{\partial {\mathcal S}_{\nu}}  \varphi_{\kappa,j} K_{\nu,m}  \, ds  .
\end{align}
By an integration by parts 
\begin{equation*}
\int_{\partial {\mathcal S}_{\nu}} 
\frac{\partial}{\partial \tau} \left[ \left( \frac{\partial \varphi_{\lambda,i}}{\partial \tau} 
- (\xi_{\lambda,i} \cdot \tau) \right) (\xi_{\nu,m} \cdot n) \right]\, \varphi_{\kappa,j} \, ds 
= - \int_{\partial {\mathcal S}_{\nu}} 
 \left( \frac{\partial \varphi_{\lambda,i}}{\partial \tau} 
- (\xi_{\lambda,i} \cdot \tau) \right)  (\xi_{\nu,m} \cdot n)  \frac{\partial \varphi_{\kappa,j}}{\partial \tau} \,  ds .
\end{equation*}
By \eqref{Eq:ExpKirchhoff1Bis}, 
$$
\left\| \frac{\partial \varphi_{\lambda,i}}{\partial \tau}  \right\|_{L^{\infty}(\partial {\mathcal S}_{\nu} ))}  =  {\mathcal O}( \varepsilon_{\lambda}^{\delta_{i3}} )
\  \text{ and } \
\left\| \frac{\partial \varphi_{\kappa,j}}{\partial \tau}  \right\|_{L^{\infty}( \partial {\mathcal S}_{\nu}))}  =  {\mathcal O}( \varepsilon_{\kappa}^{2+\delta_{j3}} ) .
$$
By integration on $ \partial {\mathcal S}_{\nu}$, using that $\varepsilon_{\lambda}^{\delta_{i3}} p_{\lambda,i} = \widehat{p}_{\lambda,i}$ and  \eqref{Eq:APEEnergy}, we obtain that the first term of the right hand side of \eqref{train1} is $ {\mathcal O}( \varepsilon_{\kappa}^{2 + \delta_{3j}}   |\widehat{\boldsymbol{p}}^{\varepsilon}|)$. 
On the other hand, by \eqref{Eq:AM-AMSA}, \eqref{Eq:CoefsMA} and Remark~\ref{Rem:MA/MASAM}, the second integral in the right hand side of \eqref{train1} is of order ${\mathcal O}( \varepsilon_{\kappa}^{2 + \delta_{3j}} \varepsilon_{\nu}^{2})$
so that by \eqref{Eq:APEEnergy}, we arrive at 
\begin{equation} \label{USOP3}
E_{\kappa,j}^{\nu,3} =   {\mathcal O}( \varepsilon_{\kappa}^{2 + \delta_{3j } }   |\widehat{\boldsymbol{p}}^{\varepsilon}|) .
\end{equation}
\noindent
{\it Estimate of $E_{\kappa,j}^{\nu,4}$.}
We deal with the term $E_{\kappa,j}^{\nu,4}$ by distinguishing two cases: \par
\begin{itemize}
\item {First case: $\lambda \neq \nu$.} %
By \eqref{Eq:HatPsiTourne}, 
\begin{equation*}
\int_{\partial {\mathcal S}_{\nu}} \partial_{t} \nabla^\perp \widehat{\psi}_{\lambda}   \cdot n \, \varphi_{\kappa,j} \, ds
=
- \int_{\partial {\mathcal S}_{\nu}}   \nabla \left( v_{{\mathcal S},\lambda} \cdot \nabla^{\perp} \widehat{\psi}_{\lambda} \right)  \cdot n  \,  \varphi_{\kappa,j} \, ds .
\end{equation*}
By \eqref{Eq:BehaviourPsi2} and the remark below \eqref{Eq:BehaviourPsi2}, we find  
\begin{equation} \label{Eq:Estnabla2Psi}
\|  \nabla \left( v_{{\mathcal S},\lambda} \cdot \nabla^{\perp} \widehat{\psi}_{\lambda} \right)  \cdot n  \|_{ L^\infty(\partial {\mathcal F} \setminus  \partial {\mathcal S}_{\lambda} )}
 = {\mathcal O}( |\widehat{p}_\lambda |).
\end{equation}
Hence since $\nu \neq \lambda$ we deduce  with \eqref{Eq:ExpKirchhoffNormalises} 
\begin{equation*}
\int_{\partial {\mathcal S}_{\nu}} \partial_{t} \nabla^\perp \widehat{\psi}_{\lambda}   \cdot n \, \varphi_{\kappa,j} \, ds
= {\mathcal O}(\varepsilon_{\kappa}^{2+\delta_{j3}} |\widehat{p}_{\nu}|) . 
\end{equation*}
\item {Second case: $\lambda=\nu$.}
Using an integration by parts and \eqref{Eq:HatPsiTourne}  we find
\begin{multline}
\label{stb} 
\int_{\partial {\mathcal S}_{\nu}} \partial_{t} \nabla^\perp \widehat{\psi}_{\lambda}   \cdot n \, \varphi_{\kappa,j} \, ds
= \int_{{\mathcal F}}  \partial_{t} \nabla^\perp \widehat{\psi}_{\lambda}  \cdot   \nabla \varphi_{\kappa,j} \, dx
- \int_{\partial {\mathcal F} \setminus \partial {\mathcal S}_{\nu}}\partial_{t} \nabla^\perp \widehat{\psi}_{\lambda}   \cdot n \, \varphi_{\kappa,j} \, ds \\
=
- \int_{{\mathcal F}} \nabla \left( v_{{\mathcal S},\nu} \cdot \nabla^{\perp} \widehat{\psi}_{\nu} \right)  \cdot   \nabla \varphi_{\kappa,j} \, dx
+\int_{\partial {\mathcal F} \setminus \partial {\mathcal S}_{\nu}}   \nabla \left( v_{{\mathcal S},\nu} \cdot \nabla^{\perp} \widehat{\psi}_{\nu} \right)  \cdot n  \,  \varphi_{\kappa,j} \, ds .
\end{multline}
With another integration by parts, the first term in the right hand side of \eqref{stb} is transformed into
\begin{equation} \nonumber
- \int_{ \partial {\mathcal S}_{\kappa}}   v_{{\mathcal S},\nu} \cdot \nabla^{\perp} \widehat{\psi}_{\nu}  \,   K_{\kappa,j} \, ds .
\end{equation}
Proceeding as for \eqref{Eq:EstT2}, we see that this term can be estimated by ${\mathcal O}(\varepsilon_{\kappa}^{2+\delta_{j3}} |\widehat{p}_{\nu}|$). 
We decompose  the second  term  in the right hand side of \eqref{stb}  into 
\begin{multline*}
\int_{\partial {\mathcal F} \setminus \partial {\mathcal S}_{\nu}}   \nabla \left( v_{{\mathcal S},\nu} \cdot \nabla^{\perp} \widehat{\psi}_{\nu} \right)  \cdot n  \,  \varphi_{\kappa,j} \, ds 
= \int_{\partial {\mathcal F} \setminus \left( \partial {\mathcal S}_{\nu} \cup  \partial {\mathcal S}_{\kappa} \right)}   \nabla \left( v_{{\mathcal S},\nu} \cdot \nabla^{\perp} \widehat{\psi}_{\nu} \right)  \cdot n  \,  \varphi_{\kappa,j} \, ds  \\
+ \int_{ \partial {\mathcal S}_{\kappa}}   \nabla \left( v_{{\mathcal S},\nu} \cdot \nabla^{\perp} \widehat{\psi}_{\nu} \right)  \cdot n  \,  \varphi_{\kappa,j} \, ds .
\end{multline*}
We use \eqref{Eq:Estnabla2Psi} and \eqref{Eq:ExpKirchhoffNormalises} to deduce that the terms in the right hand side of \eqref{stb} are of order ${\mathcal O}(\varepsilon_{\kappa}^{2+\delta_{j3}} |\widehat{p}_{\nu}|)$ (using $|\partial {\mathcal S}_{\kappa}|={\mathcal O}(\varepsilon_{\kappa})$ for the last one). 
\end{itemize}
Gathering the two cases we finally arrive at
\begin{equation} \label{USOP4}
E_{\kappa,j}^{\nu,4} =   {\mathcal O}(\varepsilon_{\kappa}^{2+\delta_{j3}} |\widehat{p}_{\nu}|)  .
\end{equation}
\noindent
{\it Estimate of $E_{\kappa,j}^{\nu,5}$.} 
By Lemma \ref{Lem:ShapeDerivativeKirchhoff}, 
\begin{equation*}
 E_{\kappa,j}^{\nu,5} =   \sum_{m=1}^3    \sum_{i=1}^2 (\alpha_{\kappa,i}+ \beta_{\kappa,i}) p_{\nu,m}  \int_{\partial {\mathcal S}_{\nu}}
\partial_n \left(\frac{\partial  \varphi_{\kappa,i}}{\partial q_{\nu,m}} \right)   \, \varphi_{\kappa,j} \, ds .
\end{equation*}
For such indices, by \eqref{Eq:EstDerivKirchhoff1,5}, $\left\| \nabla \frac{\partial \varphi_{\kappa,i}}{\partial q_{\nu,m}} \right\|_{L^\infty({ \partial {\mathcal S}_{\nu}})} = {\mathcal O}(\varepsilon_{\kappa}^{ 2 } \varepsilon_{\nu}^{-1+\delta_{m3}})$ (recall that $\nu \neq \kappa$). 
Combining with \eqref{Eq:BoundAlphaBeta}, \eqref{Eq:ExpKirchhoffNormalises} and $|\partial {\mathcal S}_{\nu}|={\mathcal O}(\varepsilon_{\nu})$, we arrive at
\begin{equation} \label{USOP5}
E_{\kappa,j}^{\nu,5} =    {\mathcal O}(\varepsilon_{\kappa}^{4+\delta_{j3}}  |\widehat{p}_{\nu}|)  .
\end{equation}
\noindent
{\it Estimate of $E_{\kappa,j}^{\nu,6}$.} 
Since $\nu \neq \kappa$, by definition of the Kirchhoff potentials, see \eqref{def-xi-j} and \eqref{Kir}, 
\begin{equation} \label{USOP6}
E_{\kappa,j}^{\nu,6} =  0  .
\end{equation}
\paragraph{Step 4.}
Gathering \eqref{USOP}, \eqref{USOP1}, \eqref{USOP2}, \eqref{USOP3}, \eqref{USOP4}, \eqref{USOP5} and \eqref{USOP6}
we deduce that for $\nu \neq \kappa$, 
\begin{equation}
\label{resser3}
 \left|\int_{\partial {\mathcal S}_{\nu}} \partial_{t} \overline{u}^{ext}_\kappa \cdot n \, \varphi_{\kappa,j} \, ds \right| \leq C \varepsilon_{\kappa}^{2+\delta_{j3}} (1+|\widehat{\bf p}|) .
\end{equation}
Finally combining \eqref{LAST}, \eqref{resser1}, \eqref{resser2} and \eqref{resser3} we conclude the proof of Lemma~\ref{Lem:I3J5}.
\end{proof}
%
%
%
%
%
%
%
%
%
%
\subsection{Main gyroscopic term}
\label{Subsec:MGT}
In this section we study the term  $J_{\kappa,j}^{2} $ defined in \eqref{saiJ2}. 
We recall that $\kappa \in {\mathcal P}_{s}$. 
\begin{Lemma} \label{Lem:J2}
The term $J_{\kappa,j}^{2}$ can be put in the form
\begin{equation*}
J_{\kappa,j}^{2} =  B_{\kappa} + A_{\kappa} + D_{\kappa},  
\end{equation*}
where $B_{\kappa} = (B_{\kappa,j} )_{j=1,2,3} $  is the main gyroscopic term given by  \eqref{Eq:Gyro},
the term $A_{\kappa}$ is weakly nonlinear in the sense of \eqref{Eq:WNL} and the term $D_{\kappa}$ is weakly gyroscopic in the sense of \eqref{Eq:WG}-\eqref{Eq:EstGyroW}.
\end{Lemma}
\begin{proof}[Proof of Lemma~\ref{Lem:J2}]
We first notice that from \eqref{Eq:DecompKappa} and \eqref{Eq:DecompMod} we have
$\overline{u}_{\kappa}^{pot} + \overline{u}^{ext}_\kappa = {u}_{\kappa}^{pot} + {u}^{ext}_\kappa$. 
Using \eqref{Eq:AssumptionApproximation} and $u^{pot}_\kappa = \mbox{Kir}_{\kappa}(v_{{\mathcal S},\kappa})$, we deduce that on $\partial \mathcal{S}_{\kappa}$
\begin{equation*}
\overline{u}_{\kappa}^{pot} + \overline{u}^{ext}_\kappa -v_{{\mathcal S},\kappa} =  (\mbox{Id} - \mbox{Kir}_{\kappa})({V}_{\kappa} - v_{{\mathcal S},\kappa}) + \varepsilon_{\kappa}^2 {u}_{\kappa}^r 
=  (\mbox{Id} - \widehat{\mbox{Kir}}_{\kappa})({V}_{\kappa} - v_{{\mathcal S},\kappa}) + \varepsilon_{\kappa}^2 \widetilde{u}_{\kappa}^r ,
\end{equation*} 
where we recall that $ \widehat{\mbox{Kir}}_{\kappa}$ is defined in \eqref{Kirchap} and where we have set
$$
\widetilde{u}_{\kappa}^r := {u}_{\kappa}^r + \varepsilon_\kappa^{-2} \left(\widehat{\mbox{Kir}}_{\kappa}({V}_{\kappa} - v_{{\mathcal S},\kappa}) -  \mbox{Kir}_{\kappa}({V}_{\kappa} - v_{{\mathcal S},\kappa}) \right).
$$
Thus for  $j \in \{1,2,3\}$ (recalling the notation \eqref{pazencor}), we have
\begin{equation} \label{Eq:ExprJ2}
J_{\kappa,j}^{2} = \tilde{J}_{\kappa,j}^{2} 
+ \varepsilon_{\kappa}^2 \gamma_{\kappa} Q_{\kappa,j}(\nabla^{\perp} \widehat{\psi}_{\kappa}, \widetilde{u}_{\kappa}^r) ,
\end{equation}
where 
\begin{equation} \label{Eq:Jk2-1}
\tilde{J}_{\kappa,j}^{2} := 
\gamma_{\kappa} Q_{\kappa,j} \left( \nabla^{\perp} \widehat{\psi}_{\kappa}, (\mbox{Id} - \widehat{\mbox{Kir}}_{\kappa})({V}_{\kappa} - v_{{\mathcal S},\kappa}) \right). 
\end{equation}
Using \eqref{Eq:DiffKir}, \eqref{Eq:EstApprox1}, $\| \nabla^{\perp} \widehat{\psi}_{\kappa} \|_{L^\infty(\partial {\mathcal S}_{\kappa})} ={\mathcal O}(1/\varepsilon_{\kappa})$ and $|\partial {\mathcal S}_{\kappa}| ={\mathcal O}(\varepsilon_{\kappa})$, we see that the last term in \eqref{Eq:ExprJ2} is weakly nonlinear. \par
To deal with the  term $\tilde{J}_{\kappa,j}^{2}$, we first observe that, by \eqref{Eq:VitSolideXi}, \eqref{Eq:DecompV}, \eqref{Eq:ParamModulation} and \eqref{Eq:VariableModulee}, 
\begin{equation}
\label{latatasse}
{V}_{\kappa} - v_{{\mathcal S},\kappa} =
- \sum_{k=1}^3  \overline{p}_{\kappa,k}  \xi_{\kappa,k}
-   \sum_{k=1}^2   \beta_{\kappa,k} \xi_{\kappa,k}
+ \sum_{k=4}^5 V_{\kappa,k} \xi_{\kappa,k}.
\end{equation}
We are therefore led to estimate $Q_{\kappa,j} \left( \nabla^{\perp} \widehat{\psi}_{\kappa}, (\mbox{Id} - \widehat{\mbox{Kir}}_{\kappa}) \xi_{\kappa,k} \right) $,
for  $\kappa \in {\mathcal P}_{s}$, $j \in \{1,2,3\}$ and  $k\in \{1,2,4,5\}$. 
We will rely on the following classical result.
\begin{Lemma}[] \label{Lem:Lamb}
Let $\mathcal S_0$ a smooth compact simply connected domain of $\R^2$. 
For any pair of vector fields $u$, $v$ in $C^\infty (\overline{\R^2 \setminus \mathcal S_0} ; \R^2 )$ satisfying 
 $\div u = \div v = \curl u  = \curl v = 0$ in $\R^2 \setminus \mathcal S_0$
and $u(x) = {\mathcal O}(1/|x|)$ and $v(x) = {\mathcal O}(1/|x|)$ as $|x| \rightarrow + \infty$,
one has, for any $j=1,2,3$, 
\begin{equation} \nonumber 
\int_{\partial \mathcal S_0} (u\cdot v) K_{j} (0,\cdot ) \, ds 
= \int_{\partial  \mathcal S_0} \xi_{j} (0,\cdot )  \cdot \Big(  (u \cdot n) v +  (v \cdot n)  u  \Big) \, ds .
\end{equation}
\end{Lemma}
We refer to \cite[Article 134a. (3) and (7)]{Lamb} for a proof of Lemma~\ref{Lem:Lamb}; see also \cite[Lemma 4.6]{GMS}).
Lemma~\ref{Lem:Lamb}  has the following consequence.
\begin{Lemma} \label{Lem:5.3}
For all $j=1,2,3$ and $k=1,2,3,4,5$, we have
\begin{equation} \label{PhiPsi}
Q_{\kappa,j}( \nabla^{\perp} \widehat{\psi}_{\kappa}, \xi_{\kappa,k}- \nabla \widehat\varphi_{\kappa,k} )
 =   \int_{\partial \mathcal{S}_{\kappa}} \partial_{n} \widehat{\psi}_{\kappa} \, \xi_{\kappa,k}^{\perp} \cdot \xi_{\kappa,j} \, ds.
\end{equation}
\end{Lemma}
\begin{proof}[Proof of Lemma~\ref{Lem:5.3}]
First, using that  the vector field $\nabla^{\perp} \widehat{\psi}_{\kappa}$ is tangent to $\partial \mathcal{S}_{\kappa}$, we split the integral into two parts
\begin{equation*}
Q_{\kappa,j}( \nabla^{\perp} \widehat{\psi}_{\kappa}, \xi_{\kappa,k}- \nabla \widehat\varphi_{\kappa,k} )
 =  \int_{\partial \mathcal{S}_{\kappa}} \partial_{n} \widehat{\psi}_{\kappa} ( \xi_{\kappa,k} \cdot \tau ) \, K_{\kappa,j} \, ds  
-  \int_{\partial \mathcal{S}_{\kappa}}   \nabla^{\perp} \widehat{\psi}_{\kappa} \cdot  \nabla \widehat\varphi_{\kappa,k}    \, K_{\kappa,j} \, ds  .
\end{equation*}
Then thanks to Lemma~\ref{Lem:Lamb}, we transform the second integral as
\begin{equation*}
- \int_{\partial \mathcal{S}_{\kappa}}  \xi_{\kappa,j} \cdot \Big(  ( \nabla \widehat\varphi_{\kappa,k}  \cdot     n)\nabla^{\perp} \widehat{\psi}_{\kappa}  \Big)  \, ds .
\end{equation*}
Finally, since  $\nabla \widehat\varphi_{\kappa,k}  \cdot n = K_{\kappa,k} = - \xi_{\kappa,k}^{\perp} \cdot \tau$, we observe that 
\begin{equation*}
- \int_{\partial \mathcal{S}_{\kappa}}  \xi_{\kappa,j}   \cdot \Big(  ( \nabla \widehat\varphi_{\kappa,k}  \cdot     n)\nabla^{\perp} \widehat{\psi}_{\kappa}  \Big)  \, ds
=  \int_{\partial \mathcal{S}_{\kappa}} \partial_{n} \widehat{\psi}_{\kappa} \, (\xi_{\kappa,k}^{\perp} \cdot \tau) (\xi_{\kappa,j} \cdot \tau)   \, ds ,
\end{equation*}
and we arrive at \eqref{PhiPsi}.
\end{proof}
%
%
%
%
%
%
%
Now with \eqref{latatasse} and Lemma~\ref{Lem:5.3}, we consequently transform \eqref{Eq:Jk2-1} into
\begin{equation} \nonumber 
\tilde{J}_{\kappa,j}^{2}=
B_{\kappa,j} \ + \widehat{J}_{\kappa,j}^{2}  ,
\end{equation}
where we recall that $B_{\kappa} = (B_{\kappa,j} )_{j=1,2,3} $ is the main gyroscopic term given by \eqref{Eq:Gyro} and where 
\begin{equation} \nonumber 
 \widehat{J}_{\kappa,j}^{2} :=
- \ \gamma_{\kappa} \sum_{k=1}^2 \beta_{\kappa,k}  \int_{\partial \mathcal{S}_{\kappa}} \partial_{n} \widehat{\psi}_{\kappa} \, \xi_{\kappa,k}^{\perp} \cdot \xi_{\kappa,j} \, ds
\ + \ \gamma_{\kappa} \sum_{k=4}^5 V_{\kappa,k}  \int_{\partial \mathcal{S}_{\kappa}} \partial_{n} \widehat{\psi}_{\kappa} \, \xi_{\kappa,k}^{\perp} \cdot \xi_{\kappa,j} \, ds .
\end{equation}
We have the following lemma, which is the main reason for the choice of $\beta_{\kappa,1}$ and $\beta_{\kappa,2}$ in \eqref{Eq:ParamModulation}.
\begin{Lemma} \label{Lem:Transfert45vers12}
Define $\beta_{\kappa,1}$ and $\beta_{\kappa,2}$ by \eqref{Eq:ParamModulation}. Then one has the following relation for $j=1,2$,
\begin{equation} \label{Eq:Transfert}
\sum_{k=1}^2 \beta_{\kappa,k} \int_{\partial \mathcal{S}_{\kappa}} \partial_{n} \widehat{\psi}_{\kappa} \, \xi_{\kappa,k}^{\perp} \cdot \xi_{\kappa,j} \, ds
= 
\sum_{k=4}^5 V_{\kappa,k} \int_{\partial \mathcal{S}_{\kappa}} \partial_{n} \widehat{\psi}_{\kappa} \, \xi_{\kappa,k}^{\perp} \cdot \xi_{\kappa,j} \, ds.
\end{equation}
\end{Lemma} 
\begin{proof}[Proof of Lemma~\ref{Lem:Transfert45vers12}] 
This is a direct consequence of \eqref{def-xi-j}, \eqref{circ-norma} and \eqref{Eq:Zeta}: for $j=1,2$ and $k=1,2$ one finds
\begin{equation*}
\int_{\partial \mathcal{S}_{\kappa}} \partial_{n} \widehat{\psi}_{\kappa} \, \xi_{\kappa,k}^{\perp} \cdot \xi_{\kappa,j} \, ds = 
\begin{pmatrix}
0 & -1 \\
1 & 0
\end{pmatrix}_{k,j=1,2},
\end{equation*}
while for  $j=1,2$ and $k=4,5$ one has
\begin{equation*}
\int_{\partial \mathcal{S}_{\kappa}} \partial_{n} \widehat{\psi}_{\kappa} \, \xi_{\kappa,k}^{\perp} \cdot \xi_{\kappa,j} \, ds = 
\begin{pmatrix}
\zeta_{\kappa,2} & \zeta_{\kappa,1} \\
\zeta_{\kappa,1} & -\zeta_{\kappa,2}
\end{pmatrix}_{\substack{{k=4,5} \\ {j=1,2}}}.
\end{equation*}
Hence \eqref{Eq:Transfert} is equivalent to $\beta_{\kappa,2}=\zeta_{\kappa,2}V_{\kappa,4}+\zeta_{\kappa,1}V_{\kappa,5}$ and $-\beta_{\kappa,1}=\zeta_{\kappa,1}V_{\kappa,4}-\zeta_{\kappa,2}V_{\kappa,5}$, that is, exactly the second relation of \eqref{Eq:ParamModulation}.
\end{proof}
From Lemma~\ref{Lem:Transfert45vers12} we readily deduce that $\widehat{J}_{\kappa,1}^{2}=\widehat{J}_{\kappa,2}^{2}=0$.
Hence it remains only to study
\begin{equation*}
\widehat{J}_{\kappa,3}^{2} =
-\gamma_{\kappa} \sum_{k=1}^2 \beta_{\kappa,k}  \int_{\partial \mathcal{S}_{\kappa}} \partial_{n} \widehat{\psi}_{\kappa} \, \xi_{\kappa,k}^{\perp} \cdot \xi_{\kappa,3} \, ds
+\gamma_{\kappa} \sum_{k=4}^5 V_{\kappa,k}  \int_{\partial \mathcal{S}_{\kappa}} \partial_{n} \widehat{\psi}_{\kappa} \, \xi_{\kappa,k}^{\perp} \cdot \xi_{\kappa,3} \, ds =: D^1_{3} + D^2_{3}.
\end{equation*}
Let us show that the term $\widehat{J}_{\kappa}^{2} = (0,0,D^1_{3}+D^2_{3})^T$ is weakly gyroscopic. First, with  \eqref{Eq:ScalePsi}, \eqref{Eq:EstApprox1} and \eqref{Eq:BoundAlphaBeta} and 
$\| \xi_{\kappa,k} \|_{L^{\infty}(\partial S_{\kappa})}= {\mathcal O}(\varepsilon_{\kappa}) $ for $ k=4,5$, it is easy to check that it satisfies \eqref{Eq:EstGyroW}. Let us now prove \eqref{Eq:WG} by treating the two terms $(0,0,D^1_{3})^T$ and $(0,0,D^2_{3})^T$ separately.\par 
\smallskip \par
We start with the term $(0,0,D^1_{3})^T$. Here \eqref{Eq:Zeta} gives  for $k=1,2$
\begin{equation*}
\int_{\partial \mathcal{S}_{\kappa}} \partial_{n} \widehat{\psi}_{\kappa} \, \xi_{\kappa,k}^{\perp} \cdot \xi_{\kappa,3} \, ds 
= \zeta^\varepsilon_{\kappa}(q_{\kappa}) \cdot e_{k} .
\end{equation*}
Moreover, due to \eqref{Eq:ParamModulation} we have
\begin{equation*}
\sum_{k=1}^2 \beta_{\kappa,k} \zeta^\varepsilon_{\kappa}(q_{\kappa}) \cdot e_{k}
=
\zeta^\varepsilon_{\kappa}(q_{\kappa}) \cdot {\mathcal A}(V_{\kappa})
\zeta^\varepsilon_{\kappa}(q_{\kappa})
\ \text{ where } \ 
{\mathcal A}(V_{\kappa}):= \begin{pmatrix} -V_{\kappa,4} & V_{\kappa,5} \\ V_{\kappa,5} & V_{\kappa,4} \end{pmatrix} .
\end{equation*}
Since the matrix ${\mathcal A}(V_{\kappa})$ is a traceless symmetric $2 \times 2$ matrix, we have $R(\vartheta)^* {\mathcal A}(V_{\kappa}) = {\mathcal A}(V_{\kappa}) R(\vartheta)$ so that, using again  $\eqref{Eq:Zeta}$, 
\begin{equation*}
\sum_{k=1}^2 \beta_{\kappa,k} \zeta^\varepsilon_{\kappa}(q_{\kappa}) \cdot e_{k}
= \varepsilon_{\kappa}^2 \zeta^1_{\kappa,0} \cdot {\mathcal A}(V_{\kappa}) R(2 \vartheta_{\kappa}) \zeta^1_{\kappa,0}.
\end{equation*}
It follows that
\begin{equation*}
\int_{0}^t \overline{p}_{\kappa,3}(\tau) D^1_{3} (\tau) \, d \tau
= - \gamma_{\kappa} \varepsilon_{\kappa}^2 \zeta^1_{\kappa,0} \cdot 
\int_{0}^t \vartheta_{\kappa}'(\tau) {\mathcal A}(V_{\kappa (\tau)}) R(2 \vartheta_{\kappa}(\tau)) \zeta^1_{\kappa,0} \, d \tau.
\end{equation*}
By integration by parts we infer
\begin{equation*}
\int_{0}^t \vartheta_{\kappa}'(\tau) {\mathcal A}(V_{\kappa}(\tau)) R(2 \vartheta_{\kappa}(\tau)) \zeta^1_{\kappa,0} \, d\tau
=
- \frac{1}{2} \int_{0}^t {\mathcal A}(V_{\kappa}'(\tau)) R(2 \vartheta_{\kappa}(\tau)) \zeta^1_{\kappa,0}\, d\tau
+ \frac{1}{2} \Big[ {\mathcal A}(V_{\kappa}(\tau)) R(2 \vartheta_{\kappa}(\tau)) \zeta^1_{\kappa,0} \Big]_{0}^t.
\end{equation*}
Since we can bound the right hand side by $C(1+ \| V_{\kappa} \|_{\infty} + t \| V_{\kappa}' \|_{\infty})$, the estimate \eqref{Eq:WG} for the term $(0,0,D^1_{3})^T$ follows from Proposition~\ref{Pro:ExistsApprox}.\par 
\smallskip \par
We now consider the term $(0,0,D^2_{3})^T$. In that case,  the integrals are given by
\begin{gather*}
\int_{\partial \mathcal{S}_{\kappa}} \partial_{n} \widehat{\psi}_{\kappa} \, \xi_{\kappa,4}^{\perp} \cdot \xi_{\kappa,3} \, ds
= \int_{\partial \mathcal{S}_{\kappa}} \partial_{n} \widehat{\psi}_{\kappa} \, \big[ (x_{2}-h_{\kappa,2})^2 - (x_{1}-h_{\kappa,1})^2 \big] \, ds \\
\text{ and } \ \ 
\int_{\partial \mathcal{S}_{\kappa}} \partial_{n} \widehat{\psi}_{\kappa} \, \xi_{\kappa,5}^{\perp} \cdot \xi_{\kappa,3} \, ds
= 2 \int_{\partial \mathcal{S}_{\kappa}} \partial_{n} \widehat{\psi}_{\kappa} \, (x_{1}-h_{\kappa,1}) (x_{2}-h_{\kappa,2}) \, ds.
\end{gather*}
We notice that
\begin{equation*}
(x-h_{\kappa})^\perp \otimes (x-h_{\kappa}) + (x-h_{\kappa}) \otimes (x-h_{\kappa})^\perp = \begin{pmatrix}
-2 (x_{1} - h_{\kappa,1}) (x_{2}-h_{\kappa,2}) & (x_{1}-h_{\kappa,1})^2 - (x_{2}-h_{\kappa,2})^2 \\
 (x_{1}-h_{\kappa,1})^2 - (x_{2}-h_{\kappa,2})^2  & 2 (x_{1} - h_{\kappa,1}) (x_{2}-h_{\kappa,2}),
\end{pmatrix}
\end{equation*}
and consequently
\begin{multline*}
\sum_{k=4}^5 V_{\kappa,k}  \int_{\partial \mathcal{S}_{\kappa}} \partial_{n} \widehat{\psi}_{\kappa} \, \xi_{\kappa,k}^{\perp} \cdot \xi_{\kappa,3} \, ds
\\
= e_{1} \cdot
\left(\int_{\partial \mathcal{S}_{\kappa}} \partial_{n} \widehat{\psi}_{\kappa} \, \big[(x-h_{\kappa})^\perp \otimes (x-h_{\kappa}) + (x-h_{\kappa}) \otimes (x-h_{\kappa})^\perp \big]\, ds\right)
\begin{pmatrix}
-V_{\kappa,5} \\ -V_{\kappa,4}
\end{pmatrix}.
\end{multline*}
Now the matrix between parentheses can be rewritten as
\begin{multline*}
\int_{\partial \mathcal{S}_{\kappa}} \partial_{n} \widehat{\psi}_{\kappa} \, \big( (x-h_{\kappa})^\perp \otimes (x-h_{\kappa}) + (x-h_{\kappa}) \otimes (x-h_{\kappa})^\perp \big) \, ds \\
=
\varepsilon_{\kappa}^2 R(\vartheta_{\kappa}) \left[\int_{\partial \mathcal{S}_{\kappa,0}} \partial_{n} \widehat{\psi}_{\kappa,0} \, \big( x^\perp \otimes x + x \otimes x^\perp \big) \, ds \right]
R(\vartheta_{\kappa})^*.
\end{multline*}
Call ${\bf Z}$ the time-independent matrix between brackets. Since ${\bf Z}$ is a traceless symmetric $2 \times 2$ matrix, we have $R(\vartheta_{\kappa}) {\bf Z} R(\vartheta_{\kappa})^* = {\bf Z} R(-2\vartheta_{\kappa})$, so that
\begin{equation*}
\sum_{k=4}^5 V_{\kappa,k}  \int_{\partial \mathcal{S}_{\kappa}} \partial_{n} \widehat{\psi}_{\kappa} \, \xi_{\kappa,k}^{\perp} \cdot \xi_{\kappa,3} \, ds
= - e_{1} \cdot {\bf Z} R(-2\vartheta_{\kappa}) \begin{pmatrix}
V_{\kappa,5} \\ V_{\kappa,4}
\end{pmatrix}.
\end{equation*}
Now we deduce
\begin{equation*}
\int_{0}^t \overline{p}_{\kappa,3}(\tau) D^2_{3} (\tau) \, d \tau
= - \gamma_{\kappa} \varepsilon_{\kappa}^2 e_{1} \cdot {\bf Z}
\int_{0}^t \vartheta_{\kappa}'(\tau) R(-2 \vartheta_{\kappa}) \begin{pmatrix}
V_{\kappa,5} \\ V_{\kappa,4}
\end{pmatrix} \, d \tau,
\end{equation*}
and we conclude as for the term $(0,0,D^1_{3})^T$ by using an integration by parts in time and the estimates of Proposition~\ref{Pro:ExistsApprox}.
\end{proof}
%

%
%
%
\subsection{Added mass term}
In this section we combine the term  $I_{\kappa,j}^{2} $ defined in \eqref{saiI2}, the term $J_{\kappa,j}^{4}$ 
defined in \eqref{saiJ4}   and the term $J_{\kappa,j}^{6}$ defined in \eqref{saiJ6}. We recall the notation \eqref{Eq:AddMassMatrix} for the added mass matrix ${\mathcal M}_{a,\kappa}$ which is time-dependent and that we say that a term is gyroscopic of lower order when it satisfies  \eqref{Eq:GyroF} and \eqref{Eq:EstGyro}. 
\begin{Lemma} \label{Lem:I2J4J6}
The term $I_{\kappa,j}^{2} + J_{\kappa,j}^{4} + J_{\kappa,j}^{6}$ can be put in the form
\begin{equation*}
I_{\kappa,j}^{2} + J_{\kappa,j}^{4} + J_{\kappa,j}^{6}
= {\mathcal M}_{a,\kappa} \overline{p}_{\kappa}' + \frac{1}{2} {\mathcal M}_{a,\kappa}' \overline{p}_{\kappa} + A_{\kappa}+ C_{\kappa},
\end{equation*}
where the term $A_{\kappa}$ is weakly nonlinear and where the term $C_{\kappa}$ is gyroscopic of lower order. 
\end{Lemma}
\begin{proof}[Proof of Lemma \ref{Lem:I2J4J6}] We proceed in three steps. \par
\ \par

\noindent
{\bf Step 1.}
Using the definition of $\overline{u}_{\kappa}^{pot}$ in \eqref{Eq:DecompMod},  we find,  for  $j \in \{1,2,3\}$,  
\begin{equation*}
I^{2}_{\kappa,j} = 
({\mathcal M}_{a,\kappa}\overline{p}_{\kappa}')_{\kappa,j}
+ \sum_{i=1}^3 \sum_{\substack{{\nu \in \{1,\dots,N\}} \\ {k \in \{1,2,3\}}}} \int_{{\mathcal F}(t)} \overline{p}_{\kappa,i} \, p_{\nu,k} \frac{\partial \nabla \varphi_{\kappa,i}}{\partial q_{\nu,k}} \cdot \nabla \varphi_{\kappa,j} \, dx .
\end{equation*}
On the other hand, by Reynolds' transport theorem, 
\begin{multline*}
({\mathcal M}_{a,\kappa}')_{ij} 
= \sum_{\substack{{\nu \in \{1,\dots,N\}} \\ {k \in \{1,2,3\}}}} \int_{{\mathcal F}(t)} p_{\nu,k} \frac{\partial \nabla \varphi_{\kappa,i}}{\partial q_{\nu,k}} \cdot \nabla \varphi_{\kappa,j} \, dx
+ \sum_{\substack{{\nu \in \{1,\dots,N\}} \\ {k \in \{1,2,3\}}}} \int_{{\mathcal F}(t)} \nabla \varphi_{\kappa,i} \cdot p_{\nu,k} \frac{\partial \nabla \varphi_{\kappa,j}}{\partial q_{\nu,k}} \, dx \\
+ \int_{\partial {\mathcal F}(t)} \nabla \varphi_{\kappa,i} \cdot \nabla \varphi_{\kappa,j} (u^{pot} \cdot n) \, ds ,
\end{multline*}
so that
\begin{multline*}
\left({\mathcal M}_{a,\kappa}\overline{p}_{\kappa}' + \frac{1}{2} {\mathcal M}_{a,\kappa}' \overline{p}_{\kappa} \right)_{j}
= I^2_{\kappa,j}
+\frac{1}{2} \sum_{i=1}^3 \sum_{\substack{{\nu \in \{1,\dots,N\}} \\ {k \in \{1,2,3\}}}} \int_{{\mathcal F}(t)} 
\overline{p}_{\kappa,i} \, p_{\nu,k} \left( \nabla \varphi_{\kappa,i} \cdot \frac{\partial \nabla \varphi_{\kappa,j}}{\partial q_{\nu,k}} 
- \frac{\partial \nabla \varphi_{\kappa,i}}{\partial q_{\nu,k}} \cdot \nabla \varphi_{\kappa,j} \right) \, dx \\
+ \frac{1}{2}  \sum_{i=1}^3   \int_{\partial {\mathcal F}(t)} \overline{p}_{\kappa,i} \nabla \varphi_{\kappa,i} \cdot \nabla \varphi_{\kappa,j} (u^{pot} \cdot n) \, ds.
\end{multline*}
We focus on the last term in the right-hand side. The idea is to replace $u^{pot} \cdot n$ with $\overline{u}_\kappa^{pot} \cdot n$, up to an error term. Adding and subtracting 
\eqref{saiJ4} in the right-hand side, and using \eqref{Eq:DecompMod} we find
\begin{eqnarray}
\nonumber
\left({\mathcal M}_{a,\kappa}\overline{p}_{\kappa}' + \frac{1}{2} {\mathcal M}_{a,\kappa}' \overline{p}_{\kappa} \right)_{j}
&=& I^2_{\kappa,j} + J_{\kappa,j}^{4} \\
\nonumber
&+& \frac{1}{2} \sum_{i=1}^3 \sum_{\substack{{\nu \in \{1,\dots,N\}} \\ {k \in \{1,2,3\}}}} \int_{{\mathcal F}(t)} 
\overline{p}_{\kappa,i} \, p_{\nu,k} \left( \nabla \varphi_{\kappa,i} \cdot \frac{\partial \nabla \varphi_{\kappa,j}}{\partial q_{\nu,k}} 
- \frac{\partial \nabla \varphi_{\kappa,i}}{\partial q_{\nu,k}} \cdot \nabla \varphi_{\kappa,j} \right) \, dx \\
\nonumber
&+& \frac{1}{2} \int_{\partial {\mathcal F}(t)}  \overline{u}_\kappa^{pot} \cdot \big( (\overline{u}_\kappa^{pot} \cdot n)  \nabla \varphi_{\kappa,j} - ( \nabla \varphi_{\kappa,j} \cdot n) \overline{u}_\kappa^{pot} \big)\, ds \\
\label{Eq:I2J4}
&+& \frac{1}{2} \int_{\partial {\mathcal F}(t)} \overline{u}_\kappa^{pot} \cdot \nabla \varphi_{\kappa,j} \, (u^{pot} - \overline{u}_\kappa^{pot}) \cdot n \, ds.
\end{eqnarray}
Call $C^1_{\kappa,j}$ the expression in the second line of \eqref{Eq:I2J4} and $C^2_{\kappa,j}$ the expression in the third line of \eqref{Eq:I2J4}. It is clear that $C_{\kappa}^1=(C^1_{\kappa,1},C^1_{\kappa,2},C^1_{\kappa,3})^T$ and $C_{\kappa}^2=(C^2_{\kappa,1},C^2_{\kappa,2},C^2_{\kappa,3})^T$ satisfy the property
 $\overline{p}_{\kappa} \cdot C^1_{\kappa} = \overline{p}_{\kappa} \cdot C^2_{\kappa}=0.$
Using  \eqref{Eq:EstDerivKirchhoff1} and an integration by parts we see that $C_{\kappa}^1$ satisfies \eqref{Eq:EstGyro}.
Using \eqref{Eq:VariableModulee}, \eqref{Eq:DecompMod}, \eqref{Eq:ExpKirchhoff1Bis} and \eqref{Eq:BoundAlphaBeta} we see that independently of $\overline{\boldsymbol\varepsilon}$, we have
\begin{equation} \label{Eq:EstUnat}
\| \overline{u}_{\kappa}^{pot} \|_{L^{\infty}(\partial {\mathcal S}_{\kappa})} \leq C (1+|\widehat{p}_{\kappa}|) \text{ and } \
\| \overline{u}_{\kappa}^{pot} \|_{L^{\infty}(\partial {\mathcal S}_{\nu})} \leq C \varepsilon_{\kappa}^2 (1+|\widehat{p}_{\kappa}|) \text{ for } \nu \neq \kappa.
\end{equation}
With $|\partial \mathcal{S}_\kappa | = \mathcal{O}(\varepsilon_\kappa)$, we deduce that $C_{\kappa}^2$ satisfies \eqref{Eq:EstGyro} .
Consequently the terms $C_{\kappa}^1$ and  $C_{\kappa}^2$ are gyroscopic of lower order (in the sense of \eqref{Eq:GyroF} and \eqref{Eq:EstGyro}). \par
\ \par
\noindent
{\bf Step 2.}
Hence we now focus on the last term in the right-hand side of \eqref{Eq:I2J4}.  We first consider the integral away from $\partial {\mathcal S}_{\kappa}$:
\begin{equation*}
\int_{\partial {\mathcal F}(t) \setminus \partial {\mathcal S}_{\kappa}} \overline{u}_\kappa^{pot} \cdot \nabla \varphi_{\kappa,j} \, (u^{pot} - \overline{u}_\kappa^{pot}) \cdot n \, ds =
\sum_{\nu \neq \kappa} \int_{\partial {\mathcal S}_{\nu}} \overline{u}_\kappa^{pot} \cdot \nabla \varphi_{\kappa,j} \, u^{pot} \cdot n \, ds ,
\end{equation*}
since $\overline{u}_\kappa^{pot} \cdot n=0$ on $\partial {\mathcal F}(t) \setminus \partial {\mathcal S}_{\kappa}$ and since moreover $u^{pot}\cdot n$ vanishes on $\partial \Omega$. From \eqref{Eq:DecompGlob} we have
\begin{equation*}
u^{pot} \cdot n = \sum_{i=1}^3 p_{\nu,i} K_{\nu,i} \ \text{ on } \partial {\mathcal S}_{\nu}.
\end{equation*}
Using the decay \eqref{Eq:ExpKirchhoff1Bis} of $\nabla \varphi_{\kappa,j}$, the energy estimates of Proposition~\ref{Pro:APEE} and $|\partial {\mathcal S}_{\nu}|= {\mathcal O}(\varepsilon_{\nu})$ we deduce that 
\begin{equation} \label{Eq:UpotUpotbar}
\left| \int_{\partial {\mathcal F}(t) \setminus \partial {\mathcal S}_{\kappa}} \overline{u}_\kappa^{pot} \cdot \nabla \varphi_{\kappa,j} \, (u^{pot} - \overline{u}_\kappa^{pot}) \cdot n \, ds \right| \leq C \varepsilon_{\kappa}^{4+\delta_{j3}} (1 + |\widehat{p}_{\kappa}|),
\end{equation}
so this term is weakly nonlinear. \par
Now we consider the integral over $\partial {\mathcal S}_{\kappa}$.
By \eqref{Eq:VariableModulee} and \eqref{Eq:DecompMod}  we have for $\kappa \in \{ 1,\dots,N \}$, 
\begin{equation} \label{Eq:up-unat}
(u^{pot} - \overline{u}_\kappa^{pot}) \cdot n = \sum_{\ell=1}^2 (\alpha_{\kappa,\ell} + \beta_{\kappa,\ell}) K_{\kappa,\ell} 
\ \text{ on } \ \partial {\mathcal S}_{\kappa}.
\end{equation}
Hence with \eqref{Eq:BoundAlphaBeta} we see that this factor is bounded. 
We want now to replace in this integral the factor $\overline{u}_\kappa^{pot} \cdot \nabla \varphi_{\kappa,j} $ by $  \widehat{u}_\kappa^{pot} \cdot \nabla \widehat \varphi_{\kappa,j},$ where we set 
\begin{equation*}
\widehat{u}_\kappa^{pot} := \sum_{i=1}^3 \overline{p}_{\kappa,i} \nabla \widehat{\varphi}_{\kappa,i}.
\end{equation*}
%
Similarly to \eqref{Eq:EstUnat}, we have
\begin{equation} \label{Eq:EstUnatprime}
\| \widehat{u}_{\kappa}^{pot} \|_{L^{\infty}(\partial {\mathcal S}_{\kappa})} \leq C (1+|\widehat{p}_{\kappa}|) .
\end{equation}
Using \eqref{Eq:ExpKirchhoff1} in Proposition~\ref{Pro:ExpKirchhoff},  \eqref{Eq:EstUnat}, the boundedness of \eqref{Eq:up-unat} and \eqref{Eq:EstUnatprime}, we find
\begin{multline} \label{Eq:I2J4B}
\frac{1}{2} \int_{\partial {\mathcal S}_{\kappa}} \overline{u}_\kappa^{pot} \cdot \nabla \varphi_{\kappa,j} ((u^{pot} - \overline{u}_\kappa^{pot}) \cdot n) \, ds 
= N_{\kappa,j}+ {\mathcal O}(\varepsilon_{\kappa}^{2+\delta_{j3}} ( 1+ |\widehat{p}_\kappa|)) \\
\text{ where } \ N_{\kappa,j}:=  \frac{1}{2} \int_{\partial {\mathcal S}_{\kappa}} \widehat{u}_\kappa^{pot} \cdot \nabla \widehat \varphi_{\kappa,j} ((u^{pot} - \overline{u}_\kappa^{pot}) \cdot n) \, ds.
\end{multline}
Of course the last term in the right-hand side of \eqref{Eq:I2J4B} is  weakly nonlinear.\par
\ \par
\noindent
{\bf Step 3.}
Hence it remains to consider the term ${N}_{\kappa,j}$. Using \eqref{Eq:up-unat} and applying Lemma~\ref{Lem:Lamb} to ${N}_{\kappa,j}$  we deduce that   
\begin{equation}
\label{Eq:I2J4J6a}
{N}_{\kappa,j}  = \frac{1}{2} \sum_{\ell=1}^2 \int_{\partial {\mathcal S}_{\kappa}}  \left( \alpha_{\kappa,\ell} + \beta_{\kappa,\ell}\right)  \xi_{\ell} 
 \cdot \left((\widehat{u}_\kappa^{pot} \cdot n) \nabla \widehat \varphi_{\kappa,j} + (\nabla \widehat\varphi_{\kappa,j} \cdot n) \widehat{u}_\kappa^{pot} \right) \, ds \\
 =  \widehat{N}_{\kappa,j} + C_{\kappa,j}^3,
\end{equation}
\vskip -0.4cm
\begin{multline*}
\text{ where } \ 
\widehat{N}_{\kappa,j} :=  \sum_{\ell=1}^2 \int_{\partial {\mathcal S}_{\kappa}}  \left( \alpha_{\kappa,\ell} + \beta_{\kappa,\ell}\right)  \xi_{\ell} \cdot \widehat{u}_\kappa^{pot} K_{\kappa,j} \, ds \\
\text{ and } \ 
C_{\kappa,j}^3 := \frac{1}{2} \sum_{\ell=1}^2 \int_{\partial {\mathcal S}_{\kappa}}  \left( \alpha_{\kappa,\ell} + \beta_{\kappa,\ell}\right)  \xi_{\ell} 
\cdot \left((\widehat{u}_\kappa^{pot} \cdot n) \nabla \widehat \varphi_{\kappa,j} - (\nabla \widehat\varphi_{\kappa,j} \cdot n) \widehat{u}_\kappa^{pot} \right) \, ds,
\end{multline*}
As before, we see that $C_{\kappa}^3=(C^3_{\kappa,1},C^3_{\kappa,2},C^3_{\kappa,3})^T$ is gyroscopic of lower order, and we are left with the term $\widehat{N}_{\kappa,j}$. 
We recombine $\widehat{N}_{\kappa,j}$ with $J^6_{\kappa,j}=Q_{\kappa,j}(\overline{u}_{\kappa}^{pot}, \overline{u}^{ext}_\kappa)$ as follows:
\begin{equation} \label{Eq:HatN}
\widehat{N}_{\kappa,j} - J^6_{\kappa,j} =  \int_{\partial {\mathcal S}_{\kappa}} 
\left(\widehat{u}_\kappa^{pot}  \cdot \left[-\overline{u}^{ext}_\kappa + \sum_{\ell=1}^2 (\alpha_{\kappa,\ell} + \beta_{\kappa,\ell}) \xi_{\ell}\right] \right) K_{\kappa,j}\, ds + Q_{\kappa,j}(\overline{u}^{ext}_\kappa,\widehat{u}_\kappa^{pot}  - \overline{u}_\kappa^{pot}).
\end{equation}
By \eqref{Eq:TildeUExt}, \eqref{Eq:BoundAlphaBeta} and Lemma~\ref{Lem:EstExtField}, $\| \overline{u}^{ext}_\kappa \|_{\infty} \leq C$.
Hence as before, with \eqref{Eq:ExpKirchhoff1} we can estimate the last term in \eqref{Eq:HatN} by ${\mathcal O}(\varepsilon_{\kappa}^{2+\delta_{j3}} ( 1 + |\widehat{p}_{\kappa}|))$.
Concerning the first term in \eqref{Eq:HatN}, using \eqref{ApproxTildeUeps}, \eqref{Eq:DecompV} and \eqref{Eq:ParamModulation} we find
\begin{equation*}
\left[-\overline{u}^{ext}_\kappa + \sum_{\ell=1}^2 (\alpha_{\kappa,\ell} + \beta_{\kappa,\ell}) \xi_{\ell}\right] = 
\sum_{\ell=1}^2 \beta_{\kappa,\ell} (\xi_{\kappa,\ell} - \nabla \varphi_{\kappa,\ell})
- \sum_{\ell=4}^5 V_{\kappa,\ell} (\xi_{\kappa,\ell} - \nabla \varphi_{\kappa,\ell}) 
- \varepsilon_{\kappa}^2 u^{r}_{\kappa}   \quad  \text{ on } \partial {\mathcal S}_{\kappa} .
\end{equation*}
Since $\beta_{\kappa,\ell}= {\mathcal O}(\varepsilon_{\kappa})$ for $\ell=1,2$ and $\| \xi_{\kappa,\ell} \|_{L^{\infty}(\partial S_{\kappa})}= {\mathcal O}(\varepsilon_{\kappa})$
for $\ell=4,5$, using  \eqref{Eq:ExpKirchhoff1Bis} and \eqref{Eq:EstApprox1} we see that these terms are all (at least) of order ${\mathcal O}(\varepsilon_{\kappa})$ in $L^{\infty}$ norm on $\partial {\mathcal S}_{\kappa}$. Using $|\partial {\mathcal S}_{\kappa}|= {\mathcal O}(\varepsilon_{\kappa})$ and \eqref{Eq:EstUnatprime}, this gives the estimate
\begin{equation*}
\widehat{N}_{\kappa,j} - J^6_{\kappa,j}  = {\mathcal O}(\varepsilon_{\kappa}^{2+\delta_{j3}} ( 1 + |\widehat{\mathbf{p}}|)).
\end{equation*}
Going back to \eqref{Eq:UpotUpotbar} and \eqref{Eq:I2J4B} and taking into account the above treatment of \eqref{Eq:I2J4J6a}, we deduce that
\begin{equation} \label{Eq:J6C3}
\frac{1}{2} \int_{\partial {\mathcal F}(t)} \overline{u}_\kappa^{pot} \cdot \nabla \varphi_{\kappa,j} ((u^{pot} - \overline{u}_\kappa^{pot}) \cdot n) \, ds 
= J^6_{\kappa,j} + C^3_{\kappa,j} +{\mathcal O}(\varepsilon_{\kappa}^{2+\delta_{j3}} ( 1 + |\widehat{\mathbf{p}}|)).
\end{equation}
Of course the last term in \eqref{Eq:J6C3} is weakly nonlinear. Then injecting \eqref{Eq:J6C3} in \eqref{Eq:I2J4} we obtain the desired result.
\end{proof}
\subsection{Conclusion of the proof of the normal form} 
Gathering \eqref{Def:IJ2}, Lemmas \ref{Lem:L}, \ref{Lem:J1},  \ref{Lem:I1J3}, \ref{Lem:I3J5bis}, \ref{Lem:I3J5}, \ref{Lem:5.3} and \ref{Lem:I2J4J6} we conclude the proof of  Proposition~\ref{Pro:MEE}.
\end{proof}
%
%
%
%
%
%
%
%
\section{Modulated energy estimates}
\label{Sec:MEE}
This section is devoted to the following crucial \textit{a priori} estimate. 
\begin{Proposition} \label{Pro:Bounded-P}
Let $\delta>0$. There exists $\varepsilon_{0}>0$ such that for all $\kappa$, $\widehat{p}_{\kappa}$ is bounded as long as $(\overline{\boldsymbol\varepsilon},{\bf q},\omega)$ stays in $\mathfrak{Q}_\delta^{\varepsilon_{0}}$.
\end{Proposition}
\begin{proof}[Proof of Proposition~\ref{Pro:Bounded-P}]
We only consider $\kappa \in {\mathcal P}_{(iii)}$, since the boundedness of $\widehat{p}_{\kappa}$ was already obtained for $\kappa \in {\mathcal P}_{(i)} \cup {\mathcal P}_{(ii)}$, see Proposition~\ref{Pro:APEE}. Now we consider \eqref{Eq:FormeNormale} and multiply it by $\overline{p}_{\kappa}$: using \eqref{Eq:VariableModulee}, we find, as long as $(\overline{\boldsymbol\varepsilon},q,\omega) \in \mathfrak{Q}_\delta^{\varepsilon_{0}}$:
\begin{equation} \nonumber 
\left({\mathcal M}_{\kappa} \overline{p}_{\kappa}' + \frac{1}{2} {\mathcal M}'_{\kappa} \overline{p}_{\kappa} \right) \cdot \overline{p}_{\kappa} = A_{\kappa} \cdot \overline{p}_{\kappa} + B_{\kappa} \cdot \overline{p}_{\kappa} 
+ C_{\kappa} \cdot \overline{p}_{\kappa} 
+ D_{\kappa} \cdot \overline{p}_{\kappa}
- {\mathcal M}_{g,\kappa} 
\mathscr{V}_{\kappa} '  \cdot \overline{p}_{\kappa} ,
\end{equation}
where $\mathscr{V}_{\kappa} := (\alpha_{\kappa,1} + \beta_{\kappa,1}, \alpha_{\kappa,2} + \beta_{\kappa,2} , 0 )^T$.
We observe that the left-hand side equals $\frac{1}{2} \left({\mathcal M}_{\kappa} \overline{p}_{\kappa} \cdot \overline{p}_{\kappa} \right)'$ and that the second and third terms in the right-hand side vanish, see \eqref{Eq:GyroF} and  \eqref{Eq:GyroFB}. 
Concerning the last term, we use \eqref{Eq:VariableModulee}, \eqref{Eq:EstApprox1}-\eqref{Eq:EstApprox2B} (recalling that $\alpha_{\kappa,i}$ and $\beta_{\kappa,i}$ are given by  \eqref{Eq:ParamModulation}) and \eqref{Eq:Family_iii}; we find
\begin{equation*}
\left| {\mathcal M}_{g,\kappa} \mathscr{V}_{\kappa}'  \cdot \overline{p}_{\kappa} \right|
\leq C \sum_{j=1}^2 \varepsilon_{\kappa}^{\alpha_{\kappa}} |\overline{p}_{\kappa,j}| (1+|\widehat{\bf p}|)
\end{equation*}
Integrating over time and using \eqref{Eq:WNL} and \eqref{Eq:WG} we deduce
\begin{equation} \label{Eq:MEE1}
\big| {\mathcal M}_{\kappa} \overline{p}_{\kappa} \cdot \overline{p}_{\kappa} (t) - {\mathcal M}_{\kappa} \overline{p}_{\kappa} \cdot \overline{p}_{\kappa}(0) \big|
\leq C  \int_{0}^t  \sum_{j=1}^3 \varepsilon_{\kappa}^{\min(2,\alpha_{\kappa}) +\delta_{j3}} |\overline{p}_{\kappa,j}| ( 1 + |\widehat{\boldsymbol p}|  ) 
+ K \varepsilon_{\kappa}^2 \left( 1 + t + \int_{0}^t |\widehat{p}_{\kappa}|^2 \right).
\end{equation}
Now we introduce the slight variant $\widetilde{p}_{\kappa}$ of the modulated variable:
\begin{equation} \nonumber 
\widetilde{p}_{\kappa,i} = \widehat{p}_{\kappa,i} - \delta_{i \in \{1,2\}} (\alpha_{\kappa,i} + \beta_{\kappa, i} ).
\end{equation}
The only difference between $\widetilde{p}_{\kappa}$ and $\overline{p}_{\kappa}$ lies in the third coordinate $i=3$: $\widetilde{p}_{\kappa,i}=\varepsilon_{\kappa} \vartheta'_{k}$ while $\overline{p}_{\kappa,i}=\vartheta'_{k}$. In particular
\begin{equation*}
\sum_{j=1}^3 \varepsilon_{\kappa}^{\delta_{j3}} |\overline{p}_{\kappa,j}|  = \sum_{j=1}^3 |\widetilde{p}_{\kappa,j}|.
\end{equation*}
Next we introduce the $3 \times 3$ matrix ${\mathcal M}^*_{\kappa}$ whose entries are given by ${\mathcal M}^*_{\kappa,ij} = \varepsilon_{\kappa}^{-\min(2,\alpha_{\kappa})-\delta_{i3}-\delta_{j3}} {\mathcal M}_{\kappa,ij}$ for $i,j=1,2,3$. We have
\begin{equation*}
{\mathcal M}_{\kappa} \overline{p}_{\kappa} \cdot \overline{p}_{\kappa} = \varepsilon_{\kappa}^{\min(2,\alpha_{\kappa})} {\mathcal M}^*_{\kappa} \widetilde{p}_{\kappa} \cdot \widetilde{p}_{\kappa}.
\end{equation*}
Hence using $\widetilde{p}_\kappa$ and ${\mathcal M}^*_{\kappa}$, \eqref{Eq:MEE1} allows to write, with $\varepsilon_{\kappa}^2 \leq \varepsilon_{\kappa}^{\min(2,\alpha_{\kappa})}$:
\begin{equation*}
\left| {\mathcal M}^*_{\kappa} \widetilde{p}_{\kappa} \cdot \widetilde{p}_{\kappa} (t)
-{\mathcal M}^*_{\kappa} \widetilde{p}_{\kappa} \cdot \widetilde{p}_{\kappa} (0) \right|
\leq C \left[ \int_{0}^t (1 + |\widehat{\mathbf{p}}|) |\widetilde{p}_{\kappa}| + \left( 1 + t + \int_{0}^t |\widehat{p}_{\kappa}|^2 \right)\right].
\end{equation*}
Now there are two cases:
\begin{itemize}
\item If $\alpha_{\kappa} >2$, then relying on the added mass one has, using Corollary~\ref{Cor:ExpAddedMass} and Remark~\ref{Rem:PasBoules}, that $| ({\mathcal M}^*_{\kappa})^{-1} | \leq C$  independently of $\overline{\boldsymbol\varepsilon}$ and $t$.
\item If $\alpha_{\kappa} \leq 2$, then we rely on the genuine mass matrix and conclude as well that $| ({\mathcal M}^*_{\kappa})^{-1} | \leq C$  independently of $ \overline{\boldsymbol\varepsilon}$ and $t$.
\end{itemize}
Consequently in both cases we can invert by ${\mathcal M}^*_{\kappa}$ and reach for all $\kappa \in {\mathcal P}_{(iii)}$:
\begin{equation*}
|\widetilde{p}_{\kappa}|^2 (t)
\leq C  \int_{0}^t (1 + |\widehat{\mathbf{p}}|) |\widetilde{p}_{\kappa}| 
+ K  \left( 1 + t + \int_{0}^t |\widehat{p}_{\kappa}|^2 \right) + C  |\widetilde{p}_{\kappa}|^2 (0) .
\end{equation*}
From \eqref{Eq:BoundAlphaBeta}, we see that $|\widehat{p}_{\kappa}| \leq C(1+ |\widetilde{p}_{\kappa}| )$ and $|\widetilde{p}_{\kappa}| \leq C(1+ |\widehat{p}_{\kappa}| )$. We sum over $\kappa \in {\mathcal P}_{(iii)}$ and use that we already have a bound on $\widehat{p}_{\kappa}$ for $\kappa \in {\mathcal P}_{(i)} \cup {\mathcal P}_{(ii)}$. We deduce that for some constant $K$ depending only on the geometry, $\delta$ and the initial condition, one has:
\begin{equation*}
|\widehat{\mathbf{p}}|^2 (t)
\leq 
K  \left( 1 + t + \int_{0}^t |\widehat{\mathbf{p}}|^2 \right) .
\end{equation*}
We conclude by Gronwall's lemma (which we can apply on any time-interval for which $(\overline{\boldsymbol\varepsilon},{\bf q},\omega) \in \mathfrak{Q}_\delta^{\varepsilon_{0}}$).
\end{proof}
%
%
%
%
%
%
%
%
%
%
%
%
%
%
%
\section{Passage to the limit}
\label{Sec:PTTL}
%
%
%
%
%
%
%
%
\subsection{A change of variable}
\label{Subsec:Diffeos}
A difficulty to prove the convergences is the dependence of the domain on $\overline{\boldsymbol{\varepsilon}}$. This dependence is twofold: first it depends directly on $\overline{\boldsymbol{\varepsilon}}$ because the small solids occupy a zone depending on this parameter; and then it depends on $\overline{\boldsymbol{\varepsilon}}$ because the solution does, and all solids whether small or of fixed size are located according to the variable ${\bf q}^\varepsilon$. We can temper the difficulty associated with the second dependence by using an adequate family of diffeomorphisms which we now describe. It will not solve the first difficulty but will help with the second one; in particular it will allow to be more precise on the convergences in the neighborhoods of large solids. \par
First, we define the following partial set of coordinates for the solids:
\begin{equation} \label{Eq:DefNewQ}
\underline{{\bf q}}^\varepsilon := (q_{1},\dots,q_{N_{(i)}},h_{N_{(i)}+1},\dots,h_{N}).
\end{equation}
This corresponds to the coordinates in which we will actually pass to the limit. Given $\delta>0$, we introduce the following configuration space $\underline{Q}_{\delta}$:
\begin{multline} \nonumber
\underline{\mathcal Q}_{\delta} := \{ \underline{{\bf q}} \in \R^{3N_{(i)} + 2N_s}  \ : \
\forall \kappa,\lambda \in  {\mathcal P}_{(i)} , \kappa \neq \lambda, \ 
\forall \mu,\nu \in  {\mathcal P}_{(s)} , \mu \neq \nu, \\
\  d({\mathcal S}_\kappa (\underline{{\bf q}}),{\mathcal S}_\lambda (\underline{{\bf q}})) > 2\delta, \ 
|h_{\mu} - h_{\nu}| > \delta, \ 
d(\mathcal S_\kappa (\underline{{\bf q}}), h_{\nu})> 2\delta, 
d(\mathcal S_\kappa (\underline{{\bf q}}),\partial\Omega)> 2\delta 
\text{ and } d(h_\nu,\partial\Omega)> 2\delta \},
\end{multline}
with the obvious abuse of notation for $\partial {\mathcal S}_{\kappa}(\underline{{\bf q}})$. We denote $\underline{{\bf q}}_{0}$ the initial value of $\underline{{\bf q}}^{\varepsilon}$ (which does not depend on $\varepsilon$). We have the following statement.
\begin{Lemma} \label{Lem:Diffeos}
There exist a  neighborhood ${\mathcal U}_{\underline{{\bf q}}_{0}}$ of $\underline{{\bf q}}_{0}$ in $\underline{\mathcal Q}_{\delta}$
and  a smooth mapping ${\mathcal T}: \underline{{\bf q}} \mapsto {\mathcal T}_{\underline{{\bf q}}}$ from $ {\mathcal U}_{\underline{{\bf q}}_{0}}$ into the group $\mbox{Diff}(\Omega)$ of the diffeomorphims of $\Omega$, independent of $\boldsymbol{\varepsilon}$ (provided that $\overline{\boldsymbol{\varepsilon}}$ is small enough), and such that ${\mathcal T}_{\underline{{\bf q}}_{0}} = \mbox{Id}_{\Omega}$, such that 
for all $\underline{{\bf q}} \in {\mathcal U}_{\underline{{\bf q}}_{0}}$, ${\mathcal T}_{\underline{{\bf q}}}$ is an orientation and area-preserving diffeomorphism of $\Omega$, which sends  $S_\kappa ({{\bf q}}_0)$  to $S_\kappa (\underline{{\bf q}})$ for  $\kappa \in {\mathcal P}_{(i)}$, $h_{\kappa,0}$ to $h_{\kappa}$ for $\kappa \in {\mathcal P}_{s}$ and such that for all $\underline{{\bf q}} \in {\mathcal U}_{\underline{{\bf q}}_{0}}$, ${\mathcal T}_{\underline{{\bf q}}}$ is rigid in a neighborhood of each ${\mathcal S}_{\kappa}(\underline{{\bf q}}_{0})$ for $\kappa \in {\mathcal P}_{(i)}$, is a translation in a neighborhood of  $h_{\kappa,0}$
 for $\kappa \in {\mathcal P}_{s}$ and is equal to identity in a neighborhood of $\partial \Omega$.
\end{Lemma}
\begin{proof}[Proof of Lemma~\ref{Lem:Diffeos}]
The construction of such a mapping is easy and classical. We first introduce ${\mathcal W}_{\kappa}$ as the $\delta$-neighborhood of ${\mathcal S}_{\kappa}$ for $\kappa \in {\mathcal P}_{(i)}$ and of $h_{\kappa}$ for $\kappa \in {\mathcal P}_{s}$. Given $\underline{{\bf q}}$ close to $\underline{{\bf q}}_{0}$, we define  ${\mathcal T}_{\underline{{\bf q}}}$ in ${\mathcal W}_{\kappa}$ as the unique rigid movement sending $q_{\kappa}^0$ to $\underline{q}_{\kappa}$ for $\kappa \in {\mathcal P}_{(i)}$,
as the unique translation sending $h_{\kappa}^0$ to $h_{\kappa}$ for $\kappa \in {\mathcal P}_{s}$ and as the identity in a neighborhood of $\partial \Omega$. Then we extend ${\mathcal T}_{\underline{{\bf q}}}$ as a global diffeomorphim on $\Omega$: it suffices to write ${\mathcal T}_{\underline{{\bf q}}}$ in ${\mathcal W}_{\kappa}$ as the flow of a vector field as in Paragraph~\ref{SSSec:ShapeDeriv} and to use extensions of vector fields. To make sure to conserve the zero-divergence of these vector fields, we extend their stream functions. 
\end{proof}
%
%
%
%
%
%
%
\subsection{First step and compactness}
\label{Subsec:FirstStep}
\subsubsection{Fixing $\varepsilon_{0}$ and $\overline{T}$.}
Given an initial data $({\boldsymbol\gamma},{\bf q}_{0},{\bf p}_{0},\omega_{0})$ we first set
 (having \eqref{Eq:ToutDansOmega} in mind):
\begin{multline} \label{Eq:DefD}
D:=\min \big\{ D_\varepsilon, \ \overline{\boldsymbol\varepsilon} \in (0,1]^{N_s} \big\}, \\
\text{where } \ D_\varepsilon := \min \Big\{
\min \{ \dist({\mathcal S}^\varepsilon_{\lambda,0},{\mathcal S}^\varepsilon_{\mu,0}), \ \lambda \neq \mu \}, \ 
\min \{ \dist({\mathcal S}^\varepsilon_{\lambda,0},\partial \Omega), \ \lambda =1, \dots, N \}, \\
\min \{ \dist({\mathcal S}^\varepsilon_{\lambda,0},\Supp(\omega_{0})), \ \lambda =1, \dots, N \}
\Big\},
\end{multline}
and we observe that $D>0$.
Next we set
\begin{equation} \nonumber 
\delta:=\frac{D}{2},
\end{equation}
and apply Proposition~\ref{Pro:Bounded-P} with this $\delta$. We deduce some $\varepsilon_{0}>0$ and some $C_{1}>0$ such that, as long as $(\overline{\boldsymbol\varepsilon},{\bf q},\omega)$ stays in $\mathfrak{Q}_\delta^{\varepsilon_{0}}$, one has 
\begin{equation} \nonumber 
\forall \kappa \in \{1, \dots,N \}, \ \ \ |\widehat{p}_{\kappa}| \leq C_{1}.
\end{equation}
We reduce if necessary $\varepsilon_{0}>0$ so that all intermediate results from Sections \ref{Sec:Expansions} to \ref{Sec:MEE} and Subsection~\ref{Subsec:Diffeos} hold as well. \par

We deduce from the existence of $C_{1}$ the existence of $C_{2}>0$ such that as long as $(\overline{\boldsymbol\varepsilon},{\bf q},\omega)$ stays in $\mathfrak{Q}_\delta^{\varepsilon_{0}}$, one has 
\begin{gather}
\label{Eq:Borne2}
\forall \kappa \in \{1, \dots,N \}, \ \  |v_{{\mathcal S},\kappa}| \leq C_{2}   \text{ in } {\mathcal S}_{\kappa} , \\
\label{Eq:Borne3}
|u^\varepsilon(t,x)| \leq C_{2} \ \text{ on } \ {\mathcal F}_{\delta}({\bf q}(t)) := \left\{ x \in {\mathcal F}({\bf q}) \ \Big/ \ d\left(x,\bigcup_{\kappa \in {\mathcal P}_{s}} {\mathcal S}_{\kappa}\right) > \delta \right\}.
\end{gather}
To get \eqref{Eq:Borne3}, we used the decomposition \eqref{Eq:DecompGlob} and Proposition~\ref{Pro:ExpKirchhoff}, Lemma~\ref{Lem:StandaloneCSF} and Lemma~\ref{Lem:UextGlob} to estimate the three terms in this decomposition. We let
\begin{equation}\label{Eq:DefOverlines}
\overline{C}:=\max(C_{1},C_{2}) \ \text{ and } \ \overline{T}:= \frac{D}{8\overline{C}}.
\end{equation}
Then using a continuous induction argument, we see as a consequence of \eqref{Eq:Vorticite} and the fact that the solids move with velocity $v_{{\mathcal S},\kappa}$ that, provided that $\overline{\varepsilon} \leq \varepsilon_{0}$, one has $(\overline{\boldsymbol\varepsilon},{\bf q},\omega)$ belongs to $\mathfrak{Q}_\delta^{\varepsilon_{0}}$ for all $t \in [0,\overline{T}]$, and in particular all the above \textit{a priori} estimates are true on $[0,\overline{T}]$. \par
In the sequel, reducing $\overline{T}$ if necessary, we may ask that for all $t \in [0,\overline{T}]$, $\underline{q}^\varepsilon(t) \in {\mathcal U}_{\underline{{\bf q}}_{0}}$, where the neighborhood ${\mathcal U}_{\underline{{\bf q}}_{0}}$ was defined in Lemma~\ref{Lem:Diffeos}. 
\subsubsection{Using compactness}
As a consequence of the \textit{a priori} estimates given in Lemma~\ref{Lem:Vorticite} and Propositions~\ref{Pro:APEE} and \ref{Pro:Bounded-P}, we have that $\widehat{p}_{\kappa}^\varepsilon$ is bounded in $W^{2,\infty}(0,\overline{T})$ for $\kappa \in {\mathcal P}_{(i)} \cup {\mathcal P}_{(ii)}$ and in $W^{1,\infty}(0,\overline{T})$ for $\kappa \in {\mathcal P}_{(iii)}$, and that $\omega^\varepsilon$ is bounded in $L^\infty((0,\overline{T}) \times \Omega)$.
Hence we may extract a subsequence (that we abusively still denote by an exponent $\varepsilon$) such that
\begin{gather}
\label{Eq:CVQ1}
q^\varepsilon_{\kappa} \longrightharpoonup {q}_{\kappa}^\star \ \text{ in } \ W^{2,\infty}(0,\overline{T}) \ \text{weak}-\star \ \text{ for } \ \kappa \in {\mathcal P}_{(i)}, \\
\label{Eq:CVQ2}
h^\varepsilon_{\kappa} \longrightharpoonup h_{\kappa}^\star \ \text{ in } \ W^{2,\infty}(0,\overline{T}) \ \text{weak}-\star \ \text{ for } \ \kappa \in {\mathcal P}_{(ii)}, \\
\label{Eq:CVQ3}
h^\varepsilon_{\kappa} \longrightharpoonup h_{\kappa}^\star \ \text{ in } \ W^{1,\infty}(0,\overline{T}) \ \text{weak}-\star \ \text{ for } \ \kappa \in {\mathcal P}_{(iii)}, \\
\label{Eq:CVOmega}
\omega^\varepsilon \longrightharpoonup \omega^\star \ \text{ in } \ L^{\infty}((0,\overline{T}) \times \Omega) \ \text{weak}-\star.
\end{gather}
The fact that we can improve the convergence \eqref{Eq:CVOmega} to the convergence
\begin{equation} \label{Eq:CVOmega2}
\omega^\varepsilon \longrightarrow \omega^\star \ \text{ in } C^{0}([0,\overline{T}]; L^\infty(\Omega)-w\star),
\end{equation}
is obtained as in \cite[Appendix C]{Lions}: this comes from the fact that, thanks to \eqref{Num41}, we have an \textit{a priori} bound on $\partial_{t} \omega^\varepsilon = -\div(u^\varepsilon \omega^\varepsilon)$ in $L^\infty(0,T;W^{-1,p}(\Omega))$. \par
Note in particular that the convergences \eqref{Eq:CVomega}, \eqref{Eq:CVh} and \eqref{Eq:CVtheta} are contained in the above convergences.
Moreover convergences \eqref{Eq:CVQ1} and \eqref{Eq:CVQ2} have naturally the following consequence:
\begin{multline}
\label{Eq:CVP}
p^\varepsilon_{\kappa} \longrightharpoonup {p}_{\kappa}^\star = ({q}_{\kappa}^\star)' \ \text{ in } \ W^{1,\infty}(0,\overline{T}) \ \text{weak}-\star \ \text{ for } \ \kappa \in {\mathcal P}_{(i)}, \\
(h^\varepsilon_{\kappa})' \longrightharpoonup  (h_{\kappa}^\star)' \ \text{ in } \ W^{1,\infty}(0,\overline{T}) \ \text{weak}-\star \ \text{ for } \ \kappa \in {\mathcal P}_{(ii)}
\text{ and in } \ L^{\infty}(0,\overline{T}) \ \text{weak}-\star \ \text{ for } \ \kappa \in {\mathcal P}_{(iii)}.
\end{multline}
%
%

%
%
%
%
%
%
\subsection{Limit dynamics of the fluid}
\label{Subsec:LDF}
Let us now see how the convergences above  involve  the convergence \eqref{Eq:CVu} of the velocity field $u^\varepsilon$
to $u^\star$ satisfying \eqref{Eq:DefUStar}. We recall that we take the convention to extend all the vector fields by $0$ inside the solids. 
The family of diffeomorphisms in Subsection~\ref{Subsec:Diffeos} will be helpful here. We denote
\begin{equation*}
\underline{{\bf q}}^\star :=(q^\star_{1},\dots,q^\star_{N_{(i)}},h^\star_{N_{(i)}+1}, \dots,h^\star_{N}).
\end{equation*}
To obtain the convergence of $u^\varepsilon$ we rely on the decomposition \eqref{Eq:DecompUeps} and show that each term converges towards its final counterpart \eqref{Eq:DecompUstar}. This is done in three separate lemmas.
\begin{Lemma} \label{Lem:CVK}
As $\overline{\boldsymbol{\varepsilon}} \rightarrow 0$  for $p \in [1,2)$ :
\begin{equation} \nonumber 
K^\varepsilon_{{\bf q}^\varepsilon}[\omega^\varepsilon] \circ {\mathcal T}_{\underline{{\bf q}}^\varepsilon} \longrightarrow \widecheck{K}_{{\bf q}_{(i)}^\star}[\omega^\star] \circ {\mathcal T}_{\underline{{\bf q}}^\star} \text{ in } C^0 ( \lbrack 0,\overline{T}\rbrack; L^p(\widecheck{{\mathcal F}}({\bf q}_0))),
\end{equation}
where ${\bf q}^\star_{(i)}:=(q^\star_{1},\dots,q^\star_{N_{(i)}})$.
\end{Lemma}
\begin{Lemma} \label{Lem:CVKir}
Let $p<+\infty$. As $\overline{\boldsymbol{\varepsilon}} \rightarrow 0$: 
\begin{gather*} 
p_{\nu,i} \nabla \varphi^\varepsilon_{\nu,i}({\bf q}^\varepsilon,{\mathcal T}_{\underline{{\bf q}}^\varepsilon}(\cdot))
\longrightarrow
p^\star_{\nu,i} \nabla \widecheck{\varphi}_{\nu,i}({\bf q}^\star_{\nu},{\mathcal T}_{\underline{{\bf q}}^\star}(\cdot)) \text{ in } L^{\infty}(0,\overline{T}; L^p(\widecheck{{\mathcal F}}({\bf q}_0)))
 \  \text{ for } \ \nu \in {\mathcal P}_{(i)} ,  \\ \nonumber 
p_{\nu,i} \nabla \varphi^\varepsilon_{\nu,i}({\bf q}^\varepsilon,{\mathcal T}_{\underline{{\bf q}}^\varepsilon}(\cdot))
\longrightarrow
0 \ \text{ in }  L^{\infty} (0,\overline{T}; L^p(\widecheck{{\mathcal F}}({\bf q}_0)))
  \text{ for } \nu \in {\mathcal P}_{(ii)}  \text{ and in }   L^{\infty}_{w\star}(0,\overline{T}; L^p(\widecheck{{\mathcal F}}(q_0)))  \text{ for } \nu \in {\mathcal P}_{(iii)} .
\end{gather*}
\end{Lemma}
\begin{Lemma} \label{Lem:CVPsi}
As $\overline{\boldsymbol{\varepsilon}} \rightarrow 0$: for $\nu \in {\mathcal P}_{(i)}$:
\begin{equation} \nonumber 
 \nabla^{\perp} {\psi}^\varepsilon_{\nu} ({\bf q}^\varepsilon,{\mathcal T}_{\underline{{\bf q}}^\varepsilon}(\cdot))
\longrightarrow
\nabla^\perp \widecheck{\psi}_{\nu} ({\bf q}^\star_{\nu},{\mathcal T}_{\underline{{\bf q}}^\star}(\cdot))  \ \text{ in } \ L^{\infty}(0,\overline{T}; L^p(\widecheck{{\mathcal F}}({\bf q}_0))) \text{ for } p <+\infty,
\end{equation}
and for $\nu \in {\mathcal P}_{s}$:
\begin{equation} \nonumber 
\nabla^{\perp} {\psi}^\varepsilon_{\nu} ({\bf q}^\varepsilon,{\mathcal T}_{\underline{{\bf q}}^\varepsilon}(\cdot))
\longrightarrow
\widecheck{K}_{{\bf q}_{(i)}^\star}[\delta_{h^\star_{\nu}}] \circ {\mathcal T}_{\underline{{\bf q}}^\star} \ \text{ in } \ L^{\infty}(0,\overline{T}; L^p(\widecheck{{\mathcal F}}({\bf q}_0))) \text{ for } p <2.
\end{equation}
\end{Lemma}

\begin{proof}[Proof of Lemma~\ref{Lem:CVK}]
For all $t \in (0,\overline{T})$ we write, using the triangle inequality and recalling that all vector fields are filled with $0$ inside the solids,
\begin{multline*}
\| K^\varepsilon_{{\bf q}^\varepsilon}[\omega^\varepsilon] \circ {\mathcal T}_{\underline{{\bf q}}^\varepsilon} 
- \widecheck{K}_{{\bf q}_{(i)}^\star}[\omega^\star] \circ {\mathcal T}_{\underline{{\bf q}}^\star} \|_{L^p(\Omega)} \leq
\| K^\varepsilon_{{\bf q}^\varepsilon}[\omega^\varepsilon] \circ {\mathcal T}_{\underline{{\bf q}}^\varepsilon} 
- \widecheck{K}_{{\bf q}_{(i)}^\varepsilon}[\omega^\varepsilon] \circ {\mathcal T}_{\underline{{\bf q}}^\varepsilon} \|_{L^p(\widecheck{\mathcal F}({\bf q}_0))} \\
+
\| \widecheck{K}_{{\bf q}_{(i)}^\varepsilon}[\omega^\varepsilon] \circ {\mathcal T}_{\underline{{\bf q}}^\varepsilon} - \widecheck{K}_{{\bf q}_{(i)}^\star}[\omega^\varepsilon] \circ {\mathcal T}_{\underline{{\bf q}}^\star} \|_{L^p(\Omega)} 
+
\| \widecheck{K}_{{\bf q}_{(i)}^\star}[\omega^\varepsilon - \omega^\star] \circ {\mathcal T}_{\underline{{\bf q}}^\star} \|_{L^p(\widecheck{\mathcal F}({\bf q}_0))}
.
\end{multline*}
For what concerns the first term, since ${\mathcal T}_{\underline{{\bf q}}^\varepsilon}$ is  measure-preserving, we have
\begin{equation*}
\| K^\varepsilon_{{\bf q}^\varepsilon}[\omega^\varepsilon] \circ {\mathcal T}_{\underline{{\bf q}}^\varepsilon} 
- \widecheck{K}_{{\bf q}_{(i)}^\varepsilon}[\omega^\varepsilon] \circ {\mathcal T}_{\underline{{\bf q}}^\varepsilon} \|_{L^p(\widecheck{\mathcal F}({\bf q}_0))}
= 
\| K^\varepsilon_{{\bf q}^\varepsilon}[\omega^\varepsilon]
- \widecheck{K}_{{\bf q}_{(i)}^\varepsilon}[\omega^\varepsilon] \|_{L^p(\widecheck{\mathcal F}({\bf q}^\varepsilon))},
\end{equation*}
which converges to zero uniformly in time thanks to Lemma~\ref{Lem:BiotSavart}.
The convergence of the third term (uniformly in time) comes from \eqref{Eq:CVOmega2}: it involves the convergence of $K_\Omega[\omega^\varepsilon]$ to $K_\Omega[\omega^\star]$ (recall \eqref{Eq:DefKOmega}) in $C^0([0,\overline{T}];L^p(\Omega))$ for $p<+\infty$ due to the classical compactness of the operator $K_{\Omega}:L^p(\Omega) \rightarrow L^p(\Omega)$ (due to the Calderon-Zygmund estimate $\left\| {K}_\Omega[\omega] \right\|_{W^{1,p}(\Omega)} \leq  C \| \omega \|_{L^{p}(\Omega)}$ and the Rellich-Kondrachov theorem.)
Note that using the support of vorticity and interior regularity, this involves the convergence in $C^0([0,\overline{T}];C^k(\mathcal{V}_\delta(\partial \mathcal{S}_\lambda)))$ for each $\lambda=1,\dots,N$.  It remains to check that the correction $R[\omega- \omega^\star]$ defined in \eqref{Eq:DefR-BS} converges to $0$ in $C^0([0,\overline{T}];L^p(\Omega))$. This is again a consequence of Propositions~\ref{Pro:DirichletPetits} and \ref{Pro:DirichletGros}. \par
Finally, concerning the second term, we consider the function
\begin{equation*}
[0,1] \longrightarrow L^p(\widecheck{\mathcal{F}}(\boldsymbol{q}_0)), \ \ 
s \longmapsto \widecheck{K}_{{\bf q}_{(i)}^\star + s({\bf q}_{(i)}^\varepsilon - {\bf q}_{(i)}^\star)}[\omega^\varepsilon] \circ {\mathcal T}_{\underline{{\bf q}}^\star + s(\underline{{\bf q}}^\varepsilon - \underline{{\bf q}}^\star)}. 
\end{equation*}
It is well-defined for small enough $\overline{\boldsymbol{\varepsilon}}$ (due to the convergences \eqref{Eq:CVQ1}-\eqref{Eq:CVQ3}, so that $\underline{{\bf q}}^\star + s(\underline{{\bf q}}^\varepsilon - \underline{{\bf q}}^\star)$ belongs to the neighborhood ${\mathcal U}_{\underline{q}_{0}}$ of Lemma~\ref{Lem:Diffeos}), and its derivative with respect to $s$  is bounded by
\begin{equation} \label{Eq:BorneDeriveeMaterielle}
C |\underline{{\bf q}}^\varepsilon - \underline{{\bf q}}^\star| \left( \left\| \frac{\partial \widecheck{K}}{\partial {\bf q}} \right\|_{L^p(\widecheck{\mathcal F})}
+ \left\| \frac{\partial \widecheck{K}}{\partial x} \right\|_{L^p(\widecheck{\mathcal F})}  \right).
\end{equation}
Together with Lemma \ref{Lem:DBSBorne} and \eqref{Eq:EstKSobolev}, this establishes Lemma~\ref{Lem:CVK}.
\end{proof}
\begin{proof}[Proof of Lemma~\ref{Lem:CVKir}]
Here we write for $\nu \in {\mathcal P}_{(i)}$:
\begin{multline} \label{Eq:CVKir3}
\| \nabla \varphi^\varepsilon_{\nu,i}({\bf q}^\varepsilon,{\mathcal T}_{\underline{{\bf q}}^\varepsilon}(\cdot))
- \nabla \widecheck{\varphi}_{\nu,i}({\bf q}^\star_{\nu},{\mathcal T}_{\underline{{\bf q}}^\star}(\cdot)) \|_{L^\infty(0,T;L^p(\Omega))}
\leq
\| \nabla \varphi^\varepsilon_{\nu,i}({\bf q}^\varepsilon,{\mathcal T}_{\underline{{\bf q}}^\varepsilon}(\cdot))
-\nabla \widecheck{\varphi}_{\nu,i}({\bf q}^\varepsilon,{\mathcal T}_{\underline{{\bf q}}^\varepsilon}(\cdot)) \|_{L^\infty(0,T;L^p(\widecheck{{\mathcal F}}_{0}))} \\
+ \| \nabla \widecheck{\varphi}_{\nu,i}({\bf q}^\varepsilon,{\mathcal T}_{\underline{{\bf q}}^\varepsilon}(\cdot))
- \nabla \widecheck{\varphi}_{\nu,i}({\bf q}^\star_{\nu},{\mathcal T}_{\underline{{\bf q}}^\star}(\cdot)) \|_{L^\infty(0,T;L^p(\widecheck{{\mathcal F}}_{0}))}.
\end{multline}
The first term in the right-hand side converges to zero as shown by Proposition~\ref{Pro:ExpKirchhoff2}. For the second we reason as in the proof of Lemma~\ref{Lem:CVKir}: we consider the function
\begin{equation*}
s \longmapsto \nabla \widecheck{\varphi}_{\nu,i}(q_{\nu}^\varepsilon+ s(q_{\nu}^\varepsilon - q_{\nu}^\star),
{\mathcal T}_{\underline{{\bf q}}^\varepsilon+ s(q_{\nu}^\varepsilon - q_{\nu}^\star)}(\cdot)),
\end{equation*}
where the abusive notation $\underline{{\bf q}}^\varepsilon+ s({\bf q}_{\nu}^\varepsilon - q_{\nu}^\star)$ means that we add $s(q_{\nu}^\varepsilon - q_{\nu}^\star)$ only on the coordinate of $\underline{q}^\varepsilon$ corresponding to $q_{\nu}$. Now we estimate the $s$-derivative as in \eqref{Eq:BorneDeriveeMaterielle}. The $x$-derivative is bounded thanks to the uniform Schauder estimates in $\widecheck{\mathcal F}$, the ${\bf q}$ derivative by following the proof of Proposition~\ref{Pro:ExpShapeDerivatives} by elliptic regularity in $\widecheck{\mathcal F}$. With \eqref{Eq:CVQ1}, this proves the convergence of the left-hand side of \eqref{Eq:CVKir3} to zero. The conclusion follows then from \eqref{Eq:CVP} for solids of family $(i)$. \par \,
\par

 \noindent Concerning small solids, the convergence to $0$ of the Kirchhoff potentials (uniform with respect to ${\bf q}$) comes from Proposition~\ref{Pro:ExpKirchhoff}, and one concludes in the same way with \eqref{Eq:CVP}.
\end{proof}
\begin{proof}[Proof of Lemma~\ref{Lem:CVPsi}]
Here we write for $\nu \in {\mathcal P}_{(i)}$ and all $t \in [0,\overline{T}]$:
\begin{multline*}
\|  \nabla^{\perp} {\psi}^\varepsilon_{\nu} ({\bf q}^\varepsilon,{\mathcal T}_{\underline{{\bf q}}^\varepsilon}(\cdot))
- \nabla^\perp \widecheck{\psi}_{\nu} (q^\star_{\nu},{\mathcal T}_{\underline{{\bf q}}^\star}(\cdot)) \|_{L^p(\widecheck{{\mathcal F}}({\bf q}_0))}
\leq 
\|  \nabla^{\perp} {\psi}^\varepsilon_{\nu} ({\bf q}^\varepsilon,{\mathcal T}_{\underline{{\bf q}}^\varepsilon}(\cdot))
- \nabla^\perp \widecheck{\psi}_{\nu} ({\bf q}^\varepsilon,{\mathcal T}_{\underline{{\bf q}}^\varepsilon}(\cdot)) \|_{L^p(\widecheck{{\mathcal F}}({\bf q}_0))} \\
+
\|  \nabla^\perp \widecheck{\psi}_{\nu} ({\bf q}^\varepsilon,{\mathcal T}_{\underline{{\bf q}}^\varepsilon}(\cdot))
- \nabla^\perp \widecheck{\psi}_{\nu} (q^\star_{\nu},{\mathcal T}_{\underline{{\bf q}}^\star}(\cdot)) \|_{L^p(\widecheck{{\mathcal F}}({\bf q}_0))}.
\end{multline*}
The convergence to zero of the first term in the right-hand side, uniformly in ${\bf q}$  is a consequence of Proposition~\ref{Pro:LimCirculation}.
The convergence of the second term  is due to \eqref{Eq:CVQ1} and the regularity of $\nabla^\perp \widecheck{\psi}_{\nu}$ with respect to ${\bf q}$ (using for instance Lemma \ref{Lem:DWBorne} and \eqref{Eq:HatPsiTourne}). \par
For $\nu \in {\mathcal P}_{s}$, for $p \in [1,2)$ and all $t \in [0,\overline{T}]$ we have:
\begin{multline*}
\| \nabla^{\perp} {\psi}^\varepsilon_{\nu} ({\bf q}^\varepsilon,{\mathcal T}_{\underline{{\bf q}}^\varepsilon}(\cdot))
 - \widecheck{K}_{{\bf q}_{(i)}^\star}[\delta_{h^\varepsilon_{\nu}}] \circ {\mathcal T}_{\underline{{\bf q}}^\star} \|_{L^p(\widecheck{{\mathcal F}}({\bf q}_0)))}
\leq 
\| \nabla^{\perp} {\psi}^\varepsilon_{\nu} ({\bf q}^\varepsilon,{\mathcal T}_{\underline{{\bf q}}^\varepsilon}(\cdot))
- \widecheck{K}_{{\bf q}_{(i)}^\star}[\delta_{h^\varepsilon_{\nu}}] \circ {\mathcal T}_{\underline{{\bf q}}^\star} \|_{L^p(\widecheck{{\mathcal F}}({\bf q}_0)))} \\
+ 
\| \widecheck{K}_{{\bf q}_{(i)}^\star}[\delta_{h^\varepsilon_{\nu}}] \circ {\mathcal T}_{\underline{{\bf q}}^\star}
- \widecheck{K}_{{\bf q}_{(i)}^\star}[\delta_{h^\star_{\nu}}] \circ {\mathcal T}_{\underline{{\bf q}}^\star} \|_{L^p(\widecheck{{\mathcal F}}({\bf q}_0)))}.
\end{multline*}
The convergence to zero of the first term in the right-hand side is due to Proposition~\ref{Pro:LimCirculation}, \eqref{Eq:DecompNatPsi} and \eqref{Eq:Kdelta}.  Concerning the second one, by  \eqref{Eq:Kdelta}  
\begin{multline*}
\| \widecheck{K}_{{\bf q}_{(i)}^\star}[\delta_{h^\varepsilon_{\nu}}] \circ {\mathcal T}_{\underline{{\bf q}}^\star}
- \widecheck{K}_{{\bf q}_{(i)}^\star}[\delta_{h^\star_{\nu}}] \circ {\mathcal T}_{\underline{{\bf q}}^\star} \|_{L^p(\widecheck{{\mathcal F}}({\bf q}_0)))}
\leq
\| \nabla^\perp \widecheck{\psi}_{\kappa}(h^\varepsilon_{\nu})\circ {\mathcal T}_{\underline{{\bf q}}^\star} - \nabla^\perp \widecheck{\psi}_{\kappa}(h^\star_{\nu}) \circ {\mathcal T}_{\underline{{\bf q}}^\star}\|_{L^p(\widecheck{{\mathcal F}}({\bf q}_0)))} \\
+ 
\| H(\cdot - h^\varepsilon_{\nu})\circ {\mathcal T}_{\underline{{\bf q}}^\star} - H(\circ - h^\star_{\nu}) \cdot {\mathcal T}_{\underline{{\bf q}}^\star}\|_{L^p(\widecheck{{\mathcal F}}({\bf q}_0)))}.
\end{multline*}
Due to the uniform convergence of 
$h^\varepsilon_{\kappa} $ to $ h_{\kappa}^\star$ both terms converge to zero, the first one by regularity with respect to $h$ of $\nabla^\perp \widecheck{\psi}_{\kappa}$, the second-one by continuity of the translations in $L^p$.
\end{proof}
Now the convergence \eqref{Eq:CVu} to $u^\star$ satisfying \eqref{Eq:DefUStar} is a direct consequence of Lemmas \ref{Lem:CVK}, \ref{Lem:CVKir}, \ref{Lem:CVPsi}, and of the decompositions \eqref{Eq:DecompUeps}  and  \eqref{Eq:DecompUstar}.
Moreover one obtains \eqref{Eq:EvolOmegaStar} by passing to the limit in \eqref{Eq:Vorticite} thanks to \eqref{Eq:CVOmega} and \eqref{Eq:CVu}.
 \par
%
%
%
%
%
%
%
%
\subsection{Limit dynamics of the solids of fixed size}
To pass to the limit in the equation of the solids of family $(i)$, we must pass to the limit in the pressure. To that purpose, we observe that the convergences described in Subsection~\ref{Subsec:LDF} are actually stronger when one restricts the space domain to the $\delta$-neighborhood of $\partial {\mathcal S}_{\kappa}$ for $\kappa \in {\mathcal P}_{(i)}$, and, for $\kappa \in {\mathcal P}_{s}$, to an annulus $B(h_{\kappa},\delta) \setminus B(h_{\kappa},\delta/2)$.
 This is given in the following statement.
\begin{Lemma} \label{Lem:MeilleureCVUeps}
For $\kappa \in \{1,\dots,N\}$ we let ${\mathcal U}_\kappa^{\delta}$ the $\delta/2$-neighborhood of $\partial {\mathcal S}_{\kappa}({\bf q}_0)$ whenever $\kappa \in {\mathcal P}_{(i)}$ and we let ${\mathcal U}_\kappa^{\delta} = B(h^0_{\kappa},\delta) \setminus B(h^0_{\kappa},3\delta/4)$ whenever $\kappa \in {\mathcal P}_{s}$. Then one has
\begin{equation} \nonumber 
u^\varepsilon \circ {\mathcal T}_{\underline{{\bf q}}^\varepsilon |{\mathcal U}_\kappa^{\delta}}
\longrightarrow
u^\star \circ {\mathcal T}_{\underline{{\bf q}}^\star |{\mathcal U}_\kappa^{\delta}}
\ \text{ in } \ W^{1,\infty}(0,\overline{T};C^k({\mathcal U}_\kappa^{\delta}))-w\star, \ \ \text{ for all } k \in \N.
\end{equation}
\end{Lemma}
\begin{proof}[Proof of Lemma~\ref{Lem:MeilleureCVUeps}]

This is due to the support of $\omega$ and the remoteness of small solids from it (since $(\overline{\boldsymbol\varepsilon},{\bf q},\omega) \in \mathfrak{Q}_\delta^{\varepsilon_{0}}$), which allow to improve the convergences of Lemmas \ref{Lem:CVK}, \ref{Lem:CVKir} and \ref{Lem:CVPsi} to the weak-$\star$ one in $W^{1,\infty}(0,\overline{T};C^k({\mathcal U}_\kappa^{\delta}))$. Since we already have the convergence in a weaker space, it suffices to prove the boundedness of $u^\varepsilon \circ {\mathcal T}_{{\bf q}^\varepsilon}$ in $W^{1,\infty}(0,T;C^k({\mathcal U}_\kappa^{\delta}))$. That $u^\varepsilon \circ {\mathcal T}_{{\bf q}^\varepsilon}$ remains bounded in $L^{\infty}(0,T;C^k({\mathcal U}_\kappa^{\delta}))$ is a direct consequence of the support of $\omega^\varepsilon$ and interior elliptic regularity, since it is already bounded in $L^{\infty}(0,T;L^p({\mathcal F}_{0}))$. \par
For what concerns $\partial_{t} (u^\varepsilon  \circ {\mathcal T}_{\underline{{\bf q}}^\varepsilon})$ we have
\begin{equation*}
\partial_{t} (u^\varepsilon  \circ {\mathcal T}_{\underline{{\bf q}}^\varepsilon}) = \left\{ \begin{array}{l}
\displaystyle [\partial_{t} u^\varepsilon + (v^\varepsilon_{{\mathcal S},\kappa} \cdot \nabla) u^\varepsilon] \circ {\mathcal T}_{\underline{{\bf q}}^\varepsilon},
 \ \text{ in } \ {\mathcal U}_\kappa^{\delta} \ \text{ for } \ \kappa \in {\mathcal P}_{(i)}, \medskip \\
\displaystyle [\partial_{t} u^\varepsilon + ((h^\varepsilon_{\kappa})' \cdot \nabla) u^\varepsilon] \circ {\mathcal T}_{\underline{{\bf q}}^\varepsilon}
 \ \text{ in } \ {\mathcal U}_\kappa^{\delta} \ \text{ for } \ \kappa \in {\mathcal P}_{s},
\end{array} \right.
\end{equation*}
so that we only have to estimate $(\partial_{t} u^\varepsilon) \circ {\mathcal T}_{\underline{{\bf q}}^\varepsilon}$. Again, by interior elliptic estimates, it suffices to bound it in $L^{\infty}$ in a slightly larger set.
We rely on decomposition \eqref{Eq:DecompGlob}:
\begin{itemize}
\item $\partial_t u^{ext}$ is bounded in $C^{0}([0,\overline{T}] \times {\mathcal U}_\kappa^{\delta})$ thanks to Lemma~\ref{Lem:UextGlob},
\item the terms $\partial_t \nabla^{\perp} \widehat{\psi}_{\nu}$ for $\nu \neq \kappa$ are bounded in $C^{0}([0,\overline{T}] \times {\mathcal U}_\kappa^{\delta})$ thanks to \eqref{Eq:HatPsiTourne}, \eqref{Eq:BehaviourPsi1}-\eqref{Eq:BehaviourPsi2} and the remoteness of ${\mathcal U}_\kappa^{\delta}$ from $\partial \mathcal{S}_\nu$,
\item all the same the term $\partial_t \nabla^{\perp} \widehat{\psi}_{\kappa}$ is bounded in $C^{0}([0,\overline{T}] \times {\mathcal U}_\kappa^{\delta})$ thanks to \eqref{Eq:HatPsiTourne},  \eqref{Eq:BehaviourPsi1}-\eqref{Eq:BehaviourPsi2} and to the choice of ${\mathcal U}_\kappa^{\delta}$ (that is at positive distance from $\partial {\mathcal S}_{\kappa}$ when $\kappa \in {\mathcal P}_{s}$), 
\item the boundedness of $\partial_t u^{pot}$ follows from Proposition~\ref{Pro:ExpShapeDerivatives}, acceleration estimates (Proposition~\ref{Pro:Acceleration}) and Proposition~\ref{Pro:ExpKirchhoff} (again thanks to the choice of ${\mathcal U}_\kappa^{\delta}$).
\end{itemize}
\end{proof}
A first consequence of Lemma~\ref{Lem:MeilleureCVUeps} is \eqref{Eq:DefPStar}. Indeed, due to \eqref{Eq:DefUStar} and \eqref{Eq:EvolOmegaStar}, we have
\begin{equation*}
\curl (\partial_{t} u^\star + (u^\star \cdot \nabla) u^\star ) = 0 \ \text{ in } \ \widecheck{\mathcal F}({\bf q}^\star_{(i)}(t)).
\end{equation*}
For each $\kappa \in \{1,\dots,N\}$, we introduce a smooth simple closed loop $\gamma_{\kappa}$ in ${\mathcal U}_\kappa^{\delta}$. Then \eqref{Eq:Euler} involve that for all $t \in [0,\overline{T}]$ and all $\overline{\boldsymbol\varepsilon}$, one has
\begin{equation*}
\oint_{\gamma_{\kappa}} (\partial_{t} u^\varepsilon + (u^\varepsilon \cdot \nabla) u^\varepsilon )(t,\cdot) \cdot \tau \, ds =0.
\end{equation*}
Passing to the limit with Lemma~\ref{Lem:MeilleureCVUeps} we infer  that for all $\kappa \in \{1,\dots,N\}$,
\begin{equation*}
\oint_{\gamma_{\kappa}} (\partial_{t} u^\star + (u^\star \cdot \nabla) u^\star ) \cdot \tau \, ds =0.
\end{equation*}
This establishes \eqref{Eq:DefPStar}. \par
Next we deduce \eqref{Eq:Evolh(i)Star}. It follows  from Lemma~\ref{Lem:MeilleureCVUeps} that in a vicinity of $\partial {\mathcal S}_{\kappa}$  for $\kappa \in {\mathcal P}_{(i)}$, the convergence of the pressure is improved: recalling that
\begin{equation*}
\nabla \pi^\varepsilon = - \partial_{t} u^\varepsilon - (u^\varepsilon \cdot \nabla) u^\varepsilon
\ \text{ and } \ 
\nabla \pi^\star = - \partial_{t} u^\star - (u^\star \cdot \nabla) u^\star,
\end{equation*}
Lemma~\ref{Lem:MeilleureCVUeps} involves that
\begin{equation*}
\nabla \pi^\varepsilon \circ {\mathcal T}_{{\bf q}^\varepsilon} \longrightarrow \nabla \pi^\star \circ {\mathcal T}_{{\bf q}^\star}
\ \text{ in } \ L^{\infty}(0,T;C^k({\mathcal V}_{\delta/2}(\partial {\mathcal S}_{\kappa}))) \text{weak-}\star.
\end{equation*}
From \eqref{Eq:Newton} we deduce, for all $\kappa \in {\mathcal P}_{(i)}$:
\begin{equation*} 
\left\{ \begin{array}{l}
m_{\kappa} (h^\varepsilon_{\kappa})''(t) 
= R(\vartheta^\varepsilon_{\kappa}) \displaystyle \int_{\partial {\mathcal S}_{\kappa}({\bf q}_0)} \pi^\varepsilon(t,{\mathcal T}_{{\bf q}^\varepsilon}(x)) \, n(t,{\mathcal T}_{{\bf q}^\varepsilon}(x)) \, ds(x),
\\
\displaystyle
J_{\kappa} (\vartheta^\varepsilon_{\kappa})''(t)
= \int_{\partial {\mathcal S}_{\kappa}({\bf q}_0)} \pi^\varepsilon(t,{\mathcal T}_{{\bf q}^\varepsilon}(x)) (x-h_{\kappa,0})^\perp \cdot n(t,{\mathcal T}_{{\bf q}^\varepsilon}(x)) \, ds(x).
\end{array} \right. 
\end{equation*}
This involves the passage to the limit in \eqref{Eq:Newton} for the first family, from which we deduce \eqref{Eq:Evolh(i)Star}.
%
%
%
%
%
%
%
%
%
%
\subsection{Limit dynamics of the small solids and end of the proof of Theorem~\ref{Joli_Theoreme}}
\label{Subsec:LimitSmall}
To get the convergence on small solids we go back to the normal form \eqref{Eq:FormeNormale}. 
Let  $\kappa \in {\mathcal P}_{s}$. 
Since we now know that $\widehat{\bf p}^{\varepsilon}$ is bounded, using \eqref{Eq:WNL}, \eqref{Eq:EstGyro} and \eqref{Eq:EstGyroW}, we infer that the terms $A_{\kappa}$, $C_{\kappa}$ and $D_{\kappa}$ converge to zero strongly in $L^{\infty}(0,T)$.

Now we use two lemmas, where we recall that $\overline{p}_{\kappa}$ is the modulated variable (before the passage to the limit) given by \eqref{Eq:VariableModulee}.
\begin{Lemma} \label{Lem:PetitTermeFaibleii&iii}
When $\kappa \in {\mathcal P}_{s}$, the term 
${\mathcal M}_{a,\kappa} \overline{p}_{\kappa}' + \frac{1}{2} {\mathcal M}'_{a,\kappa} \overline{p}_{\kappa} $ 
converges to $0$ in $W^{-1,\infty}(0,\overline{T})$ as $\overline{\boldsymbol{\varepsilon}}$ goes to $0$.
\end{Lemma}
\begin{proof}[Proof of Lemma \ref{Lem:PetitTermeFaibleii&iii}]
We proceed in three steps. \par
\ \par
\noindent
{\bf Step 1.} First ${\mathcal M}_{a,\kappa}$ converges strongly to $0$ in $L^\infty(0,T)$ due to Corollary~\ref{Cor:ExpAddedMass}. Since $\overline{p}_{\kappa}$ is bounded, it follows that $({\mathcal M}_{a,\kappa} \overline{p}_{\kappa})'$ converges to $0$ in $W^{-1,\infty}(0,\overline{T})$. \par
\ \par
\noindent
{\bf Step 2.} By  Reynold's transport theorem:
\begin{eqnarray*}
{\mathcal M}'_{a,\kappa,i,j} 
&=& \sum_{\nu=1}^N \int_{{\mathcal F}({\bf q})} \left(p_{\nu} \cdot \frac{\partial \nabla \varphi_{\kappa,i}}{\partial q_{\nu}}\right) \cdot \nabla \varphi_{\kappa,j} \, dx 
+ \sum_{\nu=1}^N \int_{{\mathcal F}({\bf q})} \nabla \varphi_{\kappa,i} \cdot  \left(p_{\nu} \cdot \frac{\partial \nabla \varphi_{\kappa,j}}{\partial q_{\nu}}\right) \, dx \\
&\ & + \int_{\partial {\mathcal F}({\bf q})} (u^\varepsilon \cdot n) \nabla \varphi_{\kappa,i} \cdot \nabla \varphi_{\kappa,j} \, ds.
\end{eqnarray*}
By an integration by parts the first two terms are transformed into integrals over $\partial {\mathcal S}_{\kappa}$ with some integrands which are bounded according to Proposition~\ref{Pro:ExpShapeDerivatives}. 
Therefore these two terms converge to $0$ uniformly in time. For the third one, we first notice that $u^\varepsilon\cdot n=u^{pot} \cdot n$ is bounded (thanks to Propositions~\ref{Pro:ExpKirchhoff} and \ref{Pro:Bounded-P}). Now using again Proposition~\ref{Pro:ExpKirchhoff} we see that on $\partial {\mathcal F}({\bf q}) \setminus \partial \mathcal{S}_\kappa$ the integrand is of order $\mathcal{O}(\varepsilon_\kappa^{4+\delta_{i3}+\delta_{j3}})$ and that on $\partial \mathcal{S}_\kappa$ it is bounded. Since $|\partial \mathcal{S}_\kappa|=\mathcal{O}(\varepsilon_\kappa)$, we obtain the convergence of this term to $0$ as well. Thus ${\mathcal M}'_{a,\kappa} \overline{p}_{\kappa}$ converges to $0$ in $L^{\infty}(0,T)$ as $\overline{\boldsymbol{\varepsilon}}$ goes to $0$. \par
\ \par
\noindent
{\bf Step 3.} Since
\begin{equation*}
{\mathcal M}_{a,\kappa} \overline{p}_{\kappa}' + \frac{1}{2} {\mathcal M}'_{a,\kappa} \overline{p}_{\kappa}  = 
 ({\mathcal M}_{a,\kappa} \overline{p}_{\kappa})' - \frac{1}{2} {\mathcal M}'_{a,\kappa}\overline{p}_{\kappa} ,
\end{equation*}
the result follows.
\end{proof}
\begin{Lemma} \label{Lem:GrosTermeii&iii}
When $\kappa \in {\mathcal P}_{s}$, one has the uniform convergence in $[0,\overline{T}]$ as $\overline{\boldsymbol{\varepsilon}}$ goes to $0$:
\begin{equation} \nonumber 
\begin{pmatrix}
B_{\kappa,1} \\ B_{\kappa,2} 
\end{pmatrix}
 \longrightarrow \gamma_{\kappa}((h^\star_{\kappa})' - u^{\star}_{\kappa}(h^\star_{\kappa}))^{\perp}.
\end{equation}
\end{Lemma}
\begin{proof}[Proof of Lemma \ref{Lem:GrosTermeii&iii}]
We consider the writing of $B_{\kappa}$ in \eqref{Eq:Gyro}. Using \eqref{Eq:VariableModulee} and \eqref{circ-norma_F}, and then \eqref{Eq:ParamModulation}, \eqref{Eq:BoundAlphaBeta} and \eqref{Eq:DefV}, we see that
\begin{equation*}
\begin{pmatrix}
B_{\kappa,1} \\ B_{\kappa,2} 
\end{pmatrix}
= \gamma_{\kappa} \left( (h^\varepsilon_{\kappa})' - \begin{pmatrix} \alpha_{1} + \beta_{1} \\ \alpha_{2}+ \beta_{2} \end{pmatrix} \right)^{\perp}
= \gamma_{\kappa} \left( (h^\varepsilon_{\kappa})' - \widecheck{u}^\varepsilon_{\kappa}(h^\varepsilon_{\kappa})  \right)^{\perp} + o(1).
\end{equation*}
It remains to prove that
\begin{equation} \label{Eq:LimiteModulation}
\widecheck{u}^\varepsilon_\kappa(h_\kappa^\varepsilon) \longrightarrow u^\star_{\kappa}(h_{\kappa}^\star) \ \text{ uniformly in time as } \overline{\boldsymbol{\varepsilon}} \rightarrow 0.
\end{equation}
To prove \eqref{Eq:LimiteModulation}, we first establish the convergence  for $p \in [1,2)$ 
\begin{equation} \label{Eq:CVUC}
\widecheck{u}_{\kappa}^\varepsilon \circ {\mathcal T}_{\underline{{\bf q}}^\varepsilon} \longrightarrow u^\star_{\kappa} \circ {\mathcal T}_{\underline{{\bf q}}^\star}
\ \text{ in } \ L^\infty(0,\overline{T};L^p(\widecheck{\mathcal F}_{0})).
\end{equation}
This derives from \eqref{Eq:DecompCheckUKappa} and the equivalents of Lemmas~\ref{Lem:CVK}, \ref{Lem:CVKir} and \ref{Lem:CVPsi} in the domain $\widecheck{\mathcal{F}}_\kappa$ 
where there is no ${\mathcal S}_{\kappa}$: 
\begin{gather*}
\nabla \varphi_{\nu}^{\varepsilon,\not \kappa} \circ {\mathcal T}_{\underline{{\bf q}}^\varepsilon} \longrightarrow
\nabla \widecheck{\varphi}_{\nu} \circ {\mathcal T}_{\underline{{\bf q}}^\star}
\text{ for } \nu \in  {\mathcal P}_{(i)} , \\
\nabla \varphi_{\nu}^{\varepsilon,\not \kappa} \circ {\mathcal T}_{\underline{{\bf q}}^\varepsilon} \longrightarrow 0
\text{ for } \nu \in {\mathcal P}_{s} \setminus \{\kappa \}, \\
\nabla^\perp \psi_{\nu}^{\varepsilon,\not \kappa}  \circ {\mathcal T}_{\underline{{\bf q}}^\varepsilon} \longrightarrow
\nabla^\perp \widecheck{\psi}_{\nu} \circ {\mathcal T}_{\underline{{\bf q}}^\star}
\text{ for } \nu \in   {\mathcal P}_{(i)} , \\
\nabla^\perp \psi_{\nu}^{\varepsilon,\not \kappa}  \circ {\mathcal T}_{\underline{{\bf q}}^\varepsilon} \longrightarrow
\widecheck{K}[\delta_{h^\star_{\nu}}] \circ {\mathcal T}_{\underline{{\bf q}}^\star}
\text{ for } \nu \in {\mathcal P}_{s} \setminus \{\kappa \}, \\
K^{\varepsilon,\not \kappa}[\omega^\varepsilon]  \circ {\mathcal T}_{\underline{{\bf q}}^\varepsilon} \longrightarrow
\widecheck{K}[\omega^\star]  \circ {\mathcal T}_{\underline{{\bf q}}^\star}.
\end{gather*}
Moreover using \eqref{Eq:CVPhantom} and reasoning as in Lemma~\ref{Lem:CVPsi}
\begin{equation*}
\nabla^\perp \psi^{\varepsilon,r,\not \kappa}_{\kappa}  \circ {\mathcal T}_{\underline{{\bf q}}^\varepsilon} \longrightarrow
\nabla^\perp \widecheck{\psi}_{\kappa}^{r} \circ {\mathcal T}_{\underline{{\bf q}}^\star}
= \big\{ \widecheck{K}[\delta_{h^\star_{\kappa}}] - H_{\kappa} \big\} \circ {\mathcal T}_{\underline{{\bf q}}^\star},
\end{equation*}
where we recall that $\widecheck{\psi}_{\kappa}^r$ was defined in \eqref{Eq:DefUKRKappa} and $\widecheck{\psi}_{\kappa}^{r, \not \kappa}$ in \eqref{Eq:PsirefKappa}.
This allows to deduce \eqref{Eq:CVUC} using the decomposition \eqref{Eq:DecompCheckUKappa} of $\widecheck{u}_\kappa$.
Then using inner regularity for the Laplace equation, we see that the convergence \eqref{Eq:CVUC} actually holds in $L^\infty(0,\overline{T};C^k({\mathcal V}_\delta (\mathcal S_{\kappa})))$ since there is no vorticity near ${\mathcal S}_{\kappa}^\varepsilon$. With the uniform convergence of 
$h^\varepsilon_{\kappa}$ toward $h_{\kappa}^\star$, this gives \eqref{Eq:LimiteModulation}.
\end{proof}

Hence we obtain \eqref{Eq:Evolh(ii)Star} and \eqref{Eq:Evolh(iii)Star} by passing to the limit in \eqref{Eq:FormeNormale} 
using the assumption that $\gamma_{\kappa} \neq 0$ when $\kappa \in {\mathcal P}_{(iii)}$ (see the last paragraph of Section \ref{sec-small}) for the latter. 
This concludes the proof of Theorem~\ref{Joli_Theoreme}. \eop
%
%
%
%
%
%
%
%
%
\subsection{Proof of Theorem~\ref{VarianteDuJoliTheoreme}}
In this subsection, we briefly sketch the proof of Theorem~\ref{VarianteDuJoliTheoreme}.
Hence we consider the particular case where the data ensures the uniqueness of the solution to the limit system, together with the separation of point vortices, of solids of fixed size and of the vorticity support in the limit.
Since the limit system enjoys uniqueness in this situation, the convergence without restriction to a subsequence is commonplace; let us explain why the  the maximal existence times $T^{\boldsymbol{\varepsilon}}$ satisfy $\liminf_{\overline{\boldsymbol{\varepsilon}} \rightarrow 0} T^{\boldsymbol{\varepsilon}} \geq T^\star$ and the convergences \eqref{Eq:CVu}-\eqref{Eq:CVtheta} hold on any time interval $[0,T] \subset [0,T^\star)$. \par
Consider $T>0$; denoting ${\mathcal S}_{\kappa}^\star(t):={\mathcal S}_{\kappa}(q_{\kappa}^\star(t))$ for $\kappa \in {\mathcal P}_{(i)}$ and ${\mathcal S}_{\kappa}^\star(t):=\{h_{\kappa}^\star(t) \}$ for $\kappa \in {\mathcal P}_{s}$,
 due to the assumption on the limit system, we can find $d_{T}>0$ such that
\begin{multline*}
\forall t \in [0,T], \ \ 
\forall \kappa \in \{1,\dots,N\}, \ \
d({\mathcal S}^\star_{\kappa}(t),\mbox{Supp}(\omega^\star(t))) \geq d_{T}, 
{     \ d({\mathcal S}^\star_{\kappa}(t), \partial \Omega) \geq d_{T}     } \\
\ \text{ and } \
\forall \lambda \in \{1,\dots,N\} \setminus \{ \kappa\}, \ \ d({\mathcal S}^\star_{\kappa}(t),{\mathcal S}^\star_{\lambda}(t)) \geq d_{T}.
\end{multline*}
Reducing $d_{T}$ if necessary, we assume that $d_{T} \leq D$ where $D$ was defined in \eqref{Eq:DefD}.
We now introduce
\begin{multline*}
T_{max} := \sup \Big\{ \tau \in [0,T] \ \Big/ \ 
\exists \varepsilon_{0} >0, \ \forall t \in [0,\tau], \ 
\forall \overline{\boldsymbol{\varepsilon}} < \varepsilon_{0}, \ 
\forall \kappa \in \{1,\dots,N\}, \ \
d({\mathcal S}^\varepsilon_{\kappa}(t),\mbox{Supp}(\omega^\varepsilon(t))) \geq d_{T} /2,
\\
{ d({\mathcal S}^\varepsilon_{\kappa}(t),\partial \Omega) \geq d_{T} /2   } 
\ \text{ and } \
\forall \lambda \in \{1,\dots,N\} \setminus \{ \kappa\}, \ \ d({\mathcal S}^\varepsilon_{\kappa}(t),{\mathcal S}^\varepsilon_{\lambda}(t))\geq d_{T} /2
\Big\}.
\end{multline*}
Due to the analysis of Subsections~\ref{Subsec:FirstStep}--\ref{Subsec:LimitSmall}, we have $T_{max} \geq \overline{T}$ where $\overline{T}$ was defined in \eqref{Eq:DefOverlines}. Moreover, the convergence analysis of Subsections~\ref{Subsec:FirstStep}--\ref{Subsec:LimitSmall} can be carried out in any $[0,\tau] \subset [0,T_{max})$ since we merely use a minimal distance between the solids and between the solids and the vorticity support to obtain the estimates. Hence to conclude, it suffices to prove that $T_{max}=T$. \par
Arguing by contradiction, we suppose that $T_{max}<T$. Using the convergences \eqref{Eq:CVh}, it is easy to see that for $\tau < T_{max}$, for suitably small $\overline{\boldsymbol{\varepsilon}}$, we do have $ d({\mathcal S}^\varepsilon_{\kappa}(t),{\mathcal S}^\varepsilon_{\lambda}(t)) \geq 3d_{T} /4$ and $d({\mathcal S}^\varepsilon_{\kappa}(t),\partial \Omega) \geq 3d_{T} /4$ on $[0,\tau]$ so that the limitation $T_{max}<T$ can only come from the vorticity.
But using the definition of $T_{max}$, \eqref{Eq:CVh}, the decomposition \eqref{Eq:DecompUeps} and the estimates of Section~\ref{Sec:Expansions}, we see that for $\tau < T_{max}$, for suitably small $\overline{\boldsymbol{\varepsilon}}$, one has the uniform log-Lipschitz estimate on the support of $\omega$:
\begin{equation*}
\| u^\varepsilon(t,\cdot)\|_{\LL({\mathcal F} \setminus \bigcup_{\kappa \in {\mathcal P}_{s}} ({\mathcal V}_{d_{T}/4}(h_{\kappa}^\star(t))))} \leq C \ \text{ uniformly for } t \in [0,\tau].
\end{equation*}
Moreover, reasoning as in Lemma~\ref{Lem:MeilleureCVUeps}, we see that for $p \in (1,+\infty)$,
\begin{equation*}
\| \partial_{t} u^\varepsilon(t,\cdot)\|_{L^p( {\mathcal F} \setminus \bigcup_{\kappa \in {\mathcal P}_{s}} ({\mathcal V}_{d_{T}/4}(h_{\kappa}^\star(t))) )} \leq C \ \text{ uniformly for } t \in [0,\tau].
\end{equation*}
This implies that the convergence \eqref{Eq:CVu} can be supplemented by
\begin{equation*}
u^\varepsilon(t,\cdot) \longrightarrow u^\star(t,\cdot) \ \text{ in } \ C^0([0,\tau];C^0({\mathcal F} \setminus \bigcup_{\kappa \in {\mathcal P}_{s}} ({\mathcal V}_{d_{T}/4}(h_{\kappa}^\star(t))))) .
\end{equation*}
This involves the convergence of the corresponding flows on $\mbox{Supp}(\omega_{0})$. In particular, $\mbox{Supp}(\omega^\varepsilon(t))$ converges to $\mbox{Supp}(\omega^\star(t))$ uniformly in time for the Hausdorff distance.
Since the convergence analysis of Subsections~\ref{Subsec:FirstStep}--\ref{Subsec:LimitSmall} is valid on any $[0,\tau] \subset [0,T_{max})$, we deduce that one can find for any such $\tau$ an $\varepsilon_{0}>0$ such that for $\overline{\boldsymbol{\varepsilon}} < \varepsilon_{0}$,  for all $\kappa \in \{1,\dots,N\}$, $d({\mathcal S}^\varepsilon_{\kappa}(t),\mbox{Supp}(\omega^\varepsilon(t))) \geq 3d_{T}/4$ on $[0,\tau]$. This puts $T_{max} <T$ and the boundedness of the velocity of the vorticity support and of the solids in contradiction. This ends the proof of Theorem~\ref{VarianteDuJoliTheoreme}. \eop
%
%
%
%
%
%
%
%
%
%
\ \par
\noindent
{\bf Acknowledgements.} 
The authors warmly thank Christophe Lacave and Alexandre Munnier for fruitful discussions on the subject. 
The authors are partially supported by the Agence Nationale de la Recherche, Project IFSMACS, grant ANR-15-CE40-0010 and Project SINGFLOWS, grant ANR-18-CE40-0027-01.
 F.S. is also partially supported by the Agence Nationale de la Recherche,  Project BORDS, grant ANR-16-CE40-0027-01, the Conseil R\'egionale d'Aquitaine, grant 2015.1047.CP, the Del Duca Foundation, and the H2020-MSCA-ITN-2017 program, Project ConFlex, Grant ETN-765579.


\begin{thebibliography}{99}
\bibitem{LM-old}
Lacave, C., Miot, E.,
\newblock Uniqueness for the vortex-wave system when the vorticity is constant
  near the point vortex.
\newblock {\em SIAM J. Math. Anal.}, 41(3):1138--1163, 2009.
\end{thebibliography}
\end{document}